\documentclass[a4paper,12pt]{amsart}
\usepackage[a4paper,hscale=0.7,vscale=0.75,centering]{geometry}
\usepackage{fullpage}
\usepackage[utf8]{inputenc} 
\usepackage[T1]{fontenc}    
\usepackage{hyperref}       
\usepackage{url}            
\usepackage{booktabs}       
\usepackage{amsfonts}       
\usepackage{nicefrac}       
\usepackage{microtype}      
\usepackage{xcolor}         
\usepackage{amsmath}
\usepackage{mathtools}
\usepackage{amssymb}
\usepackage{amsthm}
\usepackage{bbm}
\usepackage[ruled]{algorithm2e}
\usepackage{graphicx}
\usepackage{subcaption}
\usepackage{enumitem}

\usepackage{xpatch}
\xpatchcmd{\paragraph}{\normalfont}{{\normalfont\bfseries}}{}{}

\makeatletter

\renewcommand{\@tocline}[7]{%
  \ifnum #1>\c@tocdepth\else
    \par\addpenalty\@secpenalty\addvspace{#2}%
    \begingroup
      \hyphenpenalty\@M
      \@ifempty{#4}{%
        \@tempdima\csname r@tocindent\number#1\endcsname\relax
      }{%
        \@tempdima#4\relax
      }%
      \parindent\z@
      \leftskip#3\relax
      \advance\leftskip\@tempdima\relax
      \rightskip\@pnumwidth plus4em
      \parfillskip-\@pnumwidth
      #5\leavevmode
      \hskip-\@tempdima
      #6\nobreak
      \leaders\hbox{\kern .35em.\kern .35em}\hfill
      \nobreak
      \hbox to\@pnumwidth{\@tocpagenum{#7}}\par
    \endgroup
  \fi
}

\renewcommand{\l@section}{\@tocline{1}{6pt}{0pt}{2.5em}{\bfseries}}
\renewcommand{\l@subsection}{\@tocline{2}{2pt}{1.5em}{3.2em}{\normalfont\small}}

\makeatother

\DeclareMathOperator*{\argmin}{argmin}

\newtheorem{theorem}{Theorem}[section]
\newtheorem{proposition}[theorem]{Proposition}
\newtheorem{lemma}[theorem]{Lemma}
\newtheorem{corollary}[theorem]{Corollary}
\newtheorem{example}[theorem]{Example}
\newtheorem{definition}[theorem]{Definition}
\newtheorem{assumption}[theorem]{Assumption}
\newtheorem{remark}[theorem]{Remark}

\begin{document}

\title{Wasserstein Saddle-Free Newton: Saddle Escape and Local Spectral Convergence}
\author{Razvan-Andrei Lascu}
\address{Center for Advanced Intelligence Project, RIKEN, Tokyo, Japan}
\email{razvan-andrei.lascu@riken.jp}

\author{Taiji Suzuki}
\address{Department of Mathematical Informatics, The University of Tokyo, and Center for Advanced Intelligence Project, RIKEN, Tokyo, Japan}
\email{taiji@mist.i.u-tokyo.ac.jp}

\begin{abstract}
Extending saddle-free Newton methods to Wasserstein space requires controlling Hessian dependent operators along nonlinear transport trajectories and accounting for a Hessian spectrum that may accumulate at zero. We address these challenges through Wasserstein Saddle-Free Newton (WSFN), a second-order method for nonconvex optimization over Wasserstein space that preconditions the Wasserstein gradient by a regularized inverse squared Hessian operator. The resulting dynamics exploit curvature magnitude without reversing the repulsive effect of negative curvature, while retaining attraction along positive curvature directions, thereby avoiding a fundamental limitation of standard Wasserstein Newton dynamics. Whereas previous Wasserstein saddle escape results rely on perturbed first-order dynamics, WSFN uses curvature both to modify the deterministic transport and to construct the perturbation mechanism. Our analysis combines stability estimates for Hessian operators with a perturbative saddle escape argument to establish polynomial time convergence to approximate second-order stationarity under regularity of landscape assumptions. The contribution of the saddle curvature parameter $\delta$ improves from $\tilde O(\delta^{-4})$ for prior perturbed first-order Wasserstein methods to $\tilde O(\delta^{-3})$. For benign landscapes, this further implies proximity to the set of global minimizers. The cubic dependence corresponds to the scale of the objective decrease that can be guaranteed from negative curvature under Lipschitz Hessian regularity. Near a minimizer, the possible absence of a uniform spectral gap prevents a uniform contraction result and instead calls for a spectral characterization. We prove geometric contraction of the linearized dynamics on positive curvature spectral bands bounded away from zero and polynomial decay under spectral source conditions. Together, these results connect quantitative saddle escape with the local spectral behavior of second-order Wasserstein dynamics.
\end{abstract}

\maketitle
\section{Introduction}
\label{section:Introduction}
We study the minimization of a non-convex functional $F:\mathcal{P}_2(\mathbb R^d)\to\mathbb R$ over Wasserstein space, namely
\begin{equation*}
    \min_{\mu \in \mathcal{P}_2(\mathbb R^d)} F(\mu),
\end{equation*}
where $\mathcal{P}_2(\mathbb R^d)$ denotes the probability measures on $\mathbb R^d$ with finite second moment, equipped with the $L^2$-Wasserstein distance $W_2$. Throughout, we assume that $\inf_{\mu\in\mathcal{P}_2(\mathbb R^d)}F(\mu)>-\infty$. 

\subsection{Nonconvex optimization on Wasserstein space}
Optimization over probability measures arises in variational inference \cite{yao2023meanfieldvariationalinferencewasserstein,lambert2022variational,yifei-projected,Blei03042017,pmlr-v75-wibisono18a}, trajectory inference in path space \cite{chizat2022trajectory,chizat2026convergencedriftdiffusion}, generative modeling \cite{arbel,pmlr-v97-chu19a,huang2024generativemodelingminimizingwasserstein2}, and mean-field neural network training \cite{10.1214/20-AIHP1140,Nitanda2022ConvexAO,chizat2022meanfield,Chizat2018OnTG,Mei2018AMF,SIRIGNANO20201820,rotskoff,chizat2026quantitativeconvergencewassersteingradient}. In the context of these applications, much of the convergence theory for Wasserstein gradient flows and their discretizations exploits linear or geodesic convexity. For nonconvex objectives, however, $\varepsilon$-small gradient norm cannot distinguish minimizers from saddles or maxima. Some learning and discrepancy objectives have benign landscapes \cite{lu2020meanfield,kim2024transformers,boufadene}, in which every local minimizer is global. This motivates methods combining descent with saddle escape, although benignity alone provides neither an escape time bound nor a local convergence rate. Our global analysis first addresses saddle point escape to achieve approximate second-order stationarity. An additional benignity landscape condition then gives proximity to global minimizers. Motivating examples, convexity background, and related optimization methods are discussed in Appendices~\ref{section:examples-benign}, \ref{section:app_conv_in_W_space}, and~\ref{section:wasserstein-proximal-taylor}, respectively.

\subsection{From first-order saddle escape to second-order transport}
\cite{yamamoto2025hessianguided} study perturbed Wasserstein gradient flow (PWGF), using Hessian-guided Gaussian perturbations to escape saddle regions. Under suitable regularity and benignity assumptions, their method reaches a neighborhood of the global minimizer set in polynomial time. The Hessian shapes the perturbations, while the deterministic transport follows the unpreconditioned Wasserstein gradient. This leaves open the possibility of improving escape by using curvature information as preconditioner throughout the transport trajectory.

The choice of the preconditioner is crucial. Wasserstein gradient dynamics expand slowly along directions of weak negative curvature. Where Newton's inverse Hessian is well defined, its linearization instead cancels the curvature sign, so a saddle can become attractive. Table~\ref{tab:critical-point-multipliers} compares these effects. Therefore, a suitable second-order method should rescale curvature while preserving the instability of negative curvature directions. Moreover, because the eigenvalues of an infinite-dimensional Hessian can accumulate at zero, there is no clean spectral gap to guarantee a single, uniform convergence speed. Together, these considerations lead to the following question:

\begin{center}
\textit{Can curvature dependent Wasserstein transport improve saddle escape while admitting local spectral guarantees in the absence of a uniform positive spectral gap near minimizers?}
\end{center}

To address this question, we introduce Wasserstein Saddle-Free Newton (WSFN), with transport direction $-(\operatorname{H}_\mu^2+\beta \operatorname{I_{d \times d}})^{-1/2}\nabla_\mu F(\mu,\cdot)$, where $\operatorname{H}_\mu$ is the Wasserstein Hessian, $\nabla_\mu F$ is the Wasserstein gradient, and $\beta>0$ is the regularization, which makes this direction well defined for bounded self-adjoint Hessians without requiring their invertibility. For suitable stepsizes, the linearization of this transport direction retains attraction along positive curvature directions and repulsion along negative ones. In a special regime of $\beta$, it also enables curvature normalization.

WSFN combines this transport with Gaussian perturbations of covariance $\operatorname{K}_\mu^2$, where $\operatorname{H}_\mu=\operatorname{M}_\mu+\operatorname{K}_\mu$ and $\operatorname{K}_\mu$ is the integral kernel operator component. The covariance construction follows PWGF. Our contribution is its coupling to full Hessian preconditioning and the escape analysis of the resulting nonlinear transport. We also characterize the local linearized dynamics spectrally, distinguishing the geometric contraction on strictly positive curvature bands from the polynomial contraction on bands with arbitrarily small positive curvature.

\subsection{Technical challenges}
The first challenge is that saddle escape is governed by a nonlinear, state dependent transformation of the Wasserstein Hessian. Since Hessians at different iterates act on different $L_\mu^2$ spaces, they must first be represented on a common reference space. Moreover, the spectral maps defining the WSFN preconditioner (see Subsection~\ref{sec:regularized_wsfn_update}) become increasingly sensitive as $\beta\to0$. Lemma~\ref{lem:regularized-wsfn-spectral-maps} quantifies this sensitivity, and, together with pulled-back Hessian estimates on the reference space, gives a first-order expansion of the preconditioned gradient (Lemma~\ref{lem:preconditioned-gradient-expansion-sharp}), which is crucial for the saddle escape analysis.

Perturbative instability of a fixed linearization at a saddle does not guarantee finite-time escape of the nonlinear WSFN trajectory. The key step is Lemma~\ref{lem:local-wsfn-descent}. Unless sufficient objective decrease has already occurred, it controls the displacement of the transport maps emerging from the saddle, thereby limiting the accumulated nonlinear error. Proposition~\ref{discrete_lower_bound_metric} then turns the resulting spectral instability into a deterministic escape guarantee, while Proposition~\ref{discrete_decrease_around_saddle} combines it with a Gaussian perturbation estimate (Lemma~\ref{lem:independence}) to obtain a high probability net decrease. By setting $\beta=\Theta(\delta^2)$, curvature normalization reduces the number of escape episodes to $\tilde O(1)$ -- down from $\tilde O(\delta^{-1})$ for PWGF -- while guaranteeing a $\tilde\Omega(\delta^3)$ decrease per episode. This gives the final $\tilde O(\delta^{-3})$ saddle contribution in Theorem~\ref{thm:polynomial-time-sosp}.

Another challenge arises near minimizers because the Wasserstein Hessian need not have a uniform positive spectral gap. Positive spectral values may accumulate at zero, so contraction of every positive curvature mode does not imply contraction in operator norm. Theorem~\ref{thm:local-spectral-behavior-wsfn} makes this obstruction explicit. Although the linearized iterates converge strongly to the kernel projection, their operator norm remains equal to one. Geometric contraction is recovered only on positive spectral bands away from zero, while source conditions give polynomial decay for arbitrarily small curvature modes. Thus, standard finite-dimensional local convergence arguments cannot be applied directly. Moreover, these are linearized rates, and, as emphasized in Remark~\ref{rem:limitations-nonlinear-spectral-rates}, transferring them to the nonlinear WSFN iterates requires additional strong structural control of the quadratic remainder.

\subsection{Contributions}
Our contributions are as follows:
\begin{itemize}
    \item \textbf{A saddle-free second-order method on Wasserstein
    space.} We formulate WSFN, combining regularized full Hessian preconditioning with Hessian-guided perturbations, and prove the operator stability estimates needed to analyze its transport direction and linearization around critical points (Section~\ref{section:Wasserstein Saddle-Free Newton} and Appendix~\ref{appendix:aux-results}).

    \item \textbf{Improved saddle escape contribution.} Under the regularity and admissibility conditions of Theorem~\ref{thm:polynomial-time-sosp}, WSFN reaches approximate second-order stationarity in polynomial time with high probability. With $\beta=\Theta(\delta^2)$, the saddle contribution improves from $\tilde O(\delta^{-4})$ for PWGF to $\tilde O(\delta^{-3})$. Under additional benignity on the initial sublevel set of $F$, the second-order stationarity guarantee implies proximity to global minimizers (Corollary~\ref{cor:global-proximity-under-benignity}) with the same hitting time bound. Under the standard regime $\varepsilon \lesssim \delta^2$ common in saddle escape analyses, the overall hitting time remains dominated by $\varepsilon^{-2}$ (Remark~\ref{rem:improved-saddle-dependence}).

    \item \textbf{Local spectral guarantees.}
    Near a global minimizer, we prove geometric contraction on positive spectral bands bounded away from zero, and polynomial decay under spectral source conditions on bands with arbitrarily small positive curvature for the linearized WSFN dynamics. (Theorem~\ref{thm:local-spectral-behavior-wsfn}). Nonlinear rate transfer requires additional strong assumptions (Remark~\ref{rem:limitations-nonlinear-spectral-rates}).
\end{itemize}

\section{Wasserstein geometry for second-order optimization}
\label{section:wass-geom-for-second-order}
We introduce the Wasserstein gradient and Hessian operator used in the analysis of WSFN. Background on operator theory, optimal transport, and Wasserstein geometry is collected in Appendices~\ref{section:additional notation}, \ref{section:background on OT}, and \ref{sec:wass-geom}, respectively.

\subsection{Notation and Wasserstein gradient}
For $\mu\in\mathcal P_2(\mathbb R^d)$, let $L_\mu^2:=L^2(\mu;\mathbb R^d)$, with norm $\|\cdot\|_{L_\mu^2}$ and inner product $\langle\cdot,\cdot\rangle_{L_\mu^2}$. We write $\operatorname{Id}$ for the identity map, $\operatorname{T}_\#\mu$ for the pushforward of $\mu$ by $\operatorname{T}$, and $\|\cdot\|_{\mu,\infty}$ for the $\mu$-essential supremum norm. Let $\Pi(\mu,\nu)$ and $\Pi_o(\mu,\nu)$ denote, respectively, the couplings and optimal couplings of $\mu$ and $\nu$, so that $W_2^2(\mu,\nu) = \inf_{\gamma\in\Pi(\mu,\nu)}\int_{\mathbb R^d\times\mathbb R^d}\|x-y\|^2\mathrm d\gamma(x,y).$ We write $B_2(r,\hat\mu):=\{\mu:W_2(\mu,\hat\mu)<r\}$. The notation $\tilde O$, $\tilde\Theta$, and $\tilde\Omega$ suppresses logarithmic factors.

We use the following definition of Wasserstein differentiability \cite[Chapter~10]{ambrosio2008gradient}. 

\begin{definition}[Wasserstein differentiability]
\label{def:wass-differentiability}
A functional $F$ is Wasserstein differentiable at
$\mu\in\mathcal P_2(\mathbb R^d)$ if there exists
$\nabla_\mu F(\mu,\cdot)\in L_\mu^2$ such that,
for every $\gamma\in\Pi_o(\mu,\nu)$,
\begin{equation*}
\label{eq:wass_diff}
    F(\nu) = F(\mu) + \int_{\mathbb R^d\times\mathbb R^d}\langle\nabla_\mu F(\mu,x),y-x\rangle\mathrm d\gamma(x,y) + o\left(W_2(\mu,\nu)\right), \text{ as } \nu\to\mu \text{ in } W_2.
\end{equation*}
\end{definition}
Throughout, $\nabla_\mu F$ denotes the unique representative belonging to the tangent space $\mathcal T_\mu\mathcal P_2(\mathbb R^d)$ \cite[Proposition~2.11]{lanzetti}; see Appendix~\ref{subsec:wass-diff}. Under the additional assumptions in Appendix~\ref{subsection:First and second variation}, it satisfies $\nabla_\mu F=\nabla\frac{\delta F}{\delta\mu}$, where $\frac{\delta F}{\delta\mu}$ denotes the first variation of $F.$

\begin{definition}[Wasserstein regularity]
\label{def:regularity-F}
We say that $F$ is Wasserstein regular if it is Wasserstein differentiable at every $\mu\in\mathcal P_2(\mathbb R^d)$, and its gradient is Wasserstein differentiable in the measure variable and differentiable in the spatial variable, with derivatives $\nabla_\mu^2F$ and $\nabla\nabla_\mu F$, respectively. For every $\mu$, these derivatives satisfy $\nabla_\mu^2F(\mu,x,y) =\nabla_\mu^2F(\mu,y,x)^\top,$ $\mu\otimes\mu\text{-a.e.},$ $\nabla\nabla_\mu F(\mu,x)=\nabla\nabla_\mu F(\mu,x)^\top,$ $\mu\text{-a.e.},$ together with $\nabla_\mu^2F(\mu,\cdot,\cdot)\in L^2_{\mu\otimes\mu}$ and $\|\nabla\nabla_\mu F(\mu,\cdot)\|_{\mu,\infty}<\infty$.
\end{definition}

\subsection{Wasserstein Hessian along transport curves}
Our analysis uses transport curves $\mu_t=(\pi_t)_\#\mu$, where $\pi_t=\operatorname{Id}+tv$ and $v\in L_\mu^2$, which need not be Wasserstein geodesics. For a fixed transport direction $v$, define the operators on the space $L_\mu^2$, with $w\in L_\mu^2,$ by $\widetilde{\operatorname M}_{\mu,t}[w](x):= \nabla\nabla_\mu F(\mu_t,\pi_t(x))w(x),$ and $\widetilde{\operatorname K}_{\mu,t}[w](x):=\langle\nabla_\mu^2F(\mu_t,\pi_t(x),\pi_t(\cdot)),w\rangle_{L_{\mu}^2}.$ Their dependence on the fixed direction $v$ is suppressed in the notation. Set $\widetilde{\operatorname H}_{\mu,t} :=\widetilde{\operatorname M}_{\mu,t}+ \widetilde{\operatorname K}_{\mu,t},$ and $\operatorname H_\mu:=\widetilde{\operatorname H}_{\mu,0} = \operatorname M_\mu+\operatorname K_\mu,$ where $\operatorname M_\mu:=\widetilde{\operatorname M}_{\mu,0}$ and $\operatorname K_\mu:=\widetilde{\operatorname K}_{\mu,0}.$
\begin{proposition}[Hessian along transport curves]
\label{proposition:Wasserstein-hessian-appendix}
Assume that $F$ is Wasserstein regular. For $\mu\in\mathcal P_2(\mathbb R^d)$,
$v\in L_\mu^2$, and $\mu_t=(\operatorname{Id}+tv)_\#\mu$, we have
\begin{equation*}
\frac{\mathrm d^2}{\mathrm dt^2}F(\mu_t) = \langle\widetilde{\operatorname H}_{\mu,t}v,v\rangle_{L_\mu^2}, \quad t\in[0,1],    
\end{equation*}
where $\widetilde{\operatorname H}_{\mu,t}$ is a bounded self-adjoint operator on $L_\mu^2$.
\end{proposition}
The operator is evaluated along the evolving measure $\mu_t$ but acts on the fixed space $L_\mu^2$. For absolutely continuous $\mu$ and an optimal transport displacement $v$, the identity at $t=0$ recovers the usual geodesic Hessian definition in \cite[Chapter~8]{alma9912197933502466} (see Definition~\ref{def:wass_hessian}). The proof is given in Appendix~\ref{subsection: Hessian identities and second-order expansions}; see also \cite{bonet2024mirror,wang2020informationnewtonsflowsecondorder} for related formulas.

The regularity assumptions imply that $\operatorname M_\mu$ and $\operatorname K_\mu$ are bounded and self-adjoint, with $\|\operatorname M_\mu\|_{\mathrm{op}} = \|\nabla\nabla_\mu F(\mu,\cdot)\|_{\mu,\infty},$  and $\|\operatorname K_\mu\|_{\mathrm{op}} \leq \|\operatorname K_\mu\|_{\mathrm{HS}} = \|\nabla_\mu^2F(\mu,\cdot,\cdot)\|_{L^2_{\mu\otimes\mu}};$ see \cite[Theorem~2.1.3]{arveson2001short} and \cite[Theorem~VI.23]{ReedSimon1980}. In particular, $\operatorname K_\mu$ is Hilbert--Schmidt and therefore compact, so its nonzero spectral values are eigenvalues of finite multiplicity, with zero the only possible accumulation point \cite[Theorems~VI.15 and VI.22]{ReedSimon1980}. Thus, we may write $\underline{\lambda}(\operatorname K_\mu) := \inf\sigma(\operatorname K_\mu) =\inf_{\|w\|_{L_\mu^2}=1} \langle\operatorname K_\mu w,w\rangle_{L_\mu^2}.$ Expressions for the Hessian components in terms of the first and second variations of $F$ are given in Proposition~\ref{prop:wass-hessian-variations}.

\section{Wasserstein Saddle-Free Newton}
\label{section:Wasserstein Saddle-Free Newton}
A second-order method for non-convex optimization should improve curvature dynamics without destroying the repulsion needed for saddle escape. To explain this distinction, let $\mu^*$ be a critical point of $F$, and assume that $\operatorname H_{\mu^*}$ is a bounded self-adjoint operator on $L^2_{\mu^*}$. 

\subsection{Curvature normalization}
At a critical point $\mu^*$, the linearized transport acts on Hessian eigendirections through the spectral multipliers in Table~\ref{tab:critical-point-multipliers}. Their detailed derivation is given in Appendix~\ref{app:critical-point-linearizations}.

Gradient descent expands weakly along directions of small negative curvature. Newton preconditioning removes the curvature dependence, but its multiplier $1-\tau$ also makes negative curvature directions contract for $\tau\in (0,2)$. Ideal saddle-free Newton replaces the inverse Hessian by its inverse absolute value, preserving negative curvature repulsion while normalizing its magnitude. The absolute value Hessian principle is inspired by finite-dimensional saddle-free Newton methods \cite{NIPS2014_dauphin}. In Wasserstein space, however, $|\operatorname H_\mu|$ may lack a bounded inverse, and the absolute value map is not generally operator Lipschitz; see \cite[VI: Bounded Operators, Problem~17]{ReedSimon1980}. WSFN addresses these difficulties through regularization. The underlying variational construction is given in Appendix~\ref{sec:well-posedness-charac-optimizer}.

\begin{table}[h]
    \centering
    \begin{tabular}{lcc}
        \hline
        Method
        & Preconditioner at $\mu^*$
        & Multiplier $m_\tau(\lambda)$
        \\[2pt]
        \hline
        Gradient descent
        & $\operatorname{I_{d \times d}}$
        & $1-\tau\lambda$
        \\
        Newton
        & $\operatorname H_{\mu^*}^{-1}$
        & $1-\tau$
        \\
        Ideal saddle-free Newton
        & $|\operatorname H_{\mu^*}|^{-1}$
        & $1-\tau\lambda|\lambda|^{-1}$
        \\
        WSFN
        & $(\operatorname H_{\mu^*}^2+\beta \operatorname{I_{d \times d}})^{-1/2}$
        & $1-\tau\lambda(\lambda^2+\beta)^{-1/2}$\\
        \hline
    \end{tabular}
    \caption{Spectral multipliers at a critical point. Newton and ideal saddle-free Newton require boundedly invertible Hessians, whereas WSFN is well defined for $\beta>0$.}
    \label{tab:critical-point-multipliers}
\end{table}

\subsection{The regularized WSFN update}
\label{sec:regularized_wsfn_update}
For $\beta>0$ and a bounded self-adjoint operator $A$ defined on a Hilbert space, define $P_\beta(A):=(A^2+\beta \operatorname{I_{d \times d}})^{-1/2}$ and $S_\beta(A):=AP_\beta(A)$. The WSFN update is
\begin{equation}
\label{eq:wsfn-update-main}
    \mu^{n+1} = \left(\operatorname{Id}-\tau P_\beta(\operatorname H_{\mu^n})\nabla_\mu F(\mu^n,\cdot)\right)_\#\mu^n .
\end{equation}
Under Assumption~\ref{assumption:smooth-F}, these iterates are well defined. Proposition~\ref{prop:well-posedness-saddle-free} establishes their variational characterization. Regularization also controls the variation of the transport direction. For bounded self-adjoint operators $A,B$ on a common Hilbert space, Lemma~\ref{lem:regularized-wsfn-spectral-maps} gives $\|P_\beta(A)\|_{\mathrm{op}}\leq\beta^{-1/2}$, $\|S_\beta(A)\|_{\mathrm{op}}\leq1$, and
\begin{equation*}
    \|P_\beta(A)-P_\beta(B)\|_{\mathrm{op}} \leq \frac{2}{\pi\beta}\|A-B\|_{\mathrm{op}}, \quad
    \|S_\beta(A)-S_\beta(B)\|_{\mathrm{op}} \leq \beta^{-1/2}\|A-B\|_{\mathrm{op}}.
\end{equation*}
After pulling the Hessians back to a common reference space, these estimates give the first-order expansion of the preconditioned gradient of WSFN in Lemma~\ref{lem:preconditioned-gradient-expansion-sharp}. Their dependence on $\beta$ is essential for the saddle escape analysis. 

For $0<\tau\leq1$, the WSFN multiplier in Table~\ref{tab:critical-point-multipliers} lies in $(0,1)$ for $\lambda>0$ and exceeds one for $\lambda<0$. If $\lambda<-\delta$ and $\beta=b\delta^2$ with fixed $b>0$, then $m_\tau(\lambda)>1+\tau(1+b)^{-1/2}$. Thus, a fixed stepsize gives an expansion factor bounded away from one independently of $\delta$. Turning this linearized instability into nonlinear escape requires the localization argument of Proposition~\ref{discrete_lower_bound_metric}. At zero curvature the multiplier equals one. In this case, local spectral guarantees are stated in Theorem~\ref{thm:local-spectral-behavior-wsfn}.

\subsection{Perturbation episodes and stopping}
Algorithm~\ref{alg:episode-wsfn} combines WSFN steps with perturbation episodes triggered by a small gradient. The perturbation covariance is $\operatorname K_\mu^2$, following \cite{yamamoto2025hessianguided}, while the deterministic transport uses the full Hessian $\operatorname H_\mu=\operatorname M_\mu+\operatorname K_\mu$.
\\

\paragraph{Large gradients.} When $\|\nabla_\mu F(\mu,\cdot)\|_{L_\mu^2}>\varepsilon$, perform \eqref{eq:wsfn-update-main} with parameters $\beta_{\mathrm{grad}}>0$ and $0<\tau_{\mathrm{grad}}\leq
C_{\operatorname H}^{-1}\sqrt{\beta_{\mathrm{grad}}}$. Lemma~\ref{descent lemma} guarantees objective decrease in this phase.
\\

\paragraph{Perturbation episodes.} When $\|\nabla_\mu F(\mu,\cdot)\|_{L_\mu^2}\leq\varepsilon$, store $\bar\mu:=\mu$ and initialize $\nu^0=(\operatorname{Id}+\eta\xi)_\#\bar\mu,$ with $\xi\sim \mathcal N_{L_{\bar\mu}^2} \left(0,\operatorname K_{\bar\mu}^2\right),$ using a fresh Gaussian perturbation conditionally on the past. Starting from $\nu^0$, perform at most $n_{\mathrm{out}}$ updates of the form \eqref{eq:wsfn-update-main}, with $\beta_{\mathrm{sad}}=\Theta(\delta^2)$ and $\tau_{\mathrm{sad}}\in(0,1/2]$ fixed independently of $\delta$, subject to the parameter conditions of Theorem~\ref{thm:polynomial-time-sosp}. Keeping $\tau_{\mathrm{sad}}$ independent of $\delta$ is essential to the episode length improvement. Applying the global descent Lemma \ref{descent lemma} with $\beta_{\mathrm{sad}}=\Theta(\delta^2)$ would instead force $\tau_{\mathrm{sad}}=O(\delta)$, losing the uniform $n_{\mathrm{out}}=\tilde O(1)$ episode length guarantee. For more details, see also Remark \ref{rem:improved-saddle-dependence}.
\\

\paragraph{Acceptance and termination.}
Immediately after the perturbation and after each subsequent update, test whether $F(\bar\mu)-F(\nu^\ell)\geq F_0$. At the first successful test, accept $\mu:=\nu^\ell$ and return to the gradient test. If no $\ell\in\{0,\ldots,n_{\mathrm{out}}\}$ satisfies the criterion, terminate and return the stored measure $\widehat\mu:=\bar\mu$.

The local descent estimate of Lemma~\ref{lem:local-wsfn-descent} and the localization argument of Proposition~\ref{discrete_lower_bound_metric} justify the $\delta$ independent saddle stepsize $\tau_\mathrm{sad}$. The returned measure has small gradient by construction, and its second-order guarantee follows with high probability from Theorem~\ref{thm:polynomial-time-sosp}. Section~\ref{sec:coulomb-experiment} gives a numerical illustration. Appendix~\ref{section:practical implementation} describes the particle implementation and additional experiments.
\\

\paragraph{Why this covariance is well defined.}
An $L_\mu^2$-valued Gaussian requires a trace class covariance. Under our assumptions, $\operatorname{K}_\mu$ is Hilbert--Schmidt, and hence $\mathbb E\|\xi\|_{L_\mu^2}^2 =\operatorname{Tr}(\operatorname{K}_\mu^2) =\|\operatorname{K}_\mu\|_{\mathrm{HS}}^2<\infty.$ The identity operator is not trace class when $L_\mu^2$ is infinite dimensional, and the full Hessian square $\operatorname{H}_\mu^2$ need not be trace class because of the multiplication component $\operatorname{M}_\mu$. Thus, $\operatorname{K}_\mu^2$ supplies an admissible covariance with finite expected squared perturbation norm, while the full Hessian remains available for deterministic preconditioning. For additional details, see Appendix~\ref{L2 valued GPs}.

\section{Theoretical guarantees}
\label{section:theoretical guarantees}
In this section, we state our main results and start with the assumptions under which WSFN enjoys global saddle escape and local spectral convergence guarantees.

\subsection{Second-order optimality and landscape assumptions}
At a critical point $\mu^* \in \mathcal{P}_2^{\mathrm{ac}}(\mathbb R^d)$, we have $\operatorname{M}_{\mu^*}=0$, so that $\operatorname H_{\mu^*}=\operatorname K_{\mu^*}$. Thus, $\operatorname K_{\mu^*}\succeq0$ is a necessary condition for local minimality; see Proposition~\ref{prop:necessary-second-order-optimality-minimizers} and \cite{yamamoto2025hessianguided}. This motivates the following approximate notions.
\begin{definition}[$(\varepsilon,\delta)$-second-order critical point and saddle point]
Let $\varepsilon,\delta>0$. A measure $\mu\in\mathcal P_2(\mathbb R^d)$ satisfying $\|\nabla_\mu F(\mu,\cdot)\|_{L_\mu^2}\leq\varepsilon$ is an $(\varepsilon,\delta)$-second-order critical point if $\underline{\lambda}(\operatorname K_\mu)\geq-\delta$, and an $(\varepsilon,\delta)$-saddle point if $\underline{\lambda}(\operatorname K_\mu)<-\delta$.
\end{definition}
Small gradient norm alone does not control the multiplication operator of the Hessian. The escape analysis requires this control at the measures initiating perturbation episodes. Thus, we impose a local version of the gradient and Hessian relation in \cite[Assumption~2]{yamamoto2025hessianguided}.
\begin{assumption}[Uniform gradient to Hessian control on a set]
\label{assumption_regularity_wg}
Let $S\subseteq\mathcal P_2(\mathbb R^d)$. There exist $\varepsilon_0, R_F>0$ such that, for every $\mu\in S$ satisfying $\|\nabla_\mu F(\mu,\cdot)\|_{L^2_\mu}\le\varepsilon_0,$ we have $\|\operatorname{M}_\mu\|_{\mathrm{op}}\le R_F\|\nabla_\mu F(\mu,\cdot)\|_{L^2_\mu}.$
\end{assumption}
For the global convergence analysis, it is sufficient to impose this assumption on the initial sublevel set $\mathcal L_0:=\{\mu\in\mathcal P_2(\mathbb R^d):F(\mu)\le F(\mu^0)\},$ and choose the threshold $0<\varepsilon\le\varepsilon_0$. The constants $R_F$ and $\varepsilon_0$ are fixed independently of the tolerances $\varepsilon,\delta$ and the episode index. In particular, for every $\mu\in S$ with $\|\nabla_\mu F(\mu,\cdot)\|_{L^2_\mu}\le\varepsilon \le\varepsilon_0$, we have $\|\operatorname{H}_\mu-\operatorname{K}_\mu\|_{\mathrm{op}}= \|\operatorname{M}_\mu\|_{\mathrm{op}}\le R_F\varepsilon.$ Hence an $(\varepsilon,\delta)$-second-order critical point in $S$ also satisfies $\inf\sigma(\operatorname{H}_\mu) \ge -\delta-R_F\varepsilon.$ Thus, the stationarity criterion also controls the full Wasserstein Hessian. Under the condition $\varepsilon=\tilde O(\delta^2)$, the additional curvature tolerance $R_F\varepsilon$ is lower order than $\delta$, with $R_F$ fixed.

To connect approximate stationarity with global optimality, write $\mathcal M^*:=\operatorname{argmin}_{\mu\in\mathcal P_2(\mathbb R^d)} F(\mu)$ and, for $\alpha\geq0$, define $\mathcal N_\alpha(\mathcal M^*):= \left\{\mu\in\mathcal P_2(\mathbb R^d): \operatorname{dist}_{W_2}(\mu,\mathcal M^*)\leq\alpha \right\}.$ The following quantitative benignity condition is analogous to the strict saddle conditions used in Euclidean saddle escape analyses \cite{jin17a,Ge15}. It roughly says that every measure not close to the set of global minimizers either has a large gradient or is localized near a saddle.

\begin{assumption}[Benignity (saddle property) on a set]
\label{def:quantitative-strict-saddle}
Let $\mathcal S\subset\mathcal P_2(\mathbb R^d)$. Assume that $F$ satisfies the $(\varepsilon,\delta,\alpha_{\mathrm{ben}})$-saddle property on $\mathcal S,$ namely, $\mathcal M^*\neq\varnothing$, and every $\mu\in\mathcal S\setminus \mathcal N_{\alpha_{\mathrm{ben}}}(\mathcal M^*)$
satisfies $\|\nabla_\mu F(\mu,\cdot)\|_{L_\mu^2}>\varepsilon$ or $\underline{\lambda}(\operatorname K_\mu)<-\delta.$
\end{assumption}
This places every $(\varepsilon,\delta)$-second-order critical point in $\mathcal S$ inside $\mathcal N_{\alpha_{\mathrm{ben}}}(\mathcal M^*)$ (Lemma~\ref{second-order stationary close to local}). Unlike the global formulation in \cite{yamamoto2025hessianguided}, it suffices to take $\mathcal S=\mathcal L_0,$ because the outer iterates of Algorithm \ref{alg:episode-wsfn}, i.e., at which the gradient $\nabla_\mu F$ is evaluated, remain in this sublevel set (Lemma~\ref{lem:outer-iterates-sublevel}). Examples motivating this property are discussed in Appendix~\ref{section:examples-benign}.

Finally, the analysis requires uniform boundedness and coupling Lipschitz continuity of the Hessian.
\begin{assumption}[Uniform boundedness of the Wasserstein Hessian]
\label{assumption:smooth-F}
Assume that $F$ is Wasserstein regular and that
there exist $C_{\operatorname M},C_{\operatorname K}>0$
such that, for every $\mu\in\mathcal P_2(\mathbb R^d)$,
\[
    \|\operatorname M_\mu\|_{\mathrm{op}} \leq C_{\operatorname M},
    \quad
    \|\operatorname K_\mu\|_{\mathrm{HS}} \leq C_{\operatorname K}.
\]
\end{assumption}

\begin{assumption}[Coupling Lipschitzness of the Wasserstein Hessian]
\label{assumption:Hessian-lipschitz}
Assume that $F$ is Wasserstein regular and that there exist $L_{\operatorname M},L_{\operatorname K}>0$ such that, for every $\mu,\nu\in\mathcal P_2(\mathbb R^d)$ and $\gamma\in\Pi(\mu,\nu)$,
\begin{align*}
    \left\|\nabla\nabla_\mu F(\nu,\cdot)-\nabla\nabla_\mu F(\mu,\cdot)\right\|_{\gamma,\infty} \leq L_{\operatorname M}d_\gamma,\
    \left\|\nabla_\mu^2F(\nu,\cdot,\cdot)-\nabla_\mu^2F(\mu,\cdot,\cdot)
    \right\|_{L^2_{\gamma\otimes\gamma}}\leq L_{\operatorname K}d_\gamma,
\end{align*}
where $d_\gamma^2:= \int_{\mathbb R^d\times\mathbb R^d} \|y-x\|^2\mathrm d\gamma(x,y)$.
\end{assumption}
These assumptions give uniform boundedness and Lipschitz control of the Hessian on a common space, with constants $C_{\operatorname H}:=C_{\operatorname M}+C_{\operatorname K}$ and $L_{\operatorname H}:=L_{\operatorname M}+L_{\operatorname K}$, respectively; see Lemmas~\ref{lem:uniform_bound_tildeH}, \ref{lem:lip-curves}, \ref{lem:two-map-pulled-back-hessian-lipschitz}, and \ref{lemma:smoothness-of-F-2}. Analogous Hessian Lipschitz continuity assumptions have been employed in Euclidean space \cite{jin17a,Ge15}. 

Assumption~\ref{assumption:Hessian-lipschitz} imposes coupling regularity, as in \cite[Assumption~3]{yamamoto2025hessianguided}. Its multiplication operator bound is stronger than ordinary spatial Lipschitz continuity of the Hessian. This condition nevertheless provides the operator norm control needed to compare the pulled-back Hessians on a common reference space and their regularized preconditioners along transport trajectories. The resulting regularity class includes nonconvex objectives. Propositions~\ref{prop:rotational-resnet-regularity}, \ref{prop:cubic-depth-example} give explicit examples.

\subsection{Approximate stationarity and proximity to global minimizers}

We first establish approximate second-order stationarity without assuming benignity, and then use benignity to obtain proximity to the global minimizer set.
\begin{theorem}[Polynomial time approximate second-order stationarity]
\label{thm:polynomial-time-sosp}
Let Assumptions~\ref{assumption:smooth-F} and~\ref{assumption:Hessian-lipschitz} hold. Fix $\mu^0\in\mathcal P_2(\mathbb R^d)$, and suppose Assumption~\ref{assumption_regularity_wg} holds on $\mathcal L_0.$ Choose $0<\varepsilon\le\varepsilon_0$, $\delta>0$, and $\zeta\in(0,1)$. Assume $L_{\mathrm H}\varepsilon\lesssim \delta^2$. Under the parameter conditions and choices $\beta_{\mathrm{grad}}>0$, $\tau_{\mathrm{grad}} \leq C_{\operatorname H}^{-1}\sqrt{\beta_{\mathrm{grad}}},$ $\beta_{\mathrm{sad}}=\Theta(\delta^2),$ $\tau_{\mathrm{sad}}\in (0,\frac{1}{2}],$ $\kappa=\tilde O(1),$ $n_\mathrm{out}=\tilde O(1),$ $F_0=\tilde \Omega(\delta^3),$ $\eta=\tilde \Theta(\frac{\delta^3}{\varepsilon+\delta^{3/2}})$ of Theorem~\ref{thm:polynomial-time-sosp-full}, Algorithm~\ref{alg:episode-wsfn}, initialized at $\mu^0\in\mathcal P_2(\mathbb R^d)$, returns an $(\varepsilon,\delta)$-second-order critical point with probability at least $1-\zeta$, using
\begin{equation}
\label{eq:sosp-soft-complexity-simplified}
    N_{\mathrm{SOSP}} = \tilde O\left(F_{\min}\left(\varepsilon^{-2}+\delta^{-3}\right)\right)
\end{equation}
WSFN transport steps, where $F_{\min}:=F(\mu^0)-\inf_\mu F(\mu)$.
\end{theorem}
\paragraph{Proof overview.} Each large gradient step decreases the objective by order $\varepsilon^2$. For saddle episodes, the local descent estimate of Lemma~\ref{lem:local-wsfn-descent} converts insufficient objective decrease into a transport displacement bound of order $\beta_{\mathrm{sad}}^{-1/4}\sqrt{F_0}$, up to logarithmic factors and the controlled initial perturbation. The nonlinear escape argument requires this displacement to remain within a radius of order $L_{\operatorname H}^{-1}\sqrt{\beta_{\mathrm{sad}}}$, again up to logarithmic factors. Balancing these scales selects a decrease threshold at the cubic scale, $F_0=\tilde\Omega(L_{\operatorname H}^{-2}\beta_{\mathrm{sad}}^{3/2}) = \tilde\Omega(L_{\operatorname H}^{-2}\delta^3).$ Proposition~\ref{discrete_decrease_around_saddle} guarantees this net decrease with high probability within at most $n_{\mathrm{out}}=\tilde O(1)$ steps. Counting large gradient steps and successful episodes against the objective gap $F_{\min}$ gives \eqref{eq:sosp-soft-complexity-simplified}. A conditional union bound controls the probability of terminating at a saddle. A detailed proof roadmap is provided in Appendix~\ref{proofs of main results}.

\begin{remark}[Improved saddle dependence]
\label{rem:improved-saddle-dependence}
With $\beta_\mathrm{sad}=b\delta^2$, $b>0$, and a saddle episode stepsize $\tau_\mathrm{sad}$ independent of $\delta$, WSFN improves the saddle contribution from $\tilde O(\delta^{-4})$ for PWGF \cite[Theorem~5.5]{yamamoto2025hessianguided} to $\tilde O(\delta^{-3})$, with the other problem constants fixed. The gain comes from shortening each episode from $\tilde O(\delta^{-1})$ to $\tilde O(1)$ iterations, while retaining a decrease of order $\tilde \Omega(\delta^3)$. Indeed, the WSFN preconditioner normalizes the unstable growth rate, but Lipschitz Hessian regularity only guarantees negative curvature of order $\delta$ over a radius of order $L_{\operatorname H}^{-1}\delta $. Hence, the associated objective depth scale is $L_{\operatorname H}^{-2}\delta^3$. Thus, WSFN removes the dynamical slowdown, but not the depth of saddles. Appendix~\ref{app:saddle-parameter-dependence} explains this decomposition and proves sharpness of the cubic decrease scale. We emphasize that this is not a worst case iteration complexity lower bound. 

The restriction $L_{\mathrm H}\varepsilon\lesssim\delta^2$ is also required by PWGF \cite[Theorem~5.5]{yamamoto2025hessianguided}. Therefore, the improvement only concerns the saddle contribution. The large gradient phase admits $\Omega(\varepsilon^{-2})$ examples with uniform regularity constants and bounded initial objective gap $F_{\min}$ (Proposition~\ref{prop:large-gradient-obstruction}). Thus, at $\varepsilon=\tilde\Theta(\delta^2)$, the leading total bound remains $\tilde O(\delta^{-4})$. A uniform $\tilde{O}(\varepsilon^{-3/2})$ large gradient guarantee -- and thus an $\tilde{O}(\delta^{-3})$ total bound -- typically requires cubic Taylor regularization instead of WSFN's quadratic model (see Appendix \ref{sec:well-posedness-charac-optimizer}). Crucially, cubic regularization leads to an implicit update, sacrificing explicit iterations like \eqref{eq:wsfn-update-main}.

In Euclidean optimization, cubic regularization achieves $O(\varepsilon^{-3/2})$ iteration complexity for first-order stationarity under Lipschitz Hessian assumptions \cite{NesterovPolyak2006,cartis2011adaptive}; see \cite{grapiglia2022cubic} for finite-difference Hessian variants. Developing an analogous Wasserstein guarantee through a modified large gradient phase is left for future work.
\end{remark}

Under benignity on the initial sublevel set, this second-order stationarity guarantee implies proximity to the global minimizer set without additional iterations.

\begin{corollary}[Proximity under sublevel benignity]
\label{cor:global-proximity-under-benignity}
Under the hypotheses and parameter choices of Theorem~\ref{thm:polynomial-time-sosp}, suppose additionally that Assumption~\ref{def:quantitative-strict-saddle} holds on $\mathcal L_0.$ Then, with probability at least $1-\zeta$,
the returned measure satisfies $\widehat\mu \in \mathcal N_{\alpha_{\mathrm{ben}}}(\mathcal M^*),$ with the same bound \eqref{eq:sosp-soft-complexity-simplified}.
\end{corollary}

\subsection{Local spectral behavior near a global minimizer}
Fix $\tau\in(0,1]$ and $\beta>0$ as parameters of the deterministic WSFN update. The following result characterizes the linearization of this fixed parameter update. To compare WSFN updates on a fixed reference space, let $\mu^*\in\mathcal P_2^{\mathrm{ac}}(\mathbb R^d)$, set $\mu_v:=(\operatorname{Id}+v)_\#\mu^*$, and define
\begin{equation*}
    \Psi_{\tau,\beta}(v) := v-\tau\left[(\operatorname H_{\mu_v}^2+\beta \operatorname{I_{d \times d}})^{-1/2}\nabla_\mu F(\mu_v,\cdot)\right]\circ(\operatorname{Id}+v).
\end{equation*}
One WSFN step sends $\mu_v$ to $(\operatorname{Id}+\Psi_{\tau,\beta}(v))_\#\mu^*$.

\begin{theorem}[Local spectral behavior of WSFN]
\label{thm:local-spectral-behavior-wsfn}
Let Assumptions~\ref{assumption:smooth-F} and \ref{assumption:Hessian-lipschitz} hold, and suppose that $\mu^*$ is a global minimizer. Fix $\tau\in(0,1]$, $\beta>0$, and define $\mathcal L_{\tau,\beta} := \operatorname{I_{d\times d}}-\tau(\operatorname K_{\mu^*}^2+\beta \operatorname{I_{d\times d}})^{-1/2}\operatorname K_{\mu^*}.$ There exists
$C=C(L_{\operatorname H},C_{\operatorname K},\beta)>0$ such that, for every $v\in L^2_{\mu^*}$,
\begin{equation*}
    \Psi_{\tau,\beta}(v) = \mathcal L_{\tau,\beta}v+\mathcal R_{\tau,\beta}(v), \quad \|\mathcal R_{\tau,\beta}(v)\|_{L^2_{\mu^*}} \leq C\|v\|_*^2.
\end{equation*}
The linearized iterates satisfy $\mathcal L_{\tau,\beta}^n v\to P_{\ker \operatorname K_{\mu^*}}v$ strongly for every $v$, although $\|\mathcal L_{\tau,\beta}^n\|_{\mathrm{op}}=1$ for every $n\geq1$. For $m>0$, let $P_m:=\mathbbm 1_{[m,\infty)}(\operatorname K_{\mu^*})$.
Then, for every $n\geq1$,
\begin{equation*}
    \|\mathcal L_{\tau,\beta}^nP_mv\|_{L^2_{\mu^*}}\leq \left(1-\tau\frac{m}{\sqrt{m^2+\beta}}\right)^n\|P_mv\|_{L^2_{\mu^*}}.
\end{equation*}
If $(\operatorname{Id}-P_{\ker \operatorname K_{\mu^*}})v=\operatorname K_{\mu^*}^s w$ for some $s>0$ and $w\in L^2_{\mu^*}$, then
\begin{equation*}
    \|\mathcal L_{\tau,\beta}^n v-P_{\ker \operatorname K_{\mu^*}}v\|_{L^2_{\mu^*}} \leq \left(\frac{s\sqrt{C_{\operatorname K}^2+\beta}}{e\tau n}\right)^s \|w\|_{L^2_{\mu^*}},
    \quad n\geq1.
\end{equation*}
\end{theorem}

\paragraph{Spectral interpretation.}
At an absolutely continuous global minimizer, $\operatorname H_{\mu^*}=\operatorname K_{\mu^*}\succeq0$. The proof combines linearization around $\operatorname K_{\mu^*}$ with functional calculus for the multiplier $1-\tau\lambda(\lambda^2+\beta)^{-1/2};$ see Appendix~\ref{subsection:proofs_conv_theorems}. Curvature much larger than $\sqrt{\beta}$ is approximately normalized, whereas nearly flat directions remain weakly contracted and kernel components are unchanged. Therefore, the theorem provides geometric contraction on spectral bands separated from zero and polynomial decay under the source condition $(\operatorname{Id}-P_{\ker \operatorname K_{\mu^*}})v=\operatorname K_{\mu^*}^s w$. This condition suppresses small curvature components through the factor $\lambda^s$ and is standard in spectral regularization and inverse problems \cite{BAUER200752,LIN2020868,blanchard2017optimal}. For more details, we refer to Remark~\ref{rem:interpretation-local-spectral-behavior-wsfn}.

\begin{remark}[Limitations of nonlinear rate transfer]
\label{rem:limitations-nonlinear-spectral-rates}
Theorem~\ref{thm:local-spectral-behavior-wsfn} gives rates for the linearized iterates. Nonlinear rate transfer requires further structural control, such as invariance of a positive spectral band or suitable source norm bounds on the remainder. These are strong requirements, which do not follow from the present assumptions. Accordingly, the neighborhood entry guarantee of Corollary~\ref{cor:global-proximity-under-benignity} and the spectral theorem above do not together imply convergence rates for the nonlinear iteration. Perturbative estimates for a deterministic continuation from the returned measure by Algorithm \ref{alg:episode-wsfn} are given in Theorem~\ref{thm:global-to-local-spectral-wsfn}. The additional requirements for nonlinear rate transfer are discussed in more detail in Remark~\ref{rem:nonlinear-spectral-rates-appendix}.
\end{remark}

\section{Numerical illustration}
\label{sec:coulomb-experiment}
\paragraph{Particle implementation.} We represent the measure by $\mu^N=N^{-1}\sum_{i=1}^N\delta_{x_i}$. For the regularized empirical objective below, we evaluate the particle gradient and the action of the full particle Hessian analytically. WSFN applies a Lanczos approximation of $(\operatorname{H}_N^2+\beta \operatorname{I_{Nd \times Nd}})^{-1/2}$ to the gradient, without forming or diagonalizing the full $Nd\times Nd$ Hessian. Perturbations use the integral Hessian component $\operatorname{K}_N$, with their magnitudes normalized as specified in Appendix~\ref{section:practical implementation}. This experiment uses prescribed perturbation times instead of stopping and episode rules. Quantitative control of kernel regularization, particle, and Lanczos approximation errors remains outside the present analysis.
\\

\paragraph{Regularized Coulomb MMD and directional ablation.}
We consider a smooth regularization of the Coulomb MMD objective in dimension $d=3$. Given target samples $Y=(y_a)_{a=1}^M$, define $\kappa_\rho(x,y)=(\|x-y\|^2+\rho^2)^{-1/2}$ and minimize
\begin{equation}
\label{eq:coulomb-objective}
F_{\rho,N,M}(X;Y)=\frac{1}{N^2}\sum_{i,j=1}^N\kappa_\rho(x_i,x_j)
-\frac{2}{NM}\sum_{i=1}^N\sum_{a=1}^M\kappa_\rho(x_i,y_a)+\frac{1}{M^2}\sum_{a,b=1}^M\kappa_\rho(y_a,y_b).
\end{equation}
This is the squared MMD between the two empirical measures for the regularized kernel. We use $N=500$ particles, $M=400$ target samples, and $\rho=0.05$.

We compare WGF, WGF with isotropic perturbations, PWGF with Hessian-shaped perturbations, WSFN, and a scalar control. The control replaces the preconditioned direction by a gradient direction of the same norm, computed at the control's own current state. The four methods' learning rates are selected independently on pilot seeds. The control inherits WSFN's selected rate and preconditioner settings. Evaluation uses five paired trials with shared target samples and initial particles, with $120$ seconds per method and trial. Except for WGF, methods receive perturbations every $100$ updates, using Gaussian draws at corresponding perturbation events.
\begin{figure}[t]
    \centering
    \includegraphics[width=\linewidth]{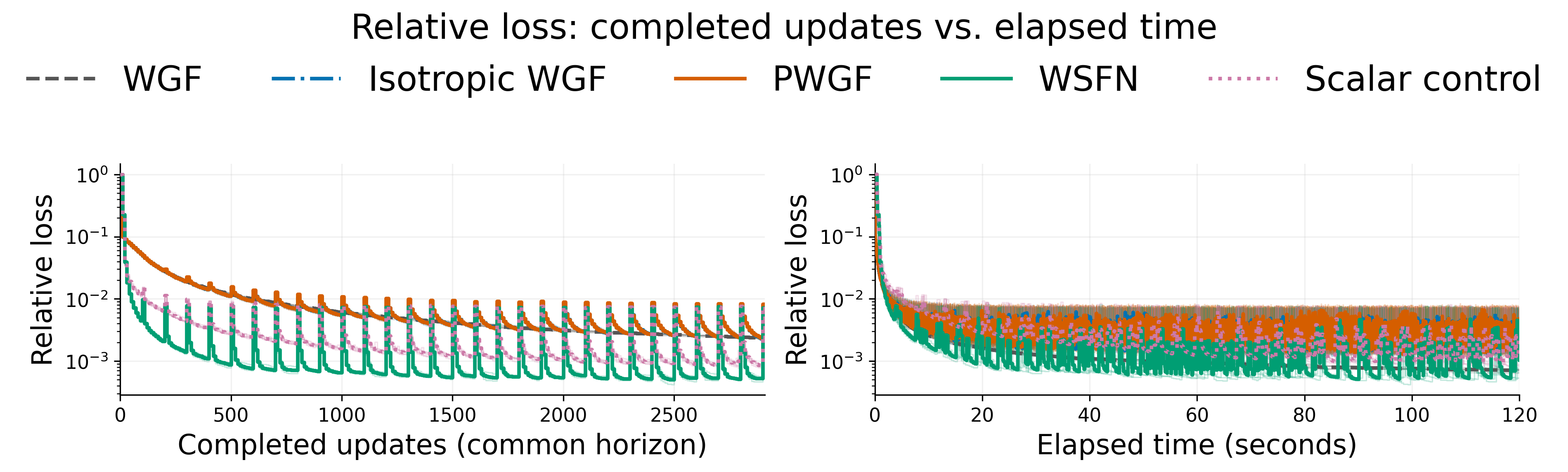}
    \caption{Regularized Coulomb MMD. Relative loss $\frac{F_{\rho,N,M}(X_t;Y)}{F_{\rho,N,M}(X_0;Y)}$ versus completed updates over a common update horizon (left) and elapsed time (right). Thick curves show means of the normalized losses over five paired trials. Thin curves show individual trials. Both vertical axes are logarithmic.}
    \label{fig:coulomb-loss}
\end{figure}

Figure~\ref{fig:coulomb-loss} shows that WSFN reduces the loss more rapidly per update and reaches low loss levels earlier in elapsed time. Between prescribed perturbations, its loss is lower than that of the scalar control, supporting a benefit from the preconditioned direction beyond its locally matched magnitude at the selected WSFN learning rate. The control is not independently tuned, so this comparison does not imply superiority over a tuned scalar adaptive method. Repeated perturbations produce transient loss increases, and WSFN does not dominate the baselines at every time. The gradient diagnostics in Appendix~\ref{section:practical implementation} also show that lower loss does not imply a uniformly smaller gradient norm.

\section{Conclusion}
\label{conclusion}

We introduced WSFN, combining regularized Hessian preconditioning with Hessian-guided perturbations. Under the stated conditions, WSFN reaches approximate second-order stationarity, improving the saddle contribution from $\tilde O(\delta^{-4})$ to $\tilde O(\delta^{-3})$. The guarantees rely on Hessian regularity and uniform control at episode starting measures. Verifying or relaxing these conditions for learning objectives remains an important direction. The local spectral rates concern linearized dynamics, and nonlinear rate guarantees require additional control of the remainder. Moreover, the complexity counts exact measure-valued transport steps, leaving particle approximation errors and computational cost outside the current guarantees. Future work should establish nonlinear rates under verifiable structural conditions and quantify how particle, preconditioner, and covariance approximations affect descent and escape. Relating these errors to $\beta$ and $\delta$ would clarify when the improved transport step bound gives computational savings.

\section*{Acknowledgements}
TS was partially supported by JSPS KAKENHI (24K02905) and JST CREST (JPMJCR2015). This research is supported by the National Research Foundation, Singapore and the Ministry of Digital Development and Information under the AI Visiting Professorship Programme (award number AIVP-2024-004). Any opinions, findings and conclusions or recommendations expressed in this material are those of the author(s) and do not reflect the views of National Research Foundation, Singapore and the Ministry of Digital Development and Information.
    
\bibliographystyle{abbrv}
\bibliography{references} 

@book{villani2008optimal,
  title={Optimal Transport: Old and New},
  author={Villani, C.},
  series={Grundlehren der mathematischen Wissenschaften},
  year={2008},
  publisher={Springer Berlin Heidelberg}
}

@article{10.1214/20-AIHP1140,
author = {Kaitong Hu and Zhenjie Ren and David \v{S}i{\v{s}}ka and {\L}ukasz Szpruch},
title = {{Mean-field {L}angevin dynamics and energy landscape of neural networks}},
volume = {57},
journal = {Annales de l'Institut Henri Poincaré, Probabilités et Statistiques},
number = {4},
publisher = {Institut Henri Poincaré},
pages = {2043-2065},
year = {2021}
}

@book{Carmona2018ProbabilisticTO,
  title={Probabilistic Theory of Mean Field Games with Applications I: Mean Field FBSDEs, Control, and Games},
  author={Ren{\'e} A. Carmona and François Delarue},
  year={2018},
publisher={Springer International Publishing}
}

@article{
chizat2022meanfield,
title={Mean-Field {L}angevin Dynamics: Exponential Convergence and Annealing},
author={L{\'e}na{\"\i}c Chizat},
journal={Transactions on Machine Learning Research},
year={2022}
}

@ARTICLE{Nitanda2022ConvexAO,
       author = {{Nitanda}, Atsushi and {Wu}, Denny and {Suzuki}, Taiji},
        title = "{Convex Analysis of the Mean Field Langevin Dynamics}",
      journal = {arXiv e-prints},
         year = 2022,
        month = jan,
          eid = {arXiv:2201.10469},
        pages = {arXiv:2201.10469},
archivePrefix = {arXiv},
       eprint = {2201.10469},
 primaryClass = {stat.ML},
       adsurl = {https://ui.adsabs.harvard.edu/abs/2022arXiv220110469N}
}

@InProceedings{jin17a,
  title = 	 {How to Escape Saddle Points Efficiently},
  author =       {Chi Jin and Rong Ge and Praneeth Netrapalli and Sham M. Kakade and Michael I. Jordan},
  booktitle = 	 {Proceedings of the 34th International Conference on Machine Learning},
  pages = 	 {1724--1732},
  year = 	 {2017},
  volume = 	 {70},
  series = 	 {Proceedings of Machine Learning Research},
  month = 	 {06--11 Aug},
  publisher =    {PMLR}
}

@InProceedings{Ge15,
  title = 	 {Escaping From Saddle Points -- Online Stochastic Gradient for Tensor Decomposition},
  author = 	 {Ge, Rong and Huang, Furong and Jin, Chi and Yuan, Yang},
  booktitle = 	 {Proceedings of The 28th Conference on Learning Theory},
  pages = 	 {797--842},
  year = 	 {2015},
  volume = 	 {40},
  series = 	 {Proceedings of Machine Learning Research},
  month = 	 {03--06 Jul},
  publisher =    {PMLR}
}

@book{conway1994course,
  title={A Course in Functional Analysis},
  author={Conway, J.B.},
  series={Graduate Texts in Mathematics},
  year={1994},
  publisher={Springer New York}
}

@article{OTTO2000361,
title = {Generalization of an Inequality by {T}alagrand and Links with the Logarithmic {S}obolev Inequality},
journal = {Journal of Functional Analysis},
volume = {173},
number = {2},
pages = {361-400},
year = {2000},
author = {F. Otto and C. Villani}
}

@article{NesterovPolyak2006,
  author    = {Yurii Nesterov and Boris T. Polyak},
  title     = {Cubic regularization of {N}ewton method and its global performance},
  journal   = {Mathematical Programming},
  volume    = {108},
  pages     = {177--205},
  year      = {2006},
  publisher = {Springer}
}

@article{cartis2011adaptive,
  author    = {Cartis, Coralia and Gould, Nicholas I. M. and Toint, Philippe L.},
  title     = {Adaptive cubic regularisation methods for unconstrained optimization. {P}art {II}: worst-case function- and derivative-evaluation complexity},
  journal   = {Mathematical Programming},
  volume    = {130},
  pages     = {295--319},
  year      = {2011},
  publisher = {Springer},
}

@article{Grapiglia2022Cubic,
  author    = {Grapiglia, G. N. and Gonçalves, M. L. N. and Silva, G. N.},
  title     = {A cubic regularization of {N}ewton's method with finite difference {H}essian approximations},
  journal   = {Numerical Algorithms},
  volume    = {90},
  pages     = {607--630},
  year      = {2022},
  publisher = {Springer}
}

@article{BAUER200752,
title = {On regularization algorithms in learning theory},
journal = {Journal of Complexity},
volume = {23},
number = {1},
pages = {52-72},
year = {2007},
author = {Frank Bauer and Sergei Pereverzev and Lorenzo Rosasco}
}

@article{cartis2020concise,
  author  = {Cartis, Coralia and Gould, Nicholas I. M.
             and Toint, Philippe L.},
  title   = {A Concise Second-order Complexity Analysis
             for Unconstrained Optimization Using
             High-order Regularized Models},
  journal = {Optimization Methods and Software},
  volume  = {35},
  number  = {2},
  pages   = {243--256},
  year    = {2020}}

@article{LIN2020868,
title = {Optimal rates for spectral algorithms with least-squares regression over {H}ilbert spaces},
journal = {Applied and Computational Harmonic Analysis},
volume = {48},
number = {3},
pages = {868-890},
year = {2020},
author = {Junhong Lin and Alessandro Rudi and Lorenzo Rosasco and Volkan Cevher}
}

@inproceedings{chizat2022trajectory,
title={Trajectory Inference via Mean-field {L}angevin in Path Space},
author={L{\'e}na{\"\i}c Chizat and Stephen Zhang and Matthieu Heitz and Geoffrey Schiebinger},
booktitle={Advances in Neural Information Processing Systems},
year={2022}
}

@book{hille1957functional,
  title={Functional analysis and semi-groups},
  author={Hille, E. and Phillips, R.S.},
  series={American Mathematical Society},
  year={1957},
  publisher={American Mathematical Soc.}
}

@article{blanchard2017optimal,
  author    = {Blanchard, Gilles and M{\"{u}}cke, Nicole},
  title     = {Optimal Rates for Regularization of Statistical Inverse Learning Problems},
  journal   = {Foundations of Computational Mathematics},
  volume    = {18},
  number    = {4},
  pages     = {971--1013},
  year      = {2018}
}

@article{Hsu-Kakade,
author = {Daniel Hsu and Sham Kakade and Tong Zhang},
title = {{A tail inequality for quadratic forms of subgaussian random vectors}},
volume = {17},
journal = {Electronic Communications in Probability},
publisher = {Institute of Mathematical Statistics and Bernoulli Society},
pages = {1--6},
year = {2012}
}

@book{alma9912197933502466,
author = {Villani, C.},
booktitle = {Topics in optimal transportation},
publisher = {American Mathematical Society},
series = {Graduate studies in mathematics},
title = {Topics in optimal transportation},
year = {2003},
}

@inproceedings{Chizat2018OnTG,
  title={On the Global Convergence of Gradient Descent for Over-parameterized Models using Optimal Transport},
  author={L{\'e}na{\"i}c Chizat and Francis R. Bach},
  booktitle={NeurIPS},
  year={2018}
}

@article{rotskoff,
author = {Rotskoff, Grant and Vanden-Eijnden, Eric},
title = {Trainability and Accuracy of Artificial Neural Networks: An Interacting Particle System Approach},
journal = {Communications on Pure and Applied Mathematics},
volume = {75},
number = {9},
pages = {1889-1935},
year = {2022}
}

@article{Mei2018AMF,
  title={A mean field view of the landscape of two-layer neural networks},
  author={Song Mei and Andrea Montanari and Phan-Minh Nguyen},
  journal={Proceedings of the National Academy of Sciences of the United States of America},
  year={2018},
  volume={115},
  pages={E7665 - E7671}
}

@article{SIRIGNANO20201820,
title = {Mean field analysis of neural networks: A central limit theorem},
journal = {Stochastic Processes and their Applications},
volume = {130},
number = {3},
pages = {1820-1852},
year = {2020},
author = {Justin Sirignano and Konstantinos Spiliopoulos}
}

@InProceedings{korba21a,
  title = 	 {Kernel {S}tein Discrepancy Descent},
  author =       {Korba, Anna and Aubin-Frankowski, Pierre-Cyril and Majewski, Szymon and Ablin, Pierre},
  booktitle = 	 {Proceedings of the 38th International Conference on Machine Learning},
  year = 	 {2021},
  publisher =    {PMLR}
}

@misc{wang2020informationnewtonsflowsecondorder,
      title={Information {N}ewton's flow: second-order optimization method in probability space}, 
      author={Yifei Wang and Wuchen Li},
      year={2020},
      note={arXiv:2001.04341},
      primaryClass={math.OC}
}

@book{ReedSimon1980,
  title     = {Methods of Modern Mathematical Physics. {I}: Functional Analysis},
  author    = {Reed, Michael and Simon, Barry},
  year      = {1980},
  publisher = {Academic Press},
  address   = {New York}
}

@book{arveson2001short,
  title={A Short Course on Spectral Theory},
  author={Arveson, W.},
  series={Graduate Texts in Mathematics},
  year={2001},
  publisher={Springer New York}
}

@book{kato1995perturbation,
  title={Perturbation Theory for Linear Operators},
  author={Kato, T.},
  series={Classics in Mathematics},
  year={1995},
  publisher={Springer Berlin Heidelberg}
}

@book{ambrosio2008gradient,
  title={Gradient Flows: In Metric Spaces and in the Space of Probability Measures},
  author={Ambrosio, L. and Gigli, N. and Savare, G.},
  series={Lectures in Mathematics. ETH Z{\"u}rich},
  year={2008},
  publisher={Birkh{\"a}user Basel}
}

@book{santambrogio2015optimal,
  title={Optimal Transport for Applied Mathematicians: Calculus of Variations, PDEs, and Modeling},
  author={Santambrogio, F.},
  series={Progress in Nonlinear Differential Equations and Their Applications},
  year={2015},
  publisher={Springer International Publishing}
}

@inproceedings{
kim2024transformers,
title={Transformers Learn Nonlinear Features In Context},
author={Juno Kim and Taiji Suzuki},
booktitle={ICLR 2024 Workshop on Mathematical and Empirical Understanding of Foundation Models},
year={2024}
}

@article{JMLR:rentian-proximal,
  author  = {Rentian Yao and Xiaohui Chen and Yun Yang},
  title   = {{W}asserstein Proximal Coordinate Gradient Algorithms},
  journal = {Journal of Machine Learning Research},
  year    = {2024},
  volume  = {25},
  number  = {269},
  pages   = {1--66},
}

@inproceedings{lu2020meanfield,
  title={A Mean-field Analysis of {Deep} {ResNet} and {Beyond}: {Towards} {Provable} {Optimization} {Via} {Overparameterization} {From} {Depth}},
  author={Lu, Yiping and Ma, Chao and Lu, Jianfeng and Ying, Lexing},
  booktitle={Proceedings of the 37th International Conference on Machine Learning},
  pages={6426--6436},
  year={2020},
  volume={119},
  series={Proceedings of Machine Learning Research},
  publisher={PMLR}
}

@article{yifei-projected,
author = {Wang, Yifei and Chen, Peng and Li, Wuchen},
title = {Projected {W}asserstein Gradient Descent for High-Dimensional {B}ayesian Inference},
journal = {SIAM/ASA Journal on Uncertainty Quantification},
volume = {10},
number = {4},
pages = {1513-1532},
year = {2022},
}

@article{jko,
author = {Jordan, Richard and Kinderlehrer, David and Otto, Felix},
title = {The Variational Formulation of the {F}okker--{P}lanck Equation},
journal = {SIAM Journal on Mathematical Analysis},
volume = {29},
number = {1},
pages = {1-17},
year = {1998}
}

@misc{zhu2025convergenceanalysiswassersteinproximal,
      title={Convergence Analysis of the {W}asserstein Proximal Algorithm beyond Geodesic Convexity}, 
      author={Shuailong Zhu and Xiaohui Chen},
      year={2025},
      note={arXiv:2501.14993},
      primaryClass={math.OC}
}

@article{yao2023meanfieldvariationalinferencewasserstein,
  author  = {Rentian Yao and Yun Yang},
  title   = {Mean-Field Variational Inference via {W}asserstein Gradient Flow},
  journal = {Journal of Machine Learning Research},
  year    = {2025},
  volume  = {26},
  number  = {309},
  pages   = {1-46}
}

@inproceedings{lambert2022variational,
  title     = {Variational Inference via {W}asserstein Gradient Flows},
  author    = {Lambert, Marc and Chewi, Sinho and Bach, Francis R. and Bonnabel, Silv{\`e}re and Rigollet, Philippe},
  booktitle = {Advances in Neural Information Processing Systems},
  volume    = {35},
  pages     = {14434--14447},
  year      = {2022}
}

@article{boufadene,
author = {Boufadene, Siwan and Vialard, Fran\c{c}ois-Xavier},
title = {On the Global Convergence of {W}asserstein Gradient Flow of the {C}oulomb Discrepancy},
journal = {SIAM Journal on Mathematical Analysis},
volume = {57},
number = {4},
pages = {4556-4587},
year = {2025}
}

@book{gustafson2020mathematical,
  title={Mathematical Concepts of Quantum Mechanics},
  author={Gustafson, S.J. and Sigal, I.M.},
  series={Universitext},
  year={2020},
  publisher={Springer International Publishing}
}

@book{rasmussen,
    author = {Rasmussen, Carl Edward and Williams, Christopher K. I.},
    title = {Gaussian {P}rocesses for Machine Learning},
    publisher = {The MIT Press},
    year = {2005}
}

@book{Dudley_2002, 
place={Cambridge}, 
series={Cambridge Studies in Advanced Mathematics}, 
title={Real Analysis and Probability}, 
publisher={Cambridge University Press}, 
author={Dudley, R. M.}, 
year={2002}, 
collection={Cambridge Studies in Advanced Mathematics}}

@book{CheNilRig25OT,
    AUTHOR = {Chewi, Sinho and Niles-Weed, Jonathan and Rigollet, Philippe},
     TITLE = {Statistical optimal transport},
    SERIES = {Lecture Notes in Mathematics},
    VOLUME = {2364},
      NOTE = {\'Ecole d'\'Et\'e{} de Probabilit\'es de Saint-Flour XLIX --
              2019},
 PUBLISHER = {Springer, Cham},
      YEAR = {2025},
     PAGES = {xiv+260},
}

@book{cardaliaguet2019master,
  title={The Master Equation and the Convergence Problem in Mean Field Games},
  author={Cardaliaguet, P. and Delarue, F. and Lasry, J.M. and Lions, P.L.},
  series={Annals of Mathematics Studies},
  year={2019},
  publisher={Princeton University Press}
}

@BOOK{figalli,
	year = {2023},
	type = {Monograph},
	author = {Figalli, Alessio and Glaudo, Federico},
	size = {152 p.},
	language = {en},
	address = {Berlin},
	publisher = {EMS Press},
	title = {An Invitation to Optimal Transport, Wasserstein Distances, and Gradient Flows}
}

@inproceedings{korbaproximal,
 author = {Salim, Adil and Korba, Anna and Luise, Giulia},
 booktitle = {Advances in Neural Information Processing Systems},
 pages = {12356--12366},
 publisher = {Curran Associates, Inc.},
 title = {The {W}asserstein Proximal Gradient Algorithm},
 volume = {33},
 year = {2020}
}

@inproceedings{
yamamoto2025hessianguided,
title={Hessian-guided Perturbed {W}asserstein Gradient Flows for Escaping Saddle Points},
author={Naoya Yamamoto and Juno Kim and Taiji Suzuki},
booktitle={The Thirty-ninth Annual Conference on Neural Information Processing Systems},
year={2025}
}

@misc{lascu2024linearconvergenceproximaldescent,
      title={Linear convergence of proximal descent schemes on the {W}asserstein space}, 
      author={Razvan-Andrei Lascu and Mateusz B. Majka and David Šiška and Łukasz Szpruch},
      year={2024},
      note={arXiv:2411.15067},
      primaryClass={math.OC}
}

@inproceedings{arbel,
 author = {Arbel, Michael and Korba, Anna and Salim, Adil and Gretton, Arthur},
 booktitle = {Advances in Neural Information Processing Systems},
 pages = {},
 title = {Maximum Mean Discrepancy Gradient Flow},
 volume = {32},
 year = {2019}
}

@article{huang2024generativemodelingminimizingwasserstein2,
author = {Huang, Yu-Jui and Malik, Zachariah},
title = {Generative Modeling by Minimizing the {W}asserstein-2 Loss},
journal = {SIAM Journal on Mathematics of Data Science},
volume = {8},
number = {3},
pages = {973-999},
year = {2026}
}

@InProceedings{pmlr-v97-chu19a,
  title = 	 {Probability Functional Descent: A Unifying Perspective on {GAN}s, {V}ariational {I}nference, and {R}einforcement {L}earning},
  author =       {Chu, Casey and Blanchet, Jose and Glynn, Peter},
  booktitle = 	 {Proceedings of the 36th International Conference on Machine Learning},
  pages = 	 {1213--1222},
  year = 	 {2019},
  volume = 	 {97},
  series = 	 {Proceedings of Machine Learning Research},
  month = 	 {09--15 Jun},
  publisher =    {PMLR}
}

@article{chizat2026convergencedriftdiffusion,
  author    = {Chizat, L{\'e}na{\"i}c and Colombo, Maria and Fern{\'a}ndez-Real, Xavier},
  title     = {Convergence of drift-diffusion {PDE}s arising as {W}asserstein gradient flows of convex functions},
  journal   = {Archive for Rational Mechanics and Analysis},
  volume    = {250},
  number    = {48},
  year      = {2026},
  publisher = {Springer}
}

@misc{chizat2026quantitativeconvergencewassersteingradient,
      title={Quantitative Convergence of {W}asserstein Gradient Flows of {K}ernel {M}ean {D}iscrepancies}, 
      author={Lénaïc Chizat and Maria Colombo and Roberto Colombo and Xavier Fernández-Real},
      year={2026},
      note={arXiv:2603.01977},
      primaryClass={math.AP}
}

@inproceedings{NIPS2014_dauphin,
 author = {Dauphin, Yann N and Pascanu, Razvan and Gulcehre, Caglar and Cho, Kyunghyun and Ganguli, Surya and Bengio, Yoshua},
 booktitle = {Advances in Neural Information Processing Systems},
 title = {Identifying and attacking the saddle point problem in high-dimensional non-convex optimization},
 year = {2014}
}

@inproceedings{
bonet2024mirror,
title={Mirror and Preconditioned Gradient Descent in {W}asserstein Space},
author={Cl{\'e}ment Bonet and Th{\'e}o Uscidda and Adam David and Pierre-Cyril Aubin-Frankowski and Anna Korba},
booktitle={The Thirty-eighth Annual Conference on Neural Information Processing Systems},
year={2024}
}

@InProceedings{pmlr-v75-wibisono18a,
  title = 	 {Sampling as optimization in the space of measures: The {L}angevin dynamics as a composite optimization problem},
  author =       {Wibisono, Andre},
  booktitle = 	 {Proceedings of the 31st  Conference On Learning Theory},
  pages = 	 {2093--3027},
  year = 	 {2018},
  volume = 	 {75},
  series = 	 {Proceedings of Machine Learning Research},
  month = 	 {06--09 Jul},
  publisher =    {PMLR}
}

@article{Blei03042017,
author = {Blei, David M. and Kucukelbir, Alp and McAuliffe, Jon D.},
title = {{V}ariational {I}nference: A Review for Statisticians},
journal = {Journal of the American Statistical Association},
volume = {112},
number = {518},
pages = {859--877},
year = {2017},
}

@article{benoit,
     author = {Bonnet, Beno{\^\i}t},
     title = {A {Pontryagin} {Maximum} {Principle} in {Wasserstein} spaces for constrained optimal control problems},
     journal = {ESAIM: Control, Optimisation and Calculus of Variations},
     year = {2019},
     publisher = {EDP Sciences},
     volume = {25}
}

@article{lanzetti,
author = {Lanzetti, Nicolas and Bolognani, Saverio and D\"{o}rfler, Florian},
title = {First-Order Conditions for Optimization in the {Wasserstein} Space},
journal = {SIAM Journal on Mathematics of Data Science},
volume = {7},
number = {1},
pages = {274-300},
year = {2025}
}
\section{Appendix}
\appendix
The appendices are organized to guide the reader from motivating examples and implementation to background material and, finally, the technical analysis. Appendix \ref{section:additional notation} collects operator theoretic notation on $L_\mu^2$. Appendix \ref{section:examples-benign} presents non-convex objectives on Wasserstein space that motivate the benignity assumption, and Appendix \ref{section:practical implementation} gives the particle-based implementation of WSFN. Appendix \ref{L2 valued GPs} introduces the $L_\mu^2$-valued Gaussian process machinery used for the perturbation step. Appendix \ref{section:background on OT} recalls the optimal transport background, Appendix \ref{sec:wass-geom} develops the differential calculus on Wasserstein space, and Appendix \ref{section:app_conv_in_W_space} reviews the main convexity notions in Wasserstein space. Appendix \ref{optimality conditions} gathers first and second-order optimality conditions, Appendix \ref{first and second order dynamics} recalls the corresponding first and second-order dynamics, Appendix \ref{Calculus along curves and flow approximations} provides calculus along curves together with flow approximation results linking discrete dynamics to their continuous-time limits, and Appendix \ref{section:wasserstein-proximal-taylor} places WSFN alongside related methods in a common Wasserstein Proximal Taylor framework. Finally, Appendix \ref{appendix:aux-results} collects the auxiliary lemmas used throughout the analysis, and additional local convergence results. Appendix \ref{proofs of main results} contains the proofs of the main results. Appendix \ref{app:saddle-parameter-dependence} explains the improvement of the saddle contribution from $\tilde O(\delta^{-4})$ to $\tilde O(\delta^{-3})$.
\tableofcontents
\appendix
\section{Additional notation for operators on $L_\mu^2$}
\label{section:additional notation}
In this section, we collect several standard notions and conventions for bounded linear operators on $L_\mu^2$ that will be used throughout the paper. For an operator $A_\mu:L_\mu^2 \to L_\mu^2$, the operator norm is defined as $\|A_\mu\|_{\mathrm{op}} = \sup_{v \neq 0} \frac{\|A_\mu v\|_{L_\mu^2}}{\|v\|_{L_\mu^2}}$. The operator $A_\mu$ is called \textit{positive} (respectively, \textit{non-negative}) and denoted by $A_\mu \succ 0$ (respectively, $A_\mu \succeq 0$) if $\langle A_\mu v,v \rangle_{L_\mu^2} > 0$ (respectively, $\langle A_\mu v,v \rangle_{L_\mu^2} \geq 0$), for all $v \in L^2_\mu$. The operator $A_\mu$ is called \textit{uniformly positive} and denoted by $A_\mu \succeq \lambda \operatorname{I_{d \times d}}$, where $\operatorname{I_{d\times d}}$ is the identity operator, if there exists $\lambda > 0$ such that $\langle A_\mu v,v \rangle_{L_\mu^2} \geq \lambda \|v\|_{L_\mu^2}^2$ for all $v \in L^2_\mu$. The spectrum of the operator $A_\mu$ is denoted by $\sigma(A_\mu).$ For a bounded linear operator $A_\mu:L_\mu^2 \to L_\mu^2$, its adjoint, denoted by $A_\mu^*$, is the bounded linear operator $A_\mu^*:L_\mu^2 \to L_\mu^2$ satisfying $\langle A_\mu v,w \rangle_{L_\mu^2} = \langle v,A_\mu^*w \rangle_{L_\mu^2}$, for all $v,w \in L_\mu^2$. A bounded linear operator $A_\mu$ is called self-adjoint if $A_\mu = A_\mu^*$. A Hilbert--Schmidt (HS) operator $A_\mu:L_\mu^2 \to L_\mu^2$ is a bounded operator that has finite HS norm $\|A_\mu\|_{\mathrm{HS}}^2:= \sum_{n=1}^\infty \|A_\mu e_n\|_{L_\mu^2}^2$, where $\{e_n\}_{n =1}^\infty$ is an orthonormal basis. We denote by $\mathcal{LS}(L_\mu^2)$ and $\mathcal{B}(L_\mu^2)$ the spaces of linear symmetric and linear bounded operators from $L_{\mu}^2$ to $L_{\mu}^2$, respectively. If $A_\mu \in \mathcal{B}(L_\mu^2)$, then $\|A_\mu\|_{\mathrm{op}} \leq \|A_\mu\|_{\text{HS}}$.

\section{Motivating examples of qualitative benign landscapes}
\label{section:examples-benign}
The examples below present different forms of qualitative benignity, such as the absence of spurious local minima or the fact that non-global critical points are saddles. These qualitative statements motivate the quantitative benignity condition used in Corollary \ref{cor:global-proximity-under-benignity}, but do not provide uniform constants \((\varepsilon,\delta,\alpha_{\mathrm{ben}})\) on the whole Wasserstein space. Proposition \ref{prop:qualitative-to-quantitative-benignity} gives sufficient compactness and continuity conditions under which the quantitative property follows.
\begin{example}[Coulomb MMD]
\label{example:Coulomb-MMD}
Let $d\ge 3$ and $\mu^* \in \mathcal{P}_2(\mathbb R^d)$ be a target probability measure. The Maximum Mean Discrepancy (MMD) associated with the Coulomb kernel is given by
\begin{equation*}
    F(\mu) = \frac{1}{d-2}\iint_{\mathbb{R}^d \times \mathbb{R}^d}
\frac{1}{\|x-y\|^{d-2}} \mathrm{d}(\mu-\mu^*)(x)\mathrm{d}(\mu-\mu^*)(y).
\end{equation*}
\cite{boufadene} showed that, although this functional is not geodesically convex in Wasserstein space, it has no local minima other than the global minimizer $\mu^*$. Hence, every local minimizer of $F$ is global. Thus, the Coulomb MMD provides a natural example of a non-convex functional with a benign optimization landscape. Moreover, they proved that any critical point coincides with $\mu^*$ on the interior of its support. If $\mu$ is a critical point of $F$, then $\mu|_{\operatorname{Int}(\operatorname{supp}\mu)} = \mu^*|_{\operatorname{Int}(\operatorname{supp}\mu)}$. In particular, if for some open set $U \subset \operatorname{Int}(\operatorname{supp}\mu)$ the measures $\mu$ and $\mu^*$ are absolutely continuous, i.e., admit densities $\rho$ and $\rho^*$ on $U$, then $\rho=\rho^*$ a.e. on $U$. Thus any discrepancy between $\mu$ and $\mu^*$ could only remain in the atomic part or on the boundary of the support. 
\end{example}
\begin{example}[Three-layer mean-field network]
\label{example:ICL}
Consider a three-layer teacher-student network in which the first two layers are expressed as a mean-field two-layer NN
\begin{equation*}
    h_\mu(x)=\int_\Theta h_\theta(x)\mathrm d\mu(\theta),
\end{equation*}
and the third layer is a linear map $T\in\mathbb{R}^{k\times k}$. Here, $h_\theta(x):= a \sigma(w^\top x)\in\mathbb{R}^k$ for $\theta=(a,w)\in\mathbb{R}^k\times\mathbb{R}^d$, input data point $x \sim \mathcal{D}$, and activation function $\sigma$, whereas the integer $k$ denotes the feature dimension of the hidden representation. The $L^2$ loss with
respect to a teacher network $x\mapsto \hat{T} h_{\hat{\mu}}(x)$ is
\begin{equation*}
    L_{\mathrm{NN}}(\mu,T) = \mathbb{E}_{x\sim \mathcal{D}}\left[\|\hat{T} h_{\hat{\mu}}(x)-Th_\mu(x)\|^2\right].
\end{equation*}
Define $\Sigma_{\mu,\nu}:=\mathbb{E}_{x\sim \mathcal{D}}[h_\mu(x)h_\nu(x)^\top]$ and $\mu^* := \hat{T}_{\#}\hat{\mu}$, so that $h_{\mu^*}(x)=\hat{T}h_{\hat{\mu}}(x)$. In the two-timescale regime where the linear layer updates infinitely fast, the optimal third-layer parameter is $T=\Sigma_{\mu^*,\mu}\Sigma_{\mu,\mu}^{-1}$, and the problem reduces optimizing the functional
\begin{equation*}
    F(\mu) = \frac12\mathbb{E}_{x\sim \mathcal{D}}\left[\left\|h_{\mu^*}(x)-\Sigma_{\mu^*,\mu}\Sigma_{\mu,\mu}^{-1}h_\mu(x)\right\|^2\right].
\end{equation*}
\cite{kim2024transformers} show that, although functional $F$ is non-convex, its landscape is benign on the non-degenerate region where $\Sigma_{\mu,\mu}$ is invertible, i.e., every local minimizer is global, and every non-global critical point is a saddle.
\end{example}

\begin{example}[Mean-field deep ResNet]
\label{example:deep-resnet}
Let $\rho$ be the distribution of the input data $x$, and let $y(x)$ denote the target function. Consider the mean-field continuum model of a deep residual network
\begin{equation*}
    \frac{\mathrm{d}}{\mathrm{d}t} X_\mu(x,t)=\int_\Omega f(X_\mu(x,t),\theta)\mu(\theta,t)\mathrm d\theta, \quad X_\mu(x,0)=\langle w_2,x\rangle.
\end{equation*}
Here $X_\mu(x,t)$ is the hidden feature at depth $t\in[0,1]$ for input $x$, $\theta\in\Omega \subset \mathbb R^{d_2 \times d_2}$ is the parameter of the residual block, where $\Omega$ is the parameter space, $\rho(\theta,t)$ is a probability distribution over residual block parameters and depth, $w_2 \in \mathbb R^{d_1 \times d_2}$ represents the first convolution layer in the ResNet, which extracts features before sending them to the residual blocks, $w_1 \in \mathbb R^{d_2}$ is a fixed linear pooling operator that transfers the final feature $X_\mu(x,1)$ to the classification result, and $f(\cdot,\theta)$ is the residual block with parameter $\theta$ that aims to learn a feature transformation from $\mathbb R^{d_2}$ to $\mathbb R^{d_2}$. For instance, $f(z,\theta)=\sigma(\theta z)$, where $\sigma:\mathbb R \to \mathbb R$ is an activation function applied component-wise. The associated $L^2$ loss is
\begin{equation*}
    F(\mu)=\mathbb{E}_{x\sim\rho}\left[\frac12\left(\langle w_1,X_\mu(x,1)\rangle-y(x)\right)^2\right].
\end{equation*}
\cite{lu2020meanfield} prove that whenever $F(\mu)>0$ there exists $\nu \in \mathcal{P}_2(\Omega)$ such that
\begin{equation*}
    \int_\Omega \frac{\delta F}{\delta \mu}(\mu, \theta, t)\mathrm{d}(\nu-\mu)(\theta) < 0.
\end{equation*}
Thus, every local minimizer of $F$ must satisfy $F(\mu)=0$. Since the loss $F$ is non-negative, every local minimizer is therefore a global minimizer.
\end{example}

\begin{example}[Matrix decomposition]
\label{eg:matrix_decomposition}
    Inspired by the strict benignity of finite-dimensional tensor decomposition \cite{Ge15}, \cite{yamamoto2025hessianguided} study the following mean-field version. Let
    \begin{align*}
        h_{\mu}(z)=\int a\sigma(w^\top z)\mathrm{d}\mu(a,w)
    \end{align*}
    be a mean-field two-layer neural network, where $z \in \mathbb R^l$ is sampled from a given input distribution $\mathcal{D}$, $x=(a,w)\in\mathbb{R}^{k+l}$ denotes the parameter of the network, and $\sigma$ is an activation function. Given a target measure $\mu^*$, define
    \begin{align*}
        F(\mu) = \mathbb{E}_{z\sim \mathcal{D}}\left[
            \left\| h_{\mu^*}(z)h_{\mu^*}(z)^\top - h_{\mu}(z)h_{\mu}(z)^\top\right\|_{\mathrm{F}}^2
        \right],
    \end{align*}
    where $\|\cdot\|_{\mathrm{F}}$ denotes the Frobenius norm. This functional corresponds to learning the rank-one matrix generated by the target representation $h_{\mu^*}(z)h_{\mu^*}(z)^\top$. \cite[Proposition~G.1]{yamamoto2025hessianguided} proves that for an $\varepsilon$-stationary point $\mu^*$ which is not a global minimizer, the objective $F$ can be strictly decreased. This suggests that $F$ possesses benignity.
\end{example}

\begin{proposition}[From qualitative to quantitative saddles]
\label{prop:qualitative-to-quantitative-benignity}
Let \(\mathcal S\subset\mathcal P_2(\mathbb R^d)\) be $W_2$-compact and suppose that \(\mathcal M^*\) is nonempty. Assume that the maps $\mu\mapsto \left\|\nabla_\mu F(\mu,\cdot)\right\|_{L^2_\mu}$ and $\mu\mapsto\underline{\lambda}(\operatorname{K}_\mu)$ are continuous on \(\mathcal S\). Assume further that every second-order critical point in \(\mathcal S\) is globally minimizing, i.e., $\left\|\nabla_\mu F(\mu,\cdot)\right\|_{L^2_\mu}=0,$ and $\underline{\lambda}(\operatorname{K}_\mu)\geq0$ imply $\mu\in\mathcal M^*.$ Then, for every \(\alpha_{\mathrm{ben}}>0\), there exists \(c_{\alpha_{\mathrm{ben}}}>0\) such that $\mu\in\mathcal S,$ $\operatorname{dist}_{W_2}(\mu,\mathcal M^*):=
\inf_{\nu^*\in\mathcal M^*}
W_2(\mu,\nu^*) \geq\alpha_{\mathrm{ben}}$ implies $\left\|\nabla_\mu F(\mu,\cdot)\right\|_{L^2_\mu} \geq c_{\alpha_{\mathrm{ben}}}$ or $\underline{\lambda}(\operatorname{K}_\mu)\leq-c_{\alpha_{\mathrm{ben}}}.$ Hence, \(F\) satisfies the saddle property on \(\mathcal S\) for any $0<\varepsilon,\delta<c_{\alpha_{\mathrm{ben}}}.$
\end{proposition}

\begin{proof}
Define $g(\mu) := \left\|\nabla_\mu F(\mu,\cdot)\right\|_{L^2_\mu},$ $\chi(\mu) := \left(-\underline{\lambda}(\operatorname{K}_\mu)\right)_+,$ where \(a_+:=\max\{a,0\}\), and set $\Phi(\mu) := \max\left\{g(\mu),\chi(\mu)\right\}.$ By assumption, \(g\) and \(\underline{\lambda}\) are continuous on \(\mathcal S\). Hence, \(\chi\) and \(\Phi\) are also continuous on \(\mathcal S\).

Fix \(\alpha_{\mathrm{ben}}>0\), and define $\mathcal C_{\alpha_{\mathrm{ben}}} := \left\{\mu\in\mathcal S: \operatorname{dist}_{W_2}\left(\mu,\mathcal M^*\right)\geq \alpha_{\mathrm{ben}}\right\}.$ The map $\mu \mapsto \operatorname{dist}_{W_2}
\left(\mu,\mathcal M^*\right)$ is \(1\)-Lipschitz with respect to \(W_2\). Indeed, for every \(\mu,\nu\in\mathcal P_2(\mathbb R^d)\),
\begin{equation*}
    \left|\operatorname{dist}_{W_2}(\mu,\mathcal M^*) - \operatorname{dist}_{W_2}(\nu,\mathcal M^*)\right| \leq W_2(\mu,\nu),
\end{equation*}
which follows directly from the triangle inequality. Therefore, \(\mathcal C_{\alpha_{\mathrm{ben}}}\) is a closed subset of the compact set \(\mathcal S\), and is therefore compact. If $\mathcal C_{\alpha_{\mathrm{ben}}} = \varnothing,$ then the asserted implication is vacuous, and we may choose any \(c_{\alpha_{\mathrm{ben}}}>0\). Hence, suppose that \(\mathcal C_{\alpha_{\mathrm{ben}}}\neq\varnothing\). 

We claim that $\Phi(\mu)>0,$ $\mu\in\mathcal C_{\alpha_{\mathrm{ben}}}.$ Indeed, suppose for a contradiction that there exists \(\mu\in\mathcal C_{\alpha_{\mathrm{ben}}}\) such that $\Phi(\mu)=0.$ Since both \(g\) and \(\chi\) are non-negative, this implies $g(\mu)=0,$ and $\chi(\mu)=0$, i.e., $\left\|\nabla_\mu F(\mu,\cdot)\right\|_{L^2_\mu}=0,$ $\underline{\lambda}(\operatorname{K}_\mu)\geq0.$ By assumption, it follows that $\mu\in\mathcal M^*.$

However, this would imply $\operatorname{dist}_{W_2}\left(\mu,\mathcal M^*\right)=0,$ contradicting $\mu\in\mathcal C_{\alpha_{\mathrm{ben}}},$ and $\alpha_{\mathrm{ben}}>0.$ Thus, $\Phi(\mu)>0$ throughout \(\mathcal C_{\alpha_{\mathrm{ben}}}\). Since \(\Phi\) is continuous and \(\mathcal C_{\alpha_{\mathrm{ben}}}\) is nonempty and compact, \(\Phi\)
attains its minimum on this set. Define
\begin{equation*}
    c_{\alpha_{\mathrm{ben}}}:=\min_{\mu\in\mathcal C_{\alpha_{\mathrm{ben}}}}
\Phi(\mu).
\end{equation*}
The strict positivity proved above implies $c_{\alpha_{\mathrm{ben}}}>0.$

Now let \(\mu\in\mathcal S\) satisfy $\operatorname{dist}_{W_2}\left(\mu,\mathcal M^*\right) \geq \alpha_{\mathrm{ben}}.$ Then \(\mu\in\mathcal C_{\alpha_{\mathrm{ben}}}\), and hence
\begin{equation*}
    \max\left\{\left\|\nabla_\mu F(\mu,\cdot)\right\|_{L^2_\mu},\left(-\underline{\lambda}(\operatorname{K}_\mu)\right)_+\right\}\geq c_{\alpha_{\mathrm{ben}}}.
\end{equation*}
It follows that either $\left\|\nabla_\mu F(\mu,\cdot)\right\|_{L^2_\mu} \geq c_{\alpha_{\mathrm{ben}}},$ or $\left(-\underline{\lambda}(\operatorname{K}_\mu)\right)_+ \geq c_{\alpha_{\mathrm{ben}}}.$ In the latter case, since \(c_{\alpha_{\mathrm{ben}}}>0\), we necessarily have $-\underline{\lambda}(\operatorname{K}_\mu) \geq c_{\alpha_{\mathrm{ben}}},$ and therefore $\underline{\lambda}(\operatorname{K}_\mu) \leq -c_{\alpha_{\mathrm{ben}}}.$ This proves the first assertion.

Finally, let $0<\varepsilon,\delta<c_{\alpha_{\mathrm{ben}}}.$ For every \(\mu\in\mathcal S\) satisfying $\operatorname{dist}_{W_2} \left(\mu,\mathcal M^* \right)\geq \alpha_{\mathrm{ben}},$ the preceding dichotomy gives either $\left\|\nabla_\mu F(\mu,\cdot)\right\|_{L^2_\mu} \geq c_{\alpha_{\mathrm{ben}}} > \varepsilon,$ or $\underline{\lambda}(\operatorname{K}_\mu)
\leq -c_{\alpha_{\mathrm{ben}}} < -\delta.$ Hence, for every \(\mu\in\mathcal S\), at least one of the following holds: $\left\|\nabla_\mu F(\mu,\cdot)\right\|_{L^2_\mu}> \varepsilon,$ $\underline{\lambda}(\operatorname{K}_\mu)<-\delta,$ or $\operatorname{dist}_{W_2}\left(\mu,\mathcal M^*\right)<\alpha_{\mathrm{ben}}.$ Thus, \(F\) satisfies Assumption \ref{def:quantitative-strict-saddle} on \(\mathcal S\) for every $0<\varepsilon,\delta<c_{\alpha_{\mathrm{ben}}}.$
\end{proof}

\subsection{Landscape analysis of a mean-field residual flow}
We analyze the specific residual network architecture with rotational weights from Example \ref{example:deep-resnet}, which gives a nonconvex objective that globally satisfies all regularity assumptions, possesses critical points with negative curvature, and satisfies the quantitative benignity condition.
\begin{proposition}[A globally regular rotational residual flow]
\label{prop:rotational-resnet-regularity}
Consider a mean-field residual flow with scalar parameter $\theta\in\mathbb R$, state in $\mathbb R^2$, and residual vector field
\begin{equation*}
    f(z,\theta)=\theta Jz, \quad J=\begin{pmatrix}0 & -1\\1 & 0\end{pmatrix}.
\end{equation*}
Let the input distribution $D$ satisfy $\mathbb E_{x\sim D}[xx^\top]=I_2$. With readout $e_1$ and target $y(x)=x_1$, define
\begin{equation*}
    F(\mu):=\frac12\mathbb E_{x\sim D}\left[\left(e_1^\top X_\mu(x,1)-x_1\right)^2\right].
\end{equation*}
Hence, Assumptions~\ref{assumption_regularity_wg}, \ref{assumption:smooth-F}, and~\ref{assumption:Hessian-lipschitz} hold globally. One admissible choice of positive constants is $R_F=\varepsilon_0=C_{\operatorname M}=C_{\operatorname K}=L_{\operatorname M}=L_{\operatorname K}=1.$ Moreover, every nonglobal critical point has a negative eigenvalue equal to $-1$ in its integral Hessian. Finally, for every $0<\alpha_{\mathrm{ben}}<\frac{\pi}{2},$ $0<\delta<1,$ $0<\varepsilon<\min\left\{\sin\alpha_{\mathrm{ben}},\sqrt{1-\delta^2}\right\},$ Assumption~\ref{def:quantitative-strict-saddle} holds on all of $\mathcal P_2(\mathbb R)$.
\end{proposition}

\begin{proof}
Write $m_\mu:=\int_{\mathbb R}\theta\mathrm d\mu(\theta)$. Since $\mu\in\mathcal P_2(\mathbb R)$, its mean $m_\mu$ is finite. The parameter distribution $\mu\in\mathcal P_2(\mathbb R)$ is the same at every depth, i.e.,
\begin{equation*}
    \frac{\mathrm d}{\mathrm dt}X_\mu(x,t)=\int_{\mathbb R}\theta JX_\mu(x,t)\mathrm d\mu(\theta), \quad X_\mu(x,0)=x, \quad t\in[0,1].
\end{equation*}
The state equation reduces to
\begin{equation*}
    \frac{\mathrm d}{\mathrm dt}X_\mu(x,t) =  m_\mu JX_\mu(x,t).
\end{equation*}
Its solution is
\begin{equation*}
    X_\mu(x,t)=R_{tm_\mu}x, \quad R_s= \begin{pmatrix}\cos s & -\sin s\\\sin s & \cos s\end{pmatrix}.
\end{equation*}
Therefore, using $\mathbb E_{x\sim D}[xx^\top]=I_2$,
\begin{equation*}
    F(\mu) =\frac12\mathbb E_{x\sim D}\left[\left((\cos m_\mu-1)x_1-\sin m_\mu x_2\right)^2\right] = \frac12\left[(\cos m_\mu-1)^2+\sin^2m_\mu
    \right] = 1-\cos m_\mu.
\end{equation*}
This proves $0\leq F\leq2$.

For any $\mu,\nu\in\mathcal P_2(\mathbb R)$ and $\gamma\in\Pi(\mu,\nu)$, we have $ m_\nu-m_\mu = \int_{\mathbb R^2}(\vartheta-\theta) \mathrm d\gamma(\theta,\vartheta),$ so $|m_\nu-m_\mu|\leq d_\gamma.$ In particular, $|m_\nu-m_\mu|\leq W_2(\mu,\nu)$. Taylor expansion of $s\mapsto1-\cos s$ gives, for every optimal coupling $\gamma\in\Pi_o(\mu,\nu)$,
\begin{equation*}
    F(\nu)-F(\mu)  = \int_{\mathbb R^2}\sin m_\mu(\vartheta-\theta)\mathrm d\gamma(\theta,\vartheta) + O\left(W_2^2(\mu,\nu)\right).
\end{equation*}
Hence $\nabla_\mu F(\mu,\theta)=\sin m_\mu$. Differentiating this expression in the spatial and measure variables gives $ \nabla\nabla_\mu F(\mu,\theta)=0,$ $\nabla_\mu^2F(\mu,\theta,\vartheta)=\cos m_\mu.$ These derivatives have the symmetry and integrability required by Wasserstein regularity.

Let $q_\mu(\theta)\equiv1$ and denote by $P_\mu$ the orthogonal projection onto the constant functions in $L^2_\mu$, i.e., $P_\mu v = \left(\int_{\mathbb R}v\,\mathrm d\mu\right)q_\mu.$ Since $\|q_\mu\|_{L^2_\mu}=1$, the Hessian components are $\operatorname M_\mu=0,$ $\operatorname K_\mu=\cos m_\mu P_\mu.$ Thus Assumption~\ref{assumption_regularity_wg} holds for any positive $R_F$ and
$\varepsilon_0$. Moreover, $\|\operatorname M_\mu\|_{\mathrm{op}}=0,$ $\|\operatorname K_\mu\|_{\mathrm{HS}} = |\cos m_\mu| \leq 1,$ which proves Assumption~\ref{assumption:smooth-F}.

The multiplication component inequality in Assumption~\ref{assumption:Hessian-lipschitz} holds because its left-hand side vanishes. For the integral component, its kernel is constant in both spatial arguments, so
\begin{equation*}
    \left\|\nabla_\mu^2F(\nu,\cdot,\cdot)-\nabla_\mu^2F(\mu,\cdot,\cdot)
    \right\|_{L^2_{\gamma\otimes\gamma}} = |\cos m_\nu-\cos m_\mu| \leq |m_\nu-m_\mu| \leq d_\gamma.
\end{equation*}
This proves Assumption~\ref{assumption:Hessian-lipschitz} with the stated constants.

The characterization of $\mathcal M^*$ follows from $F(\mu)=1-\cos m_\mu$. The global minimizer set is $\mathcal M^* = \left\{\mu\in\mathcal P_2(\mathbb R):m_\mu\in2\pi\mathbb Z\right\},$ and $\operatorname{dist}_{W_2}(\mu,\mathcal M^*) = \operatorname{dist}(m_\mu,2\pi\mathbb Z).$ For any $\nu\in\mathcal M^*$, we get
\begin{equation*}
    W_2(\mu,\nu) \geq |m_\mu-m_\nu| \geq \operatorname{dist}(m_\mu,2\pi\mathbb Z).
\end{equation*}
Conversely, choose $a\in2\pi\mathbb Z$ nearest to $m_\mu$
and set $\nu := \left(\operatorname{Id}+a-m_\mu\right)_\#\mu.$ Then $m_\nu=a$, so $\nu\in\mathcal M^*$, and the translation coupling gives $W_2(\mu,\nu) \leq |a-m_\mu|.$ Together, these inequalities prove the asserted distance identity.

The critical points satisfy $\sin m_\mu=0$. At a nonglobal critical point, $m_\mu=(2k+1)\pi$ for some $k\in\mathbb Z$, and therefore $\operatorname K_\mu q_\mu=-q_\mu.$ In particular, these critical points have negative curvature. The objective is nonconvex even along translation geodesics, since at such a measure $F\left((\operatorname{Id}+s)_\#\mu\right)=1+\cos s.$

Finally, let $r:=\operatorname{dist}(m_\mu,2\pi\mathbb Z)\in[0,\pi].$ By the distance identity, $r=\operatorname{dist}_{W_2}(\mu,\mathcal M^*)$, and $\|\nabla_\mu F(\mu,\cdot)\|_{L^2_\mu}=\sin r,$ $ \operatorname K_\mu q_\mu=\cos r q_\mu.$ Suppose $r>\alpha_{\mathrm{ben}}$. If $\cos r<-\delta$, then $\lambda(\operatorname K_\mu)<-\delta$. Otherwise, $\alpha_{\mathrm{ben}}<r\leq\arccos(-\delta),$ and hence
\begin{equation*}
    \|\nabla_\mu F(\mu,\cdot)\|_{L^2_\mu} = \sin r \geq\min\left\{ \sin\alpha_{\mathrm{ben}},\sqrt{1-\delta^2}\right\} > \varepsilon.
\end{equation*}
This proves Assumption~\ref{def:quantitative-strict-saddle} holds on $\mathcal P_2(\mathbb R)$.
\end{proof}

\section{Particle Implementation of WSFN}
\label{section:practical implementation}
In this appendix, we describe a particle-based realization of the WSFN scheme \eqref{eq:wsfn-update-main}. The goal is to translate the measure-valued Algorithm \ref{alg:episode-wsfn} into an algorithm acting on finitely many particles. This provides a practical approximation of the dynamics analyzed in the paper.

\begin{algorithm}[t]
\DontPrintSemicolon
\caption{WSFN for approximate second-order stationarity}
\label{alg:episode-wsfn}
\KwIn{Initial measure $\mu^0\in\mathcal P_2(\mathbb R^d)$; thresholds $\varepsilon,\delta>0$; regularization $\beta_\mathrm{grad}>0,$ $\beta_\mathrm{sad}=\Theta(\delta^2)$; stepsize $\tau_\mathrm{grad} \leq C_{\operatorname H}^{-1}\sqrt{\beta_{\mathrm{grad}}},$ $\tau_\mathrm{sad}\in (0,\frac{1}{2}]$; perturbation amplitude $\eta>0$; maximum episode length $n_{\mathrm{out}}\in\mathbb N$; required decrease $F_0>0$
}
\KwOut{An approximate second-order critical point $\widehat\mu$}
$\mu := \mu^0$, $N_{\mathrm{step}} := 0$\;
\While{\textnormal{true}}{
    \eIf{$\|\nabla_\mu F(\mu,\cdot)\|_{L^2_\mu}>\varepsilon$}{
        $\mu = \left(\operatorname{Id}-\tau_\mathrm{grad}(\operatorname H_{\mu}^2+\beta_\mathrm{grad} \operatorname{I_{d \times d}})^{-1/2}\nabla_\mu F(\mu,\cdot)\right)_{\#}\mu$\;
        $N_{\mathrm{step}} = N_{\mathrm{step}}+1$\;
    }{
        Store the pre-perturbation iterate $\bar\mu=\mu,$ $\bar F= F(\bar\mu).$
        
        Draw a fresh perturbation $\xi\sim\mathcal{N}_{L_{\bar \mu}^2}(0,\operatorname{K}_{\bar\mu}^2),$ and initialize the episode by
        \[
        \nu^0 = \left(\operatorname{Id}+\eta\xi\right)_\#\bar\mu.
        \]
        
        $\ell= 0$\;
        \While{$\ell<n_{\mathrm{out}}$ \textnormal{and} $\bar F-F(\nu^\ell)<F_0$}{$\displaystyle \nu^{\ell+1}=\left(
        \operatorname{Id}-\tau_{\mathrm{sad}}\left(
            \operatorname{H}_{\nu^\ell}^2 +\beta_{\mathrm{sad}}\operatorname{I_{d\times d}}\right)^{-1/2}\nabla_\mu F(\nu^\ell,\cdot)\right)_\#\nu^\ell$\;
        $\ell=\ell+1$\;
        $N_{\mathrm{step}}= N_{\mathrm{step}}+1$\;}
    \eIf{$\bar F-F(\nu^\ell)\geq F_0$}{
    Accept the episode output $\mu=\nu^\ell$\;}{$\widehat\mu=\bar\mu$\;
    \Return{$(\widehat\mu,N_{\mathrm{step}})$}\;}
    }
}
\end{algorithm}

Algorithm~\ref{alg:particle-episode-wsfn} implements WSFN using the empirical measure $\mu=N^{-1}\sum_{j=1}^N\delta_{x_j}$. When the empirical Wasserstein gradient norm exceeds $\varepsilon$, all particles are updated using the full Hessian preconditioned direction with parameters $(\tau_{\mathrm{grad}},\beta_{\mathrm{grad}})$. Otherwise, the algorithm stores the current particle configuration and initiates a perturbed saddle episode, without explicitly testing for negative curvature.

The perturbations are sampled jointly across particles using the covariance kernel $C_\mu$. Writing $k_\mu(x,y):=\nabla_\mu^2F(\mu,x,y)$, this kernel is related to $\operatorname K_\mu^2$ by
\begin{equation*}
    C_\mu(x,y) = \int_{\mathbb R^d} k_\mu(x,z)k_\mu(y,z)^\top\mathrm d\mu(z),
    \quad
    (\operatorname K_\mu^2v)(x) = \int_{\mathbb R^d}C_\mu(x,y)v(y)\mathrm d\mu(y),
\end{equation*}
where the second identity uses the symmetry of $k_\mu$. Thus, $C_\mu$ is the matrix-valued covariance kernel whose integral operator is $\operatorname K_\mu^2$, as in Lemma~\ref{lem:Gaussian-L2}. For the empirical measure, the sampled Gaussian vector has block covariance
\begin{equation*}
    [\mathbf C_\mu]_{ij} = \frac1N\sum_{k=1}^N k_\mu(x_i,x_k)k_\mu(x_j,x_k)^\top.
\end{equation*}
The perturbation uses the integral component of the Hessian, whereas the subsequent transport uses the full Hessian $\operatorname H_\mu$.

After perturbing the stored particles, the algorithm performs at most $n_{\mathrm{out}}$ WSFN steps with $\beta_{\mathrm{sad}}=\Theta(\delta^2)$ and $\tau_{\mathrm{sad}}\in(0,1/2]$. It checks the objective after the perturbation and after every subsequent update, accepting the first configuration whose objective is at least $F_0$ below the stored pre-perturbation value. If the criterion is never met within the episode budget, the algorithm terminates and returns the stored particles. The counter $N_{\mathrm{step}}$ records the transport steps actually performed, including those in the terminal unsuccessful episode.

\begin{algorithm}
\DontPrintSemicolon
\caption{Particle approximation of WSFN}
\label{alg:particle-episode-wsfn}
\KwIn{Particles $\{x_j^0\}_{j=1}^N$; thresholds $\varepsilon,\delta>0$; regularization $\beta_\mathrm{grad}>0$, $\beta_\mathrm{sad} = \Theta(\delta^2)$; step size $\tau_\mathrm{grad}\leq C_{\operatorname H}^{-1}\sqrt{\beta_{\mathrm{grad}}}$, $\tau_\mathrm{sad}\in (0,1/2]$; perturbation amplitude $\eta>0$; maximum episode length $n_{\mathrm{out}}\in\mathbb N$; required decrease $F_0>0$
}
\KwOut{An empirical approximate second-order critical point $\widehat\mu$
}
$x_j:= x_j^0$ for $j=1,\ldots,N,$ $\mu:= \frac1N\sum_{j=1}^N\delta_{x_j},$
$N_{\mathrm{step}}:=0$\;

\While{\textnormal{true}}{
    \eIf{$\left(\frac1N\sum_{j=1}^N\|\nabla_\mu F(\mu,x_j)\|^2
    \right)^{1/2}>\varepsilon$}{
        \For{$j=1,\ldots,N$}{
            $x_j\gets x_j-\tau_\mathrm{grad} \left[\left(\operatorname{H}_{\mu}^2+\beta_\mathrm{grad}\operatorname{I_{d \times d}}\right)^{-1/2}\nabla_\mu F(\mu,\cdot)\right](x_j)$\;
        }
        $\mu = \frac1N\sum_{j=1}^N\delta_{x_j}.$\;
        $N_{\mathrm{step}}\gets N_{\mathrm{step}}+1$\;
    }{
        Store the pre-perturbation particles $\bar x_j= x_j,$ $ \bar\mu=
        \frac1N\sum_{j=1}^N\delta_{\bar x_j},$ $\bar F= F(\bar\mu).$
        
        Draw the jointly Gaussian vector $\left(\xi(\bar x_1),\ldots,\xi(\bar x_N)\right) \sim \mathcal N(0,\mathbf C_{\bar\mu}),$ where the $(i,j)$ block of $\mathbf C_{\bar\mu}$ is $[\mathbf C_{\bar\mu}]_{ij}
        = C_{\bar\mu}(\bar x_i,\bar x_j).$
        
        Initialize the temporary episode particles by
        \[
        y_j^0 = \bar x_j+\eta\xi(\bar x_j), \quad j=1,\ldots,N,
        \]
        and set $\nu^0 = \frac1N\sum_{j=1}^N\delta_{y_j^0}.$

        $\ell= 0$\;
        \While{$\ell<n_{\mathrm{out}}$\textnormal{ and }$\bar F-F(\nu^\ell)<F_0$}{
            \For{$j=1,\ldots,N$}{$\displaystyle y_j^{\ell+1}= y_j^\ell
        -\tau_{\mathrm{sad}}\left[\left(\operatorname{H}_{\nu^\ell}^2 +\beta_{\mathrm{sad}}\operatorname{I_{d\times d}}\right)^{-1/2}\nabla_\mu F(\nu^\ell,\cdot)\right](y_j^\ell)$\;}
    $\displaystyle \nu^{\ell+1}= \frac1N\sum_{j=1}^N\delta_{y_j^{\ell+1}}$\;
    $\ell=\ell+1$\;
    $N_{\mathrm{step}}= N_{\mathrm{step}}+1$\;}
    \eIf{$\bar F-F(\nu^\ell)\geq F_0$}{$x_j= y_j^\ell$ for $j=1,\ldots,N$\;
    $\mu=\nu^\ell$\;}{$\widehat\mu=\bar\mu$\;
    \Return{$(\widehat\mu,N_{\mathrm{step}})$}\;}
    }
}
\end{algorithm}

\subsection{Numerical experiments}
We now illustrate the particle implementation of WSFN on a set of numerical experiments. These are not intended as a comprehensive empirical evaluation or as evidence of state-of-the-art practical performance. A detailed study of scalable approximations to the WSFN preconditioner is left for future work. The hardware and software used for the reported experiments are an Apple M4 laptop with 10-core CPU and PyTorch 2.11.
\\

\paragraph{Computational cost and future directions.} Across the reported experiments, the implemented WSFN variants attain lower or comparable losses at the common update horizons, although this behavior does not consistently translate into an advantage in elapsed time. This pattern is consistent with the additional cost of constructing and applying curvature operators and approximating the action of the regularized inverse square root preconditioner. The numerical illustrations suggest a benefit from preconditioning in terms of progress per update in several examples, while highlighting a computational limitation of the current implementations. Although the primary contribution of this work is theoretical, these observations motivate further investigation of efficient curvature representations and preconditioner applications, together with control of their approximation errors. Showing when the resulting savings translate into improved optimization performance under a common runtime budget remains an important direction for future work.

\subsection{Regularized Coulomb MMD: implementation details} For the experiment in Section~\ref{sec:coulomb-experiment}, the $M=400$ target points are sampled independently from the equally weighted mixture $\frac14\sum_{\ell=1}^4\mathcal N(c_\ell,0.25^2 I_3),$ with $c_1=(2,0,1),$ $c_2=(-2,0,-1),$ $c_3=(0,2,-1),$ $c_4=(0,-2,1).$

The $N=500$ initial particles are independent draws from $\mathcal N(0,0.05^2 I_3)$. Each trial uses a new target sample and initial particle configuration, shared by all methods within that trial. We evaluate \eqref{eq:coulomb-objective} with $\rho=0.05$, including the student and target diagonal terms. The target self-interaction term is constant during optimization and is computed once per trial. The numerical objective uses a smooth regularized kernel instead of the singular Coulomb kernel.
\\

\paragraph{Gradient and full Hessian actions.} Equip particle vectors with the norm $\|v\|_N=(N^{-1}\sum_i\|v_i\|^2)^{1/2}$. Writing $F=F_{\rho,N,M}(\cdot;Y)$, the particle gradient is $g=N\nabla_X F$, and the Hessian for this metric is $H_N=N\nabla_X^2 F$. For $u\in\mathbb R^3$, define
\begin{equation*}
    s_\rho(u)=\|u\|^2+\rho^2,\quad a_\rho(u)=\frac{u}{s_\rho(u)^{3/2}},\quad B_\rho(u)=s_\rho(u)^{-3/2} \left(I_3-\frac{3uu^\top}{s_\rho(u)}\right).
\end{equation*}
Then
\begin{equation*}
    g_i=-\frac2N\sum_j a_\rho(x_i-x_j) +\frac2M\sum_a a_\rho(x_i-y_a),
\end{equation*}
and the integral component and full Hessian act as
\begin{equation*}
(K_Nv)_i =\frac2N\sum_jB_\rho(x_i-x_j)v_j,
\end{equation*}
\begin{align}
(H_Nv)_i =-\frac2N\sum_jB_\rho(x_i-x_j)(v_i-v_j)+\frac2M\sum_aB_\rho(x_i-y_a)v_i.
\label{eq:coulomb-h-action}
\end{align}
The difference form in \eqref{eq:coulomb-h-action} cancels student self-interactions exactly. The implementation caches pairwise terms and evaluates these actions without assembling a dense particle Hessian.
\\

\paragraph{Transport directions and scalar control.} WGF, isotropic WGF, and PWGF use the direction $g$. WSFN uses $p_J\approx(H_N^2+\beta I_{3N})^{-1/2}g,$ computed by Lanczos iterations on $H_N^2+\beta \operatorname{I_{{3N} \times 3N}}$ with full reorthogonalization, a maximum of $J$ iterations, and an absolute Lanczos residual stopping tolerance $\epsilon_{\mathrm L}$. Computations check for nonfinite quantities and violations of the known lower spectral bound $\beta$. At its own current particle configuration, the scalar control computes the same $p_J$ but takes the direction
\begin{equation*}
    p_{\mathrm{sc}}= \begin{cases} \displaystyle\frac{\|p_J\|_N}{\|g\|_N}g,&\|g\|_N>0,\\
    0,&\|g\|_N=0.
\end{cases}
\end{equation*}
It shares WSFN's learning rate, $\beta$, Lanczos settings, and perturbation rule. Thus the control matches the local direction norm and incurs the curvature computation needed for that match. It is not a separately tuned inexpensive scalar method.
\\

\paragraph{Perturbations.} At updates $100,200,\ldots$, a vector $z$ with independent $\mathcal N(0,I_3)$ particle blocks generates the kick
\begin{equation*}
    r_{\mathrm{iso}}=\eta_{\mathrm{iso}}z, \quad r_{\mathrm H}=\eta_{\mathrm H}\|z\|_N\frac{K_Nz}{\|K_Nz\|_N}.
\end{equation*}
We set $r_{\mathrm H}=0$ when $K_Nz=0$. Isotropic WGF uses $r_{\mathrm{iso}}$, and PWGF, WSFN, and the scalar control use $r_{\mathrm H}$. WGF receives no kicks. Within each paired trial, the same Gaussian vector is used at the same kick index, while $K_N$ is evaluated at each method's own state. Both the transport direction and the kick are computed before the update $X^+=X-\tau p+r$. Away from prescribed events, $r=0$.

For comparison, the unnormalized theoretical perturbation has the particle realization $\xi_i=\frac1{\sqrt N}\sum_j \nabla_\mu^2F_\rho(\mu^N,x_i,x_j)z_j,$ with $\xi=\sqrt N K_Nz,$ whose block covariance is the kernel associated with $\operatorname{K}_{\mu^N}^2$. The implemented norm normalization changes this distribution, so the kicks above do not have that exact covariance in general. Together with the prescribed timing, this makes the experiment an ablation of the numerical directions and perturbations, instead of implementing the full stopping and episode rules.
\\

\paragraph{Tuning, timing, and reporting.}
The learning rates of WGF, isotropic WGF, PWGF, and WSFN are selected separately by minimizing the mean terminal relative loss on pilot trials. The pilot search uses equal numbers of candidate rates and a $20$-second budget per candidate and pilot seed. The selected rates are fixed for the five paired evaluation trials, each with a $120$-second budget per method and no simultaneous update cap. Run order is randomized. The scalar control inherits WSFN's selected rate. The final selected values are recorded in Table~\ref{tab:coulomb-run-settings}.
\begin{table}[t]
\centering
\small
\begin{tabular}{p{0.39\linewidth}p{0.55\linewidth}}
\hline
Setting & Value in the experiment\\
\hline
Selected learning rates 
& WGF, IsoWGF, PWGF: 0.01; WSFN: 0.3\\
Preconditioner $(\beta,J,\epsilon_{\mathrm L})$
& $\beta=10^{-5}$, $J=12$, $\epsilon_L=10^{-6}$\\
Kick scales $(\eta_{\mathrm{iso}},\eta_{\mathrm H})$
& $\eta_{\mathrm{iso}}=\eta_{\mathrm H}=0.1$\\
\hline
\end{tabular}
\caption{Specific settings for the regularized Coulomb MMD experiment.}
\label{tab:coulomb-run-settings}
\end{table}
Timing includes gradient and curvature calculations, Lanczos iterations, kicks, numerical checks, and diagnostic evaluations. Data generation, the target self-interaction calculation, warm-up iterations, and plotting are excluded. The update panels use a common observed horizon, while the time panels use the common $120$-second budget. Gradient panels report $\|g\|_N$, not the norm of the preconditioned direction.

\begin{figure}[t]
    \centering
    \includegraphics[width=\linewidth]{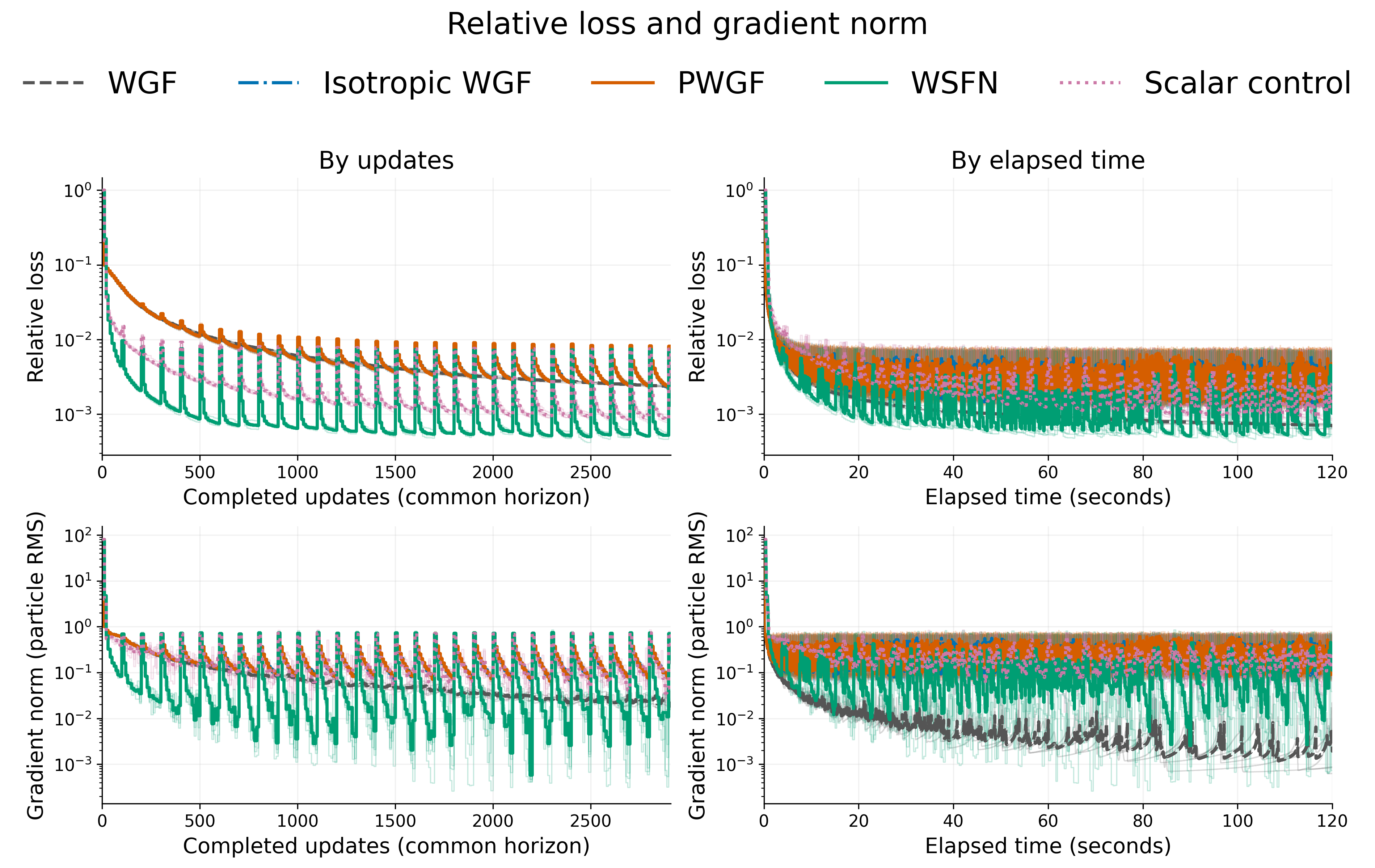}
    \caption{Regularized Coulomb MMD: relative loss (top) and particle RMS gradient norm $\|g\|_N$ (bottom), plotted against completed updates over a common horizon (left) and elapsed time (right). Thick curves show means over five paired trials. Thin curves show individual trials. All vertical axes are logarithmic. The upper panels contain the same loss data as Figure~\ref{fig:coulomb-loss}.}
    \label{fig:coulomb-diagnostics}
\end{figure}

\paragraph{Interpretation.} Figure~\ref{fig:coulomb-diagnostics} shows that WSFN reaches low relative losses in fewer updates, with faster recovery after the prescribed kicks, and this advantage is also visible in the early elapsed time loss curves. Its lower loss between perturbations relative to the scalar control supports a directional benefit under the shared WSFN settings. This is a conditional ablation since the control has not received an independent learning rate search.

The recurring loss and gradient peaks in the perturbed methods align with the prescribed kick schedule. Faster methods encounter more such events within the same time budget, producing dense fluctuations in the time panels. These excursions exclude a claim of uniform pointwise dominance or monotone loss decrease. WGF has no prescribed kicks and maintains smaller gradient norms for much of the late elapsed time interval, despite WSFN's favorable loss reduction. Hence, the experiment illustrates an advantage in reaching low loss under these settings, without showing uniformly faster convergence to stationarity or empirically verifying the theoretical saddle escape complexity bounds.

\subsection{In-context feature learning: implementation details}
We consider the in-context feature learning objective from Example~\ref{example:ICL},
\begin{equation*}
    F(\mu) =\frac12\operatorname{Tr}(\Sigma_{\mu^*,\mu^*})-\frac12\operatorname{Tr}\left(\Sigma_{\mu^*,\mu}\Sigma_{\mu,\mu}^{-1}\Sigma_{\mu,\mu^*}\right).
\end{equation*}
The teacher network is fixed, and the student is represented by an empirical measure with $N=400$ particles. We use input dimension $d=15$, feature dimension $k=5$, and $n_{\mathrm p}=300$ training prompts, on which the covariance matrices are estimated. Each evaluation seed specifies a data set and student initialization shared by all methods. The five evaluation seeds are disjoint from the two pilot seeds used for learning rate selection. The teacher seed is $1729$, the pilot seeds are $1101,1102$, and the evaluation seeds are $2101,\ldots,2105$.
\\

\paragraph{Implemented curvature operator and scalar control.} The implemented WSFN variant uses the symmetrized block curvature surrogate $\widehat H_N$, which is the full Wasserstein Hessian. With $\beta\in \{10^{-5}, 10^{-4}, 10^{-3}\}$, learned jointly with the learning rate $\tau,$ its direction is computed by $p_N\approx(\widehat H_N^2+\beta \operatorname{I_{d\times d}})^{-1/2}g_N,$ where $g_N$ is the gradient of the implemented loss in stored particle coordinates. We use at most $10$ Lanczos iterations with tolerance $10^{-6}$. At its own current state, the scalar control replaces $p_N$ by
\begin{equation*}
    d_N^{\mathrm{sc}} =\frac{\|p_N\|_N}{\|g_N\|_N}g_N, \quad \|v\|_N^2=\frac1N\sum_{i=1}^N\|v_i\|^2,
\end{equation*}
with $d_N^{\mathrm{sc}}=0$ when $g_N=0$. It inherits WSFN's selected learning rate and preconditioner settings, including the cost of computing $p_N$, and serves as a diagnostic of the contribution of the preconditioned direction.
\\

\paragraph{Perturbations.} For this controlled ablation, perturbed methods receive prescribed kicks every $5$ updates. In particle coordinates, the displacements are $\frac{1}{2N}Z$ for isotropic WGF and $0.0015\widehat H_N Z$ for PWGF, WSFN, and the scalar control, where $Z$ is standard Gaussian. Matching kick indices use common Gaussian draws, with each curvature operator evaluated at its own pre-update state. WGF receives no perturbations.
\\

\paragraph{Evaluation and timing.} Let $S,T\in\mathbb R^{n_{\mathrm p}\times k}$ contain the student and teacher features, respectively. When $S$ has full column rank, a reduced QR factorization $S=QR$ gives the implemented loss
\begin{equation*}
    L_N=\frac{1}{2n_{\mathrm p}}\|T-QQ^\top T\|_{\mathrm F}^2.
\end{equation*}
Gradients are obtained by differentiating this residual formulation.

We compare WGF, WGF with isotropic Gaussian perturbations, PWGF, WSFN, and a scalar control. Learning rates are tuned separately for the first four methods using equal sized candidate grids and $20$-second pilot runs on two seeds. Selection minimizes the mean relative loss across the pilot seeds. The selected rates (Table \ref{tab:icl-selected-rates}) are then frozen for five disjoint evaluation trials, each with a $120$-second budget per method. Within each trial, all methods share the same prompts and initial student parameters. Run order is randomized. Elapsed time includes optimization, curvature calculations, perturbations. Model and data setup, and disposable warm-up runs are excluded. Post deadline loss improvements receive no credit in the fixed-time comparison.

\begin{table}[t]
\centering
\caption{Learning rates selected from pilot runs for the reported evaluation.}
\label{tab:icl-selected-rates}
\begin{tabular}{lcc}
\hline
Method & Selected learning rate & Selected $\beta$ \\
\hline
WGF & 0.03 & N/a\\
Isotropic WGF & 0.03 & N/a \\
PWGF & 0.01 & N/a \\
WSFN & 0.000095 & $\beta=0.00001$ \\
Scalar control & Same as WSFN & Same as WSFN\\
\hline
\end{tabular}
\end{table}

\begin{figure}[t]
    \centering
    \includegraphics[width=\linewidth]{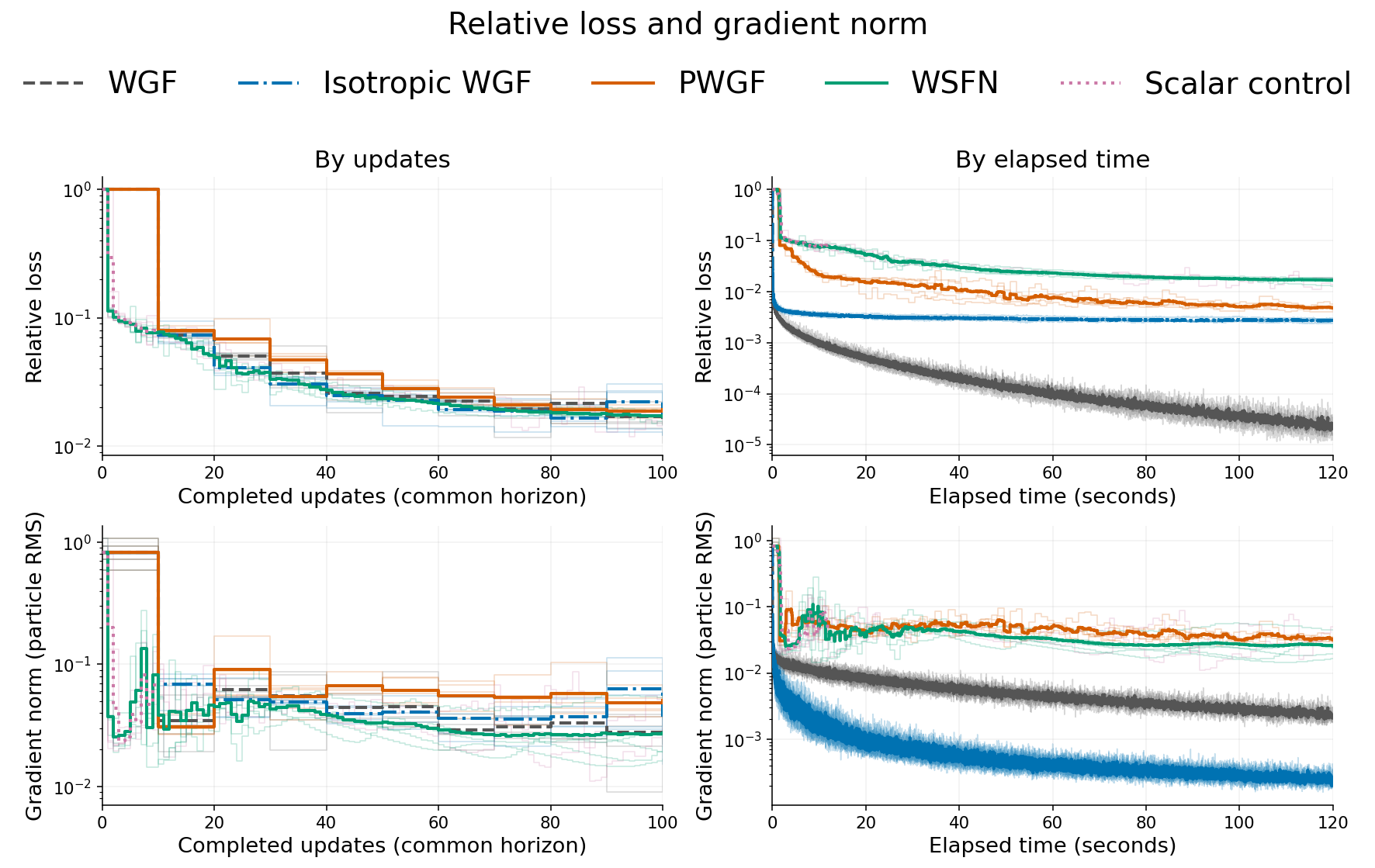}
    \caption{In-context feature learning with $N=400$ particles over five paired evaluation trials. Top: loss relative to its initial value. Bottom: particle RMS norm of the implemented gradient. Left: completed updates, restricted to the common observed horizon. Right: elapsed time, with a $120$-second budget per method and trial. Thick curves show arithmetic means across all five trials. Thin curves show individual trials.
    }
    \label{fig:icl-wsfn}
\end{figure}

\paragraph{Interpretation.} The ICL experiment shows no clear advantage of the implemented WSFN variant at the common update horizon. WSFN achieves somewhat lower losses than the first-order baselines over parts of the trajectory, but the curves approach one another as the update count increases. The scalar control remains comparable to WSFN and attains a slightly lower mean loss near the endpoint. Together with the variation across individual trials, this comparison provides no clear evidence of an additional benefit from directional preconditioning at the selected settings.

The elapsed time comparison strongly favors the first-order methods. WGF attains a loss almost three orders of magnitude smaller than WSFN within the same runtime. In this implementation, the cost of curvature construction and preconditioning is not compensated by a sufficient improvement per update.

The gradient norm panels are consistent with this comparison. At matched update counts, WSFN reaches a somewhat smaller reported particle RMS gradient norm near the endpoint, but its loss remains comparable to the other methods. Within the elapsed time budget, WGF and isotropic WGF reach substantially smaller gradient norms than WSFN, PWGF, and the scalar control. The perturbed trajectories exhibit continuing fluctuations, so individual peaks should be distinguished from the underlying decrease.

These results identify a limitation of the implemented curvature surrogate method on this ICL instance. Its computational overhead is substantial, and the scalar ablation does not reveal a clear directional advantage. The experiment does not prove a general ordering of the methods or provide direct evidence concerning saddle escape rates or second-order stationarity.

\subsection{Mean-field deep residual networks: implementation details.} We consider the mean-field deep ResNet of Example~\ref{example:deep-resnet} with residual block $f(z,\Theta)=\tanh(\Theta z)$, where the activation is applied componentwise. We use input dimension $d_1=15$, hidden dimension $d_2=4$, $L=8$ depth slices, $N=24$ matrix-valued particles per slice, and $n=300$ training samples. Writing $\mu_\ell^N=N^{-1}\sum_{j=1}^N\delta_{\Theta_{\ell j}}$, we discretize the depth dynamics by the explicit Euler scheme
\begin{equation*}
X_{r,0}=w_2^\top x_r,\quad X_{r,\ell+1}=X_{r,\ell}+\frac{1}{LN} \sum_{j=1}^N\tanh(\Theta_{\ell j}X_{r,\ell}),\quad \ell=0,\ldots,L-1.
\end{equation*}
The input map $w_2\in\mathbb R^{d_1\times d_2}$ and pooling vector $w_1\in\mathbb R^{d_2}$ are fixed. Only the residual block matrices $\Theta_{\ell j}\in\mathbb R^{d_2\times d_2}$ are trained, giving $LN d_2^2=3072$ scalar parameters. The full-batch empirical objective is
\begin{equation}
\label{eq:resnet-empirical-objective}
F_{L,N}(\Theta)=\frac{1}{2n}\sum_{r=1}^n \left(w_1^\top X_{r,L}-y_r\right)^2.
\end{equation}
We evaluate this loss directly through its squared residuals.
\\

\paragraph{Data, teacher, and initialization.} Inputs are independent samples from $\mathcal N(0,I_{d_1})$. A fixed teacher uses the same architecture and Euler grid to generate noiseless targets $y_r$. We obtain $w_2$ from the reduced QR decomposition of a standard Gaussian matrix and obtain $w_1$ by normalizing a standard Gaussian vector. For independent matrices $A_0,A_1$ with entries distributed as $\mathcal N(0,d_2^{-1})$, the teacher particles are
\begin{equation*}
\Theta^*_{\ell j}=0.8\left[A_0+0.4\sin\!\left(\frac{2\pi(\ell+1/2)}{L}\right)A_1\right] +\frac{0.25}{\sqrt{d_2}}E^*_{\ell j},
\end{equation*}
where the entries of $E^*_{\ell j}$ are independent standard Gaussians. The teacher and the fixed input and output maps are generated with seed $1729$ and held fixed throughout tuning and evaluation. Student particles are initialized as $\Theta^0_{\ell j}=0.25\frac{E_{\ell j}}{\sqrt{d_2}}$, with independent standard Gaussian entries. Each trial draws new inputs and initial particles. All methods within a trial share these draws and targets.
\\

\paragraph{Particle geometry and curvature.} We use the discretized tangent metric
\begin{equation}
\label{eq:resnet-particle-metric}
\langle U,V\rangle_{L,N} =\frac{1}{LN}\sum_{\ell,j}\operatorname{Tr}(U_{\ell j}^\top V_{\ell j}).
\end{equation}
Hence, the gradient and full Hessian in stored particle coordinates
are $G=(LN)\nabla_\Theta F_{L,N},$ $H[V]=(LN)D_\Theta^2F_{L,N}[V].$ Hessian-vector products are computed by automatic differentiation of \eqref{eq:resnet-empirical-objective}. This differentiates the full discrete objective, including interactions across depth slices.

For the Hessian-guided perturbations, we compute the integral component as $K=H-M$. Let $a_{r,\ell}$ be the discrete adjoint of the unaveraged per-sample loss, with terminal value $a_{r,L}=(w_1^\top X_{r,L}-y_r)w_1$. Holding the states and adjoints fixed when differentiating a block parameter gives
\begin{equation*}
(M[V])_{\ell j} =\frac1n\sum_{r=1}^n \left[a_{r,\ell+1}\odot\tanh''(\Theta_{\ell j}X_{r,\ell}) \odot(V_{\ell j}X_{r,\ell})\right]X_{r,\ell}^{\top},
\end{equation*}
where $\odot$ denotes componentwise multiplication. Thus $K$ retains both within-depth and cross-depth interactions, including its diagonal particle blocks. Neither a dense Hessian nor a matrix of particle pair curvature blocks is formed during optimization.
\\

\paragraph{Perturbations.} We compare WGF, isotropic WGF, PWGF, WSFN, and a scalar control. The first three methods use descent direction $G$. WSFN uses the approximation $P\approx(H^2+\beta \operatorname{I})^{-1/2}G.$ At its own current state, the scalar control computes the same approximate direction $P$ and replaces it by $P_{\mathrm{sc}}=\frac{\|P\|_{L,N}}{\|G\|_{L,N}}G,$ with $P_{\mathrm{sc}}=0$ when $G=0$. It inherits WSFN's selected learning rate, damping, and perturbation settings and incurs the cost of computing $P$. This compares direction at a locally matched length.

Perturbations are prescribed every $100$ completed updates, beginning at update $100$. For a tensor $Z$ with independent standard Gaussian entries, isotropic WGF adds $\eta_{\mathrm{iso}}\sqrt{LN}\,Z$, while PWGF, WSFN, and the scalar control add $\eta_K\sqrt{LN}\,K[Z]$, with $\eta_{\mathrm{iso}}=10^{-4}$ and $\eta_K=10^{-3}$. WGF has no perturbations. These increments are added to $\Theta-\tau D$, where $D$ is the method's descent direction. The perturbation amplitude is not multiplied by $\tau$. Both direction and perturbation are evaluated at the pre-update state. The factor $\sqrt{LN}$ makes the covariance operator of $\sqrt{LN}\,K[Z]$ equal to $K^2$ in the metric \eqref{eq:resnet-particle-metric}. Gaussian draws are shared by event index within each paired trial. Perturbations are not normalized to a common length, and their amplitudes remain fixed during tuning.
\\

\paragraph{Lanczos approximation.}
We apply at most $12$ Lanczos steps to $H^2+\beta \operatorname{I}$, using two Hessian-vector products per operator action and two passes of full reorthogonalization. Breakdown is detected with absolute tolerance $10^{-12}$ and relative tolerance $10^{-9}$. The small projected matrix is diagonalized after scaling. Projected eigenvalues violating its lower bound $\beta$ beyond relative tolerance $10^{-7}$ trigger a numerical failure.
\\

\paragraph{Parameter selection and evaluation.}
We tune on pilot seeds $1101$ and $1102$. Each of WGF, isotropic WGF, and PWGF independently searches the $12$ learning rates
\begin{equation*}
    \tau\in\{10^j,\,3\cdot10^j: j=-3,-2,-1,0,1,2\}.
\end{equation*}
WSFN searches $12$ joint configurations defined by
\begin{equation*}
    \beta\in\{10^{-4},10^{-3},10^{-2}\},\quad a\in\{0.3,1,3,10\},\quad \tau=a\sqrt{\beta}.
\end{equation*}
Each configuration receives $20$ seconds per pilot seed. Up to three nonfailing finalists per method are then restarted from the original initializations for $120$ seconds per seed. The finalist count is the same for every method. Selection minimizes the mean terminal relative loss over these longer pilot runs. The scalar control inherits the selected WSFN configuration and is not independently tuned.

After freezing the configurations, we evaluate on five fresh paired seeds, $2101, \ldots, 2105$, with $120$ seconds per method and seed. Method order is randomized within each evaluation seed. Timing includes gradient evaluation, curvature operations, Lanczos iterations, perturbations, and post-update diagnostics. It excludes data construction, initial diagnostics,
disposable warm-up, printing, and plotting.
\\

\paragraph{Reported diagnostics.}
Figure~\ref{fig:deep-resnet-diagnostics} reports relative empirical loss
and the particle gradient norm
\begin{equation*}
    R_s=\frac{F_{L,N}(\Theta_s)}{F_{L,N}(\Theta_0)}, \quad \|G_s\|_{L,N} =\left(\frac{1}{LN}\sum_{\ell,j}\|(G_s)_{\ell j}\|_{\mathrm F}^2\right)^{1/2}.
\end{equation*}
The update plots use a common recorded horizon among successful runs, while the time plots use the common $120$-second budget. Since perturbations are scheduled by update count, faster methods can experience more perturbations within the same time budget.

This experiment studies the finite-sample, finite-particle, finite-depth implementation with prescribed perturbation times. It does not implement the theoretical stopping and episode rules. The scalar ablation measures the effect of direction at WSFN's selected settings and does not imply superiority over an independently tuned scalar adaptive method. Quantitative control of the discretization and Lanczos errors remains outside the present analysis.
\\

\begin{figure}[t]
    \centering
    \includegraphics[width=\linewidth]{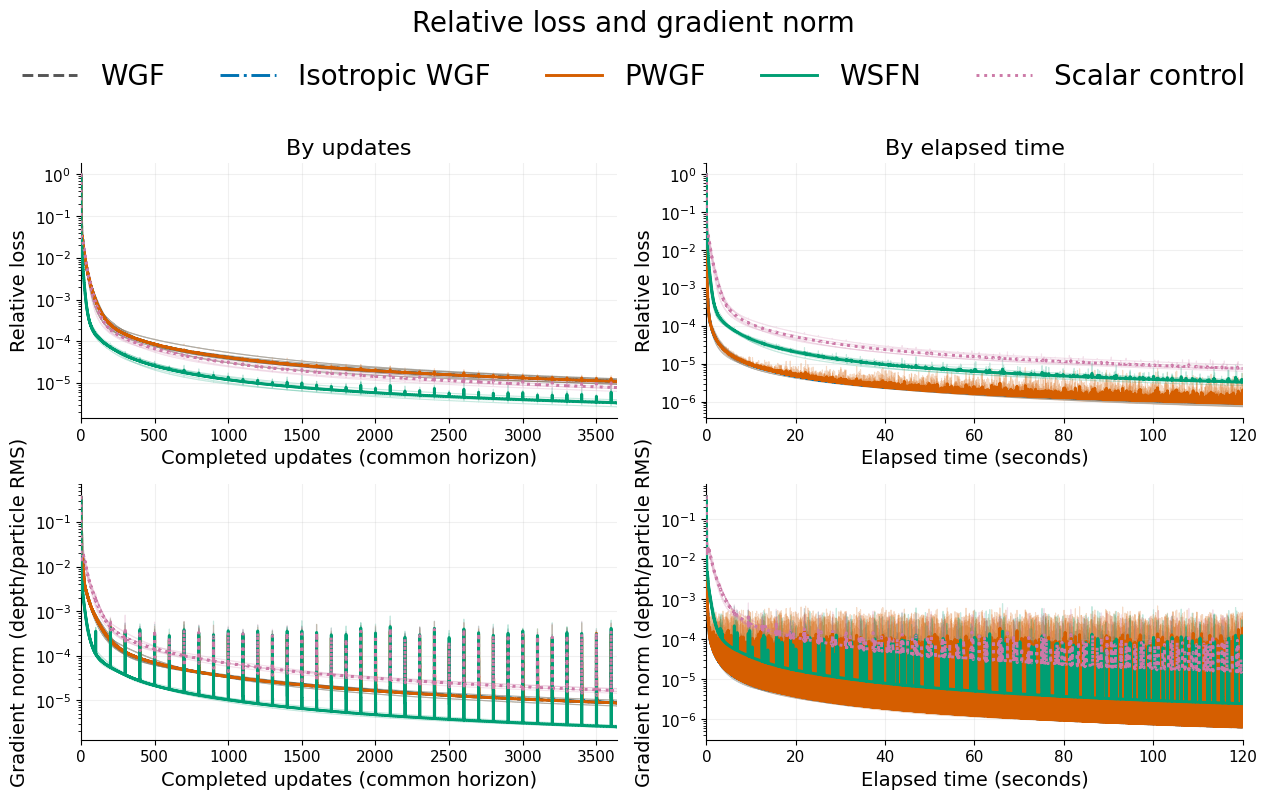}
    \caption{Mean-field deep ResNet regression. Relative empirical loss (top) and gradient norm in the particle metric (bottom), versus completed updates at a common recorded horizon (left) and elapsed time (right). Thick curves show means over five paired evaluation trials. Thin curves show individual trials}
    \label{fig:deep-resnet-diagnostics}
\end{figure}

\paragraph{Interpretation.} Figure~\ref{fig:deep-resnet-diagnostics} shows that WSFN makes more progress per update than the first-order baselines. WSFN exhibits this advantage early and maintains it throughout the common update horizon. However, the additional cost of curvature computations offsets this improvement in elapsed time. Within $120$ seconds, the first-order baselines reach relative losses around $10^{-6}$, whereas WSFN remains near $4\times10^{-6}$. The first-order loss curves largely overlap in this comparison.

The locally norm matched scalar control remains stable and attains an intermediate loss at the common update horizon. Its improvement over the first-order baselines is consistent with scalar rescaling contributing to the observed progress. WSFN achieves a further reduction and outperforms the control both per update and within the elapsed time budget, supporting a benefit from changing the update direction under the selected settings. This comparison is conditional on the control inheriting WSFN's learning rate and damping and matching the preconditioned direction's norm at its own current state.

The gradient norm panels display regular pulses aligned with the prescribed perturbations every $100$ updates. Between perturbations, WSFN attains smaller gradient norms at comparable update counts. The corresponding loss increases are small, and descent resumes after each perturbation. Faster methods experience more perturbation events within the same time budget, contributing to the dense fluctuations in the elapsed time panels.

All methods reduce the empirical training loss by several orders of magnitude. These results support improved progress of WSFN per update and a benefit from directional preconditioning at the selected settings, while the elapsed time comparison favors the first-order baselines. The trajectories do not identify a saddle escape episode, and the loss and gradient norm diagnostics alone do not prove faster saddle escape or convergence to a second-order stationary point.

\subsection{Matrix decomposition: implementation details}
We consider the matrix decomposition problem of Example~\ref{eg:matrix_decomposition}, in which a student mean-field network learns the rank-one feature matrices of a fixed teacher by minimizing
\begin{equation*}
\label{eq:matrix-objective}
F(\mu)=\frac{1}{n}\sum_{r=1}^{n} \left\|h_\mu(x_r)h_\mu(x_r)^\top-h_{\mu^*}(x_r)h_{\mu^*}(x_r)^\top\right\|_{\mathrm F}^{2}.
\end{equation*}
Here $x_r$ are the input samples and $\|\cdot\|_{\mathrm F}$ is the Frobenius norm.

We use input dimension $d=15$, feature dimension $k=5$, $n=300$ samples, and $N=300$ particles. The particle parameter dimension is $q=d+k=20$. Inputs are sampled from $\mathcal N(0,9I_d)$ and the teacher is fixed across trials. The four methods' learning rates are selected independently on pilot seeds. Each evaluation seed specifies a data set and student initialization shared by all methods. The five evaluation seeds are disjoint from the two pilot seeds used for learning rate selection. The teacher seed is $1729$, the pilot seeds are $1101,1102$, and the evaluation seeds are $2101,\ldots,2105$.
\\

\paragraph{Implemented curvature operator.}
The network uses the sigmoid activation. In the stored coordinates $\theta_i=(a_i,w_i)$, the output weight normalization is absorbed into $a_i$, so that $h_\theta(x)=\sum_{i=1}^N a_i\sigma(w_i^\top x)$. Write $h_r=h_\theta(x_r)$, $h_r^*=h_{\mu^*}(x_r)$, and $J_{ri}=D_{\theta_i}h_\theta(x_r)\in\mathbb R^{k\times q}$. The curvature action supplied to Lanczos is
\begin{align*}
(\widehat H_N v)_i =\frac{2}{nN}\sum_{r=1}^{n}\sum_{j=1}^{N} J_{ri}^{\top} B_r J_{rj}v_j, \quad B_r =2h_rh_r^\top+\|h_r\|^2I_k-h_r^*(h_r^*)^\top.
\end{align*}
This kernel type surrogate does not assemble the separate multiplication component involving second derivatives of the particle feature map. Accordingly, the reported experiment concerns this implemented curvature approximation instead of the exact full Hessian discretization. WSFN uses at most $12$ Lanczos iterations, with stopping tolerance $10^{-8}$ and full reorthogonalization.
\\

\paragraph{Scalar control.}
Let $g_N$ denote the gradient of $G$ in the stored parameter coordinates, and let $p_N$ be the computed Lanczos approximation to $(\widehat H_N^2+\beta \operatorname{I_{d\times d}})^{-1/2}g_N$. With $\|v\|_N^2=N^{-1}\sum_{i=1}^{N}\|v_i\|^2$, the control uses
\begin{equation*}
\label{eq:matrix-scalar-control}
d_N^{\mathrm{sc}}= \begin{cases}
\displaystyle\frac{\|p_N\|_N}{\|g_N\|_N}\,g_N, &\|g_N\|_N>0,\\[1ex]
0,&\|g_N\|_N=0.
\end{cases}
\end{equation*}
Thus $\|d_N^{\mathrm{sc}}\|_N=\|p_N\|_N$ at the same state. The control evaluates both vectors at its own current state and inherits WSFN's selected learning rate, regularization, and Lanczos settings. It incurs the preconditioner cost and is used to examine the effect of directional rescaling. Its learning rate is not tuned independently, and the two evolving trajectories need not have identical step lengths.
\\

\paragraph{Perturbations.}
WGF is unperturbed. The other methods receive a perturbation after every $100$ updates, irrespective of loss decrease. At each scheduled event, the methods use common underlying Gaussian blocks $Z_i\overset{\mathrm{iid}}{\sim}\mathcal N(0,I_q)$. Isotropic WGF adds $3\times10^{-4}Z$, while PWGF, WSFN, and the scalar control add $3\times10^{-3}\widehat H_N Z$. The transport direction and perturbation are computed from the same pre-update state. The curvature shaped perturbations depend on each method's own state and need not coincide as vectors. This common prescribed schedule serves the ablation and differs from stagnation triggered perturbations and the theoretical episode rules.
\\

\paragraph{Regularization and learning rate selection.}
Each of WGF, isotropic WGF, PWGF, and WSFN is tuned independently by minimizing mean relative loss over two pilot seeds, with a $20$-second
budget per candidate and seed. The initial grids are
\begin{align*}
\mathcal T_{\mathrm{first}}
&=5\times10^{-6}\{0.1, 0.3, 1, 3, 10, 30, 60, 100, 200, 300, 600, 10000\},\\
\mathcal T_{\mathrm{WSFN}}
&=5\times10^{-6}\{10^{-4}, 3\times 10^{-4}, 5\times 10^{-4}, 10^{-3}\},\\
\beta_{\mathrm{WSFN}}&=10^{-3}\{0.1, 0.3, 1\}
\end{align*}
The first grid is searched separately for each first-order method. If any selected rate lies at a grid boundary, the search adds one rate below and above every method's grid by a factor of $10$, with at most one such expansion. This preserves equal candidate counts across the four methods. The selected rates (Table \ref{tab:matrix-selected-rates}) are frozen before evaluation.
\\

\begin{table}[t]
\centering
\caption{Learning rates and $\beta$ selected from pilot runs for the reported evaluation.}
\label{tab:matrix-selected-rates}
\begin{tabular}{lcc}
\hline
Method & Selected learning rate & Selected $\beta$ \\
\hline
WGF & 0.0015 & N/a \\
Isotropic WGF & 0.0015 & N/a \\
PWGF & 0.0015 & N/a \\
WSFN & 0.001 & 0.0001 \\
Scalar control & Same as WSFN & Same as WSFN \\
\hline
\end{tabular}
\end{table}

\paragraph{Evaluation and timing.} Each method is evaluated for $120$ seconds on each of five fresh seeds. Run order is randomized. Elapsed time includes forward and backward passes, curvature actions, Lanczos iterations, perturbations, and diagnostics recorded every $10$ updates. Data and model construction, disposable warm-up runs, and plotting and printing are excluded. Only checkpoints completed within the time budget contribute to fixed-time comparisons, and the last observed value is held constant between checkpoints. The iteration plots are restricted to the common observed update horizon.
\\

\begin{figure}[t]
    \centering
    \includegraphics[width=\linewidth]{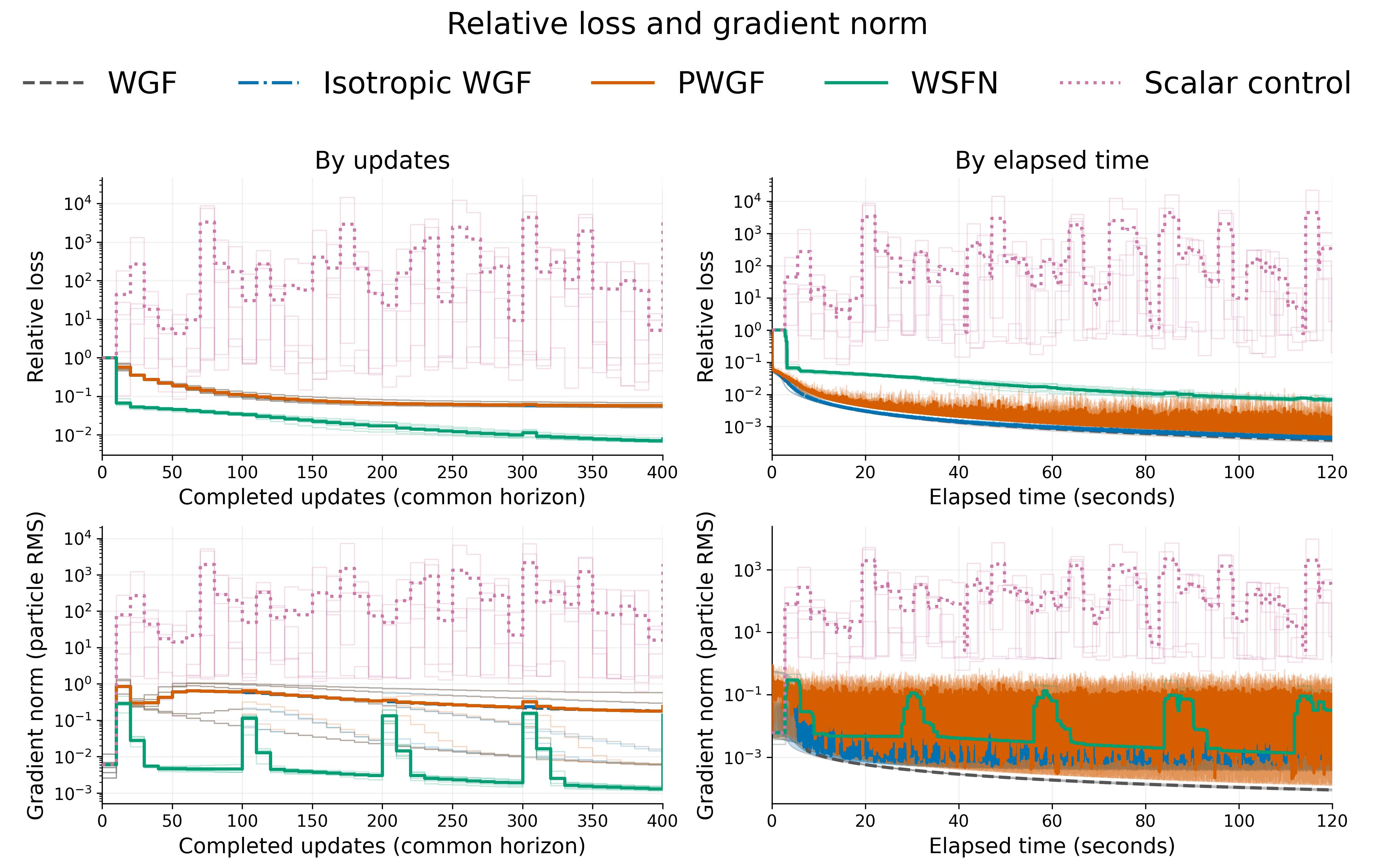}
    \caption{Matrix decomposition ablation with $N=300$ particles over five paired evaluation trials. Top: loss relative to its initial value. Bottom: particle RMS norm of the implemented gradient. Left: completed updates, restricted to the common observed horizon. Right: elapsed time, with a budget of $120$ seconds per method and trial. Thick curves show means. Thin curves show individual trials.
    }
    \label{fig:full-matrix-decomp-wsfn}
\end{figure}

\paragraph{Interpretation.} Figure~\ref{fig:full-matrix-decomp-wsfn} shows that WSFN makes substantially more progress per update than the first-order baselines. At the common horizon of approximately $400$ updates, WSFN decreases the loss rapidly during the initial phase and continues improving after the first-order curves have largely flattened over this update range. However, this advantage does not translate into better performance within the elapsed time budget. PWGF also attains lower losses than WSFN within this budget, although its trajectory exhibits persistent fluctuations. Thus, the additional cost of curvature operations outweighs WSFN's improved progress per update in this implementation.

The locally norm matched scalar control exhibits repeated large excursions in both loss and gradient norm throughout the experiment. The individual trajectories also show substantial variation, with extreme excursions strongly influencing the mean. This comparison supports a contribution from the preconditioned direction to the observed stability. However, the scalar control inherits WSFN's learning rate and damping and matches the preconditioned direction's norm at its own current state. Hence, the comparison does not show superiority over an independently tuned scalar adaptive method or identify negative curvature exploitation as the specific mechanism responsible for the improvement.

The gradient norm panels further illustrate the effect of the prescribed perturbations. WSFN's pulses near updates $100$, $200$, and $300$ coincide with the perturbation schedule. Between later perturbations, its gradient norm decreases to the order of $10^{-3}$, substantially below the first-order mean curves at comparable update counts. The accompanying loss increases are modest relative to those of the scalar control. In elapsed time, unperturbed WGF reaches the smallest gradient norm, approximately $10^{-4}$. Faster perturbed methods undergo more perturbation events within the same time budget, contributing to the dense fluctuations in their elapsed time curves.

These results support improved progress per update and greater stability than the locally norm matched control at the selected settings, while the runtime comparison favors the first-order baselines. The loss and gradient norm trajectories alone do not imply faster saddle escape or convergence to a second-order stationary point.

\section{$L_\mu^2$-valued Gaussian Processes}
\label{L2 valued GPs}
In this section, we recall the definition of Gaussian processes (GPs) and the notion of the trace of an operator, which together allow us to define $L_\mu^2$-valued Gaussian processes. Related material can also be found in \cite[Appendix D]{yamamoto2025hessianguided}. We adopt the following definition of a Gaussian process from \cite[Chapter 2]{rasmussen}.
\begin{definition}[Gaussian process]
    Let $I=\{x_1, x_2, \dots, x_n\}$ be a finite collection of points in $\mathbb R^d$. The random variable $\xi : \mathbb R^d \to \mathbb R^d$ is said to be a Gaussian process if the collection of random variables $\left\{\xi(x_i)\right\}_{1\leq i \leq n}$ is jointly normally distributed. A Gaussian process is completely determined by its mean function $m$ and covariance function $K,$ which are defined by
    \begin{align*}
        m(x) &= \mathbb{E}[\xi(x)], \quad K(x,y) = \mathbb{E}[(\xi(x) - m(x))(\xi(y) - m(y))^\top]. 
    \end{align*}
    In this case, we write $\xi(x) \sim \mathrm{GP}(m(x),K(x,y))$.
\end{definition}
We next recall a standard result, adapted from \cite[Theorem~VI.18]{ReedSimon1980}.
\begin{theorem}
\label{thm:trace-operator}
    Let $\mu\in\mathcal P_2(\mathbb R^d)$ and $\{\psi_n\}_{n=1}^\infty$ be an orthonormal basis in $L_\mu^2$. Then for any non-negative operator $A_\mu \in \mathcal{B}(L_\mu^2)$ we define $\mathrm{Tr} A = \sum_{n=1}^\infty \langle \psi_n, A_\mu\psi_n \rangle_{L_\mu^2}$. The number $\mathrm{Tr} A_\mu$ is called the trace of $A_\mu$ and is independent of the orthonormal basis chosen. 
\end{theorem}
\begin{definition}
\label{def:trace-operator}
    Let $\mu\in\mathcal P_2(\mathbb R^d)$. An operator $A_\mu \in \mathcal{B}(L_\mu^2)$ is called trace-class if and only if $\operatorname{Tr} |A_\mu| < \infty$.
\end{definition}
The next result shows that there exists a centered $L_\mu^2$-valued Gaussian random element with covariance operator $\operatorname{K}_\mu^2$.

\begin{lemma}[$L^2_\mu$-valued GP]
\label{lem:Gaussian-L2}
Let $\mu\in\mathcal P_2(\mathbb R^d)$ and assume that \(\nabla_\mu^2F(\mu,\cdot,\cdot)\in L^2_{\mu\otimes\mu}\) and that \(\nabla_\mu^2F(\mu,x,y) = \nabla_\mu^2F(\mu,y,x)^\top\) for \(\mu\otimes\mu\)-a.e. \((x,y)\). Then \(\operatorname{K}_\mu\) is self-adjoint and Hilbert--Schmidt, and \(\operatorname{K}_\mu^2\) is nonnegative, symmetric and trace-class. There exists a centered \(L^2_\mu\)-valued Gaussian random element $\xi\sim\mathcal N_{L^2_\mu}(0,\operatorname{K}_\mu^2).$ Its covariance operator has kernel
\begin{equation*}
    C_\mu(x,y) = \int_{\mathbb R^d} \nabla_\mu^2 F(\mu, x, z) \nabla_\mu^2 F(\mu, z, y)\mathrm{d}\mu(z)
\end{equation*}
\(\mu\otimes\mu\)-a.e. \((x,y)\) and 
\begin{equation*}
    \operatorname{K}_\mu^2[v](x) =\int_{\mathbb R^d} C_\mu(x,y)v(y)\mathrm d\mu(y),
\end{equation*}
for all $v \in L_\mu^2$. Moreover,
\begin{equation*}
\mathbb E \|\xi\|_{L^2_\mu}^2=\mathrm{Tr}(\operatorname{K}_\mu^2)=\|\nabla_\mu^2F(\mu,\cdot,\cdot)\|_{L^2_{\mu \otimes \mu}}^2.
\end{equation*}
For every $t > 0,$
\begin{equation*}
    \mathbb P\left(\|\xi\|_{L^2_\mu}^2 > \mathrm{Tr}(\operatorname{K}_\mu^2) +2 \sqrt{\mathrm{Tr}(\operatorname{K}_\mu^4)t} + 2\|\operatorname{K}_\mu\|_{\mathrm{op}}^2t\right) \leq e^{-t}.
\end{equation*}
\end{lemma}

\begin{proof}
By \cite[Theorem VI.23]{ReedSimon1980}, since $\nabla_\mu^2F(\mu,\cdot,\cdot) \in L^2_{\mu \otimes \mu}$, the integral operator $\operatorname{K}_\mu$ is Hilbert--Schmidt with $\|\operatorname{K}_\mu\|_{\text{HS}} = \|\nabla_\mu^2F(\mu,\cdot,\cdot)\|_{L^2_{\mu \otimes \mu}}$. We also have that $\|\operatorname{K}_\mu\|_{\text{op}} \leq \|\operatorname{K}_\mu\|_{\text{HS}} = \|\nabla_\mu^2F(\mu,\cdot,\cdot)\|_{L^2_{\mu \otimes \mu}} <\infty,$ so $\operatorname{K}_\mu$ is bounded for every $\mu \in L_\mu^2.$ Moreover, $\operatorname{K}_\mu$ is symmetric for every $\mu \in L_\mu^2$ since \(\nabla_\mu^2F(\mu,x,y) = \nabla_\mu^2F(\mu,y,x)^\top\) for \(\mu\otimes\mu\)-a.e. \((x,y)\). Thus $\operatorname{K}_\mu$ is self-adjoint for every $\mu \in L_\mu^2$.

Since $\operatorname{K}_\mu$ is HS, belongs to $\mathcal{B}(L_\mu^2)$, and is self-adjoint, it follows from \cite[Definition, p. 210]{ReedSimon1980} that $\mathrm{Tr}(\operatorname{K}_\mu^2) < \infty$. To show that $\operatorname{K}_\mu^2 \geq 0$, let $v \in L_\mu^2$. Then 
\begin{align*}
     \langle v, \operatorname{K}_\mu^2 v\rangle_{L_\mu^2} &= \int_{\mathbb R^d}\left\langle v(x), \operatorname{K}_\mu^2 v(x) \right\rangle\mathrm{d}\mu(x)= \|\operatorname{K}_\mu v\|_{L_\mu^2}^2 \geq 0,
\end{align*}
which implies that $\operatorname{K}_\mu^2$ defines a non-negative operator on $L_\mu^2$. Thus we have $|\operatorname{K}_\mu^2| = \operatorname{K}_\mu^2$, hence by Definition \ref{def:trace-operator}, combined with $\mathrm{Tr}(\operatorname{K}_\mu^2) < \infty,$ $\operatorname{K}_\mu^2$ is trace-class. The symmetry of $\operatorname{K}_\mu^2$ follows from the symmetry of $\operatorname{K}_\mu$.

For the existence of $\xi\sim\mathcal N_{L^2_\mu}(0,\operatorname{K}_\mu^2)$, it suffices (cf. \cite[Theorem 12.1.3.]{Dudley_2002}) to show that $C_\mu(x,y)=C_\mu(y,x)^\top$, $\mu\otimes\mu$-a.e., and $C_\mu$ is positive semi-definite in the matrix sense, i.e., for any
$n\in\mathbb N$, any $x_1,...,x_n\in\mathbb R^d$, and any $u_1,...,u_n\in\mathbb R^d$,
\begin{equation*}
\sum_{i,j=1}^n u_i^\top C_\mu(x_i,x_j)u_j  \ge 0.
\end{equation*}
Using the rule $(AB)^\top=B^\top A^\top$,
\[
C_\mu(y,x)^\top = \left(\int_{\mathbb R^d}\nabla_\mu^2 F(\mu, y,z)\nabla_\mu^2 F(\mu, z,x)\mathrm d\mu(z)\right)^\top =
\int_{\mathbb R^d} \nabla_\mu^2 F(\mu, z,x)^\top \nabla_\mu^2 F(\mu, y,z)^\top\mathrm d\mu(z).
\]
By regularity of $F$, we have $\nabla_\mu^2 F(\mu, z,x)^\top=\nabla_\mu^2 F(\mu,x,z)$ and $\nabla_\mu^2 F(\mu, y,z)^\top=\nabla_\mu^2 F(\mu, z,y)$, $\mu$-a.e. Hence
\[
C_\mu(y,x)^\top = \int_{\mathbb R^d} \nabla_\mu^2 F(\mu,x,z)\nabla_\mu^2 F(\mu, z,y)\mathrm d\mu(z) = C_\mu(x,y),
\]
which proves $C_\mu(x,y)=C_\mu(y,x)^\top$.

Fix $n\in\mathbb N$, $x_1,...,x_n\in\mathbb R^d$ and $u_1,...,u_n\in\mathbb R^d$. Then
\begin{align*}
\sum_{i,j=1}^n u_i^\top C_\mu(x_i,x_j)u_j &=
\sum_{i,j=1}^n \int_{\mathbb R^d} u_i^\top \nabla_\mu^2 F(\mu, x_i,z)\nabla_\mu^2 F(\mu,z,x_j)u_j\mathrm d\mu(z)\\ 
&= \int_{\mathbb R^d} \sum_{i,j=1}^n u_i^\top \nabla_\mu^2 F(\mu, x_i,z)\nabla_\mu^2 F(\mu,z,x_j)u_j\mathrm d\mu(z).
\end{align*}
Using the symmetry of $\nabla_\mu^2 F$, we can rewrite the integrand as
\begin{align*}
u_i^\top \nabla_\mu^2 F(\mu, x_i,z) \nabla_\mu^2 F(\mu,z,x_j)u_j &=
u_i^\top \nabla_\mu^2 F(\mu,z,x_i)^\top \nabla_\mu^2 F(\mu,z,x_j)u_j\\
&= \big(\nabla_\mu^2 F(\mu,z,x_i)u_i\big)^\top \big(\nabla_\mu^2 F(\mu,z,x_j)u_j\big).
\end{align*}
Therefore,
\begin{align*}
\sum_{i,j=1}^n u_i^\top C_\mu(x_i,x_j)u_j &= \int_{\mathbb R^d}
\sum_{i,j=1}^n \big(\nabla_\mu^2 F(\mu,z,x_i)u_i\big)^\top \big(\nabla_\mu^2 F(\mu,z,x_j)u_j\big)\mathrm d\mu(z)\\
&= \int_{\mathbb R^d}
\left\|\sum_{i=1}^n \nabla_\mu^2 F(\mu,z,x_i)u_i\right\|^2\mathrm d\mu(z) \ge 0,
\end{align*}
which proves positive semi-definiteness.

It remains to show that $\xi \in L_\mu^2$ almost surely. Using the definition of $C_\mu$ and Tonelli,
\begin{align*}
\mathrm{Tr}\ C_\mu(x,x)
&=
\mathrm{Tr}\left(\int_{\mathbb R^d} \nabla_\mu^2 F(\mu,x,z)\nabla_\mu^2 F(\mu,z,x)\mathrm d\mu(z)\right)\\
&=
\int_{\mathbb R^d}  \mathrm{Tr}\big(\nabla_\mu^2 F(\mu,x,z)\nabla_\mu^2 F(\mu,z,x)\big)\mathrm d\mu(z).
\end{align*}
By symmetry $\nabla_\mu^2 F(\mu,z,x)=\nabla_\mu^2 F(\mu,x,z)^\top$, hence
\[
\mathrm{Tr}\big(\nabla_\mu^2 F(\mu,x,z)\nabla_\mu^2 F(\mu,z,x)\big)
=
\mathrm{Tr}\big(\nabla_\mu^2 F(\mu,x,z)\nabla_\mu^2 F(\mu,x,z)^\top\big)
=
\|\nabla_\mu^2 F(\mu,x,z)\|^2.
\]
Therefore,
\[
\int_{\mathbb R^d} \mathrm{Tr}\,C_\mu(x,x)\mathrm d\mu(x)
=
\int_{\mathbb R^d} \int_{\mathbb R^d}  \|\nabla_\mu^2 F(\mu,x,z)\|^2\mathrm d\mu(z)\mathrm d\mu(x) 
<\infty.
\]
For each fixed $x \in \mathbb R^d$, $\xi(x)$ is a centered Gaussian vector in $\mathbb R^d$ with covariance matrix $C_\mu(x,y)$, hence
\[
\mathbb E\|\xi(x)\|^2=\mathrm{Tr}\big(\mathbb E[\xi(x)\xi(x)^\top]\big)=\mathrm{Tr}\,C_\mu(x,x).
\]
By Tonelli's theorem,
\begin{equation}
\begin{aligned}
\label{eq:second L^2 moment of xi}
\mathbb E\|\xi\|_{L^2_\mu}^2
&=
\mathbb E\int_{\mathbb R^d}\|\xi(x)\|^2\mathrm d\mu(x)
=
\int_{\mathbb R^d}\mathbb E\|\xi(x)\|^2\mathrm d\mu(x)\\
&=
\int_{\mathbb R^d}\mathrm{Tr}\,C_\mu(x,x)\mathrm d\mu(x)=\|\nabla_\mu^2F(\mu,\cdot,\cdot)\|_{L^2_{\mu \otimes \mu}}^2.
\end{aligned}
\end{equation}
Hence $\mathbb E\|\xi\|_{L^2_\mu}^2<\infty$. Since
$\|\xi\|_{L^2_\mu}^2\ge 0$, by Markov's inequality, for any $a > 0$
\begin{equation*}
    \mathbb P(\|\xi\|_{L^2_\mu}^2 \geq a) \leq \frac{\mathbb E\|\xi\|_{L^2_\mu}^2}{a}.
\end{equation*}
Sending $a \to \infty$ implies $\mathbb P(\|\xi\|_{L^2_\mu}^2 <\infty) = 1$, so $\|\xi\|_{L^2_\mu}<\infty$ a.s. and $\xi \in L^2_\mu$ a.s.

To verify the relation between $\operatorname{K}_\mu^2$ and $C_\mu,$ let $v \in L_\mu^2$. Then
\begin{align*}
    \operatorname{K}_\mu^2[v](x) = \operatorname{K}_\mu[\operatorname{K}_\mu[v]](x) &= \int_{\mathbb R^d} \nabla_\mu^2F(\mu, x, z)\operatorname{K}_\mu[v](z)\mathrm d\mu(z)\\ &= \int_{\mathbb R^d} \nabla_\mu^2F(\mu, x, z)\int_{\mathbb R^d} \nabla_\mu^2F(\mu, z, y)v(y)\mathrm d\mu(y)\mathrm d\mu(z)\\ &= \int_{\mathbb R^d} \int_{\mathbb R^d} \nabla_\mu^2F(\mu, x, z) \nabla_\mu^2F(\mu, z, y)v(y) \mathrm d\mu(y)\mathrm d\mu(z)\\ &= \int_{\mathbb R^d} \int_{\mathbb R^d} \nabla_\mu^2F(\mu, x, z) \nabla_\mu^2F(\mu, z, y)v(y) \mathrm d\mu(z)\mathrm d\mu(y)\\ 
    &= \int_{\mathbb R^d} C_\mu(x,y)v(y)\mathrm d\mu(y),
\end{align*}
where the penultimate equality follows from Fubini's theorem.

Furthermore, let $\{\psi_n\}_{n=1}^\infty$ be an orthonormal basis in $L_\mu^2$. Then 
\begin{align*}
    \mathrm{Tr} (\operatorname{K}_\mu^2) =  \sum_{n=1}^\infty \langle \psi_n, \operatorname{K}_\mu^2\psi_n \rangle_{L_\mu^2} = \sum_{n=1}^\infty \|\operatorname{K}_\mu\psi_n \|_{L_\mu^2}^2 = \|\operatorname{K}_\mu\|_{\text{HS}}^2=\|\nabla_\mu^2F(\mu,\cdot,\cdot)\|_{L^2_{\mu \otimes \mu}}^2,
\end{align*}
where the last equality follows from the definition of the HS norm. Hence by \eqref{eq:second L^2 moment of xi},
\begin{equation*}
    \mathbb E\|\xi\|_{L^2_\mu}^2 = \mathrm{Tr} (\operatorname{K}_\mu^2) = \|\nabla_\mu^2F(\mu,\cdot,\cdot)\|_{L^2_{\mu \otimes \mu}}^2.
\end{equation*}
Lastly, to prove the exponential bound we use the Karhunen--Loève representation. If $\operatorname{K_\mu}e_j=\lambda_j e_j$ is a spectral decomposition and $Z_j \sim \mathcal{N}(0,1)$ i.i.d., set
\begin{equation*}
    \xi= \sum_{j \geq 1} \lambda_j Z_j e_j.
\end{equation*} 
Since $\sum_{j \geq 1} \lambda_j^2 = \|\operatorname{K}_\mu\|_{\text{HS}}^2 < \infty,$ due to \cite[Theorem VI.22 (e)]{ReedSimon1980}, this series converges in $L_\mu^2$ and almost surely. Furthermore,
\begin{equation*}
    \|\xi\|_{L_\mu^2}^2 = \sum_{j \geq 1} \lambda_j^2 Z_j^2,
\end{equation*}
so the squared norm is a weighted chi-square random variable. Applying the standard weighted chi-square inequality to finite truncations and passing to the limit gives the displayed tail bound; see \cite[Proposition 1]{Hsu-Kakade}.
\end{proof}

\begin{lemma}[Linear functionals of a Gaussian process are Gaussian]
\label{lem:lin-func-GP}
Let $\mu\in\mathcal P_2(\mathbb R^d)$ and assume that \(\nabla_\mu^2F(\mu,\cdot,\cdot)\in L^2_{\mu\otimes\mu}\) and that \(\nabla_\mu^2F(\mu,x,y) = \nabla_\mu^2F(\mu,y,x)^\top\) for \(\mu\otimes\mu\)-a.e. \((x,y)\). Let $\xi\sim\mathcal N_{L^2_\mu}(0,\operatorname{K}_\mu^2)$ be a centered \(L^2_\mu\)-valued Gaussian random element. Then for any $\psi\in L^2_\mu$ the random variable $\langle \psi,\xi\rangle_{L^2_\mu}\sim \mathcal N\big(0,\langle \psi,\operatorname{K}_\mu^2\psi\rangle_{L^2_\mu}\big)$.
\end{lemma}
\begin{proof}
See the proof of \cite[Lemma E.9]{kim2024transformers}.
\end{proof}
\begin{lemma}
\label{lem:independence}
Let $\mu_\mathrm{sad} \in \mathcal{P}_2(\mathbb R^d)$ be an $(\varepsilon,\delta)$-saddle point, i.e., $\|\nabla_\mu F (\mu_\mathrm{sad})\|_{L^2_{\mu_\mathrm{sad}}} \leq \varepsilon$ and $\lambda_{0} := \underline{\lambda}(\operatorname{K}_{\mu_\mathrm{sad}}) < - \delta$. Let $\xi\sim\mathcal N_{L^2_{\mu_\mathrm{sad}}}(0,\operatorname{K}_{\mu_\mathrm{sad}}^2)$ and $q_0 \in L_{\mu_\mathrm{sad}}^2$ be a unit eigenvector of the operator $\operatorname{K}_{\mu_\mathrm{sad}}$ corresponding to the eigenvalue $\lambda_0$. Define the orthogonal projection of $\xi$ onto the orthogonal complement of the span of $q_0$ by $\xi_\perp := \xi - \langle q_0,\xi\rangle_{L^2_{\mu_\mathrm{sad}}} q_0$. Then $\langle q_0,\xi\rangle_{L^2_{\mu_\mathrm{sad}}} \sim \mathcal N(0,\lambda_0^2)$, and $\langle q_0,\xi\rangle_{L^2_{\mu_\mathrm{sad}}}$ and $\xi_\perp$ are independent.
\end{lemma}
\begin{proof}
By Lemma \ref{lem:lin-func-GP}, for every \(u \in L^2_{\mu_\mathrm{sad}}\), the scalar random variable \(\langle \xi,u\rangle_{L^2_{\mu_\mathrm{sad}}}\) is $\mathcal{N}(0, \langle u,\operatorname{K}_{\mu_\mathrm{sad}}^2u\rangle_{L^2_{\mu_\mathrm{sad}}})$. In particular, by 
\[
\mathbb E\big[\langle \xi,q_0\rangle_{L^2_{\mu_\mathrm{sad}}}^2\big]
= \langle \operatorname{K}_{\mu_\mathrm{sad}}^2 q_0,q_0\rangle_{L^2_{\mu_\mathrm{sad}}}=\lambda_0^2 \langle q_0,q_0\rangle_{L^2_{\mu_\mathrm{sad}}}
= \lambda_0^2,
\]
since \(\|q_0\|_{L^2_{\mu_\mathrm{sad}}} = 1\). Therefore,
\[
\langle \xi,q_0\rangle_{L^2_{\mu_\mathrm{sad}}} \sim \mathcal N(0,\lambda_0^2).
\]
Next, we show that \[
\mathbb E\big[\langle \xi,u\rangle_{L^2_{\mu_\mathrm{sad}}} \langle \xi,v\rangle_{L^2_{\mu_\mathrm{sad}}}\big]
= \langle \operatorname{K}_{\mu_\mathrm{sad}}^2 u,v\rangle_{L^2_{\mu_\mathrm{sad}}},
\]
for every $u,v \in L^2_{\mu_\mathrm{sad}}$. Indeed,
\begin{align*}
\mathbb E\big[\langle \xi,u\rangle_{L^2_{\mu_\mathrm{sad}}} \langle \xi,v\rangle_{L^2_{\mu_\mathrm{sad}}}\big]
&=
\mathbb E\left[
\left(\int_{\mathbb R^d} \xi(x)u(x)\mathrm d\mu_\mathrm{sad}(x)\right)
\left(\int_{\mathbb R^d} \xi(y)v(y)\mathrm d\mu_\mathrm{sad}(y)\right)
\right] \\
&=
\int_{\mathbb R^d}\int_{\mathbb R^d}
\mathbb E[\xi(x)\xi(y)]u(x)v(y)\mathrm d\mu_\mathrm{sad}(x)\mathrm d\mu_\mathrm{sad}(y) \\
&=
\int_{\mathbb R^d}\int_{\mathbb R^d}
C_{\mu_\mathrm{sad}}(x,y) u(x)v(y)\mathrm d\mu_\mathrm{sad}(x)\mathrm d\mu_\mathrm{sad}(y).
\end{align*}
On the other hand, by Lemma \ref{lem:Gaussian-L2},
\[
\operatorname{K}_{\mu_\mathrm{sad}}^2[u](x) =
\int_{\mathbb R^d} C_{\mu_\mathrm{sad}}(x,y)u(y)\mathrm d\mu_\mathrm{sad}(y).
\]
Therefore,
\begin{align*}
\langle \operatorname{K}_{\mu_\mathrm{sad}}^2u,v\rangle_{L^2_{\mu_\mathrm{sad}}}
&=
\int_{\mathbb R^d} \operatorname{K}_{\mu_\mathrm{sad}}^2[u](x)v(x)\mathrm d\mu_\mathrm{sad}(x) \\
&=
\int_{\mathbb R^d}
\left(\int_{\mathbb R^d} C_{\mu_\mathrm{sad}}(x,y)u(y)\mathrm d\mu_\mathrm{sad}(y)\right)
v(x)\mathrm d\mu_\mathrm{sad}(x) \\
&=
\int_{\mathbb R^d}\int_{\mathbb R^d}
C_{\mu_\mathrm{sad}}(x,y)u(y)v(x)\mathrm d\mu_\mathrm{sad}(y)\mathrm d\mu_\mathrm{sad}(x).
\end{align*}
Renaming the dummy variables \(x\) and \(y\), we obtain
\[
\langle \operatorname{K}_{\mu_\mathrm{sad}}^2u,v\rangle_{L^2_{\mu_\mathrm{sad}}}
=
\int_{\mathbb R^d}\int_{\mathbb R^d}
C_{\mu_\mathrm{sad}}(x,y) u(x)v(y)\mathrm d\mu_\mathrm{sad}(x)\mathrm d\mu_\mathrm{sad}(y).
\]
Hence
\[
\mathbb E\big[\langle \xi,u\rangle_{L^2_{\mu_\mathrm{sad}}} \langle \xi,v\rangle_{L^2_{\mu_\mathrm{sad}}}\big]
=
\langle \operatorname{K}_{\mu_\mathrm{sad}}^2u,v\rangle_{L^2_{\mu_\mathrm{sad}}}.
\]
Let \(P_\perp : L^2_{\mu_\mathrm{sad}} \to L^2_{\mu_\mathrm{sad}}\) denote the orthogonal projection onto the orthogonal complement of the span of $q_0$, that is,
\[
P_\perp v := v - \langle q_0,v\rangle_{L^2_{\mu_\mathrm{sad}}} q_0,
\quad v \in L^2_{\mu_\mathrm{sad}}.
\]
Fix \(v \in L^2_{\mu_\mathrm{sad}}\). Since \(\xi_\perp = P_\perp \xi\), we have
\[
\langle \xi_\perp, v\rangle_{L^2_{\mu_\mathrm{sad}}} = \langle \xi, P_\perp v\rangle_{L^2_{\mu_\mathrm{sad}}},
\]
due to the fact that orthogonal projections are symmetric operators. Hence
\begin{align*}
\operatorname{Cov}\bigl(\langle \xi,q_0\rangle_{L^2_{\mu_\mathrm{sad}}},\langle \xi,P_\perp v\rangle_{L^2_{\mu_\mathrm{sad}}}\bigr) = \langle \operatorname{K}_{\mu_\mathrm{sad}}^2 q_0, P_\perp v\rangle_{L^2_{\mu_\mathrm{sad}}} = \lambda_0^2 \langle q_0, P_\perp v\rangle_{L^2_{\mu_\mathrm{sad}}} = 0,
\end{align*}
because $P_\perp v$ and $q_0$ are orthogonal. Now let \((e_k)_{k\geq 1}\) be an orthonormal basis of \(L^2_{\mu_\mathrm{sad}}\). For every \(m \geq 1\), the random vector
\[
\bigl(\langle \xi,q_0\rangle_{L^2_{\mu_\mathrm{sad}}},\langle \xi_\perp,e_1\rangle_{L^2_{\mu_\mathrm{sad}}},...,\langle \xi_\perp,e_m\rangle_{L^2_{\mu_\mathrm{sad}}}\bigr)
\]
is jointly Gaussian, since each component is a continuous linear functional of the Gaussian random element \(\xi\). The computation above shows that
\[
\operatorname{Cov}\bigl(\langle \xi,q_0\rangle_{L^2_{\mu_\mathrm{sad}}},\langle \xi_\perp,e_j\rangle_{L^2_{\mu_\mathrm{sad}}}\bigr)=0,
\quad j=1,\dots,m.
\]
Therefore, by the characterization of Gaussian vectors, \(\langle \xi,q_0\rangle_{L^2_{\mu_\mathrm{sad}}}\) is independent of
\[
\bigl(\langle \xi_\perp,e_1\rangle_{L^2_{\mu_\mathrm{sad}}},...,\langle \xi_\perp,e_m\rangle_{L^2_{\mu_\mathrm{sad}}}\bigr)
\]
for every \(m\geq 1\). Since \(\xi_\perp\) is completely determined by its coordinates \(\bigl(\langle \xi_\perp,e_k\rangle_{L^2_{\mu_\mathrm{sad}}}\bigr)_{k\geq 1}\), it follows that \(\langle \xi,q_0\rangle_{L^2_{\mu_\mathrm{sad}}}\) is independent of \(\xi_\perp\). This proves the claim.
\end{proof}

\section{Background on Optimal Transport}
\label{section:background on OT}
In this appendix, we recall several standard results from optimal transport that are used throughout the paper. We begin with notation and then state Brenier's theorem, which ensures the existence and structure of optimal
transport maps when the source measure is absolutely continuous.

Let $\mathcal{P}_2^{\mathrm{ac}}(\mathbb R^d) \subset \mathcal{P}_2(\mathbb R^d)$ be the subset of probability measures that are absolutely continuous with respect to the Lebesgue measure. For $\mu,\nu \in \mathcal{P}_2(\mathbb R^d)$, define $W_2^2 (\mu, \nu) = \inf_{\gamma \in \Pi(\mu,\nu)} \int_{\mathbb R^d \times \mathbb R^d} \|x-y\|^2\ \mathrm{d}\gamma(x,y)$, where $\Pi(\mu,\nu)=\{\gamma\in\mathcal{P}_2(\mathbb R^d \times \mathbb R^d): (\pi_0)_\#\gamma=\mu, (\pi_1)_\#\gamma=\nu\}$ denotes the set of couplings between $\mu$ and $\nu$, with $\pi_0(x,y)=x$ and $\pi_1(x,y)=y$. We denote by $\Pi_o(\mu,\nu)$ the set of optimal couplings. When the optimal coupling is of the form $\gamma=(\operatorname{Id}, \operatorname{T}_\mu^\nu)_\#\mu$ with $\operatorname{T}_\mu^\nu\in L^2_\mu$ satisfying $\operatorname{T}_\mu^\nu{_\#}\mu=\nu$, we call $\operatorname{T}_\mu^\nu$ the optimal transport map from $\mu$ to $\nu$.

\begin{theorem}[Brenier's theorem; \protect{\cite[Theorem 6.2.4]{ambrosio2008gradient}}]
\label{thm:existence-optimal-coupling}
Let $\mu \in \mathcal{P}_2^{\mathrm{ac}}(\mathbb R^d)$ and $\nu \in \mathcal{P}_2(\mathbb R^d).$ Then:
\begin{itemize}
\item there exists a unique optimal coupling $\gamma^* := \left(\operatorname{Id}, \operatorname{T}_{\mu}^{\nu}\right)_{\#} \mu$ such that
\begin{equation*}
    W_2^2(\mu,\nu) = \int_{\mathbb R^d} \left|x-\operatorname{T}_{\mu}^{\nu}(x)\right|^2\mathrm{d}\mu(x),
\end{equation*}
where $\operatorname{T}_{\mu}^{\nu}:\mathbb R^d \to \mathbb R^d$ is the unique $\mu$-almost everywhere (a.e.) optimal transport map from $\mu$ to $\nu.$
\item Moreover, $\operatorname{T}_{\mu}^{\nu}(x) = \nabla \varphi(x)$ $\mu\text{-a.e.}$ for a convex function $\varphi: \mathbb R^d \to \mathbb R$ with $\varphi(x) := \frac{1}{2}|x|^2-\psi(x),$ where $\psi$ is a c-concave function.\footnote{A function $\xi:\mathbb R^d \to \mathbb R$ is c-concave if and only if $x \mapsto \frac{1}{2}|x|^2 - \xi(x)$ is convex and lower semi-continuous.}
\item If $\nu \in \mathcal{P}_2^{\lambda}(\mathbb R^d),$ then $\gamma^* = \left(\operatorname{T}_{\nu}^{\mu}, \operatorname{Id}\right)_{\#} \nu,$ where $\operatorname{T}_{\nu}^{\mu}:\mathbb R^d \to \mathbb R^d$ is the unique $\nu$-a.e. optimal transport map such that 
\begin{equation*}
\operatorname{T}_{\nu}^{\mu} \circ \operatorname{T}_{\mu}^{\nu} = \operatorname{Id}, \quad \mu\text{-a.e. and } \operatorname{T}_{\mu}^{\nu} \circ \operatorname{T}_{\nu}^{\mu} = \operatorname{Id}, \quad \nu\text{-a.e.}
\end{equation*}
\end{itemize}  
\end{theorem}
The next result gives a useful sufficient condition ensuring that a map is optimal.
\begin{theorem}[\protect{\cite[Theorem 1.48]{santambrogio2015optimal}}]
\label{thm:sufficient-cond-optimal-convexity}
Suppose $\mu \in \mathcal{P}_2(\mathbb R^d)$, $u: \mathbb R^d \to \mathbb R$ is a convex and differentiable $\mu$-a.e. and $\int_{\mathbb R^d} |\nabla u(x)|^2 \mu(\mathrm{d}x) < \infty.$ Then $\nabla u$ is an optimal transport map from $\mu$ to $(\nabla u)_{\#}\mu.$
\end{theorem}
We conclude with a standard characterization of the tangent velocity of an absolutely continuous curve in Wasserstein space in terms of the infinitesimal behavior of optimal transport.
\begin{proposition}[\protect{\cite[Proposition 8.4.6]{ambrosio2008gradient}}]
    \label{prop_infinitesimal_accurve}
    Let $T > 0$ and let $\mu:[0,T] \to \mathcal{P}_2(\mathbb R^d)$ be an absolutely continuous curve satisfying $\partial_t \mu_t + \nabla \cdot \left(v_t \mu_t\right) = 0$, $t \in (0,T)$, $\mu|_{t=0} := \mu_0$, with vector field $v_t \in \mathcal{T}_{\mu_t} \mathcal{P}_2(\mathbb R^d)$. Then, for a.e. $t \in (0,T)$, we have
    \begin{align*}
        \lim_{h \to 0} \frac{W_2(\mu_{t+h},(\operatorname{Id} + h v_t)_{\#} \mu_t)}{h} = 0. 
    \end{align*}
\end{proposition}
\section{Differential calculus on Wasserstein space}
\label{sec:wass-geom}
In this appendix, we collect the differential notions on Wasserstein space used throughout the paper. We begin with Wasserstein subdifferentials, Wasserstein differentiability, and the tangent space of $\mathcal{P}_2(\mathbb R^d)$. We then recall how these first-order objects can be expressed in terms of first variations, and how the corresponding second-order structure is described through the Wasserstein Hessian and the second variation of the functional.

\subsection{Wasserstein subdifferentials and differentiability}
\label{subsec:wass-diff}
We recall the notions of Wasserstein sub-differentiability and differentiability; see \cite[Chapter 10]{ambrosio2008gradient}, \cite[Chapter 5]{Carmona2018ProbabilisticTO}, and \cite{benoit}.

\begin{definition}[Wasserstein sub- and super-differential]
    Let $F:\mathcal{P}_2(\mathbb R^d)\to \mathbb R$ and let $\mu \in \mathcal{P}_2(\mathbb R^d)$. Then
    \begin{itemize}
        \item a map $\xi\in L^2_\mu$ belongs to the sub-differential $\partial^-F(\mu)$ of $F$ at $\mu$ if for all $\nu\in\mathcal{P}_2(\mathbb R^d)$, 
    \begin{equation*}
        F(\nu) \ge F(\mu) + \sup_{\gamma \in \Pi_o(\mu,\nu)} \int_{\mathbb R^d \times \mathbb R^d} \langle \xi(x), y-x\rangle\ \mathrm{d}\gamma(x,y) + o(W_2(\mu,\nu)).
    \end{equation*}
    If $\partial^-F(\mu) \neq \emptyset,$ we say the function $F$ is Wasserstein sub-differentiable at $\mu.$
    \item A map $\xi \in L_\mu^2$ belongs to the super-differential $\partial^+F(\mu)$ of $F$ at $\mu$ if $-\xi\in\partial^-(-F)(\mu).$ If $\partial^+F(\mu) \neq \emptyset,$ we say the function $F$ is Wasserstein super-differentiable at $\mu.$
    \end{itemize}
\end{definition}
A functional is said to be Wasserstein differentiable when its sub- and super-differentials intersect.
\begin{definition}[Wasserstein differentiability]
    Let $F:\mathcal{P}_2(\mathbb R^d)\to \mathbb R $ be a functional and let $\mu\in \mathcal{P}_2(\mathbb R^d)$. We say $F$ is Wasserstein differentiable at $\mu$ if $\partial^-F(\mu)\cap \partial^+F(\mu) \neq \emptyset$. In this case, we say $\nabla_\mu F(\mu) \in \partial^-F(\mu)\cap \partial^+F(\mu)$ is a Wasserstein gradient of $F$ at $\mu \in \mathcal{P}_2(\mathbb R^d),$ satisfying for any $\nu \in \mathcal{P}_2(\mathbb R^d),$ and $\gamma \in \Pi_o(\mu,\nu),$
    \begin{equation*}
        F(\nu) = F(\mu) + \int_{\mathbb R^d \times \mathbb R^d} \langle \nabla_\mu F(\mu, x), y-x\rangle\ \mathrm{d}\gamma(x,y) + o(W_2(\mu,\nu)).
    \end{equation*}
\end{definition}
Let $C_c^{\infty}(\mathbb R^d)$ denote the space of smooth functions with compact support in $\mathbb R^d.$ The tangent space of $\mathcal{P}_2(\mathbb R^d)$ at $\mu \in \mathcal{P}_2(\mathbb R^d)$ is defined as 
\begin{equation*}
\mathcal{T}_{\mu}\mathcal{P}_2(\mathbb R^d) = \overline{\left\{\nabla \psi: \psi \in C_c^{\infty}(\mathbb R^d) \right\} } \subset L_\mu^2,
\end{equation*}
where the closure is taken in $L_\mu^2$ \cite[Definition 8.4.1]{ambrosio2008gradient}, \cite[Section 5.4.3]{Carmona2018ProbabilisticTO}. Moreover, \cite[Proposition 2.11]{lanzetti} shows that whenever $F$ is Wasserstein differentiable,
there exists a unique Wasserstein gradient in $\mathcal T_\mu\mathcal P_2(\mathbb R^d)$, and we may therefore restrict attention to this canonical representative.

The next result, due to \cite[Proposition 2.12]{lanzetti}, shows that Wasserstein gradients provide a first-order expansion even along perturbations that are not induced by optimal transport plans.
\begin{proposition}[\protect{\cite[Proposition 2.12]{lanzetti}}] \label{prop:strong_diff_w}
    Let $\mu,\nu\in\mathcal{P}_2(\mathbb R^d)$, $\gamma\in\Pi(\mu,\nu)$ and let $F:\mathcal{P}_2(\mathbb{R}^d)\to\mathbb R$ be Wasserstein differentiable at $\mu$ with Wasserstein gradient $\nabla_\mu F(\mu)\in \mathcal{T}_\mu\mathcal{P}_2(\mathbb R^d)$. Then,
    \begin{equation*}
        F(\nu) = F(\mu) + \int_{\mathbb R^d \times \mathbb R^d} \langle \nabla_\mu F(\mu, x), y-x \rangle\mathrm{d}\gamma(x,y) + o\left(\sqrt{\int_{\mathbb R^d\times\mathbb R^d} \|x-y\|_2^2\ \mathrm{d}\gamma(x,y)}\right).
    \end{equation*}
\end{proposition}
Together, these notions provide the first-order differential structure used throughout the paper. 

\subsection{Examples of Wasserstein differentiable functionals}
\label{Examples of Wasserstein differentiable functionals}
We now give examples of Wasserstein (sub)-differentiable functionals.
\begin{example}[The $L^2$-Wasserstein distance; \protect{\cite[Proposition 2.17]{lanzetti}, \cite[Section 9.3, 10.4]{ambrosio2008gradient}}]
    Fix $\nu \in \mathcal{P}_2(\mathbb R^d)$, and for any $\mu \in \mathcal{P}_2(\mathbb R^d)$ consider the functional $\mathcal{W}(\mu) = \frac{1}{2}W_2^2(\mu,\nu)$. If $\mu \in \mathcal{P}_2^{\mathrm{ac}}(\mathbb R^d)$, the functional $\mathcal{W}$ is Wasserstein differentiable at $\mu$ with Wasserstein gradient given by $\nabla_\mu \mathcal{W}(\mu) = \operatorname{Id}-\operatorname{T}_{\mu}^\nu$, where $\operatorname{T}_{\mu}^\nu$ is the (unique) optimal transport map from $\mu$ to $\nu$.
\end{example}
\begin{example}[Potential and interaction energies; \protect{\cite[Proposition 2.20 and 2.21]{lanzetti}}]
    The potential energy $\mathcal{V}(\mu) = \int_{\mathbb R^d} V(x)\mathrm{d}\mu(x)$ and interaction energy $\mathcal{U}(\mu) = \frac12\int_{\mathbb R^d}\int_{\mathbb R^d} U(x-y) \mathrm{d}\mu(x)\mathrm{d}\mu(y)$ with $V:\mathbb R^d\to\mathbb R$ and $U:\mathbb R^d\to\mathbb R$ in $C^2(\mathbb R^d)$ and with uniformly bounded Hessian are Wasserstein differentiable. In particular, their Wasserstein gradients are $\nabla_\mu \mathcal{V}(\mu)=\nabla V$ and $\nabla_\mu \mathcal{U}(\mu) = \nabla U\star \mu:=\int_{\mathbb R^d} U(\cdot-y)\mathrm{d}\mu(y)$, provided that $\nabla U\star \mu \in L^2_\mu$. The Maximum mean Discrepancy (MMD) functional, which may expressed as a sum of the potential energy $\mathcal{V}$  and interaction energy $\mathcal{U}$ is also Wasserstein differentiable \cite{arbel}.
\end{example}
\begin{example}[The entropy functional; \protect{\cite[Proposition 2.25]{lanzetti}}]
    The entropy functional defined as $\mathcal{H}(\mu)=\int_{\mathbb R^d} \frac{\mathrm{d}\mu}{\mathrm{d}x}(x)\log \frac{\mathrm{d}\mu}{\mathrm{d}x}(x)\mathrm{d}x$ for $\mu \in \mathcal{P}_2^{\mathrm{ac}}(\mathbb R^d)$ is not Wasserstein differentiable but sub-differentiable \cite[Proposition 2.25]{lanzetti}. To guarantee that the Wasserstein sub-gradient of $\mathcal{H}$ is non-empty, $\frac{\mathrm{d}\mu}{\mathrm{d}x}$ is required to satisfy Sobolev regularity cf. \cite[Theorem 10.4.13]{ambrosio2008gradient}. Then, if $\nabla\log \frac{\mathrm{d}\mu}{\mathrm{d}x} \in L^2_\mu$, the only sub-gradient of $\mathcal{H}$ in the tangent space is $\nabla_\mu \mathcal{H}(\mu)=\nabla \log \frac{\mathrm{d}\mu}{\mathrm{d}x}$; see \cite[Theorem 10.4.17]{ambrosio2008gradient}. Thus, the relative entropy $\operatorname{KL}(\cdot|\rho)$ with respect to a fixed target measure $\rho \propto e^{-V}$, which is the sum of the potential energy $\mathcal{V}$ and entropy term $\mathcal{H}$, is Wasserstein sub-differentiable \cite[Example 2.29]{lanzetti}.
\end{example}
\subsection{Wasserstein Hessian along geodesics}
We next recall the Wasserstein Hessian along geodesics \cite[Chapter~8]{alma9912197933502466}, which plays the role of the second-order differential of $F$ in Wasserstein space and characterizes the curvature of $F$ along geodesic curves.
\begin{definition}[Wasserstein Hessian along geodesics]
\label{def:wass_hessian}
    The Wasserstein Hessian of $F$ is an operator over $\mathcal{T}_{\mu}\mathcal{P}_2(\mathbb R^d) \subset L_\mu^2$ defined by
    \begin{equation*}
    \operatorname{H}_\mu[v_0](x) := \nabla \nabla_\mu F(\mu,x)v_0(x) + \int_{\mathbb R^d} \nabla_\mu^2 F(\mu, x, \bar{x})v_0(\bar{x})\mathrm d\mu(\bar{x}),
    \end{equation*}
    verifying $\frac{\mathrm{d}^2}{\mathrm{d}t^2}\big|_{t=0}F(\mu_t) = \langle\operatorname{H}_\mu v_0, v_0\rangle_{L^2_\mu}$ if $(\mu_t,v_t)_{t\in [0,1]}$ is a Wasserstein geodesic starting at $\mu \in \mathcal{P}_2(\mathbb R^d)$, with $v_t \in \mathcal{T}_{\mu_t}\mathcal{P}_2(\mathbb R^d)$ satisfying the flow equation $\partial_t \mu_t + \nabla \cdot (v_t \mu_t) = 0$.
\end{definition} 
Definition \ref{def:wass_hessian} and Proposition \ref{proposition:Wasserstein-hessian-appendix} can be used to compute the Wasserstein Hessian of several standard functionals, including the Kullback--Leibler divergence \cite{OTTO2000361}, the maximum mean discrepancy \cite{arbel}, and the kernel Stein discrepancy \cite{korba21a}. Additional examples are provided in Example \ref{example:hessian_potential} and Example \ref{example:hessian_interaction}.

\subsection{First and second variation}
\label{subsection:First and second variation}
Under suitable regularity assumptions, the Wasserstein gradient and Hessian of a functional $F:\mathcal P_2(\mathbb R^d)\to\mathbb R$ can be expressed in terms of its first and second variations. In this subsection, we recall these notions and the assumptions under which the
Wasserstein differential structure of $F$ can be identified with the Euclidean derivatives of $\frac{\delta F}{\delta\mu}$ and $\frac{\delta^2 F}{\delta\mu^2}$. We begin with the first variation \cite[Definition 7.12]{santambrogio2015optimal}.
\begin{definition}[First variation]
\label{def:flat-derivative}
We say $F \in \mathcal{C}^1$, i.e., it admits a first variation at $\mu \in \mathcal{P}_2(\mathbb R^d)$, if there exists a jointly continuous function $\frac{\delta F}{\delta \mu}:\mathcal{P}_2(\mathbb R^d)\times \mathbb R^d \to \mathbb{R},$ unique up to an additive constant, called the first variation of $F,$ for which there exists $\kappa_1 > 0$ such that, for all $(\mu,x) \in \mathcal{P}_2(\mathbb R^d) \times \mathbb R^d,$ $\|\frac{\delta F}{\delta \mu}(\mu,x)\| \leq \kappa_1\left(1+\|x\|^2\right),$ and $\nu \in \mathcal{P}_2(\mathbb R^d),$
\begin{equation*}
\label{eq:flat-der}
\lim_{t \searrow 0 }\frac{F(\mu + t (\nu - \mu))- F(\mu)}{t} =
\int_{\mathbb R^d} \frac{\delta F}{\delta \mu} (\mu, x) \mathrm{d}(\nu-\mu)(x),
\end{equation*}
and $\int_{\mathbb R^d} \frac{\delta F}{\delta \mu} (\mu, x) \mathrm{d}\mu(x)=0.$
\end{definition}
The next assumption ensures that the Wasserstein gradient of $F$ can be identified with the Euclidean gradient of its first variation.
\begin{assumption}
\label{assumption:L-derivative}
Assume that $F \in \mathcal{C}^1$, the map $\mathbb R^d \ni x \mapsto \frac{\delta F}{\delta \mu}(\mu, x) \in \mathbb R$ is differentiable, there exists $\kappa_3 > 0$ such that $\|\nabla_x\frac{\delta F}{\delta \mu}(\mu, x)\| \leq \kappa_3(1+\|x\|)$, for all $(\mu,x)$, and the map $\mathcal{P}_2(\mathbb R^d) \times \mathbb R^d \ni (\mu,x) \mapsto \nabla \frac{\delta F}{\delta \mu}(\mu, x) \in \mathbb R^d$ is jointly continuous.
\end{assumption}
Under Assumption \ref{assumption:L-derivative}, it follows from
\cite[Proposition 5.48; Theorem 5.64]{Carmona2018ProbabilisticTO}
that $F$ is Wasserstein differentiable, with $\nabla_\mu F(\mu,\cdot)=\nabla \frac{\delta F}{\delta\mu}(\mu,\cdot)$.
See also \cite[Proposition 2.2.3]{cardaliaguet2019master}
and \cite[Proposition 5.10]{CheNilRig25OT}.

The same idea extends to second order. If, for each fixed $x\in\mathbb R^d$, the map $\mu \mapsto \frac{\delta F}{\delta\mu}(\mu,x)$
admits a first variation, then one may define the second variation of $F$ \cite[Section 2.2.2]{cardaliaguet2019master}.
\begin{definition}[Second variation]
\label{def:second-variation}
    We say $F \in \mathcal{C}^2,$ i.e., it admits a second variation $\mu \in \mathcal{P}_2(\mathbb R^d)$, if there exists a jointly continuous function $\frac{\delta^2 F}{\delta \mu^2}:\mathcal{P}_2(\mathbb R^d)\times \mathbb R^d \times \mathbb R^d \to \mathbb R$, unique up to an additive term $\phi(x) \in \mathbb R$, called the second variation $F$, for which there exists $\kappa_2(x) > 0$ such that, for all $(\mu,x') \in \mathcal{P}_2(\mathbb R^d) \times \mathbb R^d,$ $\left|\frac{\delta F}{\delta \mu}(\mu,x,x')\right| \leq \kappa_2(x)\left(1+\|x'\|^2\right),$ and $\nu \in \mathcal{P}_2(\mathbb R^d),$
\begin{equation*}
\label{def:2FlatDerivative}
\lim_{t \searrow 0}\frac{1}{t}\left(\frac{\delta F}{\delta \mu}(\mu + t (\nu - \mu),x)-\frac{\delta F}{\delta \mu}(\mu,x)\right) = \int_{\mathbb R^d}\frac{\delta^2 F}{\delta \mu^2}(\mu,x, x')\mathrm{d}(\nu- \mu)(x'),
\end{equation*}
\end{definition}
The next assumption guarantees that the Wasserstein Hessian of $F$ can be computed from the Euclidean derivatives of its first and second variations. Let $C^2$ denote the space of twice continuously differentiable functions.
\begin{assumption}
\label{assumption:2nd-Wasserstein-F}
   Assume that $F \in \mathcal{C}^2$, the maps $\mathbb R^d \ni x\mapsto\frac{\delta F}{\delta\mu}(\mu,x)$, $\mathbb R^d \times \mathbb R^d \ni (x,x')\mapsto \frac{\delta^2 F}{\delta\mu^2}(\mu,x,x')$ are $C^2$, $\|\nabla^2\frac{\delta F}{\delta\mu}(\mu,\cdot)\|_{\mu, \infty} < \infty$, $\nabla_{x'}\nabla_x\frac{\delta^2 F}{\delta\mu^2}(\mu,\cdot,\cdot) \in L^2_{\mu \otimes \mu}$, and, for every $\mu\in\mathcal{P}_2(\mathbb R^d)$, $\frac{\delta}{\delta\mu}\left(\nabla \frac{\delta F}{\delta\mu}(\mu,\cdot)\right)(\cdot) = \nabla \frac{\delta^2 F}{\delta\mu^2}(\mu,\cdot,\cdot)$.
\end{assumption}
The identity $\frac{\delta}{\delta\mu}\!\left(\nabla \frac{\delta F}{\delta\mu}(\mu,\cdot)\right)(\cdot) = \nabla \frac{\delta^2 F}{\delta\mu^2}(\mu,\cdot,\cdot)$ is used when differentiating the Wasserstein gradient $\nabla_\mu F(\mu_t,\cdot)=\nabla_x\frac{\delta F}{\delta\mu}(\mu_t,\cdot)$ along a transport curve $\mu_t=(\operatorname{Id}+tv)_\#\mu$. Under the $C^2$ regularity assumptions above, Schwarz's theorem also implies that the resulting Hessian operator is symmetric on $L_\mu^2$.
\begin{proposition}[Wasserstein Hessian via first and second variations of $F$]
\label{prop:wass-hessian-variations}
Let Assumption \ref{assumption:L-derivative}, \ref{assumption:2nd-Wasserstein-F} hold. Let $\mu \in\mathcal{P}_2(\mathbb R^d)$, $v \in L_\mu^2$ and for $t \in [0,1]$ define the curve $\mu_t = (\pi_t)_{\#}\mu$, with $\pi_t := \operatorname{Id}+tv$. Then, for any $t\in[0,1]$, 
\begin{align*}
\frac{\mathrm d}{\mathrm d t}\nabla_z\frac{\delta F}{\delta\mu}(\mu_t,\pi_t(x)) = \underbrace{\nabla_z^2\frac{\delta F}{\delta\mu}(\mu_t,\pi_t(x))v(x)}_{=\widetilde{\operatorname{M}}_{\mu, t} [v](x)}+ \underbrace{\int_{\mathbb R^d} \nabla_{\bar{z}} \nabla_{z} \frac{\delta^2 F}{\delta\mu^2}(\mu_t, \pi_t(x), \pi_t(\bar{x}))v(\bar{x})\mathrm d\mu(\bar{x})}_{=\widetilde{\operatorname{K}}_{\mu, t}[v](x)}.
\end{align*}
Hence,
\begin{equation*}
    \frac{\mathrm d^2}{\mathrm d t^2} F(\mu_t) = \left\langle (\widetilde{\operatorname{M}}_{\mu, t} + \widetilde{\operatorname{K}}_{\mu, t})v,v\right\rangle_{L_\mu^2} = \left\langle \widetilde{\operatorname{H}}_{\mu, t}v,v\right\rangle_{L_\mu^2}.
\end{equation*}
\end{proposition}
\begin{proof}
Applying Lemma \ref{lemma_for_second_order_mu} to the map $\mu \mapsto \partial_{z_i}\frac{\delta F}{\delta\mu}(\mu_t,\pi_t(x))$ for $1 \leq i \leq d$ with fixed $\pi_t(x)$ gives
\begin{align*}
    \frac{\mathrm d}{\mathrm dt} \left(\partial_{z_i}\frac{\delta F}{\delta\mu}(\mu,\pi_t(x))\right)\Bigg|_{\mu=\mu_t} &= \int_{\mathbb R^d} \left\langle\nabla_\mu \left(\partial_{z_i}\frac{\delta F}{\delta\mu}(\mu_t,\pi_t(x))\right)(\pi_t(\bar{x})),v(\bar{x})\right\rangle\mathrm d\mu(\bar{x})\\
    &= \int_{\mathbb R^d} \left\langle\nabla_{\bar{z}}\frac{\delta}{\delta \mu} \left(\partial_{z_i}\frac{\delta F}{\delta\mu}(\mu_t,\pi_t(x))\right)(\pi_t(\bar{x})),v(\bar{x})\right\rangle\mathrm d\mu(\bar{x})\\
    &=\int_{\mathbb R^d} \left\langle\nabla_{\bar{z}}\, \partial_{z_i} \frac{\delta^2 F}{\delta\mu^2}(\mu_t, \pi_t(x),\pi_t(\bar{x})), v(\bar{x})\right\rangle\mathrm d\mu(\bar{x}).
\end{align*}
Then, noting that $\partial_t \pi_t(x) = v(x)$ and using the chain rule, 
\begin{align*}
    \frac{\mathrm d}{\mathrm d t} \partial_{z_i}\frac{\delta F}{\delta\mu}(\mu_t,\pi_t(x))&= \frac{\mathrm d}{\mathrm dt} \left(\partial_{z_i}\frac{\delta F}{\delta\mu}(\mu_t,u(x))\right)\Bigg|_{u=\pi_t} + \frac{\mathrm d}{\mathrm dt} \left(\partial_{z_i}\frac{\delta F}{\delta\mu}(\mu,\pi_t(x))\right)\Bigg|_{\mu=\mu_t}\\
    &=\left\langle \nabla_z \partial_{z_i}\frac{\delta F}{\delta\mu}(\mu_t,\pi_t(x)), v(x)\right\rangle\\ 
    &+ \int_{\mathbb R^d} \left\langle\nabla_{\bar{z}}\, \partial_{z_i} \frac{\delta^2 F}{\delta\mu^2}(\mu_t,\pi_t(x),\pi_t(\bar{x}), v(\bar{x})\right\rangle\mathrm d\mu(\bar{x}).
\end{align*}
Hence, 
\begin{align*}
\frac{\mathrm d}{\mathrm d t}\nabla_z\frac{\delta F}{\delta\mu}(\mu_t,\pi_t(x))
&= \nabla_z^2\frac{\delta F}{\delta\mu}(\mu_t,\pi_t(x))v(x) + \int_{\mathbb R^d} \nabla_{\bar{z}} \nabla_{z} \frac{\delta^2 F}{\delta\mu^2}(\mu_t, \pi_t(x),\pi_t(\bar{x}))v(\bar{x})\mathrm d\mu(\bar{x}).
\end{align*}
Therefore, by Lemma \ref{lemma_for_second_order_mu},
\begin{align*}
    \frac{\mathrm d^2}{\mathrm d t^2} F(\mu_t) &= \frac{\mathrm d}{\mathrm d t}\int_{\mathbb R^d}\left\langle \nabla_z\frac{\delta F}{\delta\mu}(\mu_t,\pi_t(x)) , v(x)\right\rangle \mathrm{d}\mu(x)\\
    &= \int_{\mathbb R^d}\left\langle \frac{\mathrm d}{\mathrm d t}\nabla_z\frac{\delta F}{\delta\mu}(\mu_t,\pi_t(x)) , v(x)\right\rangle \mathrm{d}\mu(x)\\
    &= \left\langle ((\widetilde{\operatorname{M}}_{\mu, t} + \widetilde{\operatorname{K}}_{\mu, t})v,v\right\rangle_{L_\mu^2}=\left\langle \widetilde{\operatorname{H}}_{\mu, t}v,v\right\rangle_{L_\mu^2}.
\end{align*}
\end{proof}
We conclude with two examples from \cite{bonet2024mirror} for which the assumptions above hold and Proposition \ref{prop:wass-hessian-variations} yields an explicit expression for the Wasserstein Hessian.
\begin{example}[Potential energy; \protect{\cite{bonet2024mirror}}] \label{example:hessian_potential}
    Let $\mathcal{V}(\mu) = \int_{\mathbb R^d} V(x)\mathrm{d}\mu(x)$ with $V \in C^2$ with uniformly bounded Hessian. Then, we have $\frac{\delta\mathcal{V}}{\delta\mu}(\mu,x) = V(x)$ and $\frac{\delta^2\mathcal{V}}{\delta\mu^2}(\mu,x,x') = 0$ (using Definition \ref{def:flat-derivative}, \ref{def:second-variation}). Thus, applying Proposition \ref{prop:wass-hessian-variations}, we recover for $\mu_t = (\pi_t)_\#\mu$,
    \begin{equation*}
        \frac{\mathrm{d}^2}{\mathrm{d} t^2}\mathcal{V}(\mu_t) = \int_{\mathbb R^d} \left\langle \nabla^2 V(\pi_t(x))v(x),v(x)\right\rangle \mathrm{d}\mu(x).
    \end{equation*}
\end{example}
\begin{example}[Interaction energy; \protect{\cite{bonet2024mirror}}] \label{example:hessian_interaction}
    Let $\mathcal{U}(\mu) = \frac12\int_{\mathbb R^d}\int_{\mathbb R^d} U(x-y)\ \mathrm{d}\mu(x)\mathrm{d}\mu(y)$ with $U \in C^2$ symmetric and with uniformly bounded Hessian. Then, we have for all $x,y\in\mathbb R^d$, $\frac{\delta\mathcal{U}}{\delta\mu}(\mu, x) = (U \star \mu)(x)$ and $\frac{\delta^2\mathcal {U}}{\delta\mu^2}(\mu, x, y)= U(x-y)$. Therefore,
    \begin{equation*}
        \nabla_x^2 \frac{\delta\mathcal{U}}{\delta\mu}(\mu, x) = \nabla_x^2 (U \star \mu)(x) = \int_{\mathbb R^d} \nabla_x^2 U(x-y)\mathrm{d}\mu(y), 
    \end{equation*}
    \begin{equation*}
       \nabla_y\nabla_x \frac{\delta^2\mathcal {U}}{\delta\mu^2}(\mu, x, y) = -\nabla_x^2 U(x-y).
    \end{equation*}
    Thus, applying Proposition \ref{prop:wass-hessian-variations}, for $\mu_t=(\pi_t)_\#\mu$, the Hessian operator is
    \begin{equation*}
        \widetilde{\operatorname{H}}_{\mu, t}[v](x) = \int_{\mathbb R^d} \nabla^2 U(\pi_t(x)-\pi_t(\bar{x}))(v(x)-v(\bar{x})) \mathrm{d}\mu(\bar{x}),
    \end{equation*}
    and
    \begin{equation*}
        \frac{\mathrm{d}^2}{\mathrm{d} t^2}\mathcal{U}(\mu_t) = \int_{\mathbb R^d} \int_{\mathbb R^d} \left\langle \nabla^2 U(\pi_t(x)-\pi_t(\bar{x}))(v(x)-v(\bar{x})),v(x) \right\rangle \mathrm{d}\mu(\bar{x})\mathrm{d}\mu(x).
    \end{equation*}
\end{example}

\section{Convexity in Wasserstein space}\label{section:app_conv_in_W_space}
In this appendix, we recall the two notions of convexity on Wasserstein space that are most relevant to the paper: (i) convexity along Wasserstein geodesics and (ii) convexity along linear interpolations of measures. Although both notions play an important role in optimization over probability measures, they are generally distinct and lead to different classes of objective functionals.

\subsection{Geodesic convexity}
We begin with the notion of convexity along Wasserstein geodesics, which is the natural analogue of convexity induced by the geometry of the metric space $(\mathcal P_2(\mathbb R^d),W_2)$ \cite[Definition 9.1.1]{ambrosio2008gradient}.

\begin{definition}\label{def:lambda_convex}
We say a functional $F:\mathcal{P}_2(\mathbb R^d) \to \mathbb R$ is $\alpha$-convex along geodesics, for $\alpha \geq 0$, if for all $\mu,\nu \in \mathcal{P}_2(\mathbb R^d)$, there exists an optimal transport plan $\gamma^* \in \Pi_o(\mu, \nu)$ such that
\begin{equation} \label{eq:geod_curve}
    F(\mu_t) \le (1-t)F(\mu) +tF(\nu) - \alpha \frac{t(1-t)}{2} W_2^2(\mu,\nu),
\end{equation}
for all $t \in [0,1]$, where $\mu_t = (\pi_t)_\#\gamma^* := \left((1-t)x+ty\right)_\# \gamma^*$ is a Wasserstein constant-speed geodesic between $\mu$ and $\nu$. 
\end{definition}
A canonical example of geodesic convexity is provided by the relative entropy with respect to a log-concave reference measure.
\begin{example}[The relative entropy]
Suppose $\rho(\mathrm{d}x) \propto e^{-U(x)}\mathrm{d}x \in \mathcal{P}_2(\mathbb R^d)$, with a convex potential $U:\mathbb{R}^d \to \mathbb{R}$. In this case, we say $\rho$ is a log-concave measure. Thus by \cite[Theorem 9.4.10]{ambrosio2008gradient} the relative entropy $\operatorname{KL}(\cdot|\rho)$ is geodesically convex, and hence
\begin{equation*}
\operatorname{KL}\left(\mu_t| \rho\right) \leq (1-t)\operatorname{KL}(\mu|\rho) + t\operatorname{KL}(\nu|\rho).
\end{equation*}
\end{example}
The following proposition, adapted from \cite[Proposition 16.2]{villani2008optimal}, collects several equivalent ways of characterizing geodesic convexity.
\begin{proposition} \label{prop:equivalences_w_convex}
    Assume $F:\mathcal{P}_2(\mathbb R^d)\to \mathbb R$ is Wasserstein regular. Let $\mu,\nu \in\mathcal{P}_2(\mathbb R^d)$, $\gamma^* \in \Pi_o(\mu, \nu)$ and $\mu_t = (\pi_t)_\#\gamma^*$, for all $t\in [0,1]$, where $\pi_t = (1-t)x+ty$. Then, the following statements are equivalent:
    \begin{enumerate}
        \item For all $t \in [0,1]$, $F(\mu_t) \le (1-t)F(\mu) +t F(\nu)$, i.e., $F$ is convex along the geodesic $t\mapsto \mu_t$. \label{equivalence:c1}
        \item For all $\mu, \nu \in\mathcal{P}_2(\mathbb R^d)$, we have \label{equivalence:c2}
        \begin{equation*}
            F(\nu) \geq F(\mu) + \int_{\mathbb R^d \times \mathbb R^d} \left\langle \nabla_\mu F(\mu,x),y-x \right\rangle \mathrm{d}\gamma^*(x,y).
        \end{equation*}
        \item For all $\mu, \nu \in\mathcal{P}_2(\mathbb R^d)$, we have \label{equivalence:c3}
        \begin{equation*}
            \int_{\mathbb R^d \times \mathbb R^d} \left\langle \nabla_\mu F(\nu, y) - \nabla_\mu F(\mu,x),y-x \right\rangle \mathrm{d}\gamma^*(x,y) \ge 0.
        \end{equation*}
        \item For all $s\in [0,1]$, $\frac{\mathrm{d}^2}{\mathrm{d}t^2}\Big|_{t=s} F(\mu_t) \ge 0$. \label{equivalence:c4}
    \end{enumerate}
\end{proposition}
\begin{proof}
    $\eqref{equivalence:c1} \implies \eqref{equivalence:c2}.$ Let $t\in (0,1)$. Then
    \begin{equation*}
        F(\mu_t) \le (1-t)F(\mu) + t F(\nu) \iff \frac{F(\mu_t)-F(\mu)}{t} \le F(\nu)-F(\mu).
    \end{equation*}
    Passing to the limit $t\to 0$ and using Lemma \ref{lemma_for_second_order}, we get
    \begin{equation*}
        \frac{\mathrm{d}}{\mathrm{d}t}\Bigg|_{t=0} F(\mu_t) = \int_{\mathbb R^d \times \mathbb R^d} \left\langle \nabla_\mu F(\mu, x) , y-x\right\rangle \mathrm{d}\gamma^*(x, y) \le F(\nu) - F(\mu).
    \end{equation*}
    $\eqref{equivalence:c2} \implies \eqref{equivalence:c3}$. Replace the geodesic $\mu_t$ with the geodesic $\mu_{1-t}$ in \eqref{equivalence:c2} to get
    \begin{equation*}
        F(\mu) \geq F(\nu) + \int_{\mathbb R^d \times \mathbb R^d} \left\langle \nabla_\mu F(\nu,y),x-y \right\rangle \mathrm{d}\gamma^*(x,y).
    \end{equation*}
    By adding up the inequality in \eqref{equivalence:c2}, we obtain \eqref{equivalence:c3}.
    \newline
    
    \noindent $\eqref{equivalence:c3} \implies \eqref{equivalence:c4}$. First, we have
    \begin{align*}
        \int_0^1 \frac{\mathrm{d}^2}{\mathrm{d}t^2} F(\mu_t)\mathrm{d}t &= \frac{\mathrm{d}}{\mathrm{d}t}\Bigg|_{t=1} F(\mu_t) - \frac{\mathrm{d}}{\mathrm{d}t}\Bigg|_{t=0} F(\mu_t)\\ &= \int_{\mathbb R^d \times \mathbb R^d} \left\langle \nabla_\mu F(\nu, y) - \nabla_\mu F(\mu,x),y-x \right\rangle \mathrm{d}\gamma^*(x,y) \ge 0.
    \end{align*}
    Fix $s\in(0,1)$ and let $\varepsilon>0$ be such that $s+\varepsilon\le 1$. Define the rescaled curve $\nu_t^\varepsilon := \mu_{s+\varepsilon t}$, for $t\in[0,1]$. Then by the chain rule
    \begin{equation}\label{eq:nonneg-renorm}
       0 \leq \frac{1}{\varepsilon^2} \int_0^1 \frac{\mathrm{d}^2}{\mathrm{d}t^2} F(\nu_t^\varepsilon) \mathrm{d}t = \int_0^1 \frac{\mathrm{d}^2}{\mathrm{d}t^2}\Bigg|_{r = s+\varepsilon t} F(\mu_r) \mathrm{d}t.
    \end{equation}
Because $t \mapsto \frac{\mathrm{d}^2}{\mathrm{d}t^2} F(\mu_t)$ is continuous on $[0,1]$, it is bounded in a neighborhood of $s$, thus there exists $M>0$ such that $\big|\frac{\mathrm{d}^2}{\mathrm{d}t^2}\big|_{r = s+\varepsilon t} F(\mu_r)\big|\le M$ for all $t\in[0,1]$ and all sufficiently small $\varepsilon>0$. Moreover, for each fixed $t\in[0,1]$ we have $\frac{\mathrm{d}^2}{\mathrm{d}t^2}\big|_{r = s+\varepsilon t} F(\mu_r) \to \frac{\mathrm{d}^2}{\mathrm{d}t^2} F(\mu_s)$ as $\varepsilon\to 0$. Hence, by dominated convergence applied to \eqref{eq:nonneg-renorm},
\begin{equation}\label{eq:dct-limit}
0 \leq \lim_{\varepsilon\to 0} \frac{1}{\varepsilon^2}\int_0^1 \frac{\mathrm{d}^2}{\mathrm{d}t^2} F(\nu_t^\varepsilon) \mathrm{d}t =\int_0^1 \frac{\mathrm{d}^2}{\mathrm{d}t^2} F(\mu_s) \mathrm{d}t = \frac{\mathrm{d}^2}{\mathrm{d}t^2} F(\mu_s),
\end{equation}
as claimed.
\newline

\noindent $\eqref{equivalence:c4} \implies \eqref{equivalence:c1}$. Let $\varphi(t) = F(\mu_t)$ for all $t\in [0,1]$. From \cite[Equation 16.5]{villani2008optimal}, for all $t\in [0,1]$,
    \begin{equation*}
         \varphi(t) = (1-t)\varphi(0) + t \varphi(1) - \int_0^1 G(s,t)\frac{\mathrm{d}^2}{\mathrm{d}t^2}\varphi(s)\mathrm{d}s,
    \end{equation*}
    where $G$ is the Green function defined as $G(s,t) = s(1-t)\mathbbm{1}_{\{s\le t\}} + t(1-s)\mathbbm{1}_{\{t\le s\}} \ge 0$ in \cite[Equation 16.6]{villani2008optimal}. Then, $\frac{\mathrm{d}^2}{\mathrm{d}t^2}F(\mu_s) \ge 0$ implies that $\int_0^1 G(s,t)\frac{\mathrm{d}^2}{\mathrm{d}t^2}\varphi(s) \mathrm{d}s \ge 0$, and thus
    \begin{equation*}
        \varphi(t) = F(\mu_t) \le (1-t)\varphi(0) + t\varphi(1) = (1-t)F(\mu) + tF(\nu).
    \end{equation*}
\end{proof}
Geodesic convexity is intrinsic to the Wasserstein geometry, but it is not the only convexity notion that appears in optimization over measures. In many mean-field and learning problems, one instead works with convexity along linear interpolations of probability measures.

\subsection{Linear convexity}
We now turn to convexity along linear interpolations of measures, which is widely used in the mean-field optimization literature and is generally different from geodesic convexity.
\begin{definition}[Linear convexity of $F$]
\label{ass:convexity}
We say a functional $F:\mathcal{P}_2(\mathbb R^d) \to \mathbb R$ is linearly convex, if for any $\mu, \nu \in \mathcal{P}_2(\mathbb R^d)$ and any $t \in [0,1]$,
\begin{equation}\label{eq:convexity_classical}
    F((1-t)\mu + t \nu) \leq (1-t)F(\mu) + t F(\nu).
\end{equation}
\end{definition}
This notion is standard in the mean-field optimization literature; see, for instance, \cite{10.1214/20-AIHP1140,Nitanda2022ConvexAO,chizat2022meanfield,lascu2024linearconvergenceproximaldescent}. Moreover, if $F\in\mathcal C^1$, then linear convexity implies the first-order inequality
\[
F(\nu)-F(\mu)\geq \int_{\mathbb R^d}\frac{\delta F}{\delta\mu}(\mu,x)\,\mathrm d(\nu-\mu)(x),
\]
see \cite[Lemma 4.1]{10.1214/20-AIHP1140}.

The following examples illustrate that linear convexity and geodesic convexity need not coincide. In particular, some objectives of interest in generative modeling are linearly convex without being geodesically convex.
\begin{example}[$L^2$-Wasserstein and MMD distances]
    Examples of optimization problems that are linearly convex but not geodesically convex include minimizing distances with respect to a fixed target distribution, a setting common in generative modeling \cite{huang2024generativemodelingminimizingwasserstein2, arbel}. One example is the minimization of the squared $L^2$-Wasserstein distance $F(\mu) = \frac{1}{2}\mathcal{W}_2^2(\mu, \mu^*)$, studied in \cite{huang2024generativemodelingminimizingwasserstein2}, which can be interpreted as training a Generative Adversarial Network (GAN) generator to match a target measure $\mu^*$ in $\mathcal{W}_2$. While $F$ in this case fails to be geodesically convex (see, e.g., \cite[Example 9.1.5]{ambrosio2008gradient}), it is linearly convex. Another example, given in \cite{arbel}, is the minimization of the Maximum Mean Discrepancy (MMD) $F(\mu) = \frac{1}{2}\operatorname{MMD}^2(\mu, \mu^*)$. Here, again, $F$ is linearly convex but not geodesically convex \cite[Section 3.1]{arbel}.
\end{example}

\section{Optimality conditions in the Wasserstein space}
\label{optimality conditions}
In this appendix, we collect the first and second-order optimality notions used in the paper. We begin with the standard first-order characterization of minimizers in terms of the Wasserstein gradient, and then turn to second-order conditions involving the operator $\operatorname{K}_\mu$. We conclude by recording the consequences of the saddle-property assumption, which links second-order criticality to global minimality.

\subsection{First-order optimality conditions}
We start with the first-order notions of minimality and stationarity. These are the Wasserstein counterparts of the familiar first-order optimality conditions from finite-dimensional optimization.
\begin{definition}[Local, global and strict minimizers]\label{def:minima unconstrained}
A probability measure $\mu^*\in\mathcal{P}_2(\mathbb R^d)$ is a local minimizer of $F$ if there exists $r>0$ such that $F(\mu)\geq F(\mu^*)$ for all $\mu \in B_{2}(r, \mu^*)$. If $F(\mu)>F(\mu^*)$ for all $\mu\in B_{2}(r, \mu^*),$ $\mu\neq \mu^*$, then $\mu^*$ is the strict local minimizer of $F$. If this holds for any $r>0$, then $\mu^*$ is a (strict) global minimizer.
\end{definition}
\begin{definition}[First-order stationary point]
\label{def:first-order-stationary}
    We say $\mu^*\in\mathcal{P}_2(\mathbb R^d)$ is a first-order stationary point of $F$ if $\nabla_\mu F(\mu^*,\cdot) = 0$ in $L_{\mu^*}^2$.
\end{definition}
\begin{definition}[$\varepsilon$-first-order stationary point]
    We say $\mu^*\in\mathcal{P}_2(\mathbb R^d)$ is an $\varepsilon$-first-order stationary point of $F$ if $\|\nabla_\mu F(\mu^*,\cdot)\|_{L_{\mu^*}^2} \leq \varepsilon$.
\end{definition}
The following theorem, due to \cite{lanzetti}, gives the basic first-order necessary condition for local optimality.
\begin{theorem}[First-order necessary optimality condition; \protect{\cite[Theorem 3.2]{lanzetti}}]
\label{thm:first-order-optimality-minimizers}
    Let $\mu^*$ be a local minimizer of $F$ and assume $F$ is Wasserstein differentiable at $\mu^*$.  Then, the Wasserstein gradient of $F$ vanishes at $\mu^*$, i.e., $\nabla_\mu F(\mu^*,\cdot) = 0$ in $L_{\mu^*}^2$.
\end{theorem}
Under an additional geodesic convexity assumption, first-order stationarity also becomes sufficient for global optimality.
\begin{theorem}[First-order sufficient optimality condition; \protect{\cite[Theorem 3.3]{lanzetti}}]
Assume that $F$ is $\alpha$-geodesically convex with $\alpha \geq 0$. Suppose there exists $\mu^*\in \mathcal{P}_2(\mathbb R^d)$ such that $F$ is Wasserstein differentiable at $\mu^*$ and the Wasserstein gradient of $F$ vanishes at $\mu^*$, i.e., $\nabla_\mu F(\mu^*,\cdot)=0$, $\mu^*$-a.e. Then, $\mu^*$ is a global minimizer of $F$. If $\alpha > 0$, $\mu^*$ is the unique global minimizer of $F$.
\end{theorem}

\subsection{Second-order optimality conditions}
We next introduce the second-order optimality notions used in the paper. As explained in the main text, when $\mu^*$ is a critical point, the second-order behavior is determined by the operator $\operatorname{K}_{\mu^*}$, which motivates the definitions below.
\begin{definition}[Second-order critical and saddle point]
\label{def:second-order-stationary-strict-saddle}
    We say $\mu^*\in\mathcal{P}_2^{\mathrm{ac}}(\mathbb R^d)$ satisfying $\nabla_\mu F(\mu^*,\cdot) = 0$ in $L_{\mu^*}^2$ is (i) a second-order critical point of $F$ if $\langle\operatorname{K}_{\mu^*}v,v\rangle_{L_{\mu^*}^2} \geq 0$, for all $v \in L_{\mu^*}^2$, and (ii) a saddle point of $F$ if $\underline{\lambda}(\operatorname{K}_{\mu^*})< 0$. 
\end{definition}
The next result provides the second-order necessary condition satisfied by local minimizers.
\begin{proposition}[Second-order necessary optimality condition; \protect{\cite[Proposition 3.3]{yamamoto2025hessianguided}}]
\label{prop:necessary-second-order-optimality-minimizers}
    Assume $F:\mathcal{P}_2(\mathbb R^d) \to \mathbb R$ is Wasserstein regular (cf. Definition \ref{def:regularity-F}). Suppose $\mu^* \in \mathcal{P}_2^{\mathrm{ac}}(\mathbb R^d)$ is a local minimizer of $F$, i.e., there exists $r > 0$ such that $F(\mu) \geq F(\mu^*)$, for all $\mu \in B_2(r,\mu^*)$. Assume further that the maps $(\mu ,x )\mapsto \nabla \nabla_\mu F(\mu,x)$ and $(\mu ,x,\bar{x})\mapsto \nabla_\mu^2 F(\mu,x,\bar{x})$ are jointly continuous. Then $\operatorname{K}_{\mu^*} \succeq 0$. 
\end{proposition}
\begin{proof}
By Theorem~\ref{thm:first-order-optimality-minimizers}, $\nabla_\mu F(\mu^*,\cdot)=0$ in $L^2_{\mu^*}$. Since $\mu^*$ is absolutely continuous, $\operatorname{M}_{\mu^*}=0$ $\mu^*$-a.e. \cite[Lemma C.10.]{yamamoto2025hessianguided}, and hence $\operatorname{H}_{\mu^*}=\operatorname{K}_{\mu^*}$. Fix $v\in L^2_{\mu^*}$ and set $\mu_t:=(\operatorname{Id}+tv)_\#\mu^*.$ The coupling induced by this map gives $W_2(\mu_t,\mu^*)\leq |t|\|v\|_{L^2_{\mu^*}}.$ Thus $\mu_t$ belongs to the neighborhood of local minimality for sufficiently small $|t|$. The function $f(t):=F(\mu_t)$ therefore has a local minimum at $t=0$. By Proposition~\ref{proposition:Wasserstein-hessian-appendix},
\begin{equation*}
    0\leq f''(0) = \langle \operatorname{H}_{\mu^*}v,v\rangle_{L^2_{\mu^*}} = \langle \operatorname{K}_{\mu^*}v,v\rangle_{L^2_{\mu^*}}.
\end{equation*}
Hence $\operatorname{K}_{\mu^*}\succeq0$. 
\end{proof}

We continue with the proof of a second-order sufficient optimality condition, which relies only on the local second-order geometry of the objective.

\begin{proposition}[Second-order sufficient condition under local radial convexity]
\label{prop:second-order-suff}
Assume that $F:\mathcal P_2(\mathbb R^d)\to\mathbb R$ is Wasserstein
regular. Let $\mu^*\in\mathcal P_{2}^{\mathrm{ac}}(\mathbb R^d)$ satisfy $ \nabla_\mu F(\mu^*,\cdot)=0$ in $L^2_{\mu^*}.$ Assume that $\operatorname{K}_{\mu^*}$ is strictly positive on the Wasserstein tangent space, i.e., 
\begin{equation}
\label{eq:strict-positivity-K}
    \left\langle \operatorname{K}_{\mu^*}v,v\right\rangle_{L^2_{\mu^*}}>0,
\end{equation} 
for every $ v\in\mathcal{T}_{\mu^*}\mathcal P_2(\mathbb R^d)\setminus\{0\}.$ Suppose that there exists $r>0$ such that the following holds. For every $\mu\in B_2(r,\mu^*)$, let $\operatorname{T}_{\mu^*}^{\mu}$ be the optimal transport map from $\mu^*$ to $\mu$ and set $v_\mu:=\operatorname{T}_{\mu^*}^{\mu}-\operatorname{Id},$ $\mu_t:=\left(\operatorname{Id}+t v_\mu\right)_\#\mu^*,$ $t\in[0,1].$ Assume that
\begin{equation}
\label{eq:local-radial-convexity}
    \left\langle \widetilde{\operatorname{H}}_{\mu^*, t}v_\mu,v_\mu \right\rangle_{L^2_{\mu^*}} \geq 0,
\end{equation}
for every $t\in[0,1],$ and that the map $t\mapsto \langle\widetilde{\operatorname{H}}_{\mu^*, t}v_\mu,v_\mu\rangle_{L^2_{\mu^*}}$ is continuous at $t=0$. Then $\mu^*$ is a strict local minimizer of $F,$ i.e., $F(\mu)>F(\mu^*),$ for every $\mu\in B_2(r,\mu^*)\setminus\{\mu^*\}.$
\end{proposition}

\begin{proof}
Fix $\mu\in B_2(r,\mu^*)\setminus\{\mu^*\}$. Since $\mu^*$ is absolutely continuous, $\|v_\mu\|_{L^2_{\mu^*}}=W_2(\mu,\mu^*)>0.$ Moreover, $v_\mu\in\mathcal{T}_{\mu^*}\mathcal P_2(\mathbb R^d)$. Define $f(t):=F(\mu_t)$. By Proposition \ref{proposition:Wasserstein-hessian-appendix},
\[
    f''(t) = \left\langle\widetilde{\operatorname{H}}_{\mu^*, t}v_\mu,v_\mu \right\rangle_{L^2_{\mu^*}}.
\]
Since \(\mu^*\) is stationary, Definition \ref{def:first-order-stationary} implies
\[
    f'(0) = \left\langle\nabla_\mu F(\mu^*,\cdot),v_\mu\right\rangle_{L^2_{\mu^*}} =0.
\]
Furthermore, stationarity and absolute continuity give $\operatorname{M}_{\mu^*}=0$ \cite[Lemma C.10.]{yamamoto2025hessianguided}, and hence $\widetilde{\operatorname{H}}_{\mu^*, 0}=\operatorname{H}_{\mu^*}=\operatorname{K}_{\mu^*}.$ Therefore, by \eqref{eq:strict-positivity-K},
\[
    f''(0) = \left\langle \operatorname{K}_{\mu^*}v_\mu,v_\mu\right\rangle_{L^2_{\mu^*}} >0.
\]
By the continuity assumption, there exists \(s_\mu\in(0,1]\) such that
\[
    f''(t) \geq \frac12 \left\langle \operatorname{K}_{\mu^*}v_\mu,v_\mu\right\rangle_{L^2_{\mu^*}} >0,
\]
for all $t\in[0,s_\mu].$ On the other hand, local radial convexity gives $f''(t)\geq0$ on all of $[0,1]$. Taylor's formula with integral remainder therefore implies
\begin{align*}
    F(\mu)-F(\mu^*) =
    \int_0^1(1-t)f''(t)\mathrm dt &\geq
    \frac12
    \left\langle \operatorname{K}_{\mu^*}v_\mu,v_\mu\right\rangle_{L^2_{\mu^*}}
    \int_0^{s_\mu}(1-t)\mathrm dt \\
    &=
    \frac12
    \left(s_\mu-\frac{s_\mu^2}{2}\right)
    \left\langle \operatorname{K}_{\mu^*}v_\mu,v_\mu\right\rangle_{L^2_{\mu^*}}
    >0.
\end{align*}
Thus \(F(\mu)>F(\mu^*)\) for every \(\mu\in B_2(r,\mu^*)\setminus\{\mu^*\}\).
\end{proof}
\begin{remark}[Local radial convexity]
Condition \eqref{eq:local-radial-convexity} requires convexity only along geodesics starting from $\mu^*$, which is weaker than geodesic convexity over all $B_2(r,\mu^*)$. Because $\operatorname{K}_{\mu^*}$ is compact, its positive eigenvalues can accumulate at zero, preventing a uniform lower bound $\operatorname{K}_{\mu^*}\succeq m\operatorname{I_{d \times d}},$ $m>0$. Standard Hessian continuity is insufficient to guarantee preserved positivity in directions with arbitrarily small curvature. Strict positivity of $\operatorname{K}_{\mu^*}$ ensures positive initial curvature along these geodesics, continuity maintains this positivity over an initial segment, and radial convexity prevents it from dropping below zero later. Together, this gives strict local minimality.
\end{remark}

\paragraph{Spectrally projected transport neighborhoods.}
Let $\mu^*\in\mathcal P_2^{\mathrm{ac}}(\mathbb R^d)$. Under Wasserstein regularity, $\operatorname{K}_{\mu^*}$ is compact and self-adjoint on $L^2_{\mu^*}$. For $m>0$, define
\begin{equation}
\label{eq:spectral-projection-positive-modes}
    P_m:=\mathbbm 1_{[m,\infty)}(\operatorname{K}_{\mu^*}), \quad
    V_m:=\operatorname{Ran}P_m.
\end{equation}
We consider cutoffs for which $V_m\neq\{0\}$. Compactness implies that $K_{\mu^*}$ has only finitely many eigenvalues, counted with multiplicity, in $[m,\infty)$. Thus $V_m$ is finite-dimensional and is spanned by the eigenvectors whose eigenvalues are at least $m$. By the spectral calculus,
\begin{equation}
\label{eq:coercivity-on-spectral-subspace}
    \langle \operatorname{K}_{\mu^*}v,v\rangle_{L^2_{\mu^*}} \geq m\|v\|_{L^2_{\mu^*}}^2, \quad v\in V_m.
\end{equation}
Because \(\mu^*\) is absolutely continuous, for every \(\mu\in\mathcal P_2(\mathbb R^d)\) there exists a unique \(\mu^*\)-a.e. optimal transport map \(\operatorname{T}_{\mu^*}^{\mu}\) from \(\mu^*\) to \(\mu\). Define the projected transport family
\begin{equation*}
\label{eq:projected-transport-family}
    \mathcal M_m(\mu^*) :=\left\{\mu\in\mathcal P_2(\mathbb R^d): \operatorname{T}_{\mu^*}^{\mu}-\operatorname{Id}\in V_m\right\},
\end{equation*}
which represents the collection of measures reachable from $\mu^*$ through optimal displacement fields belonging to $V_m$, and its restricted Wasserstein neighborhood
\begin{equation*}
\label{eq:projected-wasserstein-ball}
    B_{2,m}(r,\mu^*) := B_2(r,\mu^*)\cap\mathcal M_m(\mu^*).
\end{equation*}
Now we can prove that \(\mu^*\) is a strict local minimizer of \(F\) relative to the projected transport family \(\mathcal M_m(\mu^*)\).

\begin{proposition}[Projected second-order sufficient optimality condition]
\label{prop:projected-second-order-sufficient-condition}
Assume that \(F:\mathcal P_2(\mathbb R^d)\to\mathbb R\) is Wasserstein
regular. Let \(\mu^*\in\mathcal P_{2}^{\mathrm{ac}}(\mathbb R^d)\)
satisfy $\nabla_\mu F(\mu^*,\cdot)=0$ in $L^2_{\mu^*}$. Fix \(m>0\) such that \(V_m\neq\{0\}\), where \(V_m\) is defined in \eqref{eq:spectral-projection-positive-modes}. For
\(\mu\in\mathcal M_m(\mu^*)\), set $v_\mu:=\operatorname{T}_{\mu^*}^{\mu}-\operatorname{Id},$ $\mu_t:=\left(\operatorname{Id}+t v_\mu\right)_\#\mu^*,$ $t\in[0,1].$ Assume that there exists $r_0>0$ and a nondecreasing function \(\omega_m:[0,\infty)\to[0,\infty)\), satisfying $\lim_{\rho\downarrow0}\omega_m(\rho)=0,$ such that, for every \(\mu\in B_{2,m}(r_0,\mu^*)\) and every \(t\in[0,1]\),
\begin{equation}
\label{eq:projected-hessian-continuity}
    \left\|P_m\left(\widetilde{\operatorname{H}}_{\mu^*,t}-\operatorname{K}_{\mu^*}\right)P_m\right\|_{\mathrm{op}} \leq \omega_m\left(W_2(\mu,\mu^*)\right).
\end{equation}
Then there exists $r_m\in(0,r_0)$ such that
\begin{equation*}
    F(\mu) \geq F(\mu^*) + \frac{m}{4}W_2^2(\mu,\mu^*), 
\end{equation*}
for all $\mu\in B_{2,m}(r_m,\mu^*).$ In particular, $\mu^*$ is a strict local minimizer relative to $\mathcal M_m(\mu^*)$.
\end{proposition}

\begin{proof}
Choose $r_m\in(0,r_0)$ sufficiently small that
\begin{equation}
\label{eq:choice-projected-radius}
    \omega_m(r_m)\leq\frac{m}{2}.
\end{equation}
Fix \(\mu\in B_{2,m}(r_m,\mu^*)\). By the definition of \(\mathcal M_m(\mu^*)\), we have \(v_\mu\in V_m\), and hence \(P_mv_\mu=v_\mu\). Since \(\operatorname{T}_{\mu^*}^{\mu}\) is optimal,
\begin{equation}
\label{eq:projected-displacement-distance}
    \|v_\mu\|_{L^2_{\mu^*}} = W_2(\mu,\mu^*).
\end{equation}
Define \(f(t):=F(\mu_t)\). Proposition \ref{proposition:Wasserstein-hessian-appendix} gives
\begin{equation*}
    f''(t) =\left\langle \widetilde {\operatorname{H}}_{\mu^*,t} v_\mu,v_\mu\right\rangle_{L^2_{\mu^*}}.
\end{equation*}
Since \(v_\mu=P_mv_\mu\), we obtain
\begin{align*}
    f''(t) &= \left\langle \operatorname{K}_{\mu^*}v_\mu,v_\mu\right\rangle_{L^2_{\mu^*}} +
    \left\langle \left(\widetilde {\operatorname{H}}_{\mu^*,t} -\operatorname{K}_{\mu^*}\right)v_\mu,v_\mu\right\rangle_{L^2_{\mu^*}}\\
    &=\left\langle \operatorname{K}_{\mu^*}v_\mu,v_\mu\right\rangle_{L^2_{\mu^*}}
    +\left\langle P_m\left(\widetilde {\operatorname{H}}_{\mu^*,t} -\operatorname{K}_{\mu^*}\right)P_mv_\mu,v_\mu \right\rangle_{L^2_{\mu^*}}.
\end{align*}
Using \eqref{eq:coercivity-on-spectral-subspace}, \eqref{eq:projected-hessian-continuity}, and \eqref{eq:choice-projected-radius}, it follows that
\begin{align*}
    f''(t) \geq m\|v_\mu\|_{L^2_{\mu^*}}^2 - \omega_m\left(W_2(\mu,\mu^*)\right)\|v_\mu\|_{L^2_{\mu^*}}^2
    \geq \frac{m}{2}\|v_\mu\|_{L^2_{\mu^*}}^2,
    \quad t\in[0,1].
\end{align*}
The stationarity condition gives
\[
    f'(0) =\left\langle \nabla_\mu F(\mu^*,\cdot),v_\mu \right\rangle_{L^2_{\mu^*}}=0.
\]
Therefore, Taylor's formula with integral remainder gives
\begin{align*}
    F(\mu)-F(\mu^*) = \int_0^1(1-t)f''(t)\mathrm dt \geq \frac{m}{2}\|v_\mu\|_{L^2_{\mu^*}}^2 \int_0^1(1-t)\mathrm dt &= \frac{m}{4}\|v_\mu\|_{L^2_{\mu^*}}^2= \frac{m}{4}W_2^2(\mu,\mu^*),
\end{align*}
where the final identity follows from \eqref{eq:projected-displacement-distance}. For $\mu\neq\mu^*$, the right-hand side is strictly positive, proving the relative strict local minimality.
\end{proof}
The conclusion concerns strict local minimality relative to $\mathcal M_m(\mu^*)$. The operator $\operatorname{K}_{\mu^*}$ may have a nontrivial kernel or negative spectrum on $V_m^\perp$. The proof requires coercivity only on $V_m$, where the spectral cutoff supplies the explicit bound $\operatorname{K}_{\mu^*}|_{V_m}\succeq mI_{V_m}$. Hence, the proposition does not assert local minimality with respect to arbitrary Wasserstein perturbations.

\paragraph{Finite-dimensional restriction.} Under the present compact Hessian setting, a uniform positive quadratic growth bound cannot hold on an entire Wasserstein neighborhood of an absolutely continuous stationary measure. Indeed, suppose that there exist $c,r>0$ such that
\begin{equation}
\label{eq:full-wasserstein-quadratic-growth-impossible}
    F(\mu)\geq F(\mu^*)+cW_2^2(\mu,\mu^*), \quad \mu\in B_2(r,\mu^*).
\end{equation}
Take $v=\nabla\phi$ with $\phi\in C_c^\infty(\mathbb R^d)$ and set $\mu_s=(\operatorname{Id}+sv)_\#\mu^*$. For sufficiently small $s>0$, the function $x\mapsto\frac12\|x\|^2+s\phi(x)$ is convex. Hence, by \ref{thm:sufficient-cond-optimal-convexity}, $\operatorname{Id}+s\nabla\phi$ is an optimal transport map, and $W_2(\mu_s,\mu^*)=s\|v\|_{L^2_{\mu^*}}.$ Stationarity, \(\operatorname{H}_{\mu^*}=\operatorname{K}_{\mu^*}\), and the second-order expansion in Lemma \ref{lem:2ndorder-expansion} give
\begin{equation*}
    F(\mu_s)-F(\mu^*) = \frac{s^2}{2}\langle \operatorname{K}_{\mu^*}v,v\rangle_{L^2_{\mu^*}} +o(s^2).
\end{equation*}
Combining the two bounds, dividing by $s^2$, and letting $s\downarrow0$ gives
\begin{equation}
\label{eq:quadratic-growth-implies-coercivity}
    \langle \operatorname{K}_{\mu^*}v,v\rangle_{L^2_{\mu^*}}
    \geq 2c\|v\|_{L^2_{\mu^*}}^2.
\end{equation}
This holds for every $v=\nabla\phi$ with $\phi\in C_c^\infty(\mathbb R^d)$. By density and boundedness of $\operatorname{K}_{\mu^*}$, it extends to $\mathcal T_{\mu^*}\mathcal P_2(\mathbb R^d)$. Choose an orthonormal sequence $(v_j)_{j\geq1}$ in $\mathcal T_{\mu^*}\mathcal P_2(\mathbb R^d)$. Then $v_j\rightharpoonup0$ in $L^2_{\mu^*}$, and compactness implies $\operatorname{K}_{\mu^*}v_j\to0$ strongly. Thus $\langle \operatorname{K}_{\mu^*}v_j,v_j\rangle_{L^2_{\mu^*}}\to0,$ contradicting the preceding lower bound. More generally, a compact operator cannot satisfy
a uniform positive quadratic lower bound on an infinite-dimensional linear subspace. A fixed positive spectral band provides a finite-dimensional family on which such a bound is available.

\paragraph{The spectral cutoff is a natural projection.} The spectral cutoff is canonical and selects exactly the eigendirections whose curvature is at least $m$. Moreover, $P_m$ commutes with $\operatorname{K}_{\mu^*}$ and with every bounded spectral function of $\operatorname{K}_{\mu^*}$, including $(\operatorname{K}_{\mu^*}^2+\beta \operatorname{I_{d \times d}})^{-1/2}$. If the positive eigenvalues accumulate at zero, decreasing $m$ retains progressively more positive spectral directions while weakening the available
uniform curvature bound. If there are only finitely many positive eigenvalues, $V_m$ stabilizes once $m$ falls below the smallest positive eigenvalue.

\subsection{Consequences under the saddle property}
We now combine the optimality conditions above with the saddle property assumption from the main text. This allows us to relate approximate second-order critical points to the global minimizer set. 

\begin{lemma}[Approximate second-order critical points are close to the global minimizer set]
\label{second-order stationary close to local}
Suppose Assumption~\ref{def:quantitative-strict-saddle} holds on $\mathcal S\subseteq\mathcal P_2(\mathbb R^d)$. If $\mu\in \mathcal S$ satisfies $\|\nabla_\mu F(\mu,\cdot)\|_{L^2_\mu} \le\varepsilon,$ and $\underline{\lambda}(\operatorname{K}_\mu)\ge-\delta,$ then $\mu\in\mathcal N_{\alpha_{\mathrm{ben}}}(\mathcal M^*)$.
\end{lemma}
\begin{proof}
Otherwise, Assumption~\ref{def:quantitative-strict-saddle} would imply either $\|\nabla_\mu F(\mu,\cdot)\|_{L^2_\mu}>\varepsilon$ or $\underline{\lambda}(\operatorname{K}_\mu)<-\delta$, contradicting the hypotheses.
\end{proof}
If Assumption \ref{def:quantitative-strict-saddle} holds with $\alpha_{\mathrm{ben}}=0$, and $\mathcal{M}^*$ is closed, then every second-order critical point in $\mathcal S$ is a global minimizer.

\section{First and second-order dynamics on Wasserstein space}
\label{first and second order dynamics}
In this appendix, we briefly recall the first and second-order transport schemes associated with the minimization of $F$ over Wasserstein space.
\subsection{Wasserstein gradient flow and Wasserstein gradient descent}
We begin with the Benamou--Brenier formula, which characterizes the $L^2$-Wasserstein distance in dynamic form; see, e.g., \cite[Theorem 4.1.3]{figalli}. Given $\mu_0,\mu_1 \in \mathcal{P}_2(\mathbb R^d),$ 
    \begin{equation*}
         W_2^2(\mu_0,\mu_1) = \inf\left\{\int_0^1 \int_{\mathbb R^d} |v_s(x)|^2\mu_s(\mathrm{d}x)\mathrm{d}s: \partial_s \mu_s + \nabla \cdot\left(v_s\mu_s\right) = 0\right\},
     \end{equation*}
where the infimum is taken over all curves $[0,1] \ni t \mapsto (\mu_t,v_t) \in \mathcal{P}_2(\mathbb R^d) \times L_{\mu_t}^2$ such that $t \mapsto \mu_t$ is continuous with respect to weak convergence topology. 

This suggests defining the steepest-descent dynamics for $F$ by choosing the velocity field $v_t=-\nabla_\mu F(\mu_t, \cdot)$, where $\nabla_\mu F$ belongs of the tangent space of $\mathcal{P}_2(\mathbb R^d)$ at $\mu_t$. Define the first-order Taylor model by
\begin{equation*}
    \mathfrak T_1^F(\mu)[v] := F(\mu)+\langle\nabla_\mu F(\mu,\cdot),v\rangle_{L^2_\mu}.
\end{equation*}
By Lemma~\ref{lem:1storder-expansion},
\begin{equation*}
    F((\operatorname{Id}+v)_\#\mu) = \mathfrak T_1^F(\mu)[v]+o(\|v\|_{L^2_\mu}), \quad \|v\|_{L^2_\mu}\to0.
\end{equation*}
Adding a quadratic penalty defines the gradient descent direction
\begin{equation*}
    v_{\mathrm{GD}}(\mu^n):=\operatorname*{argmin}_{v\in L^2_{\mu^n}}\left\{\mathfrak T_1^F(\mu^n)[v]+\frac12\|v\|_{L^2_{\mu^n}}^2\right\}.
\end{equation*}
The objective is strongly convex, and its unique minimizer satisfies
\begin{equation*}
    v_{\mathrm{GD}}(\mu^n)=-\nabla_\mu F(\mu^n,\cdot).
\end{equation*}
Thus, for a stepsize $\tau>0$, the update is
\begin{equation*}
    \mu^{n+1}=(\operatorname{Id}+\tau v_{\mathrm{GD}}(\mu^n))_\#\mu^n=(\operatorname{Id}-\tau\nabla_\mu F(\mu^n,\cdot))_\#\mu^n, \quad \mu^0 \in \mathcal{P}_2(\mathbb R^d).
\end{equation*}
If $(\mu,x) \mapsto \nabla_\mu F(\mu,x)$ is $L_F$-Lipschitz (cf. Proposition \ref{prop convergence explicit Euler}), then Lemma \ref{lemma:smoothness-of-F-1} implies that, for $\tau \leq L_F^{-1}$,
    \begin{equation*}
        F( \mu^{n+1}) \leq F(\mu^n) - \frac{\tau}{2}\|\nabla_\mu F(\mu^n,\cdot)\|_{L_{\mu^n}^2}^2 \leq F(\mu^n).
    \end{equation*}
Hence Wasserstein gradient descent is a descent method and becomes stationary exactly at first-order critical points, that is, when $\nabla_\mu F(\mu,\cdot)=0$ in $L_\mu^2$ (cf. Definition \ref{def:first-order-stationary}). By Proposition \ref{prop convergence explicit Euler}, passing to the limit $\tau\to 0$ yields the Wasserstein gradient flow
    \begin{equation}
    \label{eq:wassgf}
        \partial_t \mu_t = \nabla \cdot \left(\nabla_\mu F(\mu_t, \cdot)  \mu_t\right), \text{ for } t > 0, \quad
        \mu|_{t=0} := \mu^0.
    \end{equation}
    By Proposition \ref{prop_chain_rule}, $F$ dissipates along the flow, i.e.,
    \begin{equation}
    \label{eq:dissip-wass}
        \frac{\mathrm{d}}{\mathrm{d}t} F(\mu_t) = -\|\nabla_\mu F(\mu_t, \cdot)\|_{L_{\mu_t}^2}^2 \leq 0,
    \end{equation}
    for all $t > 0.$ Thus, Wasserstein gradient flow is the natural first-order descent dynamics for $F$ on $\mathcal P_2(\mathbb R^d)$ and becomes stationary precisely at first-order critical points.

\subsection{Wasserstein Newton flow and Wasserstein Newton method}
We work under the hypotheses of Lemma~\ref{lem:2ndorder-expansion} and define the second-order Taylor model by
\begin{equation*}
    \mathfrak T_2^F(\mu)[w] := F(\mu) +\langle\nabla_\mu F(\mu,\cdot),w\rangle_{L^2_\mu} +\frac12 \langle\operatorname{H}_\mu w,w\rangle_{L^2_\mu}.
\end{equation*}
For every fixed $v\in L^2_\mu$, Lemma~\ref{lem:2ndorder-expansion} gives
\begin{equation*}
     F((\operatorname{Id}+tv)_\#\mu) = \mathfrak T_2^F(\mu)[tv]+o(t^2), \quad t\to0.
\end{equation*}
This suggests defining the Newton direction by minimizing the quadratic model. Given $n\in\mathbb N_0$, $\mu^0\in\mathcal P_2(\mathbb R^d)$, and $\tau>0$, let
\begin{align*} 
    v_{\mathrm {N}}(\mu^n) = \argmin_{v \in L^2_{\mu^n}}\mathfrak T_2^F(\mu^n)[v],\quad \mu^{n+1} = (\operatorname{Id}+\tau v_{\mathrm {N}}(\mu^n))_\#\mu^n.
\end{align*}
Assume that $\operatorname H_{\mu^n}$ is self-adjoint and uniformly positive on $L^2_{\mu^n}$ with constant $c > 0$, so that $\operatorname H_{\mu^n}^{-1}$ exists. Then the minimizer is characterized by
\[
\operatorname H_{\mu^n} v_{\mathrm {N}}(\mu^n) = - \nabla_\mu F(\mu^n,\cdot),
\]
that is,
\[
v_{\mathrm N}(\mu^n)=-\operatorname H_{\mu^n}^{-1}\nabla_\mu F(\mu^n,\cdot).
\]
The resulting Wasserstein Newton scheme is therefore
\[
\mu^{n+1} = \Bigl(\operatorname{Id}-\tau \operatorname H_{\mu^n}^{-1}\nabla_\mu F(\mu^n,\cdot)\Bigr)_\#\mu^n, \quad \mu^0 \in \mathcal{P}_2(\mathbb R^d).
\]
Under Assumption~\ref{assumption:smooth-F}, let $C_{\mathrm H}:=C_{\mathrm M}+C_{\mathrm K}$. Suppose that $\operatorname{H}_{\mu^n} \succeq c\operatorname{I_{d \times d}},$ $c>0,$ and choose $0<\tau\le\frac{c}{C_{\mathrm H}}.$ Applying the first estimate of Lemma~\ref{lemma:smoothness-of-F-2} gives
\begin{align*}
    F(\mu^{n+1})
    &\le
    F(\mu^n)
    -\tau\langle \nabla_\mu F(\mu^n,\cdot),\operatorname{H}_{\mu^n}^{-1}\nabla_\mu F(\mu^n,\cdot)\rangle_{L^2_{\mu^n}}
    +\frac{C_{\mathrm H}\tau^2}{2}
      \|\operatorname{H}_{\mu^n}^{-1}\nabla_\mu F(\mu^n,\cdot)\|_{L^2_{\mu^n}}^2
    \\
    &=
    F(\mu^n)
    -\tau
      \langle\operatorname{H}_{\mu^n}\operatorname{H}_{\mu^n}^{-1}\nabla_\mu F(\mu^n,\cdot),\operatorname{H}_{\mu^n}^{-1}\nabla_\mu F(\mu^n,\cdot)
      \rangle_{L^2_{\mu^n}}
    +\frac{C_{\mathrm H}\tau^2}{2}
      \|\operatorname{H}_{\mu^n}^{-1}\nabla_\mu F(\mu^n,\cdot)\|_{L^2_{\mu^n}}^2
    \\
    &\le
    F(\mu^n)
    -\tau\left(c-\frac{C_{\mathrm H}\tau}{2}\right)
      \|\operatorname{H}_{\mu^n}^{-1}\nabla_\mu F(\mu^n,\cdot)\|_{L^2_{\mu^n}}^2
    \\
    &\le
    F(\mu^n)
    -\frac{\tau c}{2}
      \|\operatorname{H}_{\mu^n}^{-1}
        \nabla_\mu F(\mu^n,\cdot)\|_{L^2_{\mu^n}}^2.
\end{align*}
Thus, under uniform Hessian boundedness and the stated coercivity and stepsize conditions, the damped Newton transport step decreases the objective strictly whenever $\nabla_\mu F(\mu^n,\cdot)\ne0$.

As in the first-order case, one may view the corresponding continuous-time dynamics as the limit of the explicit transport scheme; see Proposition \ref{prop:convergence-newton-scheme}. This yields the
Wasserstein Newton flow
\[
\partial_t \mu_t = \nabla \cdot \left(\mu_t \operatorname H_{\mu_t}^{-1}\nabla_\mu F(\mu_t,\cdot)\right),\text{ for } t > 0, \quad
                 \mu|_{t=0} := \mu^0.
\]
This coincides with the Information Newton’s flow introduced by \cite{wang2020informationnewtonsflowsecondorder}. By Proposition \ref{prop_chain_rule},
    \begin{equation*}
        \frac{\mathrm{d}}{\mathrm{d}t} F(\mu_t) = -\langle \operatorname \nabla_\mu F(\mu_t, \cdot), \operatorname{H}_{\mu_t}^{-1} \nabla_\mu F(\mu_t, \cdot)\rangle_{L^2_{\mu_t}} \leq -c \|\operatorname{H}_{\mu_t}^{-1}\nabla_\mu F(\mu_t, \cdot)\|_{L_{\mu_t}^2}^2 \leq 0.
    \end{equation*}
The Newton method, however, comes with two important limitations. First, the scheme is only well defined when $\operatorname{H}_\mu^{-1}$ exists at every iteration, which is a strong requirement. Second,
inverse Hessian preconditioning does not distinguish between local minimizers and saddle points. As discussed in the main text, it may destroy repulsion along directions of negative curvature and make saddle points locally attractive.

\section{Calculus along curves and flow approximations}
\label{Calculus along curves and flow approximations}
In this appendix, we collect two types of auxiliary results used in the paper. First, we recall a chain rule formula for differentiating the objective along absolutely continuous curves in Wasserstein space. We then use this calculus to relate the discrete transport schemes introduced in the paper to their corresponding gradient flows.
\subsection{Chain rule along absolutely continuous curves}
The following result provides the basic differential identity for evaluating the variation of $F$ along an absolutely continuous curve satisfying the continuity equation.
\begin{proposition}
    \label{prop_chain_rule}
    Assume $F:\mathcal P_2(\mathbb R^d) \to \mathbb R$ is Wasserstein differentiable and let $\mu:[0,1] \to \mathcal{P}_2(\mathbb R^d)$ be an absolutely continuous curve satisfying $\partial_t \mu_t + \nabla \cdot \left(v_t \mu_t\right) = 0$, $t \in (0,1]$, $\mu|_{t=0} := \mu_0$, with vector field $v_t \in \mathcal{T}_{\mu_t} \mathcal{P}_2(\mathbb R^d)$.
    \begin{align*}
        \frac{\mathrm{d}}{\mathrm{d} t}F(\mu_t) = \langle \nabla_\mu F(\mu_t), v_t\rangle_{L^2_{\mu_t}}.
    \end{align*}
\end{proposition}
\begin{proof}
    By applying Definition \ref{def:wass-differentiability} with $\mu = \mu_{t+h}$, $\nu = (\operatorname{Id}+ h v_t)_{\#} \mu_t$, $h > 0$, and $\gamma \in \Pi_o(\mu_{t+h}, (\operatorname{Id}+ h v_t)_{\#} \mu_t)$, we have
    \begin{align*}
        \left|\frac{F(\mu_{t+h}) - F((\operatorname{Id}+ h v_t)_{\#} \mu_t)}{h}\right| &\leq \left|\frac{1}{h} \int \langle \nabla_\mu F(\mu_t, x), y-x\rangle \mathrm d\gamma(x,y)\right| + \frac{o(h)}{h} \\
        &\leq \|\nabla_\mu F(\mu_t)\|_{L^2_{\mu_t}} \frac{W_2(\mu_{t+h},(\operatorname{Id}+ h v_t)_{\#} \mu_t)}{h} + \frac{o(h)}{h}\to 0, 
    \end{align*}
    as $h\to 0$, where the second inequality follows from the Cauchy-Schwarz inequality and the optimality of $\gamma$, while the last inequality follows from Proposition \ref{prop_infinitesimal_accurve}. For fixed $t$ at which $v_t\in L^2_{\mu_t}$ is defined, Lemma~\ref{lem:1storder-expansion} gives
    \begin{equation*}
        F((\operatorname{Id}+hv_t)_\#\mu_t)-F(\mu_t) = h\langle\nabla_\mu F(\mu_t,\cdot),v_t\rangle_{L^2_{\mu_t}} +o(h), \quad h\downarrow0.
    \end{equation*}
    Thus,
    \begin{equation*}
        \lim_{h\downarrow0}\frac{F((\operatorname{Id}+hv_t)_\#\mu_t)-F(\mu_t)}{h} = \langle\nabla_\mu F(\mu_t,\cdot),v_t\rangle_{L^2_{\mu_t}}.
    \end{equation*}
    Finally,
    \begin{align*}
        \frac{\mathrm{d}}{\mathrm{d} t}F(\mu_t) = \lim_{h \to 0} \frac{F(\mu_{t+h}) - F(\mu_t)}{h} &= \lim_{h\to0} \frac{F((\operatorname{Id}+ h v_t)_{\#} \mu_t) - F(\mu_t)}{h} = \langle \nabla_\mu F(\mu_t), v_t\rangle_{L^2_{\mu_t}}.
    \end{align*}
\end{proof}

\subsection{Convergence of Wasserstein gradient descent to the flow}
We first show that the Wasserstein gradient descent scheme converges to the corresponding Wasserstein gradient flow as the time step tends to zero.
\begin{proposition}[Convergence of Wasserstein gradient descent to the flow]
\label{prop convergence explicit Euler}
Let $T > 0$ and let $\mu:[0,T] \to \mathcal{P}_2(\mathbb R^d)$ be an absolutely continuous curve satisfying $\partial_t \mu_t = \nabla \cdot \left(\nabla_\mu F(\mu_t,\cdot)\mu_t\right)$ in $\mathcal D'\left((0,T)\times\mathbb R^d\right),$ $\mu|_{t=0} := \mu_0$. Assume that $\int_0^T \|\nabla_\mu F(\mu_t,\cdot)\|_{L_{\mu_t}^2}^2 \mathrm{d}t < \infty$. For $\tau > 0$, define $E_\tau(\mu):=(\operatorname{Id}-\tau \nabla_\mu F(\mu,\cdot))_\#\mu$. Assume there exists $L_F> 0$ such that, for all $\gamma \in \Pi(\mu, \nu)$,
    \begin{equation}
    \label{eq:b-Lipschitz}
        \|\nabla_\mu F(\nu,\cdot)  - \nabla_\mu F(\mu, \cdot)\|_{L_\gamma^2} \leq L_F \left(\int_{\mathbb R^d \times \mathbb R^d} \|y-x\|^2\mathrm{d}\gamma(x,y)\right)^{1/2}.
    \end{equation}
Let $(\mu_\tau^n)_{n\ge 0}$ be defined by $\mu_\tau^{n+1}=E_\tau(\mu_\tau^n)$, with $\mu_\tau^0:=\mu_0\in \mathcal P_2(\mathbb R^d)$. Then,
\[
\lim_{\tau \to 0}\max_{n \le \lfloor\frac{T}{\tau}\rfloor} W_2(\mu_\tau^n,\mu_{n\tau}) = 0.
\]
\end{proposition}

\begin{proof}
For notational convenience, set $N_\tau:=\left\lfloor \frac{T}{\tau}\right\rfloor,$ and denote the iterates corresponding to the step size $\tau$ by $(\mu_\tau^n)_{n\geq 0}$. We first establish the stability of the Euler map. Let $\mu,\nu\in\mathcal P_2(\mathbb R^d)$, and let $\pi\in\Pi_{o}(\mu,\nu)$. The measure $\left((\operatorname{Id}-\tau \nabla_\mu F(\mu,\cdot))\times (\operatorname{Id}-\tau \nabla_\mu F(\nu,\cdot))\right)_\#\pi$ is a coupling of $E_\tau(\mu)$ and $E_\tau(\nu)$. Therefore, by Minkowski's inequality and \eqref{eq:b-Lipschitz},
\begin{equation}
\begin{aligned}
    &W_2\left(E_\tau(\mu),E_\tau(\nu)\right)
    \leq \left(\int \left\|(x-y) -\tau\left(\nabla_\mu F(\mu,x)-\nabla_\mu F(\nu,y)\right)\right\|^2\mathrm d\pi(x,y)\right)^{1/2}\\
    &\leq \left(\int \|x-y\|^2\mathrm d\pi(x,y)\right)^{1/2}
    +
    \tau
    \left(
        \int
        \|\nabla_\mu F(\mu,x)-\nabla_\mu F(\nu,y)\|^2\mathrm d\pi(x,y)
    \right)^{1/2}\\
    &\leq
    (1+L_F\tau)W_2(\mu,\nu).
    \label{eq:E-stability-corrected}
\end{aligned}
\end{equation}
We next prove a uniform one-step consistency estimate. Set $M_T :=\sup_{0\leq t\leq T} \|\nabla_\mu F(\mu_t,\cdot)\|_{L^2_{\mu_t}}.$ This quantity is finite. Indeed, if $\gamma_t\in\Pi_{o}(\mu_0,\mu_t)$, then
\begin{align*}
    \|\nabla_\mu F(\mu_t,\cdot)\|_{L^2(\mu_t)}
    &\leq
    \|\nabla_\mu F(\mu_0,\cdot)\|_{L^2(\mu_0)}
    +
    \left(
        \int
        \|\nabla_\mu F(\mu_t,y)-\nabla_\mu F(\mu_0,x)\|^2
        \,\mathrm d\gamma_t(x,y)
    \right)^{1/2}
    \\
    &\leq
    \|\nabla_\mu F(\mu_0,\cdot)\|_{L^2(\mu_0)}
    +
    L_F W_2(\mu_t,\mu_0).
\end{align*}
Since $t\mapsto\mu_t$ is continuous in $W_2$ on the compact interval $[0,T]$, the right-hand side is uniformly bounded. The curve $(\mu_t)_{t\in[0,T]}$ solves the continuity equation $\partial_t\mu_t +\nabla\cdot(v_t\mu_t)=0,$ $v_t=-\nabla_\mu F(\mu_t,\cdot).$ By the superposition principle \cite[Theorem 8.2.1]{ambrosio2008gradient}, there exists a probability measure $\boldsymbol\eta$ in $\Gamma_T$, where $\Gamma_T$ is the space of continuous maps from $[0,T]$ into $\mathbb R^d$, endowed with the sup norm, such that, writing $X_t(\omega):=\omega(t) \in AC^2([0,T];\mathbb R^d)$, we have $(X_t)_\#\boldsymbol\eta=\mu_t$ and
\[
    X_s-X_t = -\int_t^s \nabla_\mu F(\mu_r,X_r)\mathrm dr
\]
for a.e. $s,t \in (0,T).$ Fix $t\in[0,T-\tau]$. Since $\mathrm{Law}_{\boldsymbol\eta}\left(X_t-\tau \nabla_\mu F(\mu_t,X_t)\right)= E_\tau(\mu_t),$ we obtain
\begin{equation}
\begin{aligned}
    W_2\left(E_\tau(\mu_t),\mu_{t+\tau}\right)&\leq
    \left\|X_{t+\tau}-X_t+\tau \nabla_\mu F(\mu_t,X_t)\right\|_{L^2_{\boldsymbol\eta}}\\
    &=\left\|
        \int_t^{t+\tau}
        \left(
           \nabla_\mu F(\mu_t,X_t)-\nabla_\mu F(\mu_s,X_s)
        \right)\mathrm ds
    \right\|_{L^2_{\boldsymbol\eta}}\\
    &\leq
    \int_t^{t+\tau}
    \left\|
         \nabla_\mu F(\mu_t,X_t)-\nabla_\mu F(\mu_s,X_s)
    \right\|_{L^2_{\boldsymbol\eta}}\mathrm ds.
    \label{eq:local-error-first}
\end{aligned}
\end{equation}
For each $s$, the law of $(X_t,X_s)$ is a coupling of $\mu_t$ and $\mu_s$. Hence \eqref{eq:b-Lipschitz} implies
\[
    \left\|\nabla_\mu F(\mu_t,X_t)-\nabla_\mu F(\mu_s,X_s)\right\|_{L^2_{\boldsymbol\eta}} \leq L_F\|X_t-X_s\|_{L^2_{\boldsymbol\eta}}.
\]
Moreover,
\[
    \|X_s-X_t\|_{L^2_{\boldsymbol\eta}}\leq \int_t^s\| \nabla_\mu F(\mu_r,X_r)\|_{L^2_{\boldsymbol\eta}}\mathrm dr \leq M_T(s-t).
\]
Substituting these estimates into \eqref{eq:local-error-first} gives the uniform bound
\begin{equation}
\label{eq:uniform-local-error}
    \sup_{0\leq t\leq T-\tau} W_2\left(E_\tau(\mu_t),\mu_{t+\tau}\right) \leq \frac{L_FM_T}{2}\tau^2.
\end{equation}
Now set $e_n:=W_2(\mu_\tau^n,\mu_{n\tau}).$ For every $0\leq n<N_\tau$, using
\eqref{eq:E-stability-corrected} and \eqref{eq:uniform-local-error}, we obtain
\begin{align*}
    e_{n+1} &\leq W_2\left(E_\tau(\mu_\tau^n),E_\tau(\mu_{n\tau})\right)
    + W_2\left(E_\tau(\mu_{n\tau}),\mu_{(n+1)\tau}\right)\leq (1+L_F\tau)e_n + \frac{L_FM_T}{2}\tau^2.
\end{align*}
Since $e_0=0$, iteration gives
\begin{align*}
    e_n \leq \frac{L_FM_T}{2}\tau^2\sum_{k=0}^{n-1}(1+L_F\tau)^k =
    \frac{M_T}{2}\tau \left((1+L_F\tau)^n-1\right) \leq \frac{M_T}{2}\left(e^{L_FT}-1\right)\tau.
\end{align*}
Hence,
\[
    \max_{0\leq n\leq N_\tau} W_2(\mu_\tau^n,\mu_{n\tau})\leq\frac{M_T}{2}\left(e^{L_FT}-1\right)\tau,
\]
and hence
\[
    \lim_{\tau\downarrow0}
    \max_{0\leq n\leq \lfloor T/\tau\rfloor}
    W_2(\mu_\tau^n,\mu_{n\tau})
    =
    0.
\]
\end{proof}

\subsection{Convergence of the Wasserstein Newton scheme to the flow}
We next consider the second-order transport scheme associated with Wasserstein Newton's method and show that, under suitable regularity and uniform positivity assumptions, it converges to the corresponding Newton flow.

For $\mu\in\mathcal P_2(\mathbb R^d)$, write $\operatorname M_\mu(x):=\nabla\nabla_\mu F(\mu,x),$ $\operatorname K_\mu(x,x'):=\nabla_\mu^2F(\mu,x,x'),$ so that
\[
(\operatorname H_\mu v)(x) = \operatorname M_\mu(x)v(x) + \int_{\mathbb R^d} \operatorname K_\mu(x,x')v(x')\mathrm d\mu(x').
\]
Given $\mu,\nu\in\mathcal P_2(\mathbb R^d)$ and $\gamma\in\Pi(\mu,\nu)$, define the first and second coordinate lifts of the Hessians by
\begin{align*}
\left(\widehat{\operatorname H}_{\mu}^{\gamma,1}w\right)(x,y) &:=
\operatorname M_\mu(x)w(x,y) + \int_{\mathbb R^d\times\mathbb R^d}
\operatorname K_\mu(x,x') w(x',y')\mathrm d\gamma(x',y'),
\\
\label{eq:second-lifted-Hessian}
\left(\widehat{\operatorname H}_{\nu}^{\gamma,2}w\right)(x,y) &:=
\operatorname M_\nu(y)w(x,y) + \int_{\mathbb R^d\times\mathbb R^d}
\operatorname K_\nu(y,y') w(x',y')\mathrm d\gamma(x',y').
\end{align*}

\begin{assumption}[Uniform coercivity of the coupling lifts]
\label{assumption:lifted-Hessian-coercivity}
There exists $\lambda>0$ such that, for every $\mu,\nu\in\mathcal P_2(\mathbb R^d)$ and $\gamma\in\Pi(\mu,\nu)$, the operators $\widehat{\operatorname H}_{\mu}^{\gamma,1}$ and $\widehat{\operatorname H}_{\nu}^{\gamma,2}$ are bounded and self-adjoint on $L_\gamma^2$ and satisfy
\begin{align*}
\left\langle
\widehat{\operatorname H}_{\mu}^{\gamma,1}w,w\right\rangle_{L_\gamma^2} \geq
\lambda\|w\|_{L_\gamma^2}^2, \quad 
\left\langle \widehat{\operatorname H}_{\nu}^{\gamma,2}w,w \right\rangle_{L_\gamma^2} \geq \lambda\|w\|_{L_\gamma^2}^2
\end{align*}
for every $w\in L_\gamma^2$.
\end{assumption}

\begin{proposition}[Convergence of the Wasserstein Newton scheme to the flow]
\label{prop:convergence-newton-scheme}
Let Assumptions~\ref{assumption:Hessian-lipschitz} and~\ref{assumption:lifted-Hessian-coercivity} hold, and write $L_{\mathrm H}:=L_{\mathrm M}+L_{\mathrm K}.$ Assume that there exists $L_F>0$ such that, for every $\mu,\nu\in\mathcal P_2(\mathbb R^d)$ and $\gamma\in\Pi(\mu,\nu)$,
\begin{equation}
\label{eq:coupling-gradient-lipschitz-newton}
    \int_{\mathbb R^d\times\mathbb R^d}
    \|\nabla_\mu F(\nu,y)-\nabla_\mu F(\mu,x)\|^2\mathrm d\gamma(x,y) \le
    L_F^2 \int_{\mathbb R^d\times\mathbb R^d}\|y-x\|^2\mathrm d\gamma(x,y).
\end{equation}
Let $T>0$, and let $\mu:[0,T]\to\mathcal P_2(\mathbb R^d)$ be an absolutely
continuous curve satisfying $\partial_t\mu_t=\nabla\cdot(\mu_t\operatorname{H}_{\mu_t}^{-1}\nabla_\mu F(\mu_t,\cdot))$ in $\mathcal D'((0,T)\times\mathbb R^d),$ $\mu|_{t=0}=\mu^0.$ For $\tau>0$, set $E_\tau(\nu):=(\operatorname{Id}-\tau\operatorname{H}_{\nu}^{-1}\nabla_\mu F(\nu,\cdot))_\#\nu,$ $\mu_\tau^0:=\mu^0,$ $\mu_\tau^{n+1}:=E_\tau(\mu_\tau^n).$ Define
\begin{equation*}
    G_T:=\sup_{0\le t\le T}\|\nabla_\mu F(\mu_t,\cdot)\|_{L^2_{\mu_t}},
    \quad C_{N,T}:=\frac{L_F}{\lambda}+\frac{L_{\mathrm H}G_T}{\lambda^2},
\end{equation*}
where $\lambda>0$ is the coercivity constant in Assumption~\ref{assumption:lifted-Hessian-coercivity}. Then $G_T<\infty$, the discrete iterates are well defined, and
\begin{equation*}
    \lim_{\tau\downarrow0} \max_{0\le n\le\lfloor T/\tau\rfloor} W_2(\mu_\tau^n,\mu_{n\tau})=0.
\end{equation*}
\end{proposition}

\begin{proof}
The proof is divided into the following four steps.

\emph{Step 1: finite-horizon bounds and well-posedness.}
Applying \eqref{eq:coupling-gradient-lipschitz-newton} to an optimal coupling $\gamma \in \Pi_o(\mu,\nu)$ and using the reverse triangle inequality gives
\begin{equation*}
    \left|\|\nabla_\mu F(\nu,\cdot)\|_{L^2_\nu} - \|\nabla_\mu F(\mu,\cdot)\|_{L^2_\mu}\right|\le L_FW_2(\mu,\nu).
\end{equation*}
Hence, 
\begin{equation*}
    G_T \le \|\nabla_\mu F(\mu^0,\cdot)\|_{L^2_{\mu^0}} +L_F\sup_{0\le t\le T}W_2(\mu_t,\mu^0)<\infty,
\end{equation*}
because $t\mapsto\mu_t$ is continuous on $[0,T]$. Applying Assumption~\ref{assumption:lifted-Hessian-coercivity} to the diagonal coupling shows that $\operatorname{H}_\nu\succeq\lambda\operatorname{I_{d \times d}}$ on $L^2_\nu,$ and $\|\operatorname{H}_\nu^{-1}\|_{\mathrm{op}} \le\lambda^{-1},$ for every $\nu\in\mathcal P_2(\mathbb R^d)$. Hence
\begin{equation*}
    \|\operatorname{H}_{\nu}^{-1}\nabla_\mu F(\nu,\cdot)\|_{L^2_\nu}\le\lambda^{-1}\|\nabla_\mu F(\nu,\cdot)\|_{L^2_\nu}<\infty.
\end{equation*}
It follows that $E_\tau(\nu)\in\mathcal P_2(\mathbb R^d)$, so all discrete iterates are well defined. Along the exact flow,
\begin{equation}
\label{eq:Newton-direction-uniform-bound}
    \sup_{0\le t\le T}\|\operatorname{H}_{\mu_t}^{-1}\nabla_\mu F(\mu_t,\cdot)\|_{L^2_{\mu_t}} \le\frac{G_T}{\lambda}.
\end{equation}
\emph{Step 2: comparison with the exact flow.}
Fix $\mu,\nu\in\mathcal P_2(\mathbb R^d)$ and
$\gamma\in\Pi(\mu,\nu)$, and set $d_\gamma^2:=\int\|y-x\|^2\mathrm d\gamma(x,y).$ Write $\widehat{\operatorname{H}}_\mu :=\widehat{\operatorname{H}}_\mu^{\gamma,1},$ $\widehat{\operatorname{H}}_\nu :=\widehat{\operatorname{H}}_\nu^{\gamma,2}$ for the coupling lifts defined above. Assumption~\ref{assumption:Hessian-lipschitz} gives
\begin{equation}
\label{eq:lifted-Hessian-difference}
    \|\widehat{\operatorname{H}}_\nu-\widehat{\operatorname{H}}_\mu\|_{\mathrm{op}}\le L_{\mathrm H}d_\gamma.
\end{equation}
Here the multiplication part is bounded using its $L^\infty_\gamma$ norm, and the integral part using the $L^2_{\gamma\otimes\gamma}$ norm of its kernel.
Moreover, Assumption~\ref{assumption:lifted-Hessian-coercivity} yields
\begin{equation}
\label{eq:lifted-inverse-bounds}
    \|\widehat{\operatorname{H}}_\mu^{-1}\|_{\mathrm{op}} \le\lambda^{-1}, \quad \|\widehat{\operatorname{H}}_\nu^{-1}\|_{\mathrm{op}} \le\lambda^{-1}.
\end{equation}
Define the lifted gradients and Newton directions by $\widehat g_\mu(x,y):=\nabla_\mu F(\mu,x),$ $\widehat g_\nu(x,y):=\nabla_\mu F(\nu,y),$ and $\widehat N_\mu(x,y):=\operatorname H_\mu^{-1}\nabla_\mu F(\mu,x),$ $
\widehat N_\nu(x,y):=\operatorname H_\nu^{-1}\nabla_\mu F(\nu,y),$ respectively. Since the first marginal of $\gamma$ is $\mu$,
\begin{align*}
\left(\widehat{\operatorname H}_\mu\widehat N_\mu\right)(x,y)
&= \operatorname M_\mu(x)\operatorname H_\mu^{-1}\nabla_\mu F(\mu,x)+
\int \operatorname K_\mu(x,x')\operatorname H_\mu^{-1}\nabla_\mu F(\mu,x')\mathrm d\gamma(x',y')
\\
&= \operatorname M_\mu(x)\operatorname H_\mu^{-1}\nabla_\mu F(\mu,x) + \int \operatorname K_\mu(x,x') \operatorname H_\mu^{-1}\nabla_\mu F(\mu,x')\mathrm d\mu(x')
\\
&= \nabla_\mu F(\mu,x) = \widehat g_\mu(x,y).
\end{align*}
Hence, $\widehat N_\mu = \widehat{\operatorname H}_\mu^{-1}\widehat g_\mu.$ Similarly, $\widehat N_\nu = \widehat{\operatorname H}_\nu^{-1}\widehat g_\nu.$ Therefore,
\begin{equation*}
    \widehat N_\nu-\widehat N_\mu = \widehat{\operatorname{H}}_\nu^{-1}
    \left[\widehat g_\nu-\widehat g_\mu+\left(\widehat{\operatorname{H}}_\mu
        -\widehat{\operatorname{H}}_\nu\right)\widehat N_\mu\right].
\end{equation*}
Using \eqref{eq:coupling-gradient-lipschitz-newton}, \eqref{eq:lifted-Hessian-difference} and \eqref{eq:lifted-inverse-bounds} gives
\begin{align*}
    \|\widehat N_\nu-\widehat N_\mu\|_{L^2_\gamma} \le \frac{L_F}{\lambda}d_\gamma + \frac{L_{\mathrm H}}{\lambda} d_\gamma\|\operatorname H_\mu^{-1}\nabla_\mu F(\mu,\cdot)\|_{L^2_\mu} \le \left(\frac{L_F}{\lambda} + \frac{L_{\mathrm H}}{\lambda^2} \|\nabla_\mu F(\mu,\cdot)\|_{L^2_\mu}\right)d_\gamma.
\end{align*}
In particular, when $\mu=\mu_t$ belongs to the exact flow,
\begin{equation*}
    \|N_\nu-N_{\mu_t}\|_{L^2_\gamma} \le C_{N,T}d_\gamma,
\end{equation*}
uniformly over $t\in[0,T]$, $\nu\in\mathcal P_2(\mathbb R^d)$, and $\gamma\in\Pi(\mu_t,\nu)$. Taking an optimal coupling and pushing it forward by the two Euler maps, Minkowski's inequality consequently gives
\begin{equation*}
     W_2(E_\tau(\nu),E_\tau(\mu_t)) \le (1+C_{N,T}\tau)W_2(\nu,\mu_t).
\end{equation*}
\emph{Step 3: uniform consistency.} By \eqref{eq:Newton-direction-uniform-bound}, the Newton velocity $-\operatorname H_{\mu_t}^{-1}\nabla_\mu F(\mu_t,\cdot)$ satisfies
\begin{equation*}
    \int_0^T \|\operatorname H_{\mu_t}^{-1}\nabla_\mu F(\mu_t,\cdot)\|_{L^2_{\mu_t}}^2\mathrm dt \le T\left(\frac{G_T}{\lambda}\right)^2.
\end{equation*}
The superposition principle \cite[Theorem 8.2.1]{ambrosio2008gradient} gives a probability measure $\eta$ in $\Gamma_T$, where $\Gamma_T$ is the space of continuous maps from $[0,T]$ into $\mathbb R^d$, endowed with the sup norm, such that, writing $X_t(\omega):=\omega(t) \in AC^2([0,T];\mathbb R^d)$,
\begin{equation}
\label{eq:superposition-representation-newton}
(X_t)_\#\eta=\mu_t, \quad \dot X_t=-\operatorname H_{\mu_t}^{-1}\nabla_\mu F(\mu_t,X_t),
\end{equation}
for $\eta$-a.e. $X$ and a.e. $t \in (0,T)$. Thus for a.e. $0< t\le s<T$
\begin{equation*}
    \|X_s-X_t\|_{L^2_\eta} \le\frac{G_T}{\lambda}(s-t).
\end{equation*}
Applying Step~2 to the coupling $(X_t,X_s)_\#\eta$ yields
\begin{equation*}
    \|\operatorname H_{\mu_s}^{-1}\nabla_\mu F(\mu_s,X_s)-\operatorname H_{\mu_t}^{-1}\nabla_\mu F(\mu_t,X_t)\|_{L^2_\eta} \le C_{N,T}\frac{G_T}{\lambda}(s-t).
\end{equation*}
For $0\le t\le T-\tau$, the laws of $X_t-\tau \operatorname H_{\mu_t}^{-1}\nabla_\mu F(\mu_t,X_t)$ and $X_{t+\tau}$ are $E_\tau(\mu_t)$ and $\mu_{t+\tau}$, respectively. Moreover, by \eqref{eq:superposition-representation-newton},
\[
X_{t+\tau} = X_t - \int_t^{t+\tau}\operatorname H_{\mu_s}^{-1}\nabla_\mu F(\mu_s,X_s)\mathrm ds.
\]
Hence
\begin{align*}
    W_2(E_\tau(\mu_t),\mu_{t+\tau}) \le
    \left\|\int_t^{t+\tau}\left(\operatorname H_{\mu_s}^{-1}\nabla_\mu F(\mu_s,X_s)-\operatorname H_{\mu_t}^{-1}\nabla_\mu F(\mu_t,X_t)\right)\mathrm ds\right\|_{L^2_\eta} \le \frac{C_{N,T}G_T}{2\lambda}\tau^2.
\end{align*}
\emph{Step 4: global error estimate.}
Set $e_n:=W_2(\mu_\tau^n,\mu_{n\tau}).$ For $0\le n<\lfloor T/\tau\rfloor$, Steps~2 and~3 give
\begin{align*}
    e_{n+1}&\le W_2(E_\tau(\mu_\tau^n),E_\tau(\mu_{n\tau})) +
    W_2(E_\tau(\mu_{n\tau}),\mu_{(n+1)\tau})\\
    &\le (1+C_{N,T}\tau)e_n + \frac{C_{N,T}G_T}{2\lambda}\tau^2.
\end{align*}
Since $e_0=0$, summing the geometric series yields
\begin{align*}
    e_n \le \frac{G_T}{2\lambda}\left((1+C_{N,T}\tau)^n-1\right)\tau \le \frac{G_T}{2\lambda}\left(e^{C_{N,T}T}-1\right)\tau.
\end{align*}
This proves the stated estimate and convergence.
\end{proof}

\begin{remark}[A sufficient condition for coupling lift coercivity]
\label{remark:sufficient-lifted-Hessian-coercivity}
Assumption \ref{assumption:lifted-Hessian-coercivity} holds if there exist $\lambda_{\operatorname H},\lambda_{\operatorname M}>0$ such that, uniformly over $\mu\in\mathcal P_2(\mathbb R^d)$, $\operatorname H_\mu \succeq \lambda_{\operatorname H}\operatorname{I_{d \times d}}$ on $L_\mu^2,$ and $\operatorname M_\mu(x) \succeq \lambda_{\operatorname M}\operatorname{I_{d \times d}}$ for $\mu$-a.e. $x.$ In that case, Assumption
\ref{assumption:lifted-Hessian-coercivity} holds with $\lambda =
\min\{\lambda_{\operatorname H}, \lambda_{\operatorname M}.$ Indeed, fix $\gamma\in\Pi(\mu,\nu)$ and disintegrate $\gamma(\mathrm dx,\mathrm dy) =
\mu(\mathrm dx)\gamma_x(\mathrm dy).$ For $w\in L_\gamma^2$, define
\[
\overline w(x) := \int_{\mathbb R^d} w(x,y)\mathrm d\gamma_x(y), \quad
r(x,y) := w(x,y)-\overline w(x).
\]
Then $\int r(x,y)\mathrm d\gamma_x(y)=0$ for $\mu$-a.e. $x,$ and $\|w\|_{L_\gamma^2}^2 = \|\overline w\|_{L_\mu^2}^2 + \|r\|_{L_\gamma^2}^2.$ A direct calculation gives
\begin{align*}
\left\langle \widehat{\operatorname H}_{\mu}^{\gamma,1}w,w\right\rangle_{L_\gamma^2} &=
\left\langle \operatorname H_\mu\overline w,\overline w\right\rangle_{L_\mu^2}
+ \int_{\mathbb R^d\times\mathbb R^d}\left\langle\operatorname M_\mu(x)r(x,y),
r(x,y)\right\rangle\mathrm d\gamma(x,y)\\
&\geq
\lambda_{\operatorname H} \|\overline w\|_{L_\mu^2}^2 + \lambda_{\operatorname M}\|r\|_{L_\gamma^2}^2\\
&\geq \min\{\lambda_{\operatorname H},\lambda_{\operatorname M}\}\|w\|_{L_\gamma^2}^2.
\end{align*}
The same argument, using a disintegration with respect to the second
marginal, applies to $\widehat{\operatorname H}_{\nu}^{\gamma,2}$.
\end{remark}

\section{The Wasserstein Proximal Taylor (WPT-$k$) method}
\label{section:wasserstein-proximal-taylor}
In this section, we show that several iterative schemes on Wasserstein space can be viewed within a common a so-called proximal Taylor framework. This includes Wasserstein gradient descent, proximal descent, mirror descent,
preconditioned gradient descent, Newton's method, its Levenberg--Marquardt regularization, and the Wasserstein Saddle-Free Newton method. The common principle is to combine a local Taylor model of the
objective with a strongly convex penalty that controls the transport step.

Assume that for every $\mu\in\mathcal P_2(\mathbb R^d)$, the Wasserstein gradient $\nabla_\mu F(\mu)$ and Hessian $\operatorname{H}_\mu$ of $F$ exist. Let $\mathfrak T_k^F(\mu)[v]$ denote the $k$th-order Taylor expansion of $F$ at $\mu$ in the direction $v\in L_\mu^2$, for $k\in\{1,2\}$. Given $n\in\mathbb N$, $\mu^0\in\mathcal P_2(\mathbb R^d)$, and $\tau>0$, the Wasserstein
Proximal Taylor method is defined by
\begin{align}
\label{eq:general_newton}
v^{n+1} &= \argmin_{v\in L_{\mu^n}^2} \left\{
\mathfrak T_k^F(\mu^n)[v]
+
\operatorname{d}_{\mu^n}(\operatorname{Id}+v,\operatorname{Id})
\right\},
\quad
\mu^{n+1}
=
(\operatorname{Id}+\tau v^{n+1})_\#\mu^n,
\end{align}
where $\operatorname{d}_{\mu^n}:L_{\mu^n}^2\times L_{\mu^n}^2\to[0,\infty)$ is a strongly convex penalty. In the cases considered below, we assume the minimizer is well-defined and unique.

\subsection{Iterates of the WPT-$1$ method}
\paragraph{Gradient and proximal descent.} For $k=1$, the first-order Taylor model is
\begin{equation*}
    \mathfrak T_1^F(\mu^n)[v] = F(\mu^n) +\langle\nabla_\mu F(\mu^n,\cdot),v\rangle_{L^2_{\mu^n}}.
\end{equation*}
Lemma~\ref{lem:1storder-expansion} gives the approximation
\begin{equation*}
F((\operatorname{Id}+v)_\#\mu^n)= \mathfrak T_1^F(\mu^n)[v] + o(\|v\|_{L^2_{\mu^n}}), \quad \|v\|_{L^2_{\mu^n}}\to0.
\end{equation*}
If we choose $\operatorname{d}_{\mu^n}(\operatorname{Id}+v,\operatorname{Id})=\frac12\|v\|_{L_{\mu^n}^2}^2$, then \eqref{eq:general_newton} becomes
\begin{equation}
\label{eq:gd}
\mu^{n+1} = \left(\operatorname{Id}-\tau \nabla_\mu F(\mu^n,\cdot)\right)_\#\mu^n,
\end{equation}
which is Wasserstein gradient descent
\cite{korbaproximal,lascu2024linearconvergenceproximaldescent}.

We may also consider entropy-regularized objective functionals of the form $F = F_0 + \operatorname{KL}(\cdot|\pi)$, where $\pi \in \mathcal{P}_2(\mathbb R^d)$ is a fixed reference measure and $F_0$ is a smooth nonlinear functional on $\mathcal{P}_2(\mathbb R^d)$. To handle the non-smoothness of the relative entropy in the Wasserstein geometry, $F$ is typically minimized using a forward-backward splitting scheme. This method, widely known as Wasserstein proximal descent \cite{zhu2025convergenceanalysiswassersteinproximal,lascu2024linearconvergenceproximaldescent,JMLR:rentian-proximal,korbaproximal}, alternates between a forward gradient step on the smooth component $F_0$ and a backward proximal (JKO \cite{jko}) step on the relative entropy:
\begin{equation*}
\begin{cases} 
\mu^{n+\frac{1}{2}} = \left(\operatorname{Id}-\tau \nabla_\mu F_0(\mu^n,\cdot)\right)_\#\mu^n,\\        
\mu^{n+1} = \argmin_{\mu \in \mathcal{P}_2(\mathbb R^d)}\left\{\operatorname{KL}(\mu\vert{}\pi)+\frac{1}{2\tau}W_2^2(\mu, \mu^{n+\frac{1}{2}})\right\}.    
\end{cases}    
\end{equation*}
\\

\paragraph{Mirror descent.} More generally, let $\phi_{\mu^n}:L_{\mu^n}^2\to\mathbb R$ be convex and continuously G\^ateaux differentiable, with derivative $\phi_{\mu^n}'$, and let $\operatorname{d}^{\phi_{\mu^n}}(\operatorname{Id}+v,\operatorname{Id}) = \phi_{\mu^n}(\operatorname{Id}+v)-\phi_{\mu^n}(\operatorname{Id})-\langle \phi_{\mu^n}'[\operatorname{Id}],v\rangle_{L_{\mu^n}^2}$ be the associated Bregman divergence. Then \eqref{eq:general_newton} becomes
\begin{equation}
\label{wasserstein mirror}
\mu^{n+1} = \left(
(\phi_{\mu^n}^*)'\big[\phi_{\mu^n}'(\operatorname{Id})-\tau \nabla_\mu F(\mu^n)\big]
\right)_\#\mu^n,
\end{equation}
which is Wasserstein mirror descent \cite{bonet2024mirror}. Here
$(\phi_{\mu^n}^*)'=(\phi_{\mu^n}')^{-1}$ provided that $\phi_{\mu^n}$ is strictly convex, lower semi-continuous, and superlinear. In the quadratic case $\phi_{\mu^n}(v)=\frac12\|v\|_{L_{\mu^n}^2}^2$, \eqref{wasserstein mirror} reduces to \eqref{eq:gd}.
\\

\paragraph{Preconditioned gradient descent.} A further special case yields Wasserstein preconditioned gradient descent. Let $h:\mathbb R^d\to\mathbb R$ be strictly convex, define $\phi_{\mu^n}^h(v)=\int h\circ v \,\mathrm d\mu^n$, $\operatorname{d}_{\mu^n}(\operatorname{Id}+v,\operatorname{Id}) = \tau \phi_{\mu^n}^h(-v\tau^{-1}) = \int \tau h(-v(x)\tau^{-1})\,\mathrm d\mu^n(x)$. Then \eqref{eq:general_newton} yields
\begin{equation}
\label{eq:pwgf}
\mu^{n+1} = \left(\operatorname{Id}-\tau \nabla h^*\circ \nabla_\mu F(\mu^n)\right)_\#\mu^n,
\end{equation}
which is Wasserstein preconditioned gradient descent \cite{bonet2024mirror}, where $\nabla h^*=(\nabla h)^{-1}$. When $h(z)=\frac12\|z\|^2$, \eqref{eq:pwgf} again coincides with \eqref{eq:gd}.

\subsection{Iterates of the WPT-$2$ method}
\paragraph{Newton method.} For $k=2$, under the hypotheses of Lemma~\ref{lem:2ndorder-expansion}, we get the second-order Taylor model
\begin{equation*}
    \mathfrak T_2^F(\mu^n)[v] := F(\mu^n) +\langle\nabla_\mu F(\mu^n,\cdot),v\rangle_{L^2_\mu} +\frac12 \langle\operatorname{H}_{\mu^n} v,v\rangle_{L^2_\mu}.
\end{equation*}
If $\operatorname{d}_{\mu^n}(\operatorname{Id}+v,\operatorname{Id})=0$, then \eqref{eq:general_newton} becomes
\[
\mu^{n+1} = \Bigl(\operatorname{Id}-\tau \operatorname{H}_{\mu^n}^{-1}\nabla_\mu F(\mu^n,\cdot)\Bigr)_\#\mu^n,
\]
which is the Wasserstein Newton method.
\\

\paragraph{Levenberg--Marquardt regularization.} If instead $\operatorname{d}_{\mu^n}(\operatorname{Id}+v,\operatorname{Id})=\frac{1}{2\tau}\|v\|_{L_{\mu^n}^2}^2$ and $\mu^{n+1}=(\operatorname{Id}+v^{n+1})_\#\mu^n$, then Proposition \ref{prop:well-posedness} shows that \eqref{eq:general_newton} becomes
\begin{equation}
\label{eq:newton}
\mu^{n+1} = \left(
\operatorname{Id}-\left(\operatorname{H}_{\mu^n}+\tau^{-1}\operatorname{I_{d\times d}}\right)^{-1}\nabla_\mu F(\mu^n)
\right)_\#\mu^n,
\end{equation}
which is the Wasserstein analogue of the Levenberg--Marquardt regularization of Newton's method.

\subsection{Wasserstein Saddle-Free Newton as a proximal Taylor scheme}
The saddle-free Newton construction also fits into this framework. Since the quadratic Hessian term measures the curvature correction added to the first-order Taylor approximation in the direction $v$, a natural choice is
\[
\operatorname{d}_{\mu^n}(\operatorname{Id}+v,\operatorname{Id})
=
\bigl|\mathfrak T_1^F(\mu^n)[v]-\mathfrak T_2^F(\mu^n)[v]\bigr|
=
\frac12\bigl|\langle \operatorname{H}_{\mu^n}v,v\rangle_{L_{\mu^n}^2}\bigr|.
\]
Using Lemma \ref{lem:upper-bound-inner-hessian}, this term admits the upper bounds
\[
\frac12\bigl|\langle \operatorname{H}_{\mu^n}v,v\rangle_{L_{\mu^n}^2}\bigr|
\leq
\frac12\langle |\operatorname{H}_{\mu^n}|v,v\rangle_{L_{\mu^n}^2}
\leq
\frac12
\left\langle
(\operatorname{H}_{\mu^n}^2+\beta \operatorname{I_{d\times d}})^{1/2}v,v
\right\rangle_{L_{\mu^n}^2},
\]
where the first inequality follows from the proof of Lemma \ref{lem:upper-bound-inner-hessian}
(see also \cite[Problem 2.35]{kato1995perturbation}) and the second follows directly from Lemma \ref{lem:upper-bound-inner-hessian}.

Replacing the trust-region penalty $\operatorname{d}_{\mu^n}$ by these upper bounds yields two saddle-free variants. The first is the ideal saddle-free Newton scheme
\[
v^{n+1}
=
\argmin_{v\in L_{\mu^n}^2}
\left\{
F(\mu^n)
+
\langle \nabla_\mu F(\mu^n),v\rangle_{L_{\mu^n}^2}
+
\frac12\langle |\operatorname{H}_{\mu^n}|v,v\rangle_{L_{\mu^n}^2}
\right\},
\quad
\mu^{n+1}=(\operatorname{Id}+\tau v^{n+1})_\#\mu^n,
\]
which gives \eqref{eq:saddle_free_newton}.

The second is the regularized Wasserstein Saddle-Free Newton scheme
\begin{align*}
v^{n+1}
&=
\argmin_{v\in L_{\mu^n}^2}
\left\{
F(\mu^n)
+
\langle \nabla_\mu F(\mu^n),v\rangle_{L_{\mu^n}^2}
+
\frac12
\left\langle
(\operatorname{H}_{\mu^n}^2+\beta \operatorname{I_{d\times d}})^{1/2}v,v
\right\rangle_{L_{\mu^n}^2}
\right\},\\
\mu^{n+1}
&=
(\operatorname{Id}+\tau v^{n+1})_\#\mu^n,
\end{align*}
which yields \eqref{eq:wsfn-update-main}.

\section{Auxiliary results}
\label{appendix:aux-results}
This appendix collects the technical lemmas used throughout the paper. We first record basic calculus identities for the functional $F$ along transport curves, then derive the second-order identities underlying the Wasserstein Hessian, and finally gather the operator estimates needed for the analysis of the regularized WSFN preconditioner.

\subsection{Calculus along transport curves}
We begin with first-order identities describing how $F$ evolves along curves of the form $\mu_t = (\operatorname{Id}+tv)_\#\mu$. These results will be used repeatedly in the derivation of both the gradient and Hessian expansions.
\begin{lemma}[I. Derivative of $F$ along $W_2$--interpolating curves]
\label{lemma_for_second_order}
    Let $F:\mathcal{P}_2(\mathbb R^d) \to \mathbb R$ be Wasserstein differentiable. Let $\mu, \gamma \in\mathcal{P}_2(\mathbb R^d)$, $v \in L_\gamma^2$ and for $t \in [0,1]$ define the curve $\mu_t = (\pi_t)_{\#}\gamma$, with $\pi_t := \operatorname{Id}+tv$, such that $\mu_0 = \mu$ and $\mu_1 = (\pi_1)_\#\gamma$. Then, for any $t\in[0,1]$, 
    \begin{align*}
        \frac{\mathrm{d}}{\mathrm{d} t} F(\mu_t) = \int_{\mathbb R^d\times\mathbb R^d}\left\langle \nabla_\mu F(\mu_t, \pi_t(x,y)) , v(x,y)\right\rangle \mathrm{d}\gamma(x,y).
    \end{align*}
\end{lemma}
Note that for $v(x,y) = y-x$ we recover \cite[Lemma C.6]{yamamoto2025hessianguided}.
\begin{proof}
Fix $t\in[0,1)$ and let $h > 0$ be such that $t+h\in[0,1]$. Consider the coupling
\[
\gamma_{t,t+h} := (\pi_t,\pi_{t+h})_\#\gamma \in \Pi(\mu_t,\mu_{t+h}).
\]
Applying Proposition \ref{prop:strong_diff_w} with $\mu=\mu_t$, $\nu=\mu_{t+h}$, and $\gamma=\gamma_{t,t+h}$ yields
\[
F(\mu_{t+h})
= F(\mu_t)
+ \int_{\mathbb R^d\times\mathbb R^d}
\left\langle \nabla_\mu F(\mu_t, x), y-x \right\rangle\mathrm d\gamma_{t,t+h}(x,y)
+ o\left(\sqrt{\int_{\mathbb R^d\times\mathbb R^d} \|y-x\|^2\ \mathrm{d}\gamma_{t,t+h}(x,y)}\right).
\]
Since $\gamma_{t,t+h}=(\pi_t,\pi_{t+h})_\#\gamma$, 
\[
\int_{\mathbb R^d\times\mathbb R^d}
\left\langle \nabla_\mu F(\mu_t, x), y-x \right\rangle\mathrm d\gamma_{t,t+h}(x,y) = \int_{\mathbb R^d \times\mathbb R^d}
\left\langle \nabla_\mu F(\mu_t, \pi_t(x,y)),
\pi_{t+h}(x,y)-\pi_t(x,y) \right\rangle\mathrm d\gamma(x,y),
\]
\[
\int_{\mathbb R^d\times\mathbb R^d} \|y-x\|^2\ \mathrm{d}\gamma_{t,t+h}(x,y) = \int_{\mathbb R^d \times\mathbb R^d} \|\pi_{t+h}(x,y)-\pi_t(x,y)\|^2\ \mathrm{d}\gamma(x,y).
\]
Observing that
\[
\pi_{t+h}(x,y)-\pi_t(x,y)=hv(x,y),
\]
we obtain
\[
F(\mu_{t+h}) - F(\mu_t)
= h \int_{\mathbb R^d \times\mathbb R^d}
\left\langle \nabla_\mu F(\mu_t, \pi_t(x,y)), v(x,y) \right\rangle
\,\mathrm d\gamma(x,y)
+ o\left(h\sqrt{\int_{\mathbb R^d \times\mathbb R^d} \|v(x,y)\|^2\ \mathrm{d}\gamma(x,y)}\right).
\]
Dividing by $h$, letting $h\to 0$ and noting that $\frac{o(h)}{h} \to 0$, as $h \to 0$, we conclude that
\[
\frac{\mathrm d}{\mathrm d t}F(\mu_t) =\int_{\mathbb R^d \times\mathbb R^d}
\left\langle \nabla_\mu F(\mu_t, \pi_t(x,y)), v(x,y) \right\rangle\mathrm d\gamma(x,y).
\]
\end{proof}
\begin{lemma}[II. Derivative of $F$ along $W_2$--interpolating curves]
\label{lemma_for_second_order_mu}
    Let $F:\mathcal{P}_2(\mathbb R^d) \to \mathbb R$ be Wasserstein differentiable. Let $\mu\in\mathcal{P}_2(\mathbb R^d)$, $v \in L_\mu^2$ and for $t \in [0,1]$ define the curve $\mu_t = (\pi_t)_{\#}\mu$, with $\pi_t := \operatorname{Id}+tv$. Then, for any $t\in[0,1]$, 
    \begin{align*}
        \frac{\mathrm{d}}{\mathrm{d} t} F(\mu_t) = \int_{\mathbb R^d}\left\langle \nabla_\mu F(\mu_t, \pi_t(x)) , v(x)\right\rangle \mathrm{d}\mu(x).
    \end{align*}
\end{lemma}
\begin{proof}
    Fix $t\in[0,1)$ and let $h > 0$ be such that $t+h\in[0,1]$. Consider the coupling
\[
\gamma_{t,t+h} := (\pi_t,\pi_{t+h})_\#\mu \in \Pi(\mu_t,\mu_{t+h}).
\]
Applying Proposition \ref{prop:strong_diff_w} with $\mu=\mu_t$, $\nu=\mu_{t+h}$, $\gamma=\gamma_{t,t+h}$ and following the proof of Lemma \ref{lemma_for_second_order} yields the conclusion.
\end{proof}
\begin{lemma}[First-order expansion along $L^2_\mu$ perturbations]
\label{lem:1storder-expansion}
Let $F:\mathcal{P}_2(\mathbb R^d) \to \mathbb R$ be Wasserstein differentiable. Let $\mu\in\mathcal{P}_2(\mathbb R^d)$ and $\mu_t:=(\pi_t)_\#\mu$, with
$\pi_t:=\operatorname{Id}+tv$, for any $v\in L^2_\mu$. Then, we have the first-order expansion
\begin{equation*}
F(\mu_t)-F(\mu) = t\left\langle \nabla_\mu F(\mu,\cdot), v\right\rangle_{L_\mu^2} + o(t\|v\|_{L_\mu^2}), \quad t \to 0.
\end{equation*}
\end{lemma}
\begin{proof}
    Fix $t\in[0,1)$. Consider the coupling
\[
\gamma_t := (\operatorname{Id},\pi_t)_\#\mu \in \Pi(\mu,\mu_t).
\]
Applying Proposition \ref{prop:strong_diff_w} with $\mu=\mu$, $\nu=\mu_t$, and $\gamma=\gamma_t$ yields
\[
F(\mu_t) = F(\mu)
+ \int_{\mathbb R^d\times\mathbb R^d}
\left\langle \nabla_\mu F(\mu, x), y-x \right\rangle\mathrm d\gamma_t(x,y)
+ o\left(\sqrt{\int_{\mathbb R^d\times\mathbb R^d} \|y-x\|^2\ \mathrm{d}\gamma_t(x,y)}\right).
\]
Since $\gamma_t =(\operatorname{Id},\pi_t)_\#\mu$, 
\begin{align*}
\int_{\mathbb R^d\times\mathbb R^d}
\left\langle \nabla_\mu F(\mu, x), y-x \right\rangle\mathrm d\gamma_t(x,y) &= \int_{\mathbb R^d}
\left\langle \nabla_\mu F(\mu, x),
\pi_t(x)-x \right\rangle\mathrm d\mu(x)\\ 
&= t\int_{\mathbb R^d}
\left\langle \nabla_\mu F(\mu, x), v(x) \right\rangle\mathrm d\mu(x),
\end{align*}
\[
\int_{\mathbb R^d\times\mathbb R^d} \|y-x\|^2\ \mathrm{d}\gamma_t(x,y) = \int_{\mathbb R^d} \|\pi_t(x)-x\|^2\ \mathrm{d}\mu(x) = t^2\|v\|_{L_\mu^2}^2.
\]
We obtain
\[
F(\mu_t) - F(\mu)
= t\int_{\mathbb R^d}
\left\langle \nabla_\mu F(\mu, x), v(x) \right\rangle\mathrm d\mu(x)
+ o(t\|v\|_{L_\mu^2}).
\]
\end{proof}
\begin{lemma}[First-order smoothness of $F$]
    \label{lemma:smoothness-of-F-1}
     Let $F:\mathcal{P}_2(\mathbb R^d) \to \mathbb R$ be Wasserstein differentiable and assume $\nabla_\mu F$ is $L_F$-Lipschitz, i.e., there exists $L_F> 0$ such that, for all $\gamma \in \Pi(\mu, \nu)$,
    \begin{equation*}
        \|\nabla_\mu F(\nu,\cdot)  - \nabla_\mu F(\mu, \cdot)\|_{L_\gamma^2} \leq L_F \left(\int_{\mathbb R^d \times \mathbb R^d} \|y-x\|^2\mathrm{d}\gamma(x,y)\right)^{1/2}.
    \end{equation*}
    Let $\mu\in\mathcal{P}_2(\mathbb R^d)$, $v\in L^2_\mu$ and define $\mu_t:=(\operatorname{Id}+t v)_\#\mu$, for $t \in [0,1]$. Then
    \begin{equation*}
        F\left((\operatorname{Id}+v)_\#\mu\right) \leq  F(\mu) + \left\langle\nabla_\mu F(\mu,\cdot),v\right\rangle_{L_\mu^2} + \frac{L_F}{2}\|v\|_{L_\mu^2}^2.
    \end{equation*}
\end{lemma}
\begin{proof}
    By Lemma \ref{lemma_for_second_order_mu},
    \begin{equation*}
        F((\operatorname{Id}+ v)_\#\mu) -F(\mu) = \int_0^1  \int_{\mathbb R^d}\left\langle \nabla_\mu F(\mu_t, \pi_t(x)) , v(x)\right\rangle \mathrm{d}\mu(x)\mathrm{d}t.
    \end{equation*}
   Take $\gamma = (\operatorname{Id}, \operatorname{Id}+ tv)_{\#}\mu \in \Pi(\mu, \mu_t)$. Subtracting $\left\langle\nabla_\mu F(\mu,\cdot),v\right\rangle_{L_\mu^2}$ from both sides gives
    \begin{align*}
        F((\operatorname{Id}+ v)_\#\mu) -F(\mu) - \left\langle\nabla_\mu F(\mu,\cdot),v\right\rangle_{L_\mu^2} &= \int_0^1  \left\langle \nabla_\mu F(\mu_t) \circ \pi_t - \nabla_\mu F(\mu,\cdot), v\right\rangle_{L_\mu^2}\mathrm{d}t\\
        &\leq \int_0^1  \left\| \nabla_\mu F(\mu_t) \circ \pi_t - \nabla_\mu F(\mu,\cdot)\right\|_{L_\mu^2} \|v\|_{L_\mu^2}\mathrm{d}t\\
        &\leq L_F\int_0^1 t\|v\|_{L_\mu^2}^2\mathrm{d}t = \frac{L_F}{2}\|v\|_{L_\mu^2}^2.
    \end{align*}
\end{proof}

\subsection{Hessian identities and second-order expansions}
\label{subsection: Hessian identities and second-order expansions}
We now turn to the second-order structure of $F$ along transport curves. The results in this subsection justify the form of the Wasserstein Hessian used in the main text and yield the corresponding second-order expansions of both the objective and its gradient.
\begin{proof}[Proof of Proposition \ref{proposition:Wasserstein-hessian-appendix}]
We follow the proofs of \cite[Lemma 11, Proposition 12]{bonet2024mirror}. By Lemma \ref{lemma_for_second_order_mu},
\begin{align*}
        \frac{\mathrm{d}^2}{\mathrm{d}t^2}F(\mu_t) &= \frac{\mathrm{d}}{\mathrm{d}t}\left\langle \nabla_\mu F(\mu_t, \pi_t), v\right\rangle_{L^2_\mu} = \left\langle \frac{\mathrm{d}}{\mathrm{d}t}\nabla_\mu F(\mu_t, \pi_t), v\right\rangle_{L^2_\mu}\\
        &= \left\langle\frac{\mathrm d}{\mathrm dt}\Big|_{\mu=\mu_t} \nabla_\mu F(\mu,\pi_t), v\right\rangle_{L^2_\mu} + \left\langle\frac{\mathrm d}{\mathrm dt}\Big|_{u=\pi_t} \nabla_\mu F(\mu_t,u), v\right\rangle_{L^2_\mu}\\
        &= \int_{\mathbb R^d}\left\langle \int_{\mathbb R^d}\nabla_\mu^2 F(\mu_t, \pi_t(x), \pi_t(\bar x))v(\bar x)\mathrm d\mu(\bar x), v(x)\right\rangle \mathrm{d}\mu(x)\\ 
        &+ \int_{\mathbb R^d}\left\langle \nabla\nabla_\mu F(\mu_t, \pi_t(x))\partial_t\pi_t( x), v(x)\right\rangle \mathrm{d}\mu(x)\\
        &=\int_{\mathbb R^d}\int_{\mathbb R^d}\left\langle \nabla_\mu^2 F(\mu_t, \pi_t(x), \pi_t(\bar x))v(\bar x), v(x)\right\rangle \mathrm d\mu(\bar x)\mathrm{d}\mu(x)\\ 
        &+ \int_{\mathbb R^d}\left\langle \nabla\nabla_\mu F(\mu_t, \pi_t(x))v( x), v(x)\right\rangle \mathrm{d}\mu(x),
\end{align*}
where the last equality follows since $\partial_t \pi_t = v$. Hence,
\begin{equation*}
        \frac{\mathrm{d}^2}{\mathrm{d}t^2}F(\mu_t)= \int_{\mathbb R^d}\left\langle (\widetilde{\operatorname{M}}_{\mu, t} + \widetilde{\operatorname{K}}_{\mu, t})[v](x),v(x)\right\rangle\mathrm{d}\mu(x)= \left\langle \widetilde{\operatorname{H}}_{\mu, t}v, v\right\rangle_{L^2_\mu}.
\end{equation*}
Now we prove that $\widetilde{\operatorname{H}}_{\mu, t} \in \mathcal{LS}(L_\mu^2)$, i.e., for all $v,w \in L_\mu^2$ and $\alpha_1, \alpha_2 \in \mathbb R$,
    \begin{equation*}
        \widetilde{\operatorname{H}}_{\mu, t}[\alpha_1 v+\alpha_2 w] = \alpha_1 \widetilde{\operatorname{H}}_{\mu, t}[v] + \alpha_2 \widetilde{\operatorname{H}}_{\mu, t}[w],
    \end{equation*}
    \begin{equation*}
        \left\langle \widetilde{\operatorname{H}}_{\mu, t}v,w\right\rangle_{L^2_\mu} = \left\langle v,\widetilde{\operatorname{H}}_{\mu, t}w\right\rangle_{L^2_\mu}.
    \end{equation*}
    The linearity follows immediately from a direct calculation. By assumption, $\nabla \nabla_\mu F(\mu,x) \in \mathbb R^{d \times d}$ is symmetric, i.e.,
    \begin{equation*}
        \nabla \nabla_\mu F(\mu,x) = \nabla \nabla_\mu F(\mu,x)^\top.
    \end{equation*}
    Hence, using the identity $\langle u,v \rangle = u^\top v,$ for all $u,v \in \mathbb R^d,$
    \begin{equation}
    \label{1st part hessian sym}
    \begin{aligned}
        \left\langle \widetilde{\operatorname{M}}_{\mu, t}v, w\right\rangle_{L^2_\mu}&=\int_{\mathbb R^d} \left\langle \nabla \nabla_\mu F(\mu,\pi_t(x))v(x), w(x)\right\rangle\ \mathrm{d}\mu(x)\\
        &= \int_{\mathbb R^d} \left(\nabla_\mu F(\mu,\pi_t(x))v(x))\right)^\top w(x) \mathrm{d}\mu(x)\\
        &= \int_{\mathbb R^d} v(x)^\top\nabla_\mu F(\mu,\pi_t(x))^\top w(x) \mathrm{d}\mu(x)\\
        &=\int_{\mathbb R^d} \left\langle v(x), \nabla_\mu F(\mu,\pi_t(x))^\top w(x)\right\rangle\mathrm{d}\mu(x)\\
        &=\int_{\mathbb R^d} \left\langle v(x), \nabla_\mu F(\mu,\pi_t(x))w(x)\right\rangle\mathrm{d}\mu(x)\\ 
        &=\left\langle v, \widetilde{\operatorname{M}}_{\mu, t}w\right\rangle_{L^2_\mu}.
    \end{aligned}
    \end{equation}
    By regularity of $F$, $\nabla_\mu^2 F \in \mathbb R^{d \times d}$ is symmetric, i.e.,
    \begin{equation*}
        \nabla_\mu^2 F(\mu, x,\bar{x})^\top = \nabla_\mu^2 F(\mu, \bar x, x).
    \end{equation*}
    Hence,
    \begin{equation}
    \label{2nd part hessian sym}
    \begin{aligned}
        \left\langle \widetilde{\operatorname{K}}_{\mu, t}v, w\right\rangle_{L^2_\mu}&
        =\int_{\mathbb R^d} \int_{\mathbb R^d} \left\langle \nabla_\mu^2 F(\mu, \pi_t(x), \pi_t(\bar{x}))v(\bar{x}), w(x)\right\rangle\mathrm d\mu(\bar{x})\mathrm{d}\mu(x)\\ 
        &=\int_{\mathbb R^d} \int_{\mathbb R^d} \left(\nabla_\mu^2 F(\mu, \pi_t(x), \pi_t(\bar{x}))v(\bar{x})\right)^\top w(x)\mathrm d\mu(\bar{x})\mathrm{d}\mu(x)\\
        &=\int_{\mathbb R^d} \int_{\mathbb R^d} v(\bar{x})^\top \nabla_\mu^2 F(\mu, \pi_t(x), \pi_t(\bar{x}))^\top w(x)\mathrm d\mu(\bar{x})\mathrm{d}\mu(x)\\
        &=\int_{\mathbb R^d} \int_{\mathbb R^d}v(\bar{x})^\top \nabla_\mu^2 F(\mu, \pi_t(\bar x),\pi_t(x)) w(x)\mathrm d\mu(\bar{x})\mathrm{d}\mu(x)\\
        &=\int_{\mathbb R^d} \int_{\mathbb R^d} \left\langle v(\bar{x}), \nabla_\mu^2 F(\mu, \pi_t(\bar x),\pi_t(x)) w(x)\right\rangle\mathrm{d}\mu(x)\mathrm d\mu(\bar{x})\\
        &=\int_{\mathbb R^d} \left\langle v(\bar{x}), \int_{\mathbb R^d}\nabla_\mu^2 F(\mu, \pi_t(\bar x),\pi_t(x)) w(x)\mathrm{d}\mu(x)\right\rangle\mathrm d\mu(\bar{x})\\
        &=\left\langle v,  \widetilde{\operatorname{K}}_{\mu, t}w\right\rangle_{L^2_\mu},
    \end{aligned}
    \end{equation}
    where the penultimate equality follows by Fubini's theorem. Adding \eqref{1st part hessian sym} and \eqref{2nd part hessian sym} gives the conclusion. 
\end{proof}
\begin{proposition}[Wasserstein Hessian of $F$ along $W_2$--curves]
\label{proposition:Wasserstein-hessian}
    Assume $F:\mathcal{P}_2(\mathbb R^d) \to \mathbb R$ is Wasserstein regular. Let $\mu \in\mathcal{P}_2(\mathbb R^d)$, $v \in L_\mu^2$ and for $t \in [0,1]$ define the curve $\mu_t = (\pi_t)_{\#}\mu$, with $\pi_t = \operatorname{Id}+tv$. If $\pi_t$ is invertible $\mu$-a.e., for all $t \in [0,1]$, setting $v_t = v\circ \pi_t^{-1}$, 
\begin{equation*}
\label{eq:hessian_formula}
\frac{\mathrm d^2}{\mathrm d t^2}F(\mu_t)
= \left\langle \operatorname{H}_{\mu_t}v_t, v_t\right\rangle_{L^2_{\mu_t}},
\end{equation*}
where, the Hessian operator $\operatorname{H}_\mu: L_\mu^2 \to L_\mu^2$ is defined by
\begin{equation*}
    \operatorname{H}_\mu[v](x) = \underbrace{\nabla \nabla_\mu F(\mu,x)v(x)}_{=:\operatorname{M}_\mu[v](x)} +\underbrace{\int_{\mathbb R^d} \nabla_\mu^2F(\mu, x, \bar{x})v(\bar{x})\mathrm d\mu(\bar{x})}_{=:\operatorname{K}_\mu[v](x)},
\end{equation*}
and $\operatorname{H}_{\mu} \in \mathcal{L}\mathcal{S}(L_\mu^2)$.
\end{proposition}
\begin{proof}
    We follow the calculations presented after of \cite[Proposition 12]{bonet2024mirror}. The relation between $\widetilde{\operatorname{H}}_{\mu, t}$ and $\operatorname{H}_{\mu_t}$ can be established for example if $\pi_t$ is invertible for all $t$. In this case, setting $v_t = v\circ \pi_t^{-1} \in L_{\mu_t}^2$, we can write by Proposition \ref{proposition:Wasserstein-hessian-appendix},
\begin{equation*}
    \begin{aligned}
        \frac{\mathrm{d}^2}{\mathrm{d}t^2}F(\mu_t) &= \langle \widetilde{\operatorname{H}}_{\mu, t} v, v\rangle_{L^2_\mu} \\
        &= \int_{\mathbb R^d} \langle \widetilde{\operatorname{H}}_{\mu, t}[v](x), v(x)\rangle\ \mathrm{d}\mu(x) \\
        &= \int_{\mathbb R^d} \langle \widetilde{\operatorname{H}}_{\mu, t}[v]\big(\pi_t^{-1}(x_t)\big), v_t(x_t)\rangle\ \mathrm{d}\mu_t(x_t) \\
        &=\langle \mathrm{H}_{\mu_t} v_t, v_t\rangle_{L^2_{\mu_t}}.
    \end{aligned}
\end{equation*}
The last equality follows from Proposition \ref{proposition:Wasserstein-hessian-appendix} and that $\mu_t=(\pi_t)_\#\mu$, as for all $x\in\mathrm{supp}(\mu)$, which give
\begin{equation*}
    \begin{aligned}
        \widetilde{\operatorname{H}}_{\mu, t}[v](x) &= \nabla \nabla_\mu F(\mu_t,\pi_t(x))v(x) +\int_{\mathbb R^d} \nabla_\mu^2F(\mu_t, \pi_t(x), \pi_t(\bar{x}))v(\bar{x})\mathrm d\mu(\bar{x}) \\
        &= \nabla \nabla_\mu F(\mu_t,x_t)v_t(x_t) +\int_{\mathbb R^d} \nabla_\mu^2F(\mu_t, x_t, \bar{x}_t))v_t(\bar{x}_t)\mathrm d\mu_t(\bar{x}_t) \\
        &= \mathrm{H}_{\mu_t}[v_t](x_t),
    \end{aligned}
\end{equation*}
and thus $\widetilde{\operatorname{H}}_{\mu, t}[v](\pi_t^{-1}(x_t)) = \mathrm{H}_{\mu_t}[v_t](x_t)$. The fact that $\operatorname{H}_\mu \in \mathcal{LS}(L_\mu^2)$ follows directly from Proposition \ref{proposition:Wasserstein-hessian-appendix}.
\end{proof}

\begin{lemma}[Second-order expansion of $F$ along $L^2_\mu$ perturbations]
\label{lem:2ndorder-expansion}
Assume that Assumption~\ref{assumption:smooth-F} holds and that the maps $(\nu,x)\mapsto \nabla\nabla_\mu F(\nu,x),$ $(\nu,x,y)\mapsto \nabla_\mu^2 F(\nu,x,y)$ are jointly continuous, with the measure variable endowed with
the $W_2$ topology. Fix $\mu\in\mathcal P_2(\mathbb R^d)$ and $v\in L^2_\mu$,
and define $\pi_t:=\operatorname{Id}+tv,$ $\mu_t:=(\pi_t)_\#\mu.$ Then
\begin{equation*}
    F(\mu_t)=F(\mu)+t\langle\nabla_\mu F(\mu,\cdot),v\rangle_{L^2_\mu} +\frac{t^2}{2}\langle\operatorname{H}_\mu v,v\rangle_{L^2_\mu} +o(t^2),
    \quad t\to0.
\end{equation*}
\end{lemma}
\begin{proof}
Write
\[
    m_t(x):=\nabla\nabla_\mu F(\mu_t,\pi_t(x)),
    \qquad
    k_t(x,y):=
    \nabla_\mu^2F(\mu_t,\pi_t(x),\pi_t(y)).
\]
Since $ W_2(\mu_t,\mu_s)\le |t-s|\|v\|_{L^2_\mu},$ joint continuity implies that, as $t\to s$, $\nabla\nabla_\mu F(\mu_t,\pi_t(x)) \to \nabla\nabla_\mu F(\mu_s,\pi_s(x)),$ and $\nabla_\mu^2F(\mu_t,\pi_t(x),\pi_t(y)) \to \nabla_\mu^2F(\mu_s,\pi_s(x),\pi_s(y)),$ for $\mu$-a.e. $x$ and $\mu\otimes\mu$-a.e. $(x,y)$, respectively. Assumption~\ref{assumption:smooth-F} and the pushforward identities give
\begin{equation*}
    \|\nabla\nabla_\mu F(\mu_t,\cdot)\|_{L^\infty_\mu}\le C_{\mathrm M}, \quad \|\nabla_\mu^2F(\mu_t,\cdot,\cdot)\|_{L^2_{\mu\otimes\mu}}\le C_{\mathrm K}.
\end{equation*}
Hence, dominated convergence implies
\begin{equation*}
    \int \langle \nabla\nabla_\mu F(\mu_t,\pi_t(x))v(x),v(x)\rangle \mathrm \mu(x) \to \int \langle \nabla\nabla_\mu F(\mu_s,\pi_s(x))v(x),v(x)\rangle \mathrm d\mu(x).
\end{equation*}
For the kernel term, pointwise convergence and uniform $L^2_{\mu\otimes\mu}$ boundedness imply
\begin{equation*}
    \nabla_\mu^2F(\mu_t,\pi_t(x),\pi_t(y))\rightharpoonup \nabla_\mu^2F(\mu_s,\pi_s(x),\pi_s(y)) \text{ in }L^2_{\mu\otimes\mu}.
\end{equation*}
Indeed, the uniform $L^2$ bound implies uniform integrability in $L^1$, so pointwise convergence gives $L^1$ convergence. Convergence against bounded test functions follows, and truncation of an arbitrary $L^2$ test function, together with the uniform $L^2$ bound, gives weak $L^2$ convergence. Since $(x,y)\mapsto v(x)v(y)^\top$ belongs to $L^2_{\mu\otimes\mu},$ $\|v(x)v(y)^\top\|_{L^2_{\mu\otimes\mu}} =\|v\|_{L^2_\mu}^2,$ we obtain
\begin{equation*}
    \iint \langle \nabla_\mu^2F(\mu_t,\pi_t(x),\pi_t(y))v(y),v(x)\rangle\mathrm d\mu(x)\mathrm d\mu(y) \to \iint \langle \nabla_\mu^2F(\mu_t,\pi_t(x),\pi_t(y))v(y),v(x)\rangle\mathrm d\mu(x)\mathrm d\mu(y).
\end{equation*}
Thus $q(t):= \langle\widetilde{\operatorname{H}}_{\mu,t}v,v\rangle_{L^2_\mu}$ is continuous, and $q(0)=\langle\operatorname{H}_\mu v,v\rangle_{L^2_\mu}$.

Set $\phi(t):=F(\mu_t)$. By Lemma~\ref{lemma_for_second_order_mu} and Proposition~\ref{proposition:Wasserstein-hessian-appendix},
\begin{equation*}
    \phi'(0) =\langle\nabla_\mu F(\mu,\cdot),v\rangle_{L^2_\mu}, \quad \phi''(t)=q(t).
\end{equation*}
Taylor's formula with integral remainder therefore gives
\begin{equation*}
    \phi(t) =  \phi(0)+t\phi'(0)+\frac{t^2}{2}q(0) +r_{\mu,v}(t),
\end{equation*}
where $r_{\mu,v}(t):= t^2\int_0^1(1-s)\left(q(st)-q(0)\right)\mathrm ds.$ Hence
\begin{equation*}
    \frac{|r_{\mu,v}(t)|}{t^2} \le \frac12 \sup_{|u|\le |t|}|q(u)-q(0)| \to 0,
\end{equation*}
which proves the assertion.
\end{proof}

\begin{lemma}[First-order expansion of the Wasserstein gradient]
\label{lem:grad-expansion}
Let Assumption~\ref{assumption:Hessian-lipschitz} hold. Let $\mu\in\mathcal P_2(\mathbb R^d)$ and $v\in L^2_\mu$. For $t\in[0,1]$, define $\pi_t:=\operatorname{Id}+tv,$ $\mu_t:=(\pi_t)_\#\mu.$ Then
\begin{equation*}
    \nabla_\mu F(\mu_t,\cdot)\circ\pi_t - \nabla_\mu F(\mu,\cdot) = t\operatorname{H}_\mu v+o(t) \text{ in }L^2_\mu, \quad t\downarrow0.
\end{equation*}
\end{lemma}
\begin{proof}
By Proposition \ref{proposition:Wasserstein-hessian-appendix},
\begin{equation*}
    \widetilde{\operatorname{H}}_{\mu,t}v(x) = \nabla\nabla_\mu F(\mu_t,\pi_t(x))v(x) + \int_{\mathbb R^d}\nabla_\mu^2F(\mu_t,\pi_t(x),\pi_t(y))v(y)\mu(\mathrm dy).
\end{equation*}
Fix an arbitrary $w\in L^2_\mu$. Apply the chain rule componentwise to $\nabla_\mu F(\mu_t,\pi_t(x))$, differentiating both its spatial argument and its measure argument. For the latter, use the coupling $\gamma_{t,t+h}:=(\pi_t,\pi_{t+h})_\#\mu \in\Pi(\mu_t,\mu_{t+h}),$ whose displacement is $hv$. The first-order expansion in Proposition~\ref{prop:strong_diff_w} gives the weak derivative identity
\begin{align*}
\frac{\mathrm d}{\mathrm dt} \langle \nabla_\mu F(\mu_t,\cdot)\circ\pi_t,w\rangle_{L^2_\mu} &= \int_{\mathbb R^d} \langle \nabla\nabla_\mu F(\mu_t,\pi_t(x))v(x),w(x)\rangle\mu(\mathrm dx)\\ 
&+ \int_{\mathbb R^d}\int_{\mathbb R^d}\langle \nabla_\mu^2F(\mu_t,\pi_t(x),\pi_t(y))v(y),w(x)\rangle \mu(\mathrm dy)\mu(\mathrm dx)\\&= \langle\widetilde {\operatorname{H}}_{\mu,t}v,w\rangle_{L^2_\mu}.
\end{align*}
The terms are integrable as their absolute values are bounded, respectively, by $\|\nabla\nabla_\mu F(\mu_t,\cdot)\|_{L^\infty_\mu}\|v\|_{L^2_\mu}\|w\|_{L^2_\mu}$ and $\|\nabla_\mu^2F(\mu_t,\cdot,\cdot)\|_{L^2_{\mu\otimes\mu}} \|v\|_{L^2_\mu}\|w\|_{L^2_\mu}.$ Assumption~\ref{assumption:Hessian-lipschitz} gives
\begin{align*}
&\|\nabla\nabla_\mu F(\mu_t,\cdot)-\nabla\nabla_\mu F(\mu_s,\cdot)\|_{L^\infty_\mu} \le |t-s|L_{\operatorname{M}}\|v\|_{L^2_\mu},
\\
&\|\nabla_\mu^2F(\mu_t,\cdot,\cdot)-\nabla_\mu^2F(\mu_s,\cdot,\cdot)\|_{L^2_{\mu\otimes\mu}} \le |t-s|L_{\operatorname{K}}\|v\|_{L^2_\mu}.
\end{align*}
In particular, these coefficient norms are bounded along the curve, and $t\mapsto\widetilde{\operatorname{H}}_{\mu,t}v$ is continuous in $L^2_\mu$. Integrating the weak identity therefore gives
\begin{equation*}
    \langle \nabla_\mu F(\mu_t,\cdot)\circ\pi_t - \nabla_\mu F(\mu,\cdot),w\rangle_{L^2_\mu} = \left\langle \int_0^t\widetilde {\operatorname{H}}_{\mu,s}v\mathrm ds,w \right\rangle_{L^2_\mu}.
\end{equation*}
Since this holds for every $w\in L^2_\mu$, we obtain
\begin{equation*}
    \nabla_\mu F(\mu_t,\cdot)\circ\pi_t - \nabla_\mu F(\mu,\cdot) = \int_0^t\widetilde{\operatorname{H}}_{\mu,s}v\mathrm ds \text{ in } L^2_\mu.
\end{equation*}
Finally, $\widetilde{\operatorname{H}}_{\mu,0}=\operatorname{H}_\mu$, so
\begin{align*}
\|\nabla_\mu F(\mu_t,\cdot)\circ\pi_t - \nabla_\mu F(\mu,\cdot)-t\operatorname{H}_\mu v\|_{L^2_\mu} &\le \int_0^t \|(\widetilde{\operatorname{H}}_{\mu,s}-\operatorname{H}_\mu)v\|_{L^2_\mu} \mathrm ds\\ &\le t\sup_{0\le s\le t} \|(\widetilde{\operatorname{H}}_{\mu,s}-\operatorname{H}_\mu)v\|_{L^2_\mu} = o(t).
\end{align*}
\end{proof}

\begin{lemma}[$L_\mu^2$-smoothness of the Wasserstein gradient of $F$]
\label{lemma:smoothness gradient of F}
    Let Assumption \ref{assumption:Hessian-lipschitz} hold. Let $\mu \in\mathcal{P}_2(\mathbb R^d)$ and $v \in L_\mu^2$. For $t \in [0,1]$, define $\pi_t = \operatorname{Id}+tv,$ $\mu_t = (\pi_t)_{\#}\mu$. Then
    \begin{equation*}
        \left\|\nabla_\mu F((\operatorname{Id}+v)_{\#}\mu,\cdot) \circ \pi_1 - \nabla_\mu F(\mu,\cdot) -  \operatorname{H}_\mu v\right\|_{L_\mu^2} \leq \frac{L_{\operatorname{M}}+L_{\operatorname{K}}}{2}\|v\|^2_{L_\mu^2}.
    \end{equation*}
\end{lemma}
\begin{proof}
By the integral identity in Lemma~\ref{lem:grad-expansion},
\begin{equation*}
    \nabla_\mu F((\operatorname{Id}+v)_\#\mu,\cdot) \circ(\operatorname{Id}+v)-\nabla_\mu F(\mu,\cdot)-\operatorname{H}_\mu v = \int_0^1 (\widetilde{\operatorname{H}}_{\mu,t}-\operatorname{H}_\mu)v\mathrm dt.
\end{equation*}
Lemma~\ref{lem:lip-curves} gives
\begin{equation*}
    \|\widetilde{\operatorname{H}}_{\mu,t}-\operatorname{H}_\mu\|_{\mathrm{op}} \le t(L_{\operatorname{M}}+L_{\operatorname{K}})\|v\|_{L^2_\mu}.
\end{equation*}
Hence,
\begin{equation*}
    \left\|\int_0^1(\widetilde{\operatorname{H}}_{\mu,t}-\operatorname{H}_\mu)v\mathrm dt\right\|_{L^2_\mu} \le (L_{\operatorname{M}}+L_{\operatorname{K}})\int_0^1
t\|v\|_{L^2_\mu}^2\mathrm dt
=
\frac12(L_{\operatorname{M}}+L_{\operatorname{K}})\|v\|_{L^2_\mu}^2.
\end{equation*}
\end{proof}

\subsection{Operator bounds and preconditioner identities}
We conclude with several technical operator estimates. These bounds control the size and regularity of the Hessian along transport curves and are used to justify the regularized saddle-free preconditioner.
\begin{lemma}
\label{lem:upper-bound-inner-hessian}
    Let Assumption \ref{assumption:smooth-F} hold and let $\beta \geq 0$. Then $(\operatorname{H}_\mu)_{\mu \in \mathcal{P}_2(\mathbb R^d)}$ is self-adjoint and, for all $v \in L_\mu^2$,
    \begin{equation*}
        \left|\langle \operatorname{H}_{\mu}v,v\rangle_{L^2_{\mu}}\right| \leq \left\langle\left(\operatorname{H}_{\mu}^2+\beta\operatorname{I_{d \times d}}\right)^{\frac{1}{2}}v, v\right\rangle_{L^2_{\mu}},
    \end{equation*}
    where $\operatorname{H}_{\mu}^2$ denotes the composition $\operatorname{H}_{\mu}^2[v] := \operatorname{H}_{\mu}[\operatorname{H}_{\mu}[v]]$.
\end{lemma}
\begin{proof}
By Proposition \ref{proposition:Wasserstein-hessian}, $(\operatorname{H}_{\mu})_{\mu} \subset \mathcal{L}\mathcal{S}(L_\mu^2)$, i.e., $(\operatorname{H}_{\mu})_{\mu}$ is linear and symmetric, and since $(\operatorname{H}_\mu)_{\mu}$ is uniformly bounded by Assumption \ref{assumption:smooth-F}, it follows by \cite[Theorem 23.22]{gustafson2020mathematical} that $(\operatorname{H}_\mu)_{\mu}$ is self-adjoint, i.e., $\operatorname{H}_{\mu} = \operatorname{H}_{\mu}^*$. 

Since $\operatorname{H}_{\mu}$ is self-adjoint and $\operatorname{H}_{\mu}^*\operatorname{H}_{\mu} = \operatorname{H}_{\mu}^2 \succeq 0$, using \cite[Definition; VI: Bounded Operators, page 196]{ReedSimon1980}, we may define the absolute value operator by $|\operatorname{H}_{\mu}| := \left(\operatorname{H}_{\mu}^2\right)^{\frac{1}{2}}$. By \cite[Problem 2.35]{kato1995perturbation},
\begin{equation*}
    \left|\langle \operatorname{H}_{\mu}v,v\rangle_{L^2_{\mu}}\right| \leq \langle \left|\operatorname{H}_{\mu}\right|v,v\rangle_{L^2_{\mu}},
\end{equation*}
for any $v \in L^2_{\mu}$. 

Using the property of adjoint, $(AB)^* = B^* A^*$, we get
\begin{equation*}
    (\operatorname{H}_{\mu}^2)^* = (\operatorname{H}_{\mu}\operatorname{H}_{\mu})^* = \operatorname{H}_{\mu}^*\operatorname{H}_{\mu}^* = \operatorname{H}_{\mu}\operatorname{H}_{\mu} = \operatorname{H}_{\mu}^2,
\end{equation*}
hence $\operatorname{H}_{\mu}^2$ is self-adjoint. The identity operator $\operatorname{I_{d \times d}}$ is trivially self-adjoint and the sum of two self-adjoint operators is self-adjoint. Thus, $\operatorname{H}_{\mu}^2+\beta\operatorname{I_{d \times d}}$ is self-adjoint. Furthermore, because $\operatorname{H}_{\mu}^2 \succeq 0$, the operator $\operatorname{H}_{\mu}^2+\beta\operatorname{I_{d \times d}}\succeq 0$. Then by \cite[Theorem VI.9 (square root lemma)]{ReedSimon1980} the square root operator $\left(\operatorname{H}_{\mu}^2+\operatorname{I_{d \times d}}\right)^{\frac{1}{2}}$ is well-defined. 

Let $\sigma(\operatorname{H}_\mu)$ be the spectrum of $\operatorname{H}_\mu$ and $E(\lambda)$ be the spectral measure associated to the eigenvalue $\lambda \in \sigma(\operatorname{H}_\mu)$. Since $\operatorname{H}_{\mu}$ is self-adjoint, we have $\sigma(\operatorname{H}_\mu) \subset \mathbb R$.  Then by spectral calculus
\begin{align*}
    \langle\left|\operatorname{H}_{\mu}\right|v,v\rangle_{L^2_{\mu}} = \int_{\sigma(\operatorname{H}_\mu)} |\lambda| \mathrm{d}\langle v, E(\lambda) v\rangle_{L^2_{\mu}} &= \int_{\sigma(\operatorname{H}_\mu)} |\lambda| \mathrm{d}\langle v, E(\lambda) v\rangle_{L^2_{\mu}}
    \\&= \int_{\sigma(\operatorname{H}_\mu)} (\lambda^2)^{\frac{1}{2}} \mathrm{d}\langle v, E(\lambda) v\rangle_{L^2_{\mu}}\\ 
    &\leq \int_{\sigma(\operatorname{H}_\mu)} (\lambda^2+\beta)^{\frac{1}{2}} \mathrm{d}\langle v, E(\lambda) v\rangle_{L^2_{\mu}}\\ 
    &= \left\langle\left(\operatorname{H}_{\mu}^2+\beta\operatorname{I_{d \times d}}\right)^{\frac{1}{2}}v, v\right\rangle_{L^2_{\mu}}.
\end{align*}
\end{proof}
\begin{lemma}[Essential supremum under pushforward]\label{lem:esssup_pushforward}
Let $\mu\in\mathcal P(\mathbb R^d)$, let $T:\mathbb R^d\to\mathbb R^d$ be measurable and set $\nu:=T_\#\mu$. Let $f:\mathbb R^d\to \mathbb R$ be measurable. Then
\begin{equation*}
\|f\circ T\|_{\mu,\infty}=\|f\|_{\nu,\infty},
\end{equation*}
\end{lemma}
\begin{proof}
By definition of essential supremum, for any measurable $f$,
\[
\|f\|_{\mu,\infty}:=\inf\{M\ge 0: |f(x)|\le M \text{ for }\mu\text{-a.e. }x\}.
\]
Fix $M\ge 0$. We claim that $f\le M \ \nu\text{-a.e.} \Longleftrightarrow f\circ T\le M \ \mu\text{-a.e.}$ Indeed,
\[
f\le M\ \nu\text{-a.e.} \Longleftrightarrow
\nu(\{y\in\mathbb R^d: f(y)>M\})=0.
\]
Since $\nu=T_\#\mu$, by definition of pushforward, we have $\nu(\{y: f(y)>M\}) = \mu\big(T^{-1}(\{y: f(y)>M\})\big)$. But $T^{-1}(\{y: f(y)>M\}) = \{x\in\mathbb R^d: f(T(x))>M\} = \{x\in\mathbb R^d: (f\circ T)(x)>M\}$.
Therefore,
\[
\nu(\{y: f(y)>M\})=0 \Longleftrightarrow \mu(\{x: (f\circ T)(x)>M\})=0.
\]
Thus the sets $\{M\ge 0: f\le M\ \nu\text{-a.e.}\}$ and $\{M\ge 0: f\circ T\le M\ \mu\text{-a.e.}\}$ coincide. Taking infimum over $M\ge 0$ yields the conclusion.
\end{proof}
\begin{lemma}[Uniform operator norm bound for Hessian]\label{lem:uniform_bound_tildeH}
Let Assumption \ref{assumption:smooth-F} hold. Let $\mu\in\mathcal P_2(\mathbb R^d)$, $v\in L^2_\mu$ and define $\mu_t:=(\pi_t)_\#\mu$, with $\pi_t = \operatorname{Id}+tv$, for all $t\in[0,1]$. Then, for every $t\in[0,1]$,
\[
\|\widetilde{\operatorname{H}}_{\mu,t}\|_{\mathrm{op}}\le C_{\operatorname{M}}+C_{\operatorname{K}}.
\]
\end{lemma}
\begin{proof}
Fix $t\in[0,1]$. Let $v\in L^2_\mu$. By definition, $\widetilde {\operatorname{M}}_{\mu,t}[v](x)=\nabla\nabla_\mu F(\mu_t,\pi_t(x))v(x)$. Hence
\[
\|\widetilde {\operatorname{M}}_{\mu,t}[v](x)\|
\le
\big\|\nabla\nabla_\mu F(\mu_t,\pi_t(x))\big\|\|v(x)\|.
\]
Squaring and integrating with respect to $\mu$ yields
\begin{align*}
\|\widetilde {\operatorname{M}}_{\mu,t}v\|_{L^2_\mu}^2
&\le
\left(\mu\operatorname*{-ess\sup}_{x}\big\|\nabla\nabla_\mu F(\mu_t,\pi_t(x))\big\|\right)^2
\|v\|_{L^2_\mu}^2.
\end{align*}
By applying Lemma \ref{lem:esssup_pushforward} with $T=\pi_t$, $\nu=\mu_t$, and $f(y)=\|\nabla\nabla_\mu F(\mu_t,y)\|$, we get
\begin{equation*}
    \mu\operatorname*{-ess\sup}_{x}\big\|\nabla\nabla_\mu F(\mu_t,\pi_t(x))\big\| = \|\nabla\nabla_\mu F(\mu_t,\cdot)\|_{\mu_t,\infty}.
\end{equation*}
By Assumption \ref{assumption:smooth-F}, $\|\nabla\nabla_\mu F(\mu_t,\cdot)\|_{\mu_t,\infty}\le C_{\operatorname{M}}.$ Therefore $\|\widetilde{\operatorname{M}}_{\mu,t}v\|_{L^2_\mu}\le C_{\operatorname{M}}\|v\|_{L^2_\mu},$ and hence
\begin{equation}\label{eq:bound_tildeM}
\|\widetilde{\operatorname{M}}_{\mu,t}\|_{\mathrm{op}}\le C_{\operatorname{M}}.
\end{equation}
By definition 
\[
\widetilde{\operatorname{K}}_{\mu,t}[v](x)=\int_{\mathbb R^d}\nabla_\mu^2F(\mu_t,\pi_t(x),\pi_t(\bar x))v(\bar x)\mathrm d\mu(\bar x).
\]
By Cauchy--Schwarz in $\bar x$,
\begin{align*}
\|\widetilde{\operatorname{K}}_{\mu,t}[v](x)\|^2
&=
\left\|\int_{\mathbb R^d} \nabla_\mu^2F(\mu_t,\pi_t(x),\pi_t(\bar x))v(\bar x)\mathrm d\mu(\bar x)\right\|^2\\
&\le
\left(\int_{\mathbb R^d}\|\nabla_\mu^2F(\mu_t,\pi_t(x),\pi_t(\bar x))\|^2\mathrm d\mu(\bar x)\right)
\left(\int_{\mathbb R^d}\|v(\bar x)\|^2\mathrm d\mu(\bar x)\right).
\end{align*}
Integrating in $x$ gives
\begin{align*}
\|\widetilde{\operatorname{K}}_{\mu,t}v\|_{L^2_\mu}^2
&\le
\left(\int_{\mathbb R^d}\int_{\mathbb R^d}|\nabla_\mu^2F(\mu_t,\pi_t(x),\pi_t(\bar x))|^2\mathrm d\mu(x)\mathrm d\mu(\bar x)\right)
\|v\|_{L^2_\mu}^2.
\end{align*}
Thus
\begin{equation}\label{eq:K_hs_bound}
\|\widetilde{\operatorname{K}}_{\mu,t}\|_{\mathrm{op}}\le \|\nabla_\mu^2F(\mu_t,\cdot,\cdot) \circ \pi_t\|_{L^2_{\mu\otimes\mu}}.
\end{equation}
Now use $\mu_t=(\pi_t)_\#\mu$ in both variables to get
\[
\|\nabla_\mu^2F(\mu_t,\cdot,\cdot) \circ \pi_t\|_{L^2_{\mu\otimes\mu}}
=
\|\nabla_\mu^2F(\mu_t,\cdot,\cdot)\|_{L^2_{\mu_t\otimes\mu_t}}
\le C_{\operatorname{K}}.
\]
by Assumption \ref{assumption:smooth-F}. Combining with \eqref{eq:K_hs_bound},
\begin{equation}\label{eq:bound_tildeK}
\|\widetilde {\operatorname{K}}_{\mu,t}\|_{\mathrm{op}}\le  C_{\operatorname{K}}.
\end{equation}
Since $\widetilde{\operatorname{H}}_{\mu,t}=\widetilde {\operatorname{M}}_{\mu,t}+\widetilde{\operatorname{K}}_{\mu,t}$, we have by the triangle inequality
\[
\|\widetilde{\operatorname{H}}_{\mu,t}\|_{\mathrm{op}}
\le
\|\widetilde{\operatorname{M}}_{\mu,t}\|_{\mathrm{op}}+\|\widetilde{\operatorname{K}}_{\mu,t}\|_{\mathrm{op}}
\le C_{\operatorname{M}}+C_{\operatorname{K}},
\]
using \eqref{eq:bound_tildeM} and \eqref{eq:bound_tildeK}.
\end{proof}
\begin{lemma}[Coupling Lipschitzness implies Lipschitzness along $W_2$--curves]
\label{lem:lip-curves}
    Let Assumption \ref{assumption:Hessian-lipschitz} hold. Let $\mu\in\mathcal{P}_2(\mathbb R^d)$, $v\in L^2_\mu$ and define $\mu_t:=(\operatorname{Id}+t v)_\#\mu$, for $t \in [0,1]$. Then, for any $t,s\in [0,1]$,
    \begin{align*}
        \|\widetilde{\operatorname{H}}_{\mu,t}-\widetilde{\operatorname{H}}_{\mu,s}\|_{\mathrm{op}} \leq |t-s|(L_{\operatorname{M}} + L_{\operatorname{K}})\|v\|_{L_\mu^2}.
    \end{align*}
\end{lemma}
\begin{proof}
    By Proposition \ref{proposition:Wasserstein-hessian-appendix}, the operator $\widetilde{\operatorname{H}}_{\mu, t}$ is defined by
\begin{equation*}
    \widetilde{\operatorname{H}}_{\mu, t}[v](x) = \nabla \nabla_\mu F(\mu_t,\pi_t(x))v(x) + \int_{\mathbb R^d} \nabla_\mu^2 F(\mu_t, \pi_t(x), \pi_t(\bar{x}))v(\bar{x})\mathrm d\mu(\bar{x}),
\end{equation*}
where
\begin{equation*}
    \widetilde{\operatorname{M}}_{\mu, t}[v](x) = \nabla \nabla_\mu F(\mu_t,\pi_t(x))v(x), \quad \widetilde{\operatorname{K}}_{\mu, t}[v](x)= \int_{\mathbb R^d} \nabla_\mu^2 F(\mu_t, \pi_t(x), \pi_t(\bar{x}))v(\bar{x})\mathrm d\mu(\bar{x}).
\end{equation*}
For any $t,s\in [0,1]$, we have  
\begin{equation*}
    \|\widetilde{\operatorname{H}}_{\mu, t} - \widetilde{\operatorname{H}}_{\mu, s}\|_{\mathrm{op}} \leq \|\widetilde{\operatorname{M}}_{\mu, t} -\widetilde{\operatorname{M}}_{\mu, s}\|_{\mathrm{op}} + \|\widetilde{\operatorname{K}}_{\mu, t} - \widetilde{\operatorname{K}}_{\mu, s}\|_{\mathrm{op}}.
\end{equation*}
Taking $\gamma = (\pi_s,\pi_t)_{\#}\mu \in \Pi(\mu_s, \mu_t)$ gives 
\begin{align*}
    \|(\widetilde{\operatorname{M}}_{\mu, t} -\widetilde{\operatorname{M}}_{\mu, s})v\|_{L_\mu^2}^2 &\leq \int_{\mathbb R^d} \left\| (\widetilde{\operatorname{M}}_{\mu, t}-\widetilde{\operatorname{M}}_{\mu, s})[v](x)\right\|^2\mathrm{d}\mu(x)\\
    &= \int_{\mathbb R^d}\|(\nabla \nabla_\mu F(\mu_t,\pi_t(x))-\nabla \nabla_\mu F(\mu_s,\pi_s(x)))v(x)\|^2\mathrm{d}\mu(x)\\
    &\leq \|\nabla  \nabla_\mu F(\mu_t,\cdot)  - \nabla \nabla_\mu F(\mu_s,\cdot)\|_{\mu ,\infty}^2 \|v\|_{L_\mu^2}^2\\
    &\leq (t-s)^2L_{\operatorname{M}}^2\|v\|_{L_\mu^2}^2\|v\|_{L_\mu^2}^2,
\end{align*}
for all $v \in L_\mu^2$. Thus, $\|\widetilde{\operatorname{M}}_{\mu, t} -\widetilde{\operatorname{M}}_{\mu, s}\|_{\mathrm{op}} \leq |t-s|L_{\operatorname{M}}\|v\|_{L_\mu^2}$. Similarly,
\begin{align*}
    \|(\widetilde{\operatorname{K}}_{\mu, t} -\widetilde{\operatorname{K}}_{\mu, s})v\|_{L_\mu^2}^2 &\leq \int_{\mathbb R^d} \left\| (\widetilde{\operatorname{K}}_{\mu, t}-\widetilde{\operatorname{K}}_{\mu, s})[v](x)\right\|^2\mathrm{d}\mu(x)\\
    &= \int_{\mathbb R^d}\left\|\int_{\mathbb R^d}(\nabla_\mu^2 F(\mu_t,\pi_t(x),\pi_t(\bar x))-\nabla_\mu^2 F(\mu_s,\pi_s(x),\pi_s(\bar x)))v(\bar x)\mathrm{d}\mu(\bar x)\right\|^2\mathrm{d}\mu(x)\\
    &\leq \|\nabla_\mu^2 F(\mu_t,\cdot,\cdot)  -  \nabla_\mu^2 F(\mu_s,\cdot,\cdot)\|_{L_{\mu \otimes \mu}^2} \|v\|_{L_\mu^2}^2\\
    &\leq (t-s)^2L_{\operatorname{K}}^2\|v\|_{L_\mu^2}^2\|v\|_{L_\mu^2}^2,
\end{align*}
where the penultimate inequality follows from Cauchy-Schwarz inequality. Thus, $\|(\widetilde{\operatorname{K}}_{\mu, t} -\widetilde{\operatorname{K}}_{\mu, s})\|_{\mathrm{op}} \leq |t-s|L_{\operatorname{K}}\|v\|_{L_\mu^2}$. Finally,
\begin{align*}
    \|\widetilde{\operatorname{H}}_{\mu, t} - \widetilde{\operatorname{H}}_{\mu, s}\|_{\mathrm{op}} &\leq \|\widetilde{\operatorname{M}}_{\mu, t} -\widetilde{\operatorname{M}}_{\mu, s}\|_{\mathrm{op}} + \|\widetilde{\operatorname{K}}_{\mu, t} - \widetilde{\operatorname{K}}_{\mu, s}\|_{\mathrm{op}}\\
    &\leq |t-s|L_{\operatorname{M}}\|v\|_{L_\mu^2} + |t-s|L_{\operatorname{K}}\|v\|_{L_\mu^2}\\
    &=|t-s|(L_{\operatorname{M}} + L_{\operatorname{K}})\|v\|_{L_\mu^2}.
\end{align*}
\end{proof}

\begin{lemma}[Two-map Lipschitz continuity of the pulled-back Hessian]
\label{lem:two-map-pulled-back-hessian-lipschitz}
Let Assumption \ref{assumption:Hessian-lipschitz} hold. Let
\(\mu\in\mathcal P_2(\mathbb R^d)\), let \(u,v\in L^2_\mu\), and define $Z_u:=\operatorname{Id}+u,$ $Z_v:=\operatorname{Id}+v,$ together with $\mu_u:=(Z_u)_{\#}\mu,$ $ \mu_v:=(Z_v)_{\#}\mu.$ Define the corresponding pulled-back Hessians on \(L^2_\mu\) by
\begin{align*}
    \widetilde{\operatorname{H}}_{\mu,u}[h](x) := \nabla\nabla_\mu F(\mu_u,Z_u(x))h(x) + \int_{\mathbb R^d}\nabla_\mu^2F\left(\mu_u,Z_u(x),Z_u(y)\right)h(y)\mathrm d\mu(y)
\end{align*}
and analogously for \(\widetilde{\operatorname{H}}_{\mu,v}[h]\). Then
\begin{equation*}
    \left\|\widetilde{\operatorname{H}}_{\mu,u}-\widetilde{\operatorname{H}}_{\mu,v}\right\|_{\mathrm{op}}\leq L_{\operatorname H} \|v-u\|_{L^2_\mu}.
\end{equation*}
\end{lemma}
\begin{proof}
Consider the coupling $\gamma_{u,v}:=(Z_u,Z_v)_{\#}\mu \in\Pi(\mu_u,\mu_v).$ Its quadratic transportation cost is
\begin{align}
    \int_{\mathbb R^d\times\mathbb R^d}\|y-x\|^2\mathrm d\gamma_{u,v}(x,y) =
    \int_{\mathbb R^d}\|Z_v(x)-Z_u(x)\|^2\mathrm d\mu(x) = \|v-u\|_{L^2_\mu}^2.
    \label{eq:two-map-coupling-cost}
\end{align}
Decompose the pulled-back Hessian as $\widetilde{\operatorname{H}}_{\mu,v}[h] = \widetilde{\operatorname{M}}_{\mu,v}[h] + \widetilde{\operatorname{K}}_{\mu,v}[h],$ where $\widetilde{\operatorname{M}}_{\mu,v}[h]= \nabla\nabla_\mu F(\mu_v,Z_v(x))h(x),$ and
\[
    \widetilde{\operatorname{K}}_{\mu,v}[h](x) = \int_{\mathbb R^d}\nabla_\mu^2F\left(\mu_v,Z_v(x),Z_v(y)\right)h(y)\mathrm d\mu(y).
\]
Applying the first part of Assumption \ref{assumption:Hessian-lipschitz} to
\(\gamma_{u,v}\), and using \eqref{eq:two-map-coupling-cost}, gives
\begin{align*}
    \mu\text{-}\operatorname{ess\,sup}_{x\in\mathbb R^d}
    \left\|\nabla\nabla_\mu F(\mu_v,Z_v(x))-\nabla\nabla_\mu F(\mu_u,Z_u(x))
    \right\| \leq L_{\operatorname M}\|v-u\|_{L^2_\mu}.
\end{align*}
Consequently,
\begin{equation}
\label{eq:two-map-multiplication-bound}
    \left\|\widetilde{\operatorname{M}}_{\mu,v}-\widetilde{\operatorname{M}}_{\mu,u}
    \right\|_{\mathrm{op}} \leq L_{\operatorname M}\|v-u\|_{L^2_\mu}.
\end{equation}
Similarly, the second part of Assumption \ref{assumption:Hessian-lipschitz} gives
\begin{align*}
    \left(\iint\left\|
        \nabla_\mu^2F\left(\mu_v,Z_v(x),Z_v(y)\right) - \nabla_\mu^2F\left(\mu_u,Z_u(x),Z_u(y)\right)
    \right\|^2 \mathrm d\mu(x)\mathrm d\mu(y)
    \right)^{1/2} \leq
    L_{\operatorname K}\|v-u\|_{L^2_\mu}.
\end{align*}
The operator norm of an integral operator is bounded by its Hilbert--Schmidt norm. Hence,
\begin{equation}
\label{eq:two-map-integral-operator-bound}
    \left\|\widetilde{\operatorname{K}}_{\mu,v}-\widetilde{\operatorname{K}}_{\mu,u}\right\|_{\mathrm{op}}
    \leq L_{\operatorname K}\|v-u\|_{L^2_\mu}.
\end{equation}
Combining \eqref{eq:two-map-multiplication-bound} and \eqref{eq:two-map-integral-operator-bound} gives
\[
    \left\|\widetilde{\operatorname{H}}_{\mu,v}-\widetilde{\operatorname{H}}_{\mu,u}\right\|_{\mathrm{op}}
    \leq(L_{\operatorname M}+L_{\operatorname K})\|v-u\|_{L^2_\mu},
\]
as claimed.
\end{proof}

\begin{lemma}[Second-order smoothness of $F$]
\label{lemma:smoothness-of-F-2}
    Assume $F:\mathcal{P}_2(\mathbb R^d) \to \mathbb R$ is Wasserstein regular. Let $\mu \in\mathcal{P}_2(\mathbb R^d)$, $v \in L_\mu^2$ and for $t \in [0,1]$ define the curve $\mu_t = (\pi_t)_{\#}\mu$, with $\pi_t = \operatorname{Id}+tv$. If Assumption \ref{assumption:smooth-F} holds, then 
    \begin{align*}
        F\left((\operatorname{Id}+v)_\#\mu\right) \leq F(\mu) + \left\langle \nabla_\mu F(\mu, \cdot), v\right\rangle_{L_\mu^2} + \frac{C_{\operatorname{M}}+C_{\operatorname{K}}}{2}\|v\|^2_{L^2_\mu},
    \end{align*}
    If Assumption \ref{assumption:Hessian-lipschitz} holds,
    \begin{align*}
        F\left((\operatorname{Id}+v)_\#\mu\right) &\leq  F(\mu) + \left\langle\nabla_\mu F(\mu,\cdot),v\right\rangle_{L_\mu^2} + \frac{1}{2}\left\langle \operatorname{H}_{\mu}v, v\right\rangle_{L^2_\mu} + \frac{L_{\operatorname{M}}+L_{\operatorname{K}}}{6}\|v\|_{L_\mu^2}^3.
    \end{align*}
\end{lemma}
\begin{proof}
We start with the proof of the first statement. By Lemma \ref{lemma_for_second_order_mu},
\begin{equation*}
         \frac{\mathrm{d}}{\mathrm{d} t}\Big|_{t=0} F(\mu_t) = \int_{\mathbb R^d} \left\langle \nabla_\mu F (\mu, x), v(x)\right\rangle \mathrm{d}\mu(x),
    \end{equation*}
By the fundamental theorem of calculus and integration by parts,
\begin{align*}
       F\left((\operatorname{Id}+v)_\#\mu\right) -F(\mu) &=  \int_0^1\frac{\mathrm{d}}{\mathrm{d}t}F(\mu_t)\mathrm{d}t = \left(t\frac{\mathrm{d}}{\mathrm{d}t}F(\mu_t)\right)\Big|_{t=0}^{t=1} - \int_0^1 t \frac{\mathrm{d}^2}{\mathrm{d}t^2}F(\mu_t)\mathrm{d}t\\
       &=\frac{\mathrm{d}}{\mathrm{d}t}\Big|_{t=1}F(\mu_t)- \int_0^1 t \frac{\mathrm{d}^2}{\mathrm{d}t^2}F(\mu_t)\mathrm{d}t\\
       &=\frac{\mathrm{d}}{\mathrm{d}t}\Big|_{t=0}F(\mu_t)+ \int_0^1 \frac{\mathrm{d}^2}{\mathrm{d}t^2}F(\mu_t)\mathrm{d}t- \int_0^1 t \frac{\mathrm{d}^2}{\mathrm{d}t^2}F(\mu_t)\mathrm{d}t\\
       &= \left\langle \nabla_\mu F(\mu, \cdot), v\right\rangle_{L^2_\mu} + \int_0^1 (1-t)\frac{\mathrm{d}^2}{\mathrm{d}t^2}F(\mu_t)\mathrm{d}t.
\end{align*}
By Proposition \ref{proposition:Wasserstein-hessian-appendix}, 
    \begin{equation*}
        \frac{\mathrm d^2}{\mathrm d t^2}F(\mu_t) = \left\langle \widetilde{\operatorname{H}}_{\mu,t}v, v\right\rangle_{L^2_{\mu}},
    \end{equation*}
and so
\begin{equation*}
    F\left((\operatorname{Id}+v)_\#\mu\right) -F(\mu) = \left\langle \nabla_\mu F(\mu, \cdot), v\right\rangle_{L^2_\mu} + \int_0^1 (1-t)\left\langle \widetilde{\operatorname{H}}_{\mu,t}v, v\right\rangle_{L^2_{\mu}}\mathrm{d}t.
\end{equation*}
    Hence, by Lemma \ref{lem:uniform_bound_tildeH},
    \begin{align*}
        F\left((\operatorname{Id}+v)_\#\mu\right) &= F(\mu) + \left\langle \nabla_\mu F(\mu, \cdot), v\right\rangle_{L_\mu^2} + \int_0^1 (1-t)\left\langle \widetilde{\operatorname{H}}_{\mu,t}v, v\right\rangle_{L^2_{\mu}}\mathrm{d}t\\
        &\leq F(\mu) + \left\langle \nabla_\mu F(\mu, \cdot), v\right\rangle_{L_\mu^2} + \sup_{t \in [0,1]}\|\widetilde{\operatorname{H}}_{\mu,t}\|_{\mathrm{op}}\int_0^1 (1-t)\|v\|^2_{L^2_{\mu}}\mathrm{d}t\\
        &\leq F(\mu) + \left\langle \nabla_\mu F(\mu, \cdot), v\right\rangle_{L_\mu^2} + \frac{C_{\operatorname{M}}+C_{\operatorname{K}}}{2}\|v\|^2_{L^2_\mu}.
    \end{align*}
     For the proof of the second statement, we proceed identically until we obtain
    \begin{equation*}
        F\left((\operatorname{Id}+v)_\#\mu\right) = F(\mu) + \left\langle \nabla_\mu F(\mu, \cdot), v\right\rangle_{L_\mu^2} + \int_0^1 (1-t)\left\langle \widetilde{\operatorname{H}}_{\mu, t}v, v\right\rangle_{L^2_{\mu}}\mathrm{d}t.
    \end{equation*}
    Hence, by Lemma \ref{lem:lip-curves} with $s=0$,
    \begin{align*}
        F\left((\operatorname{Id}+v)_\#\mu\right) 
        &=F(\mu) + \left\langle \nabla_\mu F(\mu, \cdot), v\right\rangle_{L_\mu^2} + \frac{1}{2}\left\langle \operatorname{H}_{\mu}v, v\right\rangle_{L^2_\mu} +\int_0^1 (1-t)\left\langle (\widetilde{\operatorname{H}}_{\mu, t} - \operatorname{H}_\mu)v, v \right\rangle_{L_\mu^2}\mathrm{d}t\\
        &\leq F(\mu) + \left\langle \nabla_\mu F(\mu, \cdot), v\right\rangle_{L_\mu^2} + \frac{1}{2}\left\langle \operatorname{H}_{\mu}v, v\right\rangle_{L^2_\mu} + \int_0^1 (1-t)\|\widetilde{\operatorname{H}}_{\mu, t}-\operatorname{H}_\mu\|_{\mathrm{op}}\|v\|_{L^2_\mu}^2\mathrm{d}t\\
        &\leq F(\mu) + \left\langle \nabla_\mu F(\mu, \cdot), v\right\rangle_{L_\mu^2} + \frac{1}{2}\left\langle \operatorname{H}_{\mu}v, v\right\rangle_{L^2_\mu} + \frac{L_{\operatorname{M}}+L_{\operatorname{K}}}{6}\|v\|_{L^2_\mu}^3.
    \end{align*}
\end{proof}

\begin{lemma}[Composition identity for the regularized Hessian]\label{lem:intertwining_preconditioner}
Let Assumption \ref{assumption:smooth-F} hold. Let $\mu^*\in\mathcal P_2(\mathbb R^d)$, let $\operatorname{T}_t:\mathbb R^d\to\mathbb R^d$ be measurable and set $\mu_t:=(\operatorname{T}_t)_\#\mu^*$ for $t \in [0,1)$. Define
\begin{equation*}
    \widetilde {\operatorname{H}}_{\mu,t}[v](x) := \nabla\nabla_\mu F(\mu_t,\operatorname{T}_t(x))v(x) + \int_{\mathbb R^d}\nabla_\mu^2F(\mu_t,\operatorname{T}_t(x),\operatorname{T}_t(y))v(y)\mathrm d\mu^*(y), \quad v \in L_{\mu^*}^2,
\end{equation*}
\begin{equation*}
    \operatorname{H}_{\mu_t}[w](x) := \nabla\nabla_\mu F(\mu_t,x)w(x) + \int_{\mathbb R^d}\nabla_\mu^2F(\mu_t,x,y)w(y)\mathrm d\mu_t(y), \quad w \in L_{\mu_t}^2.
\end{equation*}
Then
\begin{equation*}
(\widetilde {\operatorname{H}}_{\mu,t} ^2+\beta\operatorname{I_{d \times d}})^{-\frac{1}{2}}(\nabla_\mu F(\mu_t)\circ \operatorname{T}_t)
=
\Big((\operatorname{H}_{\mu_t}^2+\beta\operatorname{I_{d \times d}})^{-\frac{1}{2}}\nabla_\mu F(\mu_t)\Big)\circ \operatorname{T}_t, \quad \beta > 0.
\end{equation*}
\end{lemma}
\begin{proof}
Define the composition operator $\operatorname{J}_t:L^2_{\mu_t}\to L^2_{\mu^*}$ by $(\operatorname{J}_t w)(x):=w(\operatorname{T}_t(x))$. For $w\in L^2_{\mu_t}$,
\[
\|\operatorname{J}_t w\|_{L^2_{\mu^*}}^2 = \int_{\mathbb R^d}|w(\operatorname{T}_t(x))|^2\mathrm d\mu^*(x) = \int_{\mathbb R^d}|w(y)|^2\mathrm d\mu_t(y),
\]
because $\mu_t=(\operatorname{T}_t)_\#\mu^*$. Thus
\[
\|\operatorname{J}_t w\|_{L^2_{\mu^*}}=\|w\|_{L^2_{\mu_t}}.
\]
Then
\begin{align*}
\widetilde{\operatorname{H}}_{\mu,t}[\operatorname{J}_t w](x)
=\nabla\nabla_\mu F(\mu_t,\operatorname{T}_t(x))w(\operatorname{T}_t(x))+
\int_{\mathbb R^d}\nabla_\mu^2F(\mu_t,\operatorname{T}_t(x),\operatorname{T}_t(\bar x))w(\operatorname{T}_t(\bar x))\mathrm d\mu^*(\bar x).
\end{align*}
For the first term,
\[
\nabla\nabla_\mu F(\mu_t,\operatorname{T}_t(x))w(\operatorname{T}_t(x))
=
\big(\nabla\nabla_\mu F(\mu_t,\cdot)w(\cdot)\big)(\operatorname{T}_t(x)).
\]
For the second term, use again $\mu_t=(\operatorname{T}_t)_\#\mu^*$ to get
\begin{align*}
\int_{\mathbb R^d}\nabla_\mu^2F(\mu_t,\operatorname{T}_t(x),\operatorname{T}_t(\bar x))w(\operatorname{T}_t(\bar x))\mathrm d\mu^*(\bar x) =
\int_{\mathbb R^d}\nabla_\mu^2F(\mu_t,\operatorname{T}_t(x),\bar y)w(\bar y)\mathrm d\mu_t(\bar y).
\end{align*}
Therefore
\[
\widetilde {\operatorname{H}}_{\mu,t}[\operatorname{J}_t w](x) =
(\operatorname{H}_{\mu_t}w)(\operatorname{T}_t(x)) = \operatorname{J}_t(\operatorname{H}_{\mu_t}w)(x).
\]
Hence
\[
\widetilde {\operatorname{H}}_{\mu,t} \operatorname{J}_t = \operatorname{J}_t \operatorname{H}_{\mu_t}.
\]
We prove that for every integer $k\ge 1$,
\begin{equation*}
\widetilde {\operatorname{H}}_{\mu,t}^k \operatorname{J}_t = \operatorname{J}_t \operatorname{H}_{\mu_t}^k.
\end{equation*} 
by induction on $k$, where $\operatorname{H}^k$ denotes the $k$-times composition of the operator $\operatorname{H}$ with itself. The case $k=1$ is already proved. Assume \begin{equation*}
\widetilde {\operatorname{H}}_{\mu,t}^k \operatorname{J}_t = \operatorname{J}_t \operatorname{H}_{\mu_t}^k,
\end{equation*} 
for some $k \geq 1$. Then
\[
\widetilde {\operatorname{H}}_{\mu,t}^{k+1}\operatorname{J}_t
=
\widetilde {\operatorname{H}}_{\mu,t}(\widetilde {\operatorname{H}}_{\mu,t}^{k}\operatorname{J}_t)
=
\widetilde {\operatorname{H}}_{\mu,t}(\operatorname{J}_t \operatorname{H}_{\mu_t}^{k})
=
(\widetilde {\operatorname{H}}_{\mu,t} \operatorname{J}_t)\operatorname{H}_{\mu_t}^{k}
=
\operatorname{J}_t \operatorname{H}_{\mu_t}^{k+1},
\]
which concludes the induction. Let $
p_n(\lambda)=\sum_{j=0}^n a_j\lambda^j$, with $\{a_j\}_{j=0}^n \subset \mathbb R$.
Then
\[
p_n(\widetilde {\operatorname{H}}_{\mu,t} )\operatorname{J}_t
=
\sum_{j=0}^n a_j \widetilde {\operatorname{H}}_{\mu,t} ^{j}\operatorname{J}_t
=
\sum_{j=0}^n a_j \operatorname{J}_t \operatorname{H}_{\mu_t}^{j}
=\operatorname{J}_t\sum_{j=0}^n a_j \operatorname{H}_{\mu_t}^{j}
=
\operatorname{J}_t p_n(\operatorname{H}_{\mu_t}).
\]
Because $\operatorname{H}_{\mu_t}$ and $\widetilde {\operatorname{H}}_{\mu,t} $ are bounded self-adjoint due to Assumption \ref{assumption:smooth-F}, their spectra are compact subsets of $[-M,M]$, where $M:=\max\{\|\operatorname{H}_{\mu_t}\|_{\mathrm{op}},\|\widetilde {\operatorname{H}}_{\mu,t} \|_{\mathrm{op}}\}$. Let $f \in C([-M,M])$. By the Weierstrass theorem, there exists a sequence of polynomials $(p_n)_n$ such that $\sup_{\lambda\in[-M,M]}|p_n(\lambda)-f(\lambda)|\to 0$. Hence, by functional calculus,
\[
\|p_n(\operatorname{H}_{\mu_t})-f(\operatorname{H}_{\mu_t})\|_{\mathrm{op}}\to 0,
\qquad
\|p_n(\widetilde {\operatorname{H}}_{\mu,t} )-f(\widetilde {\operatorname{H}}_{\mu,t} )\|_{\mathrm{op}}\to 0, \quad n \to \infty.
\]
From earlier,
\[
p_n(\widetilde {\operatorname{H}}_{\mu,t} )\operatorname{J}_t = \operatorname{J}_t p_n(\operatorname{H}_{\mu_t}).
\]
Passing to the limit in operator norm and using that $\operatorname{J}_t$ is bounded since it is an isometry, we obtain
\[
f(\widetilde {\operatorname{H}}_{\mu,t} )\operatorname{J}_t = \operatorname{J}_t f(\operatorname{H}_{\mu_t}).
\]
Take $ f(\lambda)=\frac{1}{\sqrt{\beta+\lambda^2}}, \beta > 0$. Then
\[
(\widetilde {\operatorname{H}}_{\mu,t} ^2+\beta\operatorname{I_{d \times d}})^{-\frac{1}{2}}\operatorname{J}_t = \operatorname{J}_t(\operatorname{H}_{\mu_t}^2+\beta\operatorname{I_{d \times d}})^{-\frac{1}{2}}.
\]
Evaluating this identity on $w\in L^2_{\mu_t}$ yields
\[
(\widetilde {\operatorname{H}}_{\mu,t} ^2+\beta\operatorname{I_{d \times d}})^{-\frac{1}{2}}(w\circ \operatorname{T}_t)
=
\Big((\operatorname{H}_{\mu_t}^2+\beta\operatorname{I_{d \times d}})^{-\frac{1}{2}}w\Big)\circ \operatorname{T}_t,
\]
Taking $w=\nabla_\mu F(\mu_t)$ gives the conclusion.
\end{proof}

\begin{lemma}[Lipschitz estimates for the regularized WSFN spectral maps]
\label{lem:regularized-wsfn-spectral-maps}
Let $\mathcal H$ be a complex or real Hilbert space, let $A,B\in\mathcal B(\mathcal H)$ be bounded self-adjoint operators, and let $\beta>0$. Define
\begin{equation}
\label{eq:regularized-preconditioner-and-sign}
    P_\beta(A) := \left(A^2+\beta \operatorname{I_{d \times d}}\right)^{-1/2}, \quad
    S_\beta(A) := A\left(A^2+\beta \operatorname{I_{d \times d}}\right)^{-1/2}.
\end{equation}
Then
\begin{equation}
\label{eq:inverse-square-root-operator-lipschitz}
    \left\|P_\beta(A)-P_\beta(B)\right\|_{\mathrm{op}} \leq \frac{2}{\pi\beta} \|A-B\|_{\mathrm{op}},
\end{equation}
and
\begin{equation}
\label{eq:regularized-sign-operator-lipschitz}
    \left\|S_\beta(A)-S_\beta(B)\right\|_{\mathrm{op}} \leq \frac{1}{\sqrt{\beta}}\|A-B\|_{\mathrm{op}}.
\end{equation}
Moreover, $\|P_\beta(A)\|_{\mathrm{op}} \leq \frac{1}{\sqrt{\beta}},$ $\|S_\beta(A)\|_{\mathrm{op}} \leq 1.$ The dependence $\beta^{-1/2}$ in \eqref{eq:regularized-sign-operator-lipschitz} is sharp.
\end{lemma}

\begin{proof}
We first consider a complex Hilbert space. For a real Hilbert
space, the same argument can be applied to its complexification. Self-adjointness and the relevant operator norms are preserved.

For $a>0$, define
\begin{equation*}
\label{eq:resolvent-spectral-maps}
    Q_a(A) := \left(A^2+a^2\operatorname{I_{d \times d}}\right)^{-1},
    \quad
    R_a(A) := A\left(A^2+a^2\operatorname{I_{d \times d}}\right)^{-1}.
\end{equation*}
For any nonzero vector $v \in \mathcal{H},$ we have
\begin{equation*}
    \langle (A\pm ia\operatorname{I_{d \times d}})v,v \rangle = \langle Av,v\rangle \pm ia\|v\|^2.
\end{equation*}
Since $A$ is self-adjoint, its eigenvalues are real, and $\langle Av,v\rangle = \langle v,Av\rangle\in \mathbb R$. Taking squared norm
\begin{equation*}
    \|(A\pm ia\operatorname{I_{d \times d}})v\|^2 = \|Av\|^2+a^2\|v\|^2 \geq a^2\|v\|^2. 
\end{equation*}
Taking square root gives
\begin{equation*}
    \|(A\pm ia\operatorname{I_{d \times d}})v\| \geq a\|v\|, 
\end{equation*}
for all nonzero $v \in \mathcal{H}.$ Therefore, the operators $A\pm \mathrm{i}a\operatorname{I_{d \times d}}$ are invertible and
\begin{equation*}
\label{eq:self-adjoint-resolvent-bound}
    \left\|\left(A\pm \mathrm{i}a\operatorname{I_{d \times d}}\right)^{-1}
    \right\|_{\mathrm{op}}
    \leq
    \frac{1}{a}.
\end{equation*}
The resolvent identity \cite[Theorem 4.8.2]{hille1957functional} gives, for either choice of sign,
\begin{equation*}
\left(A\pm\mathrm{i}a\operatorname{I_{d \times d}}\right)^{-1}-
\left(B\pm\mathrm{i}a\operatorname{I_{d \times d}}\right)^{-1} =
\left(A\pm\mathrm{i}a\operatorname{I_{d \times d}}\right)^{-1}
(B-A)
\left(B\pm\mathrm{i}a\operatorname{I_{d \times d}}\right)^{-1}.
\end{equation*}
Hence,
\begin{equation}
\label{eq:resolvent-difference-bound}
    \left\|\left(A\pm\mathrm{i}a\operatorname{I_{d \times d}}\right)^{-1}-
        \left(B\pm\mathrm{i}a\operatorname{I_{d \times d}}\right)^{-1}\right\|_{\mathrm{op}} \leq \frac{1}{a^2} \|A-B\|_{\mathrm{op}}.
\end{equation}

Using the functional calculus and the fact that $A-\mathrm{i}a\operatorname{I_{d \times d}}$ and $A+\mathrm{i}a\operatorname{I_{d \times d}}$ commute, we have
\begin{equation}
\label{eq:Qa-resolvent-representation}
    Q_a(A) = \frac{1}{2\mathrm{i}a}\left[\left(A-\mathrm{i}a\operatorname{I_{d \times d}}\right)^{-1}-\left(A+\mathrm{i}a\operatorname{I_{d \times d}}\right)^{-1}\right]
\end{equation}
and
\begin{equation}
\label{eq:Ra-resolvent-representation}
    R_a(A) =\frac{1}{2}\left[\left(A-\mathrm{i}a\operatorname{I_{d \times d}}\right)^{-1}+\left(A+\mathrm{i}a\operatorname{I_{d \times d}}\right)^{-1}
    \right].
\end{equation}
Combining
\eqref{eq:resolvent-difference-bound} with
\eqref{eq:Qa-resolvent-representation} gives
\begin{equation}
\label{eq:Qa-difference-bound}
    \|Q_a(A)-Q_a(B)\|_{\mathrm{op}}
    \leq
    \frac{1}{a^3}
    \|A-B\|_{\mathrm{op}},
\end{equation}
whereas
\eqref{eq:resolvent-difference-bound} and
\eqref{eq:Ra-resolvent-representation} give
\begin{equation}
\label{eq:Ra-difference-bound}
    \|R_a(A)-R_a(B)\|_{\mathrm{op}}
    \leq
    \frac{1}{a^2}
    \|A-B\|_{\mathrm{op}}.
\end{equation}

We now use the scalar identities
\begin{equation*}
\label{eq:scalar-inverse-square-root-integral}
    \frac{1}{\sqrt{x^2+\beta}} = \frac{2}{\pi}\int_0^\infty\frac{1}{x^2+\beta+t^2}\mathrm{d}t
\end{equation*}
and
\begin{equation*}
\label{eq:scalar-regularized-sign-integral}
    \frac{x}{\sqrt{x^2+\beta}} = \frac{2}{\pi}\int_0^\infty \frac{x}{x^2+\beta+t^2}\mathrm{d}t.
\end{equation*}
Since $A$ and $B$ are bounded, the corresponding integrals
converge in operator norm. Indeed, using the sub-multiplicativity of the operator norm
\begin{equation*}
    \left\Vert{} A \left( A^2 + (\beta + t^2)\operatorname{I_{d \times d}} \right)^{-1} \right\Vert{}_{\mathrm{op}}  \leq \Vert{}A\Vert{}_{\mathrm{op}} \left\Vert{} \left( A^2 + (\beta + t^2)\operatorname{I_{d \times d}} \right)^{-1} \right\Vert{}_{\mathrm{op}}
\end{equation*}
Since operator $A$ is self-adjoint, its spectrum $\sigma(A) \subset \mathbb{R}$, the spectrum of $A^2 + (\beta + t^2)\operatorname{I_{d \times d}}$ lies in $[\beta + t^2, \infty)$. Therefore
\begin{equation*}
\left\Vert{} \left( A^2 + (\beta + t^2)\operatorname{I_{d \times d}} \right)^{-1} \right\Vert{}_{\mathrm{op}}  = \sup_{\lambda \in \sigma(A)} \frac{1}{\lambda^2 + \beta + t^2}  \leq \frac{1}{\beta + t^2}.    
\end{equation*}
Multiplying by $\Vert{}A\Vert{}_{\mathrm{op}}$ gives
\begin{equation*}
    \left\Vert{} A \left( A^2 + (\beta + t^2)\operatorname{I_{d \times d}} \right)^{-1} \right\Vert{}_{\mathrm{op}} \leq \frac{\Vert{}A\Vert{}_{\mathrm{op}}}{\beta + t^2}.
\end{equation*}
Hence
\begin{equation*}
\int_0^\infty \left\Vert{} A \left( A^2 + (\beta + t^2)I \right)^{-1} \right\Vert{}_{\mathrm{op}} \mathrm{d}t  \leq \Vert{}A\Vert{}_{\mathrm{op}} \int_0^\infty \frac{1}{\beta + t^2} \mathrm{d}t = \frac{\Vert{}A\Vert{}_{\mathrm{op}} \, \pi}{2 \sqrt{\beta}} < \infty    
\end{equation*}
Since the scalar integral converges for $\beta > 0$, the operator integral converges absolutely in the operator norm. Therefore, by continuous functional calculus,
\begin{equation}
\label{eq:operator-inverse-square-root-integral}
    P_\beta(A) = \frac{2}{\pi}\int_0^\infty Q_{\sqrt{\beta+t^2}}(A)\mathrm{d}t
\end{equation}
and
\begin{equation}
\label{eq:operator-regularized-sign-integral}
    S_\beta(A) = \frac{2}{\pi}\int_0^\infty R_{\sqrt{\beta+t^2}}(A)\mathrm{d}t.
\end{equation}

It follows from
\eqref{eq:operator-inverse-square-root-integral} and
\eqref{eq:Qa-difference-bound} that
\begin{align*}
    \|P_\beta(A)-P_\beta(B)\|_{\mathrm{op}} \leq \frac{2}{\pi} \int_0^\infty\frac{\|A-B\|_{\mathrm{op}}}
    {(\beta+t^2)^{3/2}}\mathrm{d}t =\frac{2}{\pi\beta}\|A-B\|_{\mathrm{op}},
\end{align*}
where we used $\int_0^\infty \frac{1}{(\beta+t^2)^{3/2}}\mathrm{d}t = \frac{1}{\beta}.$ This proves \eqref{eq:inverse-square-root-operator-lipschitz}.

Similarly, by \eqref{eq:operator-regularized-sign-integral} and \eqref{eq:Ra-difference-bound},
\begin{align*}
    \|S_\beta(A)-S_\beta(B)\|_{\mathrm{op}}\leq \frac{2}{\pi}\int_0^\infty
    \frac{\|A-B\|_{\mathrm{op}}}{\beta+t^2}\mathrm{d}t =\frac{1}{\sqrt{\beta}}\|A-B\|_{\mathrm{op}},
\end{align*}
because $\int_0^\infty \frac{1}{\beta+t^2}\mathrm{d}t= \frac{\pi}{2\sqrt{\beta}}.$ This proves
\eqref{eq:regularized-sign-operator-lipschitz}.

The norm bounds follow immediately from the scalar inequalities
\begin{equation*}
    \frac{1}{\sqrt{x^2+\beta}} \leq \frac{1}{\sqrt{\beta}},
    \quad \frac{|x|}{\sqrt{x^2+\beta}} \leq 1,
\end{equation*}
and the spectral calculus.

Finally, the dependence $\beta^{-1/2}$ is sharp already in dimension one. Indeed, for $s_\beta(x):= \frac{x}{\sqrt{x^2+\beta}},$ we have $s_\beta'(0)= \frac{1}{\sqrt{\beta}}.$ Hence no smaller uniform Lipschitz constant is possible.
\end{proof}

\begin{lemma}[First-order expansion of the preconditioned gradient]
\label{lem:preconditioned-gradient-expansion-sharp}
Let Assumptions \ref{assumption:smooth-F}, \ref{assumption:Hessian-lipschitz} hold. Let $\mu^*\in\mathcal P_2(\mathbb R^d)$ and assume that, for some $R_F>0$,
\begin{equation}
\label{eq:multiplication-part-controlled-at-reference}
    \|\operatorname{M}_{\mu^*}\|_{\mathrm{op}}  \leq R_F \|\nabla_\mu F(\mu^*,\cdot)\|_{L^2_{\mu^*}}.
\end{equation}
In particular, condition~\eqref{eq:multiplication-part-controlled-at-reference} follows from Assumption~\ref{assumption_regularity_wg} whenever $\mu^*\in S$ and $\|\nabla_\mu F(\mu^*,\cdot)\|_{L^2_{\mu^*}} \le\varepsilon_0$. Let $Z,\widetilde Z \in L^2_{\mu^*}(\mathbb R^d;\mathbb R^d)$ be measurable maps, and define $\mu_Z:=Z_\#\mu^*,$ $\mu_{\widetilde Z}:=\widetilde Z_\#\mu^*.$ For $\beta>0$, define the pulled-back preconditioned gradient by
\begin{equation*}
\label{eq:pulled-back-preconditioned-gradient-sharp}
    \mathfrak G_\beta(Z) := \left[\left(\operatorname{H}_{\mu_Z}^2+\beta\operatorname{I_{d\times d}}
        \right)^{-1/2}\nabla_\mu F(\mu_Z,\cdot)\right]\circ Z.
\end{equation*}
Then
\begin{equation*}
\label{eq:sharp-preconditioned-gradient-linearization}
    \mathfrak G_\beta(\widetilde Z)- \mathfrak G_\beta(Z) =
    \left(\operatorname{K}_{\mu^*}^2+\beta\operatorname{I_{d \times d}}
    \right)^{-1/2}\operatorname{K}_{\mu^*}\left(\widetilde Z-Z\right) + R_\beta(Z,\widetilde Z)
\end{equation*}
in $L^2_{\mu^*}$. Then the remainder satisfies
\begin{equation}
\label{eq:sharp-preconditioned-gradient-remainder}
\begin{aligned}
    \|R_\beta(Z,\widetilde Z)\|_{L^2_{\mu^*}}
    &\leq
    \Bigg[\frac{2L_H}{\sqrt{\beta}}\left(\|Z-\operatorname{Id}\|_{L^2_{\mu^*}}+\|\widetilde Z-\operatorname{Id}\|_{L^2_{\mu^*}}\right)\\
        &+ \frac{L_H^2}{2\pi\beta}\left(\|Z-\operatorname{Id}\|_{L^2_{\mu^*}}^2+\|\widetilde Z-\operatorname{Id}\|_{L^2_{\mu^*}}^2
        \right)\\
        &+\left(\frac{R_F}{\sqrt{\beta}}+\frac{2L_H}{\pi\beta}
        \right)\|\nabla_\mu F(\mu^*,\cdot)\|_{L^2_{\mu^*}}
    \Bigg]\|\widetilde Z-Z\|_{L^2_{\mu^*}}.
\end{aligned}
\end{equation}
If $\|Z-\operatorname{Id}\|_{L^2_{\mu^*}} \leq r,$ $\|\widetilde Z-\operatorname{Id}\|_{L^2_{\mu^*}} \leq r,$ then
\begin{equation}
\label{eq:sharp-preconditioned-gradient-remainder-ball}
    \|R_\beta(Z,\widetilde Z)\|_{L^2_{\mu^*}}\leq
    \Bigg[\frac{4L_Hr}{\sqrt{\beta}}+\frac{L_H^2r^2}{\pi\beta}+\left(\frac{R_F}{\sqrt{\beta}}+
            \frac{2L_H}{\pi\beta}
        \right)\|\nabla_\mu F(\mu^*,\cdot)\|_{L^2_{\mu^*}}\Bigg]\|\widetilde Z-Z\|_{L^2_{\mu^*}}.
\end{equation}
\end{lemma}

\begin{proof}
For simplicity, set $u:=Z-\operatorname{Id},$ $v:=\widetilde Z-\operatorname{Id},$ and $h:=v-u=\widetilde Z-Z.$ Thus, For $w\in L_{\mu^*}^2$, define $Z_w:=\operatorname{Id}+w,$ $\mu_w:=(Z_w)_\#\mu^*,$ and introduce the pulled-back Hessian $\widehat H_w: L^2_{\mu^*}\to L^2_{\mu^*}$ by \begin{equation*}
\label{eq:pulled-back-hessian-modified-l17}
    \widehat H_w[\varphi](x) := \nabla\nabla_\mu F\left(\mu_w,Z_w(x)\right)\varphi(x) + \int_{\mathbb R^d} \nabla_\mu^2F\left(\mu_w,Z_w(x),Z_w(y)\right)\varphi(y)\mu^*(\mathrm dy).
\end{equation*}
By Lemma \ref{lem:two-map-pulled-back-hessian-lipschitz},
\begin{equation}
\label{eq:pulled-back-hessian-lipschitz-modified-l17}
    \|\widehat{\operatorname{H}}_w-\widehat{\operatorname{H}}_z\|_{\mathrm{op}} \leq
    L_H \|w-z\|_{ L^2_{\mu^*}}, \quad w,z\in L^2_{\mu^*}.
\end{equation}
Moreover,
\begin{equation}
\label{eq:reference-pulled-back-hessian}
    \widehat{\operatorname{H}}_0 = \operatorname{H}_{\mu^*} = \operatorname{M}_{\mu^*}+\operatorname{K}_{\mu^*}.
\end{equation}
For $w\in L^2_{\mu^*}$, set
\begin{equation*}
\label{eq:pulled-back-gradient-and-spectral-maps}
    g_w := \nabla_\mu F(\mu_w,\cdot)\circ Z_w, \
    P_w := \left(\widehat {\operatorname{H}}_w^2+\beta \operatorname{I_{d \times d}}\right)^{-1/2}, \ 
    S_w:= \widehat{\operatorname{H}}_w\left(\widehat{\operatorname{H}}_w^2+\beta \operatorname{I_{d \times d}}\right)^{-1/2}.
\end{equation*}
Since $P_w$ is a spectral function of the self-adjoint operator $\widehat {\operatorname{H}}_w$, it commutes with $\widehat{\operatorname{H}}_w$ \cite[Theorem VII.2 (g)]{ReedSimon1980}, and therefore
\begin{equation}
\label{eq:normalized-hessian-commutation}
    S_w=P_w \widehat{\operatorname{H}}_w=\widehat{\operatorname{H}}_wP_w.
\end{equation}
The composition identity in Lemma \ref{lem:intertwining_preconditioner} gives
\begin{equation}
\label{eq:composition-identity-used-modified-l17}
    \mathfrak G_\beta(Z_w) = P_wg_w.
\end{equation}
We first derive a convenient decomposition of $g_w$. Along the curve $t\mapsto \mu_{tw}=(\operatorname{Id}+tw)_\#\mu^*,$ the pulled-back gradient differentiation identity gives
\begin{equation*}
    \frac{\mathrm d}{\mathrm dt}g_{tw} = \widehat{\operatorname{H}}_{tw}w
    \text{ in } L_{\mu^*}^2.
\end{equation*}
Integrating from $0$ to $1$ gives $ g_w = g_0 + \int_0^1 \widehat{\operatorname{H}}_{tw}w \mathrm dt.$ Define $E(w) := \int_0^1\left(\widehat{\operatorname{H}}_{tw}-\widehat {\operatorname{H}}_w\right)w \mathrm dt.$ Then
\begin{equation}
\label{eq:gradient-endpoint-hessian-decomposition}
    g_w = g_0+\widehat{\operatorname{H}}_ww+E(w).
\end{equation}
By \eqref{eq:pulled-back-hessian-lipschitz-modified-l17},
\begin{align}
    \|E(w)\|_{L_{\mu^*}^2} \leq \int_0^1 \|\widehat{\operatorname{H}}_{tw}-\widehat{\operatorname{H}}_w\|_{\mathrm{op}}\|w\|_{L_{\mu^*}^2}\mathrm dt \leq L_H \int_0^1(1-t)\|w\|_{L_{\mu^*}^2}^2\mathrm dt =
    \frac{L_H}{2} \|w\|_{L_{\mu^*}^2}^2.
\label{eq:E-quadratic-bound}
\end{align}

We also need a two-point estimate for $E$. Adding and subtracting $(\widehat{\operatorname{H}}_{tv}-\widehat{\operatorname{H}}_v)u$ gives
\begin{align*}
    E(v)-E(u) =\int_0^1\left(\widehat{\operatorname{H}}_{tv}-\widehat{\operatorname{H}}_v
    \right)(v-u)\mathrm dt +
    \int_0^1
    \left[
        \left(
            \widehat{\operatorname{H}}_{tv}-\widehat{\operatorname{H}}_v
        \right)
        -
        \left(
            \widehat{\operatorname{H}}_{tu}-\widehat{\operatorname{H}}_u
        \right)
    \right]u\mathrm dt.
\end{align*}
For the first term,
\begin{align}
    \left\|
        \int_0^1
        \left(
            \widehat{\operatorname{H}}_{tv}-\widehat{\operatorname{H}}_v
        \right)(v-u)\mathrm dt
    \right\|_{L_{\mu^*}^2} \leq
    L_H
    \int_0^1
    (1-t)\|v\|_{L_{\mu^*}^2}
    \|v-u\|_{L_{\mu^*}^2}\mathrm dt =
    \frac{L_H}{2}
    \|v\|_{L_{\mu^*}^2}
    \|v-u\|_{L_{\mu^*}^2}.
\label{eq:E-difference-first-term}
\end{align}
For the second term, again using
\eqref{eq:pulled-back-hessian-lipschitz-modified-l17},
\begin{align*}
    \left\|
        \left(
            \widehat{\operatorname{H}}_{tv}-\widehat{\operatorname{H}}_v
        \right)
        -
        \left(
            \widehat{\operatorname{H}}_{tu}-\widehat{\operatorname{H}}_u
        \right)
    \right\|_{\mathrm{op}} \leq
    \|\widehat{\operatorname{H}}_{tv}-\widehat{\operatorname{H}}_{tu}\|_{\mathrm{op}}
    +
    \|\widehat{\operatorname{H}}_v-\widehat{\operatorname{H}}_u\|_{\mathrm{op}} \leq
    L_H(1+t)
    \|v-u\|_{L_{\mu^*}^2}.
\end{align*}
Hence,
\begin{equation}
\begin{aligned}
    \left\|\int_0^1
        \left[
            \left(
                \widehat{\operatorname{H}}_{tv}-\widehat{\operatorname{H}}_v
            \right)
            -
            \left(
                \widehat{\operatorname{H}}_{tu}-\widehat{\operatorname{H}}_u
            \right)
        \right]u\mathrm dt
    \right\|_{L_{\mu^*}^2} &\leq
    L_H
    \|u\|_{L_{\mu^*}^2}
    \|v-u\|_{L_{\mu^*}^2}
    \int_0^1(1+t)\mathrm dt\\
    &=
    \frac{3L_H}{2}
    \|u\|_{L_{\mu^*}^2}
    \|v-u\|_{L_{\mu^*}^2}.
\label{eq:E-difference-second-term}
\end{aligned}
\end{equation}
Combining \eqref{eq:E-difference-first-term} and \eqref{eq:E-difference-second-term}, we obtain
\begin{equation}
\label{eq:E-two-point-bound}
    \|E(v)-E(u)\|_{L_{\mu^*}^2}  \leq \frac{L_H}{2} \left(\|v\|_{L_{\mu^*}^2}
        + 3\|u\|_{L_{\mu^*}^2} \right) \|v-u\|_{L_{\mu^*}^2}.
\end{equation}
We now expand the preconditioned gradient. Using \eqref{eq:gradient-endpoint-hessian-decomposition}, \eqref{eq:normalized-hessian-commutation}, and \eqref{eq:composition-identity-used-modified-l17}, we have
\begin{equation}
\begin{aligned}
    \mathfrak G_\beta(\widetilde Z)
    -
    \mathfrak G_\beta(Z)
    &=
    P_vg_v-P_ug_u
    \notag\\
    &=
    P_v
    \left(
        g_0+\widehat{\operatorname{H}}_vv+E(v)
    \right)
    -
    P_u
    \left(
        g_0+\widehat{\operatorname{H}}_uu+E(u)
    \right)
    \notag\\
    &=
    S_v(v-u)
    +
    (S_v-S_u)u
    \notag\\
    &+
    P_v\left(E(v)-E(u)\right)
    +
    (P_v-P_u)E(u)
    +
    (P_v-P_u)g_0.
\label{eq:preconditioned-gradient-five-term-decomposition}
\end{aligned}
\end{equation}

Let
\begin{equation}
\label{eq:reference-normalized-hessian}
    S_* :=
    \operatorname{K}_{\mu^*}
    \left(
        \operatorname{K}_{\mu^*}^2+\beta \operatorname{I_{d \times d}}
    \right)^{-1/2}.
\end{equation}
Because $S_*$ is a spectral function of $\operatorname{K}_{\mu^*}$, $S_*
    =
    \left(
        \operatorname{K}_{\mu^*}^2+\beta \operatorname{I_{d \times d}}
    \right)^{-1/2}
    \operatorname{K}_{\mu^*}.$
Subtracting $S_*(v-u)$ from \eqref{eq:preconditioned-gradient-five-term-decomposition} gives
\begin{equation}
\begin{aligned}
    R_\beta(Z,\widetilde Z)
    &=
    (S_v-S_*)(v-u)
    +
    (S_v-S_u)u\\
    &+
    P_v\left(E(v)-E(u)\right)
    +
    (P_v-P_u)E(u)
    +
    (P_v-P_u)g_0.
\label{eq:sharp-remainder-five-terms}
\end{aligned}
\end{equation}

We estimate the five terms in \eqref{eq:sharp-remainder-five-terms}. By
\eqref{eq:reference-pulled-back-hessian}, \eqref{eq:pulled-back-hessian-lipschitz-modified-l17}, and \eqref{eq:multiplication-part-controlled-at-reference},
\begin{align*}
    \|\widehat{\operatorname{H}}_v-\operatorname{K}_{\mu^*}\|_{\mathrm{op}} \leq
    \|\widehat{\operatorname{H}}_v-\widehat{\operatorname{H}}_0\|_{\mathrm{op}}
    +
    \|\operatorname{M}_{\mu^*}\|_{\mathrm{op}}\leq
    L_H\|v\|_{L_{\mu^*}^2}
    +
    R_F\|g_0\|_{L_{\mu^*}^2}.
\end{align*}
Applying the Hessian estimate from Lemma \ref{lem:regularized-wsfn-spectral-maps}, i.e.,
\[
    \left\|
        A(A^2+\beta I)^{-1/2}
        -
        B(B^2+\beta I)^{-1/2}
    \right\|_{\mathrm{op}}
    \leq
    \frac{1}{\sqrt{\beta}}
    \|A-B\|_{\mathrm{op}},
\]
we obtain
\begin{align}
    \|(S_v-S_*)(v-u)\|_{L_{\mu^*}^2}
    &\leq
    \left(
        \frac{L_H}{\sqrt{\beta}}
        \|v\|_{L_{\mu^*}^2}
        +
        \frac{R_F}{\sqrt{\beta}}
        \|g_0\|_{L_{\mu^*}^2}
    \right)
    \|v-u\|_{L_{\mu^*}^2}.
\label{eq:remainder-term-one}
\end{align}
Similarly,
\begin{equation}
\label{eq:remainder-term-two}
    \|(S_v-S_u)u\|_{L_{\mu^*}^2}
    \leq
    \frac{L_H}{\sqrt{\beta}}
    \|u\|_{L_{\mu^*}^2}
    \|v-u\|_{L_{\mu^*}^2}.
\end{equation}

Since $\|P_v\|_{\mathrm{op}}\leq \frac{1}{\sqrt{\beta}},$ \eqref{eq:E-two-point-bound} gives
\begin{equation}
\label{eq:remainder-term-three}
\begin{aligned}
    \|P_v(E(v)-E(u))\|_{L_{\mu^*}^2}
    \leq
    \frac{L_H}{2\sqrt{\beta}}
    \left(
        \|v\|_{L_{\mu^*}^2}
        +
        3\|u\|_{L_{\mu^*}^2}
    \right)
    \|v-u\|_{L_{\mu^*}^2}.
\end{aligned}
\end{equation}

The other estimate from Lemma \ref{lem:regularized-wsfn-spectral-maps} and \eqref{eq:pulled-back-hessian-lipschitz-modified-l17} imply
\begin{equation}
\label{eq:preconditioner-two-point-bound}
    \|P_v-P_u\|_{\mathrm{op}} \leq \frac{2L_H}{\pi\beta}\|v-u\|_{L_{\mu^*}^2}.
\end{equation}
Combining \eqref{eq:preconditioner-two-point-bound} with \eqref{eq:E-quadratic-bound}, we get
\begin{equation}
\label{eq:remainder-term-four}
    \|(P_v-P_u)E(u)\|_{L_{\mu^*}^2}
    \leq
    \frac{L_H^2}{\pi\beta}
    \|u\|_{L_{\mu^*}^2}^2
    \|v-u\|_{L_{\mu^*}^2}.
\end{equation}
Finally,
\begin{equation}
\label{eq:remainder-term-five}
    \|(P_v-P_u)g_0\|_{L_{\mu^*}^2}
    \leq
    \frac{2L_H}{\pi\beta}
    \|g_0\|_{L_{\mu^*}^2}
    \|v-u\|_{L_{\mu^*}^2}.
\end{equation}

Combining
\eqref{eq:remainder-term-one}--\eqref{eq:remainder-term-five}
gives the asymmetric estimate
\begin{equation}
\begin{aligned}
    \|R_\beta(Z,\widetilde Z)\|_{L_{\mu^*}^2}
    &\leq
    \Bigg[
        \frac{L_H}{2\sqrt{\beta}}
        \left(
            5\|u\|_{L_{\mu^*}^2}
            +
            3\|v\|_{L_{\mu^*}^2}
        \right)\\
        &+
        \frac{L_H^2}{\pi\beta}
        \|u\|_{L_{\mu^*}^2}^2+
        \left(
            \frac{R_F}{\sqrt{\beta}}
            +
            \frac{2L_H}{\pi\beta}
        \right)
        \|g_0\|_{L_{\mu^*}^2}
    \Bigg]
    \|v-u\|_{L_{\mu^*}^2}.
\label{eq:asymmetric-sharp-remainder-one}
\end{aligned}
\end{equation}
Interchanging $u$ and $v$ in the preceding argument gives
\begin{equation}
\begin{aligned}
    \|R_\beta(Z,\widetilde Z)\|_{L_{\mu^*}^2}
    &\leq
    \Bigg[
        \frac{L_H}{2\sqrt{\beta}}
        \left(
            3\|u\|_{L_{\mu^*}^2}
            +
            5\|v\|_{L_{\mu^*}^2}
        \right)
        +
        \frac{L_H^2}{\pi\beta}
        \|v\|_{L_{\mu^*}^2}^2\\
        &+
        \left(
            \frac{R_F}{\sqrt{\beta}}
            +
            \frac{2L_H}{\pi\beta}
        \right)
        \|g_0\|_{L_{\mu^*}^2}
    \Bigg]
    \|v-u\|_{L_{\mu^*}^2}.
\label{eq:asymmetric-sharp-remainder-two}
\end{aligned}
\end{equation}
Indeed, after interchanging the two maps, the corresponding
remainder is $-R_\beta(Z,\widetilde Z)$. Since the norm of the remainder is bounded by both right-hand sides, it is also bounded by their arithmetic mean. Averaging \eqref{eq:asymmetric-sharp-remainder-one} and \eqref{eq:asymmetric-sharp-remainder-two} gives
\begin{align*}
    \|R_\beta(Z,\widetilde Z)\|_{L_{\mu^*}^2}
    &\leq
    \Bigg[
        \frac{2L_H}{\sqrt{\beta}}
        \left(
            \|u\|_{L_{\mu^*}^2}
            +
            \|v\|_{L_{\mu^*}^2}
        \right) +
        \frac{L_H^2}{2\pi\beta}
        \left(
            \|u\|_{L_{\mu^*}^2}^2
            +
            \|v\|_{L_{\mu^*}^2}^2
        \right)\\
        &+
        \left(
            \frac{R_F}{\sqrt{\beta}}
            +
            \frac{2L_H}{\pi\beta}
        \right)
        \|g_0\|_{L_{\mu^*}^2}
    \Bigg]
    \|v-u\|_{L_{\mu^*}^2}.
\end{align*}
Substituting $u=Z-\operatorname{Id},$ $v=\widetilde Z-\operatorname{Id},$ $g_0=\nabla_\mu F(\mu^*,\cdot)$ proves \eqref{eq:sharp-preconditioned-gradient-remainder}. If $\|u\|_{L_{\mu^*}^2},\|v\|_{L_{\mu^*}^2}\leq r$, then
\[
    \|u\|_{L_{\mu^*}^2}+\|v\|_{L_{\mu^*}^2}\leq 2r,
    \quad
    \|u\|_{L_{\mu^*}^2}^2+\|v\|_{L_{\mu^*}^2}^2\leq 2r^2,
\]
which gives \eqref{eq:sharp-preconditioned-gradient-remainder-ball}.
\end{proof}

\subsection{Well-posedness and characterization of the optimizer}
\label{sec:well-posedness-charac-optimizer}
In this subsection, we verify that the variational problems defining the Newton-type updates are well-posed. More precisely, we show that each update admits a unique minimizer in the corresponding Hilbert space and derive the associated first-order optimality condition. This yields the explicit form of the update whenever the relevant operator is invertible. We begin with the Levenberg--Marquardt regularization of the Wasserstein Newton step.
\\

\paragraph{Levenberg--Marquardt regularized Newton method.}
It is worth emphasizing that WSFN differs from the Wasserstein analogue of the Levenberg--Marquardt regularized Newton method, namely, 
\[
\mu^{n+1} = \left(\operatorname{Id} - \left(\operatorname{H}_{\mu^n}+\tau^{-1}\operatorname{I_{d\times d}}\right)^{-1}\nabla_\mu F(\mu^n,\cdot)\right)_\#\mu^n.
\]
The latter interpolates between Newton's method (as $\tau\to+\infty$) and gradient descent (as $\tau\to 0$). However, for the regularized Hessian $\operatorname{H}_{\mu^n}+\tau^{-1}\operatorname{I_{d\times d}}$ to be invertible, one must either assume that $\operatorname{H}_{\mu^n} \succeq 0$, or impose a restriction of the form $\tau<\lambda^{-1}$ when $\operatorname{H}_{\mu^n}\succeq -\lambda\operatorname{I_{d \times d}}$. In other words, the regularization $\tau^{-1}\operatorname{I_{d\times d}}$ suppresses negative curvature in non-convex regions, thereby enforcing a descent direction but potentially at the expense of a prohibitively small stepsize. By contrast, WSFN regularizes the non-negative operator $\operatorname{H}_{\mu^n}^2$. As a result, the parameter $\beta$ can be chosen independently of the eigenvalues of $\operatorname{H}_{\mu^n}$, enough to guarantee invertibility of $\operatorname{H}_{\mu^n}^2+\beta \operatorname{I_{d\times d}}$.

\begin{proposition}[Well-posedness and first-order optimality condition of Levenberg--Marquardt algorithm]
\label{prop:well-posedness}
Assume $F:\mathcal{P}_2(\mathbb R^d) \to \mathbb R$ is Wasserstein regular. Let $\mu^0 \in \mathcal{P}_2(\mathbb R^d)$ and $\tau > 0$. Define the sequence of iterates $(v^n)_{n \in \mathbb{N}}$ generated by the update 
\begin{align*} 
    &v^{n+1} = \argmin_{v \in L^2_{\mu^n}}\left\{\langle \nabla_\mu F(\mu^n,\cdot),v\rangle_{L_{\mu^n}^2} +\frac{1}{2}\langle \operatorname{H}_{\mu^n}v,v\rangle_{L_{\mu^n}^2}+\frac{1}{2\tau}\|v\|_{L_{\mu^n}^2}^2\right\},\\ &\mu^{n+1} = (\operatorname{Id}+ v^{n+1})_\#\mu^n.
\end{align*}
Assume that there exists $\delta > 0$ such that $\langle \operatorname{H}_{\mu^n}v,v\rangle_{L_{\mu_n}^2} \geq -\delta\|v\|_{L_{\mu_n}^2}^2$, for all $v \in L_{\mu_n}^2$, and $\tau < \delta^{-1}$, or $\operatorname{H}_{\mu^n}$ is non-negative, for all $n.$ Then the sequence $(v^n)_{n \in \mathbb{N}}$ is well-defined and remains in $L^2_{\mu^n}$, for all $n$. Moreover, $v^{n+1}$ satisfies
    \begin{equation*}
      v^{n+1} = - \left(\operatorname{H}_{\mu^n} + \tau^{-1}\operatorname{I_{d \times d}}\right)^{-1}\nabla_\mu F(\mu^n, \cdot)\ \text{ in } L^2_{\mu^n}.
    \end{equation*}
\end{proposition}
\begin{proof}
    For all $v \in L^2_{\mu^n}$ and $\tau > 0$, define 
    \begin{equation*}
        J_n(v) = \left\langle \nabla_\mu F(\mu^n, \cdot), v\right\rangle_{L^2_{\mu^n}} + \frac{1}{2}\langle \operatorname{H}_{\mu^n}v, v\rangle_{L^2_{\mu^n}} +\frac{1}{2\tau}\|v\|_{L^2_{\mu^n}}^2.
    \end{equation*}
    Note that, given $\mu^0 \in \mathcal{P}_2(\mathbb R^d)$, for each $n \in \mathbb N$, the functional $J_n(v)$ admits a unique minimizer, denoted $v^{n+1}$, since 
    \begin{equation*}
        J_n(v) \geq -\|\nabla_\mu F(\mu^n, \cdot)\|_{L^2_{\mu^n}}\|v\|_{L^2_{\mu^n}} + \frac{1}{2}\left(\frac{1}{\tau}-\delta\right)\|v\|_{L^2_{\mu^n}}^2.
    \end{equation*}

    Now, for $\varepsilon \in [0,1]$, and all $v \in L^2_{\mu^n}$, define $v^{n,\varepsilon} := v^{n+1} + \varepsilon(v - v^{n+1}) \in L^2_{\mu^n}$. 
    By linearity of the first term of $J_n$, 
    \begin{align*}
        \lim_{\varepsilon \downarrow 0} \frac{1}{\varepsilon}\left(\left\langle \nabla_\mu F(\mu^n, \cdot), v^{\varepsilon,n}\right\rangle_{L^2_{\mu^n}} - \left\langle \nabla_\mu F(\mu^n, \cdot), v^{n+1}\right\rangle_{L^2_{\mu^n}}\right) = \left\langle \nabla_\mu F(\mu^n, \cdot), v-v^{n+1}\right\rangle_{L^2_{\mu^n}}.
    \end{align*}
    By symmetry of $\operatorname{H}_{\mu^n}$, a straightforward calculation gives
    \begin{align*}
        &\lim_{\varepsilon \downarrow 0} \frac{1}{2\varepsilon}\left(\langle \operatorname{H}_{\mu^n}v^{n,\varepsilon}, v^{n,\varepsilon}\rangle_{L^2_{\mu^n}} - \langle \operatorname{H}_{\mu^n}v^{n+1}, v^{n+1}\rangle_{L^2_{\mu^n}}\right)\\
        &= \lim_{\varepsilon \downarrow 0} \frac{1}{\varepsilon}\left(\varepsilon\langle \operatorname{H}_{\mu^n}v^{n+1}, v-v^{n+1}\rangle_{L^2_{\mu^n}} + \frac{\varepsilon^2}{2} \langle \operatorname{H}_{\mu^n}(v-v^{n+1}), v-v^{n+1}\rangle_{L^2_{\mu^n}}\right)\\
        &=\langle \operatorname{H}_{\mu^n}v^{n+1}, v-v^{n+1}\rangle_{L^2_{\mu^n}}.
    \end{align*}
    Similarly,
    \begin{align*}
        \lim_{\varepsilon \downarrow 0}\frac{1}{2\tau\varepsilon}\left(\|v^{\varepsilon, n}\|_{L^2_{\mu^n}}^2 - \|v^{n+1}\|_{L^2_{\mu^n}}^2\right) = \tau^{-1} \langle v^{n+1}, v-v^{n+1}\rangle_{L^2_{\mu^n}}.
    \end{align*}
    Since $v^{n+1}$ is the unique minimizer of $J_n$,
    \begin{align*}
        0 &\leq \lim_{\varepsilon \downarrow 0}\frac{1}{\varepsilon}\left(J_n(v^{n,\varepsilon}) - J_n(v^{n+1})\right)\\
        &= \left\langle \nabla_\mu F(\mu^n, \cdot), v-v^{n+1}\right\rangle_{L^2_{\mu^n}} + \langle \operatorname{H}_{\mu^n}v^{n+1}, v-v^{n+1}\rangle_{L^2_{\mu^n}} + \tau^{-1} \langle v^{n+1}, v-v^{n+1}\rangle_{L^2_{\mu^n}}\\
        &= \left\langle \nabla_\mu F(\mu^n, \cdot) + \operatorname{H}_{\mu^n}v^{n+1} + \tau^{-1}v^{n+1} , v-v^{n+1}\right\rangle_{L^2_{\mu^n}},
    \end{align*}
    for all $v \in L^2_{\mu^n}$. Hence,
    \begin{align*}
        \nabla_\mu F(\mu^n, \cdot) + \operatorname{H}_{\mu^n}v^{n+1} + \tau^{-1}v^{n+1} = 0\ \text{ in } L^2_{\mu^n}.
    \end{align*}
    Equivalently,
    \begin{equation}
    \label{foc-noninv}
        \left(\operatorname{H}_{\mu^n} + \tau^{-1}\operatorname{I_{d \times d}}\right)v^{n+1} = -\nabla_\mu F(\mu^n, \cdot)\ \text{ in } L^2_{\mu^n}.
    \end{equation}
    Since there exists $\delta > 0$ such that $\langle \operatorname{H}_{\mu^n}v,v\rangle_{L_{\mu_n}^2} \geq -\delta\|v\|_{L_{\mu_n}^2}^2$, for all $v \in L_{\mu_n}^2$, and $\tau < \delta^{-1}$, then
    \begin{equation*}
        \operatorname{H}_{\mu^n} + \tau^{-1}\operatorname{I_{d \times d}} \succeq m\operatorname{I_{d \times d}},
    \end{equation*}
    with $m := \tau^{-1} -\delta > 0$. Therefore, $\operatorname{H}_{\mu^n} + \tau^{-1}\operatorname{I_{d \times d}}$ is an invertible operator in $L^2_{\mu^n}$, and by \eqref{foc-noninv}
    \begin{equation*}
        v^{n+1} = - \left(\operatorname{H}_{\mu^n} + \tau^{-1}\operatorname{I_{d \times d}}\right)^{-1}\nabla_\mu F(\mu^n, \cdot)\ \text{ in } L^2_{\mu^n}.
    \end{equation*}
    If $\operatorname{H}_{\mu^n}$ is non-negative, for all $n$, then
    \begin{equation*}
         \operatorname{H}_{\mu^n} + \tau^{-1}\operatorname{I_{d \times d}} \succeq \tau^{-1}\operatorname{I_{d \times d}},
    \end{equation*}
    hence $\operatorname{H}_{\mu^n} + \tau^{-1}\operatorname{I_{d \times d}}$ is again an invertible operator in $L^2_{\mu^n}$ and the same expression for $v^{n+1}$ holds.
    
    Finally, $v^{n+1} \in  L^2_{\mu^n}$, since $\nabla_\mu F(\mu^n,\cdot) \in L^2_{\mu^n}$ and either $\|\left(\operatorname{H}_{\mu^n} + \tau^{-1}\operatorname{I_{d \times d}}\right)^{-1}\|_{\mathrm{op}}\le m^{-1}$ or $\|\left(\operatorname{H}_{\mu^n} + \tau^{-1}\operatorname{I_{d \times d}}\right)^{-1}\|_{\mathrm{op}}\le \tau$.
\end{proof}

\paragraph{Wasserstein saddle-free Newton method.}
\label{subsec:wass-saddle-free}
Assume that $\operatorname H_\mu$ is bounded and self-adjoint on $L^2_\mu$. Given $v \in L_\mu^2,$ the quadratic Taylor model that approximates $F$ around $\mu$ is $\mathfrak T_2^F(\mu)[v] = F(\mu) + \langle \nabla_\mu F(\mu,\cdot),v\rangle_{L^2_\mu} + \frac12 \langle \operatorname H_\mu v,v\rangle_{L^2_\mu}.$ When $\operatorname H_\mu$ is invertible, the Newton direction solves the stationarity equation $\operatorname H_\mu v=-\nabla_\mu F(\mu,\cdot)$. For an indefinite Hessian, however, this quadratic model is unbounded below along negative curvature directions, so the Newton direction is not its minimizer.

Instead, define $|\operatorname H_\mu|:=(\operatorname H_\mu^2)^{1/2}$ and replace the quadratic model by $\mathcal Q_\mu^{\mathrm{SFN}}(v):= \mathfrak T_1^F(\mu)[v] +\frac12\langle |\operatorname H_\mu|v,v\rangle_{L^2_\mu},$ where $T_1^F(\mu)[v]:=F(\mu) + \langle \nabla_\mu F(\mu,\cdot),v\rangle_{L^2_\mu}$ is the linear Taylor model. Since $\operatorname H_\mu \preceq |\operatorname H_\mu|,$ we have
$\mathfrak T_2^F(\mu)[v] \leq \mathcal Q_\mu^{\mathrm{SFN}}(v)$; see Lemma~\ref{lem:upper-bound-inner-hessian} and \cite[Problem~2.35]{kato1995perturbation}. This is an upper bound on the quadratic Taylor model. If $|\operatorname H_\mu|\succeq m\operatorname{I_{d \times d}}$ for some $m>0$, the surrogate $\mathcal Q_\mu^{\mathrm{SFN}}$ is strongly convex and has the unique minimizer $v_{\mathrm{SFN}}(\mu) = -|\operatorname H_\mu|^{-1}\nabla_\mu F(\mu,\cdot).$ The absolute value Hessian principle is inspired by finite-dimensional saddle-free Newton methods. However, its use on Wasserstein space must address the following operator level difficulties.
\\

\paragraph{Lack of bounded invertibility.} Non-negativity of $|\operatorname H_\mu|$ does not guarantee that its inverse exists as a bounded operator, as the Hessian may have a nontrivial kernel. Even when its kernel is trivial, spectral values may accumulate at zero due to compactness. Thus, injectivity alone does not ensure that $|\operatorname H_\mu|^{-1}\nabla_\mu F(\mu,\cdot)$ is defined for every gradient in $L^2_\mu$ or depends continuously on it. Requiring a uniform gap around zero would exclude precisely the small curvature regimes that our local spectral analysis is intended to accommodate.
\\

\paragraph{Lack of operator Lipschitz stability.}
Although the scalar function $x\mapsto|x|$ is Lipschitz, the map $A\mapsto|A|$ need not be operator Lipschitz in the operator norm for bounded
self-adjoint operators; see \cite[VI: Bounded Operators, Problem~17]{ReedSimon1980}. Thus, a Lipschitz estimate for the Hessian does not necessarily give the corresponding estimate for its absolute value. This matters because quantitative saddle escape requires comparing preconditioners along evolving nonlinear transport trajectories, not only computing their spectra at a fixed critical point.

We address these difficulties by replacing $|\operatorname H_\mu|$ with $(\operatorname H_\mu^2+\beta \operatorname{I_{d \times d}})^{1/2},$ $\beta>0.$ Functional calculus gives $\operatorname H_\mu \preceq |\operatorname H_\mu| \preceq (\operatorname H_\mu^2+\beta \operatorname{I_{d \times d}})^{1/2},$ and $(\operatorname H_\mu^2+\beta \operatorname{I_{d\times d}})^{1/2} \succeq \sqrt{\beta}\operatorname{I_{d \times d}}.$ Thus, the regularized quadratic model still bounds the second-order Taylor model and is now coercive without a spectral gap assumption on
$\operatorname H_\mu$.

Given $\mu^0\in\mathcal P_2(\mathbb R^d)$, $\beta>0$, and $\tau>0$, we define WSFN by
\begin{equation}
\label{eq:saddle_free_newton}
\begin{aligned}
    v_{\mathrm{WSFN}}(\mu^n) &= \argmin_{v\in L^2_{\mu^n}}\left\{\mathfrak T_1^F(\mu^n)[v]+\frac12\left\langle(\operatorname H_{\mu^n}^2+\beta \operatorname{I_{d \times d}})^{1/2}v,v\right\rangle_{L^2_{\mu^n}}\right\},
    \\
    \mu^{n+1} &= \left(\operatorname{Id}+\tau v_{\mathrm{WSFN}}(\mu^n)\right)_\#\mu^n.
\end{aligned}
\end{equation}
The minimizer is unique and satisfies $v_{\mathrm{WSFN}}(\mu)=-P_\beta(\operatorname H_\mu)\nabla_\mu F(\mu,\cdot).$ Hence,
\begin{equation*}
    \mu^{n+1} =\left(\operatorname{Id} - \tau(\operatorname H_{\mu^n}^2+\beta \operatorname{I_{d \times d}})^{-1/2} \nabla_\mu F(\mu^n,\cdot) \right)_\#\mu^n.
\end{equation*}

\begin{proposition}[Well-posedness and first-order optimality condition of \eqref{eq:saddle_free_newton}]
\label{prop:well-posedness-saddle-free}
Let Assumption \ref{assumption:smooth-F} hold. Let $\mu^0 \in \mathcal{P}_2(\mathbb R^d)$ and $\tau, \beta > 0$. Then the sequence of iterates $(v^n)_{n \in \mathbb{N}}$ generated by the update \eqref{eq:saddle_free_newton} is well-defined and remains in $L^2_{\mu^n}$, for all $n$. Moreover, $v^{n+1}$ satisfies
    \begin{equation*}
      v^{n+1} = - \left(\operatorname{H}_{\mu^n}^2 + \beta\operatorname{I_{d \times d}}\right)^{-\frac{1}{2}}\nabla_\mu F(\mu^n, \cdot)\ \text{ in } L^2_{\mu^n}.
    \end{equation*}
\end{proposition}
\begin{proof}
    By Lemma \ref{lem:upper-bound-inner-hessian}, $A_{\mu^n} :=\operatorname{H}_{\mu^n}^2+\beta\operatorname{I_{d \times d}}$ is self-adjoint. Furthermore, because $\operatorname{H}_{\mu^n}^2 \succeq 0$, and $\beta > 0$ the operator $A_{\mu^n} \succeq \beta$. Let $\sigma(A_{\mu^n})$ be the spectrum of $A_{\mu^n}$. Since $A_{\mu^n}$ is self-adjoint, we have $\sigma(A_{\mu^n}) \subset \mathbb R$, and in particular $\sigma(A_{\mu^n}) \subset [\beta, \infty)$.

    Let $f$ be a continuous function over the spectrum $\sigma(A_\mu)$. By \cite[Theorem VII.1 (continuous functional calculus)]{ReedSimon1980},
    \begin{equation*}
        f(A_{\mu^n})^* = \bar{f}(A_{\mu^n}),
    \end{equation*}
    where $\bar{f}$ denotes the complex conjugate of $f$. Take $f(\lambda) = \lambda^{\frac{1}{2}}$ for $\lambda \in \sigma(A_{\mu^n})$. Since $\sigma(A_{\mu^n}) \subset [\beta, \infty)$, it follows that $f(\lambda) = \bar{f}(\lambda)$. Therefore,
    \begin{equation*}
        \left(\left(\operatorname{H}_{\mu^n}^2+\beta\operatorname{I_{d \times d}}\right)^{\frac{1}{2}}\right)^* = \left(\operatorname{H}_{\mu^n}^2+\beta\operatorname{I_{d \times d}}\right)^{\frac{1}{2}},
    \end{equation*}
    i.e., $\left(\operatorname{H}_{\mu^n}^2+\beta\operatorname{I_{d \times d}}\right)^{\frac{1}{2}}$ is self-adjoint. 
    
    For all $v \in L^2_{\mu^n}$, define 
    \begin{equation*}
        J_n(v) = \left\langle \nabla_\mu F(\mu^n, \cdot), v\right\rangle_{L^2_{\mu^n}} + \frac{1}{2}\langle \left(\operatorname{H}_{\mu^n}^2+\beta\operatorname{I_{d \times d}}\right)^{\frac{1}{2}}v, v\rangle_{L^2_{\mu^n}}.
    \end{equation*}
    Therefore, given $\mu^0 \in \mathcal{P}_2(\mathbb R^d)$, for each $n \in \mathbb N$, the functional $J_n(v)$ admits a unique minimizer, denoted $v^{n+1}$, since it is the sum of a linear and a quadratic term. 

    Hence, following the argument in Proposition \ref{prop:well-posedness}, since $\left(\operatorname{H}_{\mu^n}^2+\beta\operatorname{I_{d \times d}}\right)^{\frac{1}{2}}$ is self-adjoint, it is symmetric, so we get
    \begin{equation}
    \label{foc-saddle-free}
       \left(\operatorname{H}_{\mu^n}^2+\beta\operatorname{I_{d \times d}}\right)^{\frac{1}{2}}v^{n+1} = -\nabla_\mu F(\mu^n, \cdot)\ \text{ in } L^2_{\mu^n}.
    \end{equation}
    Note that $\left(\operatorname{H}_{\mu^n}^2+\beta\operatorname{I_{d \times d}}\right)^{\frac{1}{2}} \subset [\sqrt{\beta}, \infty)$. It follows that $\left(\operatorname{H}_{\mu^n}^2+\beta\operatorname{I_{d \times d}}\right)^{\frac{1}{2}}$ is invertible and so \eqref{foc-saddle-free} becomes
    \begin{equation*}
        v^{n+1} = -  \left(\operatorname{H}_{\mu^n}^2+\beta\operatorname{I_{d \times d}}\right)^{-\frac{1}{2}}\nabla_\mu F(\mu^n, \cdot)\ \text{ in } L^2_{\mu^n}.
    \end{equation*}
    Finally, $v^{n+1} \in  L^2_{\mu^n}$, since $\nabla_\mu F(\mu^n,\cdot) \in L^2_{\mu^n}$ and $\left\|\left(\operatorname{H}_{\mu^n}^2+\beta\operatorname{I_{d \times d}}\right)^{-\frac{1}{2}}\right\|_{\mathrm{op}}\le \beta^{-\frac{1}{2}}$.
\end{proof}
These two results justify the explicit formulas for the regularized Newton and WSFN updates used in the main text.

\subsection{Linearization of preconditioned Wasserstein updates}
\label{app:critical-point-linearizations}
This appendix derives the spectral multipliers in Table~\ref{tab:critical-point-multipliers}. All comparisons are made on the common reference space $L^2_{\mu^*},$ where $\mu^*$ is a critical point, i.e., $\nabla_\mu F(\mu^*,\cdot)=0$.
\\

\paragraph{Curvature interpretation.}
The second-order expansion in Lemma~\ref{lem:2ndorder-expansion} gives, along
$\operatorname{H}_{\mu^*}e=\lambda e$,
\[
    F\left((\operatorname{Id}+ae)_\#\mu^*\right) = F(\mu^*)
    + \frac12\lambda a^2\|e\|_{L_{\mu^*}^2}^2 + o(a^2)
    \quad\text{as }a\to0.
\]
Thus, the sign of $\lambda$ describes the curvature of the objective along this perturbation curve. The multipliers below describe how each linearized
update acts on that direction.
\\

\paragraph{Representation on a common reference space.}
For $v\in L^2_{\mu^*}$, define $\operatorname{T}_v:=\operatorname{Id}+v,$ $\mu_v:=(\operatorname{T}_v)_\#\mu^*.$ Composition defines an isometric embedding $\mathcal C_v:L^2_{\mu_v}\to L^2_{\mu^*},$ given by $\mathcal C_v h:=h\circ \operatorname{T}_v.$ In particular, this construction does not require $\operatorname{T}_v$ to be invertible. The pulled-back Hessian is the operator $\widetilde{\operatorname{H}}_{\mu^*,v}:L_{\mu^*}^2 \to L_{\mu^*}^2$ given by
\begin{align*}
    \widetilde{\operatorname{H}}_{\mu^*,v}[h](x) :=
    \nabla\nabla_\mu F(\mu_v,\operatorname{T}_v(x))h(x) +
    \int_{\mathbb R^d}\nabla_\mu^2F (\mu_v,\operatorname{T}_v(x),\operatorname{T}_v(y))h(y)\mathrm d\mu^*(y).
\end{align*}
Under regularity of $F$, in particular symmetry, and Assumption \ref{assumption:smooth-F}, $\widetilde{\operatorname{H}}_{\mu^*,v}$ is bounded and self-adjoint near $v=0$, $\widetilde{\operatorname{H}}_{\mu^*,0}=\operatorname{H}_{\mu^*}$, and, under Assumption \ref{assumption:Hessian-lipschitz}, by Lemma \ref{lem:two-map-pulled-back-hessian-lipschitz},
\[
    \|\widetilde{\operatorname{H}}_{\mu^*,v}-\widetilde{\operatorname{H}}_{\mu^*,w}\|_{\mathrm{op}}\leq L_{\operatorname H}\|v-w\|_{L_{\mu^*}^2}.
\]
Define $G(v):=\nabla_\mu F(\mu_v,\cdot) \circ \operatorname{T}_v.$ Proposition~\ref{proposition:Wasserstein-hessian-appendix} gives $ G(v) = \int_0^1 \widetilde{\operatorname{H}}_{\mu^*,sv}v\mathrm ds,$ because $G(0)=0$. Thus,
\begin{equation}
\label{eq:app-gradient-linearization}
    G(v)=\operatorname{H}_{\mu^*}v+R_G(v),
    \quad
    \|R_G(v)\|_{L_{\mu^*}^2}
    \leq
    \frac{L_{\operatorname H}}{2}
    \|v\|_{L_{\mu^*}^2}^2.
\end{equation}
Directly from the definitions, $\widetilde{\operatorname{H}}_{\mu^*,v}\mathcal C_v = \mathcal C_v\operatorname H_{\mu_v}.$ Since the operators are self-adjoint, this intertwining identity extends by functional calculus to operator functions $f(\widetilde{\operatorname{H}}_{\mu^*,v})\mathcal C_v = \mathcal C_v f(\operatorname H_{\mu_v}),$ whenever the indicated functions of the operators are well defined; see \cite[Theorem VII.2 (g)]{ReedSimon1980}. Thus, preconditioned gradients can be represented on $L_{\mu^*}^2$ without identifying composition with an invertible change of coordinates.
\\

\paragraph{A common linearization formula.}
Let $\operatorname{B}(v)$ be the common reference preconditioner, and assume that $\operatorname{B}(v)$ is continuous in operator norm at $v=0$. Define the displacement update on $L_{\mu^*}^2$ by
\[
    \Psi_\tau^{\operatorname{B}}(v) := v-\tau \operatorname{B}(v)G(v).
\]
Equivalently, the updated transport map from $\mu^*$ is $\operatorname{Id}+\Psi_\tau^{\operatorname{B}}(v)$. Using \eqref{eq:app-gradient-linearization}, we have
\begin{align*}
    \operatorname{B}(v)G(v)-\operatorname{B}(0)\operatorname{H}_{\mu^*}v =
    \operatorname{B}(v)\left(G(v)-\operatorname{H}_{\mu^*}v\right) + \left(\operatorname{B}(v)-\operatorname{B}(0)\right)\operatorname{H}_{\mu^*}v.
\end{align*}
The first term is $O(\|v\|_{L_{\mu^*}^2}^2)$ because $\operatorname{B}(v)$ is locally bounded, and the second is $o(\|v\|_{L_{\mu^*}^2})$ by operator norm continuity. Hence
\begin{equation}
\label{eq:app-preconditioned-update-linearization}
    \Psi_\tau^{\operatorname{B}}(v) = \left(\operatorname{I_{d \times d}}-\tau \operatorname{B}(0)\operatorname{H}_{\mu^*}\right)v + o(\|v\|_{L_{\mu^*}^2}),
\end{equation}
and therefore by definition of Fr\'echet derivative,
\[
    D\Psi_\tau^{\operatorname{B}}(0)=\operatorname{I_{d \times d}}-\tau \operatorname{B}(0)\operatorname{H}_{\mu^*}.
\]
No derivative of $\operatorname{B}(v)$ is needed. Continuity suffices because the gradient vanishes at the reference measure $\mu^*.$
\\

\paragraph{Gradient descent.}
For $\operatorname{B}(v)=\operatorname{I_{d \times d}}$, \eqref{eq:app-preconditioned-update-linearization} gives $D\Psi_\tau^{\mathrm{GD}}(0)=\operatorname{I_{d \times d}}-\tau \operatorname{H}_{\mu^*}.$ Thus, if $\operatorname{H}_{\mu^*}e=\lambda e$, then $D\Psi_\tau^{\mathrm{GD}}(0)e = (1-\tau\lambda)e.$
\\

\paragraph{Newton.} Suppose that $\operatorname{H}_{\mu^*}$ has a bounded inverse. By continuity of $\widetilde{\operatorname{H}}_{\mu^*,v}$, the operators $\widetilde{\operatorname{H}}_{\mu^*,v}$ are also boundedly invertible for sufficiently small $\|v\|_{L_{\mu^*}^2}$, and
$\widetilde{\operatorname{H}}_{\mu^*,v}^{-1}\to \operatorname{H}_{\mu^*}^{-1}$ in operator norm. Taking $\operatorname{B}(v)=\widetilde{\operatorname{H}}_{\mu^*,v}^{-1}$ gives $ D\Psi_\tau^{\mathrm N}(0) = \operatorname{I_{d \times d}}-\tau \operatorname{H}_{\mu^*}^{-1}\operatorname{H}_{\mu^*} = (1-\tau)\operatorname{I_{d \times d}}.$ The multiplier is therefore independent of the sign of the Hessian eigenvalue.
\\

\paragraph{Ideal saddle-free Newton.}
Again suppose that $\operatorname{H}_{\mu^*}$ is boundedly invertible. Its spectrum is then separated from zero. The same holds for $\widetilde{\operatorname{H}}_{\mu^*,v}$ in a sufficiently small neighborhood of $v=0$, and continuous functional calculus on this neighborhood gives $|\widetilde{\operatorname{H}}_{\mu^*,v}|^{-1}\to|\operatorname{H}_{\mu^*}|^{-1}$ in operator norm. Taking $\operatorname{B}(v)=|\widetilde{\operatorname{H}}_{\mu^*,v}|^{-1}$ gives $D\Psi_\tau^{\mathrm{SFN}}(0) = \operatorname{I_{d \times d}}-\tau|\operatorname{H}_{\mu^*}|^{-1}\operatorname{H}_{\mu^*}.$ For $\operatorname{H}_{\mu^*}e=\lambda e$, then $D\Psi_\tau^{\mathrm{SFN}}(0)e= \left(1-\tau\lambda|\lambda|^{-1}\right)e.$ 

Thus, for $0<\tau<2$, positive curvature directions contract, whereas negative curvature directions expand by $1+\tau$. The continuity used here does not contradict the lack of a general operator Lipschitz bound for
$A\mapsto|A|$ since this argument is local to a neighborhood whose spectra are separated from zero.
\\

\paragraph{Regularized WSFN.}
For $\beta>0$, take $\operatorname{B}(v)=P_\beta(\widetilde{\operatorname{H}}_{\mu^*,v}) = (\widetilde{\operatorname{H}}_{\mu^*,v}^2+\beta \operatorname{I_{d \times d}})^{-1/2}.$ Lemma~\ref{lem:regularized-wsfn-spectral-maps} ensures operator norm continuity without a spectral gap assumption. Therefore,
\[
    D\Psi_\tau^{\mathrm{WSFN}}(0) = \operatorname{I_{d \times d}}-\tau P_\beta(\operatorname{H}_{\mu^*})\operatorname{H}_{\mu^*} = \operatorname{I_{d \times d}}-\tau S_\beta(\operatorname{H}_{\mu^*}).
\]
In an eigendirection,
\[
    D\Psi_\tau^{\mathrm{WSFN}}(0)e = \left(1-\tau\frac{\lambda}{\sqrt{\lambda^2+\beta}}\right)e.
\]
The normalized Hessian estimate also gives a quantitative remainder with explicit dependence on $\beta$. Indeed, write $G(v)=\widetilde{\operatorname{H}}_{\mu^*,v}v+E(v),$ where
\[
    E(v) = \int_0^1 \left(\widetilde{\operatorname{H}}_{\mu^*,sv}-\widetilde{\operatorname{H}}_{\mu^*,v}\right)v\mathrm ds,
    \quad \|E(v)\|_{L_{\mu^*}^2} \leq \frac{L_{\operatorname H}}{2}\|v\|_{L_{\mu^*}^2}^2.
\]
Then
\begin{align*}
    P_\beta(\widetilde{\operatorname{H}}_{\mu^*,v})G(v)-S_\beta(\operatorname{H}_{\mu^*})v = \left(S_\beta(\widetilde{\operatorname{H}}_{\mu^*,v})-S_\beta(\operatorname{H}_{\mu^*})\right)v + P_\beta(\widetilde{\operatorname{H}}_{\mu^*,v})E(v).
\end{align*}
Applying Lemma~\ref{lem:regularized-wsfn-spectral-maps} to the two terms gives
\begin{equation*}
\label{eq:app-wsfn-linearization-remainder}
    \left\|P_\beta(\widetilde{\operatorname{H}}_{\mu^*,v})G(v)-S_\beta(\operatorname{H}_{\mu^*})v
    \right\|_{L_{\mu^*}^2} \leq \frac{3L_{\operatorname H}}{2\sqrt{\beta}}\|v\|_{L_{\mu^*}^2}^2.
\end{equation*}
In particular,
\[
    \left\|\Psi_\tau^{\mathrm{WSFN}}(v) - \left(\operatorname{I_{d \times d}}-\tau S_\beta(\operatorname{H}_{\mu^*})\right)v\right\|_{L_{\mu^*}^2} \leq \frac{3\tau L_{\operatorname H}}{2\sqrt{\beta}}
    \|v\|_{L_{\mu^*}^2}^2.
\]

\subsection{Relation with Newton's method}
\label{subsection:additional-local-rates}
In this subsection, we record one complementary local result concerning the quadratic convergence rate of the Wasserstein Newton method on positive spectral subspaces. The relation between WSFN and Wasserstein Newton must be understood spectrally. Since \(\operatorname{K}_{\mu^*}\) is compact, its inverse is generally unbounded on \((\ker \operatorname{K}_{\mu^*})^\perp\). Thus, one does not have $(\operatorname{K}_{\mu^*}^2+\beta \operatorname{I_{d \times d}})^{-1/2} \to \operatorname{K}_{\mu^*}^{-1}$ in operator norm on \(L^2_{\mu^*}\) as \(\beta\downarrow0\).

The Newton interpretation is nevertheless valid on every positive spectral subspace. Indeed, since \(\operatorname{K}_{\mu^*}\) is compact, \(\operatorname{Ran}(P_m)\) is finite-dimensional, and the restricted operator $\operatorname{K}_m:=\operatorname{K}_{\mu^*}|_{\operatorname{Ran}(P_m)}$ is boundedly invertible, with $\|\operatorname{K}_m^{-1}\|_{\mathrm{op}}\leq \frac{1}{m}.$ The spectral functional calculus gives
\begin{align*}
    \left\|\left((\operatorname{K}_{\mu^*}^2+\beta \operatorname{I_{d \times d}})^{-1/2} - \operatorname{K}_m^{-1}\right)P_m
    \right\|_{\mathrm{op}} \leq \frac{1}{m} - \frac{1}{\sqrt{m^2+\beta}} \to 0
\end{align*}
as $\beta\downarrow 0.$ Thus, on \(\operatorname{Ran}(P_m)\), the WSFN preconditioner converges in operator norm to the Newton preconditioner. In particular, for the unit stepsize \(\tau=1\),
\[
    \|\mathcal L_{1,\beta}P_m\|_{\mathrm{op}} \leq 1-\frac{m}{\sqrt{m^2+\beta}} \to 0,
\]
as $\beta\downarrow 0.$ Hence the first-order term vanishes in the limit
\(\beta\downarrow0\) on every fixed positive spectral subspace, which is precisely the cancellation responsible for the quadratic convergence rate of Newton's method.

On the full space, however, one only has the strong convergence $(\operatorname{K}_{\mu^*}^2+\beta \operatorname{I_{d \times d}})^{-1/2}\operatorname{K}_{\mu^*} \to \operatorname{Id}-P_{\ker \operatorname{K}_{\mu^*}},$ as $\beta\downarrow 0,$ and so $\mathcal L_{1,\beta} \to P_{\ker \operatorname{K}_{\mu^*}}$ strongly. This convergence is not uniform. If the positive eigenvalues of \(\operatorname{K}_{\mu^*}\) accumulate at zero, then $\|\mathcal L_{1,\beta}\|_{\mathrm{op}}=1$ for every \(\beta>0\). Accordingly, the limit \(\beta\downarrow0\) does not imply a full space quadratic convergence theorem for the nonlinear WSFN iteration. The corresponding quadratic Newton result is valid after restricting the dynamics to a fixed positive spectral subspace, as the following theorem shows.

\begin{theorem}[Quadratic convergence of spectrally projected Wasserstein Newton]
\label{thm:projected-newton-quadratic}
Let Assumptions \ref{assumption:smooth-F} and \ref{assumption:Hessian-lipschitz} hold. Let \(\mu^*\in\mathcal P_2^{\mathrm{ac}}(\mathbb R^d)\) be a global minimizer of \(F\). Fix \(m>0\), and define $P_m:=\mathbbm{1}_{[m,\infty)}(\operatorname{K_{\mu^*}}).$ Assume that \(\operatorname{Ran}(P_m)\neq\{0\}\). For \(v\in \operatorname{Ran}(P_m)\), let $Z_v:=\operatorname{Id}+v,$ and $\mu_v:=(Z_v)_{\#}\mu^*.$ Its gradient on \(\operatorname{Ran}(P_m)\) is $G_m(v):=P_m\left(\nabla_\mu F(\mu_v,\cdot) \circ Z_v\right).$ Let $A_m(v):=P_m{\widetilde{\operatorname{H}}_{\mu^*,v}}|_{\operatorname{Ran}(P_m)}:\operatorname{Ran}(P_m)\to \operatorname{Ran}(P_m)$ denote the corresponding projected pulled-back Hessian defined by
\begin{align*}
    \widetilde{\operatorname{H}}_{\mu^*,v}[h](x) := \nabla\nabla_\mu F(\mu_v,Z_v(x))h(x) + \int_{\mathbb R^d}\nabla_\mu^2F\left(\mu_v,Z_v(x),Z_v(y)\right)h(y)\mathrm d\mu^*(y),
\end{align*}
for $h \in L_{\mu^*}^2.$ Then
\begin{equation*}
    \|A_m(v)-A_m(u)\|_{\mathrm{op}} \leq L_{\operatorname H} \|v-u\|_{L^2_{\mu^*}}, \quad u,v\in \operatorname{Ran}(P_m).
\end{equation*}
Let $0<\alpha< \frac{m}{2L_{\operatorname H}},$ and consider the projected Newton iteration
\begin{equation*}
    v^{n+1} = v^n-A_m(v^n)^{-1}G_m(v^n), \quad \mu^n:=(\operatorname{Id}+v^n)_{\#}\mu^*,
\end{equation*}
initialized at \(v^0\in \operatorname{Ran}(P_m)\) with \(\|v^0\|_{L^2_{\mu^*}}<\alpha\). Then \(A_m(v^n)\) is invertible, \(v^n\in B_{\operatorname{Ran}(P_m)}(\alpha)\) for every \(n\geq0\), and
\begin{equation*}
    W_2(\mu^n,\mu^*)\leq\frac{m}{L_{\operatorname H}}\left(\frac12\right)^{2^n}.
\end{equation*}
\end{theorem}

\begin{proof}
Since \(\mu^*\) is a global minimizer, $\nabla_\mu F(\mu^*,\cdot)=0$ in $L^2_{\mu^*},$ and hence \(G_m(0)=0\). Moreover, $\operatorname{M}_{\mu^*}=0$, so $\operatorname{H}_{\mu^*}=\operatorname{K}_{\mu^*}$. It follows that $A_m(0)=P_m \operatorname{K}_{\mu^*}|_{\operatorname{Ran}(P_m)}=\operatorname{K}_{\mu^*}|_{\operatorname{Ran}(P_m)}.$

By the definition of \(\operatorname{Ran}(P_m)\), $\langle A_m(0)z,z\rangle_{L^2_{\mu^*}} \geq m\|z\|_{L^2_{\mu^*}}^2,$ for every $z\in \operatorname{Ran}(P_m).$  For $h \in \operatorname{Ran}(P_m),$ 
\begin{equation*}
    \left(A_m(v)-A_m(u)\right)h = P_m\left(\widetilde{\operatorname{H}}_{\mu^*,v}-\widetilde{\operatorname{H}}_{\mu^*,u}\right)h,
\end{equation*}
for every $v,u \in \operatorname{Ran}(P_m).$ Since $P_m$ is orthogonal projection, $\|P_mz\|_{L^2_{\mu^*}} \leq \|z\|_{L^2_{\mu^*}},$ and therefore by Lemma \ref{lem:two-map-pulled-back-hessian-lipschitz},
\begin{equation}
\label{eq:projected-hessian-lipschitz}
\|A_m(v)-A_m(u)\|_{\mathrm{op}} \leq L_{\operatorname H}\|v-u\|_{L^2_{\mu^*}}.
\end{equation}
For \(v\in B_{\operatorname{Ran}(P_m)}(\alpha)\), this Lipschitz estimate gives
\[
    \|A_m(v)-A_m(0)\|_{\mathrm{op}}
    \leq
    L_{\operatorname H}\|v\|_{L^2_{\mu^*}}
    <
    \frac{m}{2}.
\]
Therefore, $A_m(v)\succeq \frac{m}{2}I_{\operatorname{Ran}(P_m)},$ so \(A_m(v)\) is invertible and
\begin{equation}
\label{eq:projected-hessian-inverse-bound}
    \|A_m(v)^{-1}\|_{\mathrm{op}}\leq\frac{2}{m}.
\end{equation}
Because \(G_m(0)=0\), the fundamental theorem of calculus gives
\[
    G_m(v) = \int_0^1 A_m(tv)v\mathrm dt.
\]
Consequently, one projected Newton step satisfies
\begin{align*}
    v-A_m(v)^{-1}G_m(v) =A_m(v)^{-1}\left(A_m(v)v-G_m(v)\right)= A_m(v)^{-1}\int_0^1\left(A_m(v)-A_m(tv)\right)v\mathrm dt.
\end{align*}
Using \eqref{eq:projected-hessian-lipschitz},
\begin{align*}
    \left\|\int_0^1\left(A_m(v)-A_m(tv)\right)v\mathrm dt \right\|_{L^2_{\mu^*}} \leq \int_0^1 L_{\operatorname H}(1-t)\|v\|_{L^2_{\mu^*}}^2\mathrm dt = \frac{L_{\operatorname H}}{2}\|v\|_{L^2_{\mu^*}}^2.
\end{align*}
Combining this with
\eqref{eq:projected-hessian-inverse-bound} gives 
\[
    \|v^{n+1}\|_{L^2_{\mu^*}} \leq \frac{L_{\operatorname H}}{m} \|v^n\|_{L^2_{\mu^*}}^2.
\]
Since $\frac{L_{\operatorname H}}{m} \|v^0\|_{L^2_{\mu^*}}<\frac12,$ the one-step estimate shows inductively that
\[
    \|v^{n+1}\|_{L^2_{\mu^*}} \leq \frac12\|v^n\|_{L^2_{\mu^*}}< \alpha.
\]
Thus all the iterates remain in \(B_{\operatorname{Ran}(P_m)}(\alpha)\). Setting $q_n :=\frac{L_{\operatorname H}}{m}\|v^n\|_{L^2_{\mu^*}},$ we have \(q_{n+1} \leq q_n^2\), and therefore $ q_n\leq q_0^{2^n}.$ This gives
\[
    \|v^n\|_{L^2_{\mu^*}} \leq \frac{m}{L_{\operatorname H}}\left(\frac{L_{\operatorname H}}{m} \|v^0\|_{L^2_{\mu^*}}\right)^{2^n}.
\]
Finally, the coupling $\left(\operatorname{Id},\operatorname{Id}+v^n\right)_{\#}\mu^*\in\Pi(\mu^*,\mu^n)$ implies
\[
    W_2^2(\mu^n,\mu^*) \leq \int_{\mathbb R^d}\|v^n(x)\|^2\mathrm d\mu^*(x),
\]
which proves the Wasserstein estimate.
\end{proof}

\section{Proofs of main results}
\label{proofs of main results}
In this section, we prove the main theoretical results of the paper. 

\subsection*{Proof roadmap}
\phantomsection
\label{app:wsfn-proof-roadmap}
The analysis proceeds through the following steps, using the auxiliary operator estimates of Appendix~\ref{appendix:aux-results}.
\\

\paragraph{Operator stability and transport direction.} We first represent Hessians and preconditioners on a common reference space. Lemma~\ref{lem:two-map-pulled-back-hessian-lipschitz} controls the variation of the pulled-back Hessian under changes in the transport map, while Lemma~\ref{lem:regularized-wsfn-spectral-maps} gives operator Lipschitz estimates for the regularized preconditioner and normalized Hessian. Together with the composition identity of Lemma~\ref{lem:intertwining_preconditioner}, these estimates give the expansion of the preconditioned gradient in Lemma~\ref{lem:preconditioned-gradient-expansion-sharp}. This expansion is the main tool for comparing nearby WSFN trajectories. All estimates track the dependence on regularization, which is essential when $\beta_{\mathrm{sad}}=\Theta(\delta^2)$.
\\

\paragraph{Turn insufficient decrease into localization and escape.} Lemma~\ref{lem:local-wsfn-descent} relates objective decrease to transport displacement. Thus, unless a trajectory has already achieved the required decrease, its transport maps remain in a controlled neighborhood of the episode's starting measure. Within this neighborhood, Hessian regularity and Lemma~\ref{lem:preconditioned-gradient-expansion-sharp} control the nonlinear errors relative to the frozen linearized dynamics. Proposition~\ref{discrete_lower_bound_metric} uses this localization to establish a deterministic two-trajectory escape guarantee for initializations differing along an unstable direction.
\\

\paragraph{High probability net decrease.} The perturbation covariance is $\operatorname{K}_\mu^2$, whereas the subsequent transport uses the full Hessian $\operatorname{H}_\mu$. Lemma~\ref{lem:independence} decomposes the Gaussian perturbation into its component along a negative eigendirection of $\operatorname{K}_\mu$ and an independent orthogonal component. This decomposition connects the random initialization to the deterministic comparison in Proposition~\ref{discrete_lower_bound_metric}. Proposition~\ref{discrete_decrease_around_saddle} then combines Gaussian estimates with control of the perturbation cost to obtain a high probability decrease measured relative to the \emph{pre-perturbation} objective.
\\

\paragraph{Count descent steps and saddle episodes.} Theorem~\ref{thm:polynomial-time-sosp} combines the large gradient descent estimate in Lemma~\ref{descent lemma} with the uniform saddle episode guarantee. Under the prescribed parameter choices, $\beta_{\mathrm{sad}}=\Theta(\delta^2)$ gives an episode length $n_{\mathrm{out}}=\tilde O(1)$ and a required net decrease $F_0=\tilde\Omega(\delta^3)$. Writing $F_{\min}:=F(\mu^0)-\inf_\mu F(\mu)$, objective decrease bounds the number of successful episodes by $\frac{F_{\min}}{F_0}$. Including the possible terminal episode, the number of transport steps spent in episodes satisfies
\begin{equation*}
    N_{\mathrm{sad}} \leq n_{\mathrm{out}} \left(\frac{F_{\min}}{F_0}+1\right)
    = \tilde O\left(F_{\min}\delta^{-3}+1\right).
\end{equation*}
A conditional union bound over fresh perturbations controls the probability of terminating at a saddle.

For PWGF, the corresponding episode duration is $\tilde O(\delta^{-1})$
\cite{yamamoto2025hessianguided}. Thus, shortening the episodes accounts for the improvement of the saddle contribution from $\tilde O(\delta^{-4})$ to $\tilde O(\delta^{-3})$. The inverse cubic factor remains because of the guaranteed objective decrease per episode. Appendix~\ref{app:saddle-parameter-dependence} explains this balance. Remark~\ref{rem:improved-saddle-dependence} clarifies why, under the required tolerance coupling, this improvement does not change the leading total complexity bound.
\\

\paragraph{Characterize the local linearized dynamics.}
The local analysis returns to the operator expansion and applies spectral calculus at a minimizer. Theorem~\ref{thm:local-spectral-behavior-wsfn} separates the kernel component from positive curvature spectral bands, proving geometric contraction on bands bounded away from zero and polynomial decay under spectral source conditions. The source conditions control the contribution of small positive curvature components. These rates concern the linearized iterates. Remark~\ref{rem:limitations-nonlinear-spectral-rates} explains the additional control of nonlinear error accumulation and spectral transfer needed to extend them to the full iteration.

We now turn to the analysis of the WSFN iteration. The next lemmas provide the basic control needed on perturbed initializations and on the objective decrease along the discrete dynamics.

\subsection{Auxiliary lemmas for the WSFN dynamics}
In this subsection, we collect the deterministic estimates used in the analysis of the algorithm. These results control the effect of the initial perturbation and the decrease of the objective along the WSFN iterates.
\begin{lemma}
    \label{not_very_increase}
     Let Assumption \ref{assumption:smooth-F} hold, with $C_{\operatorname{H}} : = C_{\operatorname{M}} + C_{\operatorname{K}}$. Let $\mu_\mathrm{sad}$ be an $(\varepsilon,\delta)$-saddle point, let $\xi \sim \mathcal{N}_{L_{\mu_\mathrm{sad}}^2}(0, \operatorname{K}_{\mu_\mathrm{sad}}^2)$. If $\|\xi\|_{L_{\mu_\mathrm{sad}}^2} \leq \kappa$ and $\eta = \frac{2 F_0}{\kappa(\varepsilon + \sqrt{\varepsilon^2 +2 C_{\operatorname{H}}F_0})}$, for some $\kappa > 0$ and $F_0 > 0$, then $\mu^{n_{\mathrm{in}}} = (\operatorname{Id}+\eta \xi)_{\#}\mu_\mathrm{sad}$ satisfies
    \begin{align*}
        F(\mu^{n_{\mathrm{in}}}) - F(\mu_\mathrm{sad}) \leq F_0.
    \end{align*}
\end{lemma}
\begin{proof}
    Let $\pi_t = (\operatorname{Id} + t \eta \xi)_{\#}\mu_\mathrm{sad}$ for $t \in [0,1]$, with $\pi_0 = \mu_\mathrm{sad}$ and $\pi_1 = (\operatorname{Id}+\eta \xi)_{\#}\mu_\mathrm{sad}=\mu^{n_{\mathrm{in}}}$. Applying Lemma \ref{lemma:smoothness-of-F-2} gives
    \begin{align*}
        F(\mu^{n_{\mathrm{in}}}) &\leq F(\mu_\mathrm{sad}) + \eta \langle \nabla_\mu F(\mu_\mathrm{sad}), \xi  \rangle_{L_{\mu_\mathrm{sad}}^2} + \frac{C_{\operatorname{H}}}{2}\eta^2\|\xi\|_{L_{\mu_\mathrm{sad}}^2}^2\\
        &\leq F(\mu_\mathrm{sad}) + \eta \kappa \varepsilon + \frac{C_{\operatorname{H}}}{2}\eta^2\kappa^2\\
        &= F(\mu_\mathrm{sad}) + F_0,
    \end{align*}
   where the second inequality follows from $\|\xi\|_{L_{\mu_\mathrm{sad}}^2} \leq \kappa$ and the last equality follows from the choice of $\eta$.
\end{proof}
\begin{lemma}[Large gradient descent lemma]
\label{descent lemma}
    Let Assumption \ref{assumption:smooth-F} hold. Let $\{\mu^k\}_{k=0}^{n-1} \subset \mathcal{P}_2(\mathbb R^d)$ be a sequence of iterates generated by scheme \eqref{eq:wsfn-update-main} with stepsize $\tau_\mathrm{grad} \leq \frac{\sqrt{\beta_\mathrm{grad}}}{C_{\operatorname{H}}}$. Then
    \begin{align*}
        F(\mu^{k+1}) \leq  F(\mu^k) - \frac{\tau_\mathrm{grad}\sqrt{\beta_\mathrm{grad}}}{2}\|(\operatorname{H}_{\mu^k}^2+\beta_\mathrm{grad}\operatorname{I_{d \times d}})^{-\frac{1}{2}}\nabla_\mu F(\mu^k)\|^2_{L^2_{\mu^k}},
    \end{align*}
    for all $k \geq 0,$ and
    \begin{equation*}
        F(\mu^0) - F(\mu^n) \geq \frac{\tau_\mathrm{grad}\sqrt{\beta_\mathrm{grad}}}{2}\sum_{k=0}^{n-1} \|(\operatorname{H}_{\mu^k}^2+\beta_\mathrm{grad}\operatorname{I_{d \times d}})^{-\frac{1}{2}}\nabla_\mu F(\mu^k)\|_{L^2_{\mu^k}}^2.
    \end{equation*}
\end{lemma}
\begin{proof}
    Applying Lemma \ref{lemma:smoothness-of-F-2} with $\mu=\mu^k$ and $v = -\tau(\operatorname{H}_{\mu^k}^2+\beta_\mathrm{grad}\operatorname{I_{d \times d}})^{-\frac{1}{2}}\nabla_\mu F(\mu^k)$ gives
    \begin{align*}
        &F\left((\operatorname{Id}-\tau (\operatorname{H}_{\mu^k}^2+\beta_\mathrm{grad}\operatorname{I_{d \times d}})^{-\frac{1}{2}}\nabla_\mu F(\mu^k))_\#\mu^k\right)\\ &\leq F(\mu^k)  -\tau_\mathrm{grad} \left\langle \nabla_\mu F(\mu^k), (\operatorname{H}_{\mu^k}^2+\beta_\mathrm{grad}\operatorname{I_{d \times d}})^{-\frac{1}{2}}\nabla_\mu F(\mu^k)\right\rangle_{L_{\mu^k}^2}\\ 
        &+ \frac{\tau_\mathrm{grad}^2(C_{\operatorname{M}}+C_{\operatorname{K}})}{2}\|(\operatorname{H}_{\mu^k}^2+\beta_\mathrm{grad}\operatorname{I_{d \times d}})^{-\frac{1}{2}}\nabla_\mu F(\mu^k)\|^2_{L^2_{\mu^k}}.
    \end{align*}
    By scheme \eqref{eq:wsfn-update-main},
    \begin{equation*}
        F\left((\operatorname{Id}-\tau_\mathrm{grad} (\operatorname{H}_{\mu^k}^2+\beta_\mathrm{grad}\operatorname{I_{d \times d}})^{-1/2}\nabla_\mu F(\mu^k))_\#\mu^k\right) = F(\mu^{k+1}).
    \end{equation*}
    Since $\operatorname{H}_{\mu^k}$ is bounded and symmetric by Assumption \ref{assumption:smooth-F}, it is self-adjoint, and hence so is $(\operatorname{H}_{\mu^k}^2+\beta_\mathrm{grad}\operatorname{I_{d \times d}})^{-\frac{1}{2}}$. Therefore, using that $\operatorname{H}_{\mu^k}^2+\beta_\mathrm{grad}\operatorname{I_{d \times d}} \succeq \beta_\mathrm{grad}\operatorname{I_{d \times d}}$ implies $(\operatorname{H}_{\mu^k}^2+\beta_\mathrm{grad}\operatorname{I_{d \times d}})^{-\frac{1}{2}} \preceq \beta_\mathrm{grad}^{-\frac{1}{2}}\operatorname{I_{d \times d}}$ gives
    \begin{align*}
        \|(\operatorname{H}_{\mu^k}^2+\beta_\mathrm{grad}\operatorname{I_{d \times d}})^{-\frac{1}{2}}\nabla_\mu F(\mu^k)\|^2_{L^2_{\mu^k}} 
        &\leq \frac{1}{\sqrt{\beta_\mathrm{grad}}}\left\langle \nabla_\mu F(\mu^k), (\operatorname{H}_{\mu^k}^2+\beta_\mathrm{grad}\operatorname{I_{d \times d}})^{-\frac{1}{2}}\nabla_\mu F(\mu^k)\right\rangle_{L_{\mu^k}^2},
    \end{align*}
    Therefore, using $\tau_\mathrm{grad} \leq \sqrt{\beta_\mathrm{grad}}(C_{\operatorname{M}} + C_{\operatorname{K}})^{-1}$ gives
    \begin{align*}
        F(\mu^{k+1}) &\leq F(\mu^k)  -\tau_\mathrm{grad}\sqrt{\beta_\mathrm{grad}} \|(\operatorname{H}_{\mu^k}^2+\beta_\mathrm{grad}\operatorname{I_{d \times d}})^{-\frac{1}{2}}\nabla_\mu F(\mu^k)\|^2_{L^2_{\mu^k}}\\
        &+ \frac{\tau_\mathrm{grad}\sqrt{\beta_\mathrm{grad}}}{2}\|(\operatorname{H}_{\mu^k}^2+\beta_\mathrm{grad}\operatorname{I_{d \times d}})^{-\frac{1}{2}}\nabla_\mu F(\mu^k)\|^2_{L^2_{\mu^k}}\\
        &= F(\mu^k) - \frac{\tau_\mathrm{grad}\sqrt{\beta_\mathrm{grad}}}{2}\|(\operatorname{H}_{\mu^k}^2+\beta_\mathrm{grad}\operatorname{I_{d \times d}})^{-\frac{1}{2}}\nabla_\mu F(\mu^k)\|^2_{L^2_{\mu^k}}.
    \end{align*}
    Summing from $k=0$ to $k=n-1$ yields
    \begin{equation*}
        F(\mu^0) - F(\mu^n) \geq \frac{\tau_\mathrm{grad}\sqrt{\beta_\mathrm{grad}}}{2}\sum_{k=0}^{n-1} \|(\operatorname{H}_{\mu^k}^2+\beta_\mathrm{grad}\operatorname{I_{d \times d}})^{-\frac{1}{2}}\nabla_\mu F(\mu^k)\|_{L^2_{\mu^k}}^2.
    \end{equation*}
\end{proof}

\begin{lemma}[Local descent of WSFN with an order-one stepsize]
\label{lem:local-wsfn-descent}
Let Assumptions \ref{assumption:smooth-F} and \ref{assumption:Hessian-lipschitz} hold. Let $\mu^*\in\mathcal P_2(\mathbb R^d)$ satisfy $\|\nabla_\mu F(\mu^*,\cdot)\|_{L^2_{\mu^*}}\leq \varepsilon.$ For $z\in L^2_{\mu^*}(\mathbb R^d;\mathbb R^d)$, define $Z_z:=\operatorname{Id}+z,$ $ \mu_z:=(Z_z)_\#\mu^*,$ and let
\begin{equation*}
\label{eq:local-descent-wsfn-direction}
    d_z :=\left(\operatorname{H}_{\mu_z}^2+\beta_\mathrm{sad}\operatorname{I_{d \times d}}
    \right)^{-1/2}\nabla_\mu F(\mu_z,\cdot),
\end{equation*}
for $\beta_\mathrm{sad}>0.$ For $r>0$, set $D_{\beta_\mathrm{sad}}(r,\varepsilon) := \frac{\varepsilon}{\sqrt{\beta_\mathrm{sad}}} + r + \frac{L_H r^2}{2\sqrt{\beta_\mathrm{sad}}}.$ Then, for every $\|z\|_{L^2_{\mu^*}}\leq r$ and every $\tau_\mathrm{sad}\in(0,1]$,
\begin{equation}
\label{eq:general-local-wsfn-descent}
    F\left((\operatorname{Id}-\tau_\mathrm{sad} d_z)_\#\mu_z\right) \leq F(\mu_z)
    - \tau\left[1-\frac{\tau_\mathrm{sad}}{2}-\frac{L_H\tau_\mathrm{sad}^2}{6\sqrt{\beta_\mathrm{sad}}}D_{\beta_\mathrm{sad}}(r,\varepsilon)\right]\left\langle \nabla_\mu F(\mu_z,\cdot), d_z\right\rangle_{L^2_{\mu_z}}.
\end{equation}
In particular, if $0<\tau_\mathrm{sad} \leq \frac12,$ $ r\leq\frac{\sqrt{\beta_\mathrm{sad}}}{L_H},$ $\varepsilon\leq\frac{\beta_\mathrm{sad}}{L_H},$ then
\begin{equation}
\label{eq:local-descent-half-energy}
    F\left((\operatorname{Id}-\tau_\mathrm{sad} d_z)_\#\mu_z\right) \leq F(\mu_z)
    - \frac{\tau_\mathrm{sad}}{2}\left\langle \nabla_\mu F(\mu_z,\cdot), d_z\right\rangle_{L^2_{\mu_z}},
\end{equation}
and
\begin{equation}
\label{eq:local-descent-direction-square}
    F\left((\operatorname{Id}-\tau_\mathrm{sad} d_z)_\#\mu_z\right) \leq F(\mu_z) - \frac{\tau_\mathrm{sad}\sqrt{\beta_\mathrm{sad}}}{2}\|d_z\|_{L^2_{\mu_z}}^2.
\end{equation}
\end{lemma}
\begin{proof}
Denote the pulled-back gradient and Hessian by $\widehat g_z :=\nabla_\mu F(\mu_z,\cdot)\circ Z_z$ and $\widehat{\operatorname{H}}_z,$ respectively. Lemma \ref{lem:intertwining_preconditioner} gives
\begin{equation*}
    \widehat d_z := d_z\circ Z_z = \left(\widehat{\operatorname{H}}_z^2+\beta_\mathrm{sad} \operatorname{I_{d \times d}}\right)^{-1/2}\widehat g_z.
\end{equation*}
In particular, $\|\widehat d_z\|_{L_{\mu^*}^2} = \|d_z\|_{L^2_{\mu_z}}.$ Along the curve $t\mapsto \mu_{tz} = (\operatorname{Id}+tz)_\#\mu^*,$ we have $\widehat g_z =\widehat g_0 + \int_0^1\widehat{\operatorname{H}}_{tz}z\mathrm dt.$ Adding and subtracting $\widehat{\operatorname{H}}_zz$, we obtain
\begin{equation}
\label{eq:local-descent-gradient-decomposition}
    \widehat g_z = \widehat g_0 + \widehat{\operatorname{H}}_zz + \int_0^1
    \left(\widehat{\operatorname{H}}_{tz}-\widehat{\operatorname{H}}_z\right)z\mathrm dt.
\end{equation}
By Lemma \ref{lem:two-map-pulled-back-hessian-lipschitz},
\begin{align*}
    \left\|\int_0^1
    \left(\widehat{\operatorname{H}}_{tz}-\widehat{\operatorname{H}}_z\right)z\mathrm dt\right\|_{L_{\mu^*}^2}
    \leq
    \int_0^1
    \|\widehat{\operatorname{H}}_{tz}-\widehat{\operatorname{H}}_z\|_{\mathrm{op}}
    \|z\|_{L_{\mu^*}^2}\mathrm dt &\leq L_H
    \int_0^1 (1-t)\|z\|_{L_{\mu^*}^2}^2 \mathrm dt\\ 
    &=
    \frac{L_H}{2}
    \|z\|_{L_{\mu^*}^2}^2.
\end{align*}
Let $P_z := \left(\widehat{\operatorname{H}}_z^2+\beta_\mathrm{sad} \operatorname{I_{d \times d}}\right)^{-1/2},$ $S_z := P_z\widehat{\operatorname{H}}_z.$ By spectral calculus, $\|P_z\|_{\mathrm{op}} \leq \frac{1}{\sqrt{\beta_\mathrm{sad}}},$ $\|S_z\|_{\mathrm{op}} \leq 1.$ Using \eqref{eq:local-descent-gradient-decomposition}, we get
\begin{equation*}
    \widehat d_z = P_z\widehat g_0 + S_zz + P_z\int_0^1
    \left(\widehat{\operatorname{H}}_{tz}-\widehat{\operatorname{H}}_z\right)z\mathrm dt.
\end{equation*}
Thus
\begin{align*}
    \|d_z\|_{L^2_{\mu_z}} = \|\widehat d_z\|_{L_{\mu^*}^2} &\leq
    \frac{1}{\sqrt{\beta_\mathrm{sad}}}
    \|\widehat g_0\|_{L_{\mu^*}^2} + \|z\|_{L_{\mu^*}^2} +
    \frac{1}{\sqrt{\beta_\mathrm{sad}}}
    \left\|\int_0^1 \left(\widehat{\operatorname{H}}_{tz}-\widehat{\operatorname{H}}_z\right)z\mathrm dt\right\|_{L_{\mu^*}^2}\\
    &\leq \frac{\varepsilon}{\sqrt{\beta_\mathrm{sad}}} + \|z\|_{L_{\mu^*}^2} + \frac{L_H}{2\sqrt{\beta_\mathrm{sad}}}\|z\|_{L_{\mu^*}^2}^2\\
    &\leq D_{\beta_\mathrm{sad}}(r,\varepsilon).
\end{align*}
We now apply the cubic Taylor estimate of Lemma \ref{lemma:smoothness-of-F-2} at $\mu_z$ in the direction $-\tau d_z$. This gives
\begin{equation}
\begin{aligned}
    F\left((\operatorname{Id}-\tau_\mathrm{sad} d_z)_\#\mu_z\right) &\leq
    F(\mu_z) -
    \tau_\mathrm{sad} \left\langle \nabla_\mu F(\mu_z,\cdot), d_z \right\rangle_{L^2_{\mu_z}}\\ 
    &+ \frac{\tau_\mathrm{sad}^2}{2} \left\langle \operatorname{H}_{\mu_z}d_z,d_z \right\rangle_{L^2_{\mu_z}} + \frac{L_H\tau_\mathrm{sad}^3}{6} \|d_z\|_{L^2_{\mu_z}}^3.
\label{eq:local-descent-taylor}
\end{aligned}
\end{equation}
Since $d_z=\left(\operatorname{H}_{\mu_z}^2+\beta_\mathrm{sad} \operatorname{I_{d \times d}} \right)^{-1/2}\nabla_\mu F(\mu_z)$, we have
\begin{equation*}
    \nabla_\mu F(\mu_z,\cdot) = \left(\operatorname{H}_{\mu_z}^2+\beta_\mathrm{sad}\operatorname{I_{d \times d}}\right)^{1/2}d_z.
\end{equation*}
Therefore,
\begin{equation*}
\label{eq:local-descent-energy-representation}
\begin{aligned}
   \left\langle \nabla_\mu F(\mu_z,\cdot), d_z\right\rangle_{L^2_{\mu_z}} =
    \left\langle \left(\operatorname{H}_{\mu_z}^2+\beta_\mathrm{sad} \operatorname{I_{d \times d}}\right)^{1/2}d_z,d_z\right\rangle_{L^2_{\mu_z}}.
\end{aligned}
\end{equation*}
The scalar inequality $\lambda \leq \sqrt{\lambda^2+{\beta_\mathrm{sad}}},$ $ \lambda\in\mathbb R,$ implies, by spectral calculus, $\operatorname{H}_{\mu_z} \preceq \left(\operatorname{H}_{\mu_z}^2+\beta_\mathrm{sad} \operatorname{I_{d \times d}}\right)^{1/2}.$ Hence
\begin{equation*}
\label{eq:hessian-term-bounded-by-energy}
    \left\langle \operatorname{H}_{\mu_z}d_z,d_z\right\rangle_{L^2_{\mu_z}}
    \leq \left\langle \nabla_\mu F(\mu_z,\cdot), d_z\right\rangle_{L^2_{\mu_z}}.
\end{equation*}
Moreover, $\left(\operatorname{H}_{\mu_z}^2+\beta_\mathrm{sad} \operatorname{I_{d \times d}}\right)^{1/2} \succeq \sqrt{\beta_\mathrm{sad}} \operatorname{I_{d \times d}},$ so
\begin{equation}
\label{eq:energy-controls-direction}
 \left\langle \nabla_\mu F(\mu_z,\cdot), d_z\right\rangle_{L^2_{\mu_z}} \geq \sqrt{\beta_\mathrm{sad}} \|d_z\|_{L^2_{\mu_z}}^2.
\end{equation}
It follows that
\[
    \|d_z\|_{L^2_{\mu_z}}^3 \leq \frac{ \|d_z\|_{L^2_{\mu_z}}}{\sqrt{\beta_\mathrm{sad}}} \left\langle \nabla_\mu F(\mu_z,\cdot), d_z\right\rangle_{L^2_{\mu_z}}.
\]
Substituting these estimates into \eqref{eq:local-descent-taylor} gives
\begin{align*}
    F\left((\operatorname{Id}-\tau_\mathrm{sad} d_z)_\#\mu_z\right)&\leq
    F(\mu_z) - \tau_\mathrm{sad} \left(1-\frac{\tau_\mathrm{sad}}{2}\right)
    \left\langle \nabla_\mu F(\mu_z,\cdot), d_z\right\rangle_{L^2_{\mu_z}}\\ 
    &+
    \frac{L_H\tau_\mathrm{sad}^3}{6\sqrt{\beta_\mathrm{sad}}}\|d_z\|_{L^2_{\mu_z}}
    \left\langle \nabla_\mu F(\mu_z,\cdot), d_z\right\rangle_{L^2_{\mu_z}}\\
    &\leq
    F(\mu_z) -\tau_\mathrm{sad}\left[1-\frac{\tau_\mathrm{sad}}{2} - \frac{L_H\tau_\mathrm{sad}^2}{6\sqrt{\beta_\mathrm{sad}}}D_{\beta_\mathrm{sad}}(r,\varepsilon)
    \right]\left\langle \nabla_\mu F(\mu_z,\cdot), d_z\right\rangle_{L^2_{\mu_z}},
\end{align*}
which proves
\eqref{eq:general-local-wsfn-descent}. By the assumption on the constants, $\frac{\varepsilon}{\sqrt{\beta_\mathrm{sad}}} \leq \frac{\sqrt{\beta_\mathrm{sad}}}{L_H},$ $r\leq \frac{\sqrt{\beta_\mathrm{sad}}}{L_H},$ and
\[
    \frac{L_Hr^2}{2\sqrt{\beta_\mathrm{sad}}} \leq \frac{\sqrt{\beta_\mathrm{sad}}}{2L_H}.
\]
Therefore, 
\[
    D_{\beta_\mathrm{sad}}(r,\varepsilon) \leq \frac{5\sqrt{\beta_\mathrm{sad}}}{2L_H}.
\]
Since $\tau_\mathrm{sad}\leq \frac12$,
\begin{align*}
    1-\frac{\tau_\mathrm{sad}}{2} -
    \frac{L_H\tau_\mathrm{sad}^2}{6\sqrt{\beta_\mathrm{sad}}}
    D_{\beta_\mathrm{sad}}(r,\varepsilon) \geq 1-\frac14 -
    \frac{5}{12}\frac14 = \frac{31}{48} > \frac12.
\end{align*}
This proves \eqref{eq:local-descent-half-energy}. Finally, \eqref{eq:energy-controls-direction} gives \eqref{eq:local-descent-direction-square}.
\end{proof}

\begin{lemma}[Bootstrap control of a WSFN saddle episode]
\label{lem:episode-radius-bootstrap}
Let the assumptions of Lemma \ref{lem:local-wsfn-descent} hold, and suppose $0<\tau_\mathrm{sad}\leq\frac12,$ $ \varepsilon\leq\frac{\beta_\mathrm{sad}}{L_H}.$ Let $ \mu^0=(Z^0)_\#\mu^*,$ and define recursively $Z^{k+1}=\left(\operatorname{Id}-\tau_\mathrm{sad} d_k\right)\circ Z^k,$ $\mu^k=(Z^k)_\#\mu^*,$ where
\[
    d_k:= \left(\operatorname{H}_{\mu^k}^2+\beta_\mathrm{sad} \operatorname{I_{d \times d}}\right)^{-1/2}\nabla_\mu F(\mu^k,\cdot).
\]
Fix $N\in\mathbb N$ and assume that
\begin{equation}
\label{eq:episode-contradiction-decrease}
    F(\mu^0)-F(\mu^n) \leq 2F_0, \quad 0\leq n\leq N.
\end{equation}
Let $r\leq \frac{\sqrt{\beta_\mathrm{sad}}}{L_H}$. If $ \|Z^0-\operatorname{Id}\|_{L^2_{\mu^*}} + 2\sqrt{\tau N}\beta_\mathrm{sad}^{-1/4}F_0^{1/2}< r,$
then $\|Z^n-\operatorname{Id}\|_{L^2_{\mu^*}} < r,$ $0\leq n\leq N,$ and
\begin{equation}
\label{eq:episode-radius-sharp-conclusion}
    \|Z^n-\operatorname{Id}\|_{L^2_{\mu^*}} \leq \|Z^0-\operatorname{Id}\|_{L^2_{\mu^*}} +2\sqrt{\tau n}\beta_\mathrm{sad}^{-1/4}F_0^{1/2}.
\end{equation}
\end{lemma}
\begin{proof}
Suppose, for a contradiction, that an iterate leaves the ball of
radius $r$. Let
\[
    m := \min \left\{n\in\{1,\ldots,N\}:\|Z^n-\operatorname{Id}\|_{L^2_{\mu^*}}\geq r \right\}.
\]
Then
\[
    \|Z^k-\operatorname{Id}\|_{L^2_{\mu^*}} < r, \quad 0\leq k<m.
\]
Therefore, Lemma \ref{lem:local-wsfn-descent} applies at every
iteration $k=0,\ldots,m-1$, and gives
\[
    F(\mu^0)-F(\mu^m) \geq \frac{\tau_\mathrm{sad}\sqrt{\beta_\mathrm{sad}}}{2}\sum_{k=0}^{m-1} \|d_k\|_{L^2_{\mu^k}}^2.
\]
Furthermore, $Z^{k+1}-Z^k = -\tau_\mathrm{sad} d_k\circ Z^k,$ and hence $\|Z^{k+1}-Z^k\|_{L^2_{\mu^*}} = \tau_\mathrm{sad}\|d_k\|_{L^2_{\mu^k}}.$ It follows from the triangle inequality and Cauchy--Schwarz that
\begin{align*}
    \|Z^m-Z^0\|_{L^2_{\mu^*}} \leq \tau_\mathrm{sad}\sum_{k=0}^{m-1}
    \|d_k\|_{L^2_{\mu^k}}\leq
    \tau_\mathrm{sad}\sqrt m\left(\sum_{k=0}^{m-1}\|d_k\|_{L^2_{\mu^k}}^2
    \right)^{1/2} \leq \left(\frac{2\tau_\mathrm{sad} m}{\sqrt{\beta_\mathrm{sad}}}\left[F(\mu^0)-F(\mu^m)\right]\right)^{1/2}.
\end{align*}
Using \eqref{eq:episode-contradiction-decrease},
\[
    \|Z^m-Z^0\|_{L^2_{\mu^*}} \leq 2\sqrt{\tau_\mathrm{sad} m}\beta_\mathrm{sad}^{-1/4}F_0^{1/2}.
\]
Hence,
\begin{align*}
    \|Z^m-\operatorname{Id}\|_{L^2_{\mu^*}}
    &\leq \|Z^0-\operatorname{Id}\|_{L^2_{\mu^*}} + \|Z^m-Z^0\|_{L^2_{\mu^*}}\\
    &\leq \|Z^0-\operatorname{Id}\|_{L^2_{\mu^*}}
    +2\sqrt{\tau_\mathrm{sad} N}
    \beta_\mathrm{sad}^{-1/4}F_0^{1/2}<r,
\end{align*}
contradicting the definition of $m$. Hence no such first exit
time exists. The same calculation with $m$ replaced by any $n\leq N$ proves
\eqref{eq:episode-radius-sharp-conclusion}.
\end{proof}

\subsection{Escape from saddle regions}
With these preliminary estimates in hand, we now fix the parameter choices used in the saddle escape analysis. Let us define $\tilde{\delta} := \frac{\delta}{\sqrt{\delta^2 + \beta}} > 0$. Let $\zeta_\mathrm{ep} \in (0,1)$ and choose the parameters $\beta$, $\tau$, $\kappa$, $n_{\mathrm{out}}$, $F_0$, and $\eta$ as follows. Fix $b>0$ and $\tau_{\mathrm{sad}}\in(0,\frac{1}{2}]$, and set
\begin{equation}
\label{def_beta_tau_sad}
    \beta:=\beta_{\mathrm{sad}}:=b\delta^2, \quad \tau:=\tau_{\mathrm{sad}}.
\end{equation}
Take $\delta$ sufficiently small that
\begin{equation}
\label{eq:new-small-beta-condition}
    \beta_{\mathrm{sad}}=b\delta^2 \leq \min\left\{1,\left(\frac{C_{\operatorname H}\tau_{\mathrm{sad}}}{2}\right)^2\right\}.
\end{equation}
Moreover,
\begin{equation*}
    \widetilde\delta := \frac{\delta}{\sqrt{\delta^2+\beta_{\mathrm{sad}}}}
    = \frac{1}{\sqrt{1+b}}, \quad \varrho := 1+\tau_\mathrm{sad}\widetilde\delta > 1.
\end{equation*}
In particular, $\widetilde\delta$ is now a positive constant independent of $\delta$. Define $c_{\mathrm{ep}}:=\frac{\sqrt{2\pi}}{4}\delta\zeta_{\mathrm{ep}},$ and take $\kappa >0$ such that
\begin{equation}
    \label{eq:def_kappa}
    \kappa \geq C_{\operatorname{K}}\left(1+2\sqrt{\log \frac{4}{\zeta_{\mathrm{ep}}}}+2\log \frac{4}{\zeta_{\mathrm{ep}}}\right) = \tilde O(1).
\end{equation}
For sufficiently small $\delta$, define
\begin{equation}
\label{eq:new-def-nout}
    n_{\mathrm{out}} := 1+\left\lceil\frac{2}{\log\varrho}
        \log\left(\frac{16\sqrt{2C_{\operatorname H}\tau_{\mathrm{sad}}}\kappa}{\sqrt e \beta_{\mathrm{sad}}^{1/4}c_{\mathrm{ep}}\log^{1/2}\varrho}\right)\right\rceil.
\end{equation}
Note that 
\begin{equation*}
    \beta_\mathrm{sad}^{1/4} c_{\mathrm{ep}}=b^{1/4}\delta^{1/2}\frac{\sqrt{2\pi}}{4}\delta\zeta_{\mathrm{ep}}=b^{1/4}\delta^{3/2}\frac{\sqrt{2\pi}}{4}\zeta_{\mathrm{ep}}.
\end{equation*}
Therefore
\begin{equation*}
    n_{\mathrm{out}} = O\left(1+\log \frac{\kappa}{\delta^{3/2}\zeta_{\mathrm{ep}}}\right) = \tilde{O}(1),
\end{equation*}
where $\tilde{O}$ hides logarithmic dependence on \(\delta^{-1}\), \(\zeta_{\mathrm{ep}}^{-1}\), and the Gaussian norm control parameter \(\kappa\). Choose $c_F>0$, independent of $\delta$, sufficiently small as specified below, and set
\begin{equation}
\label{eq:new-def-f0}
    F_0 := c_F\frac{\log^2(3/2)\beta_{\mathrm{sad}}^{3/2}}{L_{\operatorname H}^2{(\tau_{\mathrm{sad}} n_{\mathrm{out}})}^3}.
\end{equation}
Choose $c_F$ sufficiently small that
\begin{equation}
\label{eq:cF-bootstrap-condition}
    \sqrt{c_F}\log\frac{3}{2}\left(\frac{2}{\tau_{\mathrm{sad}}}+\frac{3\sqrt{2}}{\sqrt{C_{\operatorname H}}\tau_{\mathrm{sad}}^{3/2}}\right)\leq\frac12,
\end{equation}
and
\begin{align}
\label{eq:cF-remainder-condition}
    8\sqrt{c_F}+8\sqrt{\frac{2c_F}{C_{\operatorname H}\tau_{\mathrm{sad}}}}+
    c_F\log\frac{3}{2}\left(\frac{10}{\pi C_{\operatorname H}\tau_{\mathrm{sad}}^2}+\frac{8\sqrt2}{\pi\sqrt{C_{\operatorname H}}\tau_{\mathrm{sad}}^{3/2}} +\frac{4}{\pi\tau_{\mathrm{sad}}}\right)
    \leq \frac34.
\end{align}
The definition of $F_0$ becomes
\begin{equation*}
    F_0 = c_F\frac{\log^2(3/2)b^{3/2}\delta^3}{L_{\operatorname H}^2{(\tau_{\mathrm{sad}} n_{\mathrm{out}})}^3} = \tilde{\Omega}\left(\delta^3\right).
\end{equation*}
Define the perturbation amplitude
\begin{equation}
\label{eq:new-def-eta}
    \eta :=\frac{2F_0}{\kappa\left(\varepsilon+\sqrt{\varepsilon^2+2C_{\operatorname H}F_0}\right)} = \tilde{\Theta}\left(\frac{\delta^3}{\varepsilon+\delta^{3/2}}\right).
\end{equation}
Define $Q_\beta := \frac{R_F}{\sqrt{\beta_{\mathrm{sad}}}} + \frac{2L_{\operatorname H}}{\pi\beta_{\mathrm{sad}}}.$ The small gradient condition is
\begin{equation}
\label{eq:new-epsilon-condition}
    \varepsilon \leq \min\left\{\frac{3}{4}\sqrt{2C_{\operatorname H}F_0},
        \frac{\beta_{\mathrm{sad}}}{L_{\operatorname H}},\frac{\log(3/2)}{4\tau_{\mathrm{sad}} n_{\mathrm{out}}Q_\beta}\right\},
\end{equation}
which, up to logarithmic factors, becomes
\begin{equation*}
    \varepsilon =\tilde{O}\left(\frac{\delta^2}{L_{\operatorname H}}\right).
\end{equation*}
We next analyze the behavior of WSFN near an $(\varepsilon,\delta)$-saddle. The key idea is that, after a small perturbation, the component of the iterate in a direction of negative curvature is amplified by the saddle-free dynamics, which leads to a decrease in the objective within a finite number of steps.

\begin{proposition}[Deterministic escape from a saddle region]
\label{discrete_lower_bound_metric}
Let Assumptions~\ref{assumption:smooth-F} and~\ref{assumption:Hessian-lipschitz} hold.
Let $\mu^{\mathrm{sad}}\in\mathcal P_2(\mathbb R^d)$ satisfy $ \|\nabla_\mu F(\mu^{\mathrm{sad}},\cdot)\|_{L^2_{\mu^{\mathrm{sad}}}}\le\varepsilon,$ $\lambda_0:=\underline{\lambda}(\operatorname{K}_{\mu^{\mathrm{sad}}})<-\delta.$ Assume that, for some $R_F>0$, $\|\operatorname{M}_{\mu^{\mathrm{sad}}}\|_{\mathrm{op}} \le R_F\|\nabla_\mu F(\mu^{\mathrm{sad}},\cdot)\|_{L^2_{\mu^{\mathrm{sad}}}}.$ Let $q_0\in L^2_{\mu^{\mathrm{sad}}}$ be a unit eigenvector of $\operatorname{K}_{\mu^{\mathrm{sad}}}$ corresponding to $\lambda_0$. Fix \(\zeta_{\mathrm{ep}}\in(0,1)\), and choose parameters $\beta$, $\tau$, $\kappa$, $n_\mathrm{out}$, $F_0$ and $\eta$ as in \eqref{def_beta_tau_sad}--\eqref{eq:new-def-eta}. Suppose that the small gradient condition \eqref{eq:new-epsilon-condition} holds. Let \(\xi\in L^2_{\mu^{\mathrm{sad}}}\) satisfy $ \|\xi\|_{L^2_{\mu^{\mathrm{sad}}}}\leq\kappa,$ and, for some \(c\in\mathbb R\) satisfying $c_{\mathrm{ep}}\leq |c|\leq 2\kappa,$ set $\tilde\xi:=\xi+cq_0.$ Let \(\mu^{n_{\mathrm{in}}+n}\) and \(\tilde\mu^{n_{\mathrm{in}}+n}\) be the trajectories of the WSFN scheme \eqref{eq:wsfn-update-main} initialized respectively from $\mu^{n_{\mathrm{in}}} := (\operatorname{Id}+\eta\xi)_\#\mu^{\mathrm{sad}},$ $\tilde\mu^{n_{\mathrm{in}}} := (\operatorname{Id}+\eta\tilde\xi)_\#\mu^{\mathrm{sad}}.$ Assume that $\delta$ is sufficiently small, depending on the problem constants and the logarithmic parameters. Then, there exists
\(n\in\{1,\ldots,n_{\mathrm{out}}\}\) such that
\[
\max\left\{F(\mu^{n_{\mathrm{in}}})-F(\mu^{n_{\mathrm{in}}+n}), F(\tilde\mu^{n_{\mathrm{in}}})-F(\tilde\mu^{n_{\mathrm{in}}+n})\right\} \geq 2F_0.
\]
\end{proposition}

\begin{proof}
    We give a proof by contradiction. Assume that for any $n \in \{1,..., n_{\mathrm{out}}\}$, 
    \begin{align*}
        \max\left\{F(\mu^{n_{\mathrm{in}}}) - F(\mu^{n_{\mathrm{in}}+n}), F (\tilde{\mu}^{n_{\mathrm{in}}}) - F(\tilde{\mu}^{n_{\mathrm{in}}+n})\right\} \leq 2 F_0. 
    \end{align*}
    For every $n \in \{0,..., n_{\mathrm{out}}\}$, define
    \begin{equation*}
        \Psi_\tau[\mu^{n_{\mathrm{in}}+n}] = \operatorname{Id} - \tau(\operatorname{H}_{\mu^{n_{\mathrm{in}}+n}}^2+\beta\operatorname{I_{d \times d}})^{-\frac{1}{2}} \nabla_\mu F(\mu^{n_{\mathrm{in}}+n}).
    \end{equation*}
    Then the $n$-step maps of the SFN scheme starting from $\mu^*$ are
    \begin{align*}
        &Z^{n_{\mathrm{in}}+n} := \underbrace{\Psi_\tau[\mu^{{n_{\mathrm{in}}+n}-1}] \circ ... \circ \Psi_\tau[\mu^{n_{\mathrm{in}}}]}_{=:Y^{n_{\mathrm{in}}+n}}\circ (\operatorname{Id} + \eta \xi), \quad n \geq 1, \quad Z^{n_{\mathrm{in}}} := \operatorname{Id} + \eta \xi,\\
        &\tilde{Z}^{n_{\mathrm{in}}+n} = \underbrace{\Psi_\tau[\tilde \mu^{n_{\mathrm{in}}+n-1}] \circ ... \circ \Psi_\tau[\tilde \mu^{n_{\mathrm{in}}}]}_{=:\tilde Y^{n_{\mathrm{in}}+n}} \circ (\operatorname{Id} + \eta \tilde \xi), \quad n \geq 1,\quad  \tilde Z^{n_{\mathrm{in}}} := \operatorname{Id} + \eta \tilde \xi.
        \end{align*}
    Note that $Y^{n_{\mathrm{in}}} = \tilde Y^{n_{\mathrm{in}}} = \operatorname{Id}$. We have
    \begin{equation}
    \begin{aligned}
     \label{evolution_of_y_k}
        W_2(\mu^{n_{\mathrm{in}}+n}, \mu^{\mathrm{sad}})&\leq \|Z^{n_{\mathrm{in}}+n} - \operatorname{Id}\|_{L^2_{\mu^{\mathrm{sad}}}}\\
        &\leq \eta \|\xi\|_{L^2_{\mu^{\mathrm{sad}}}} + \|Z^{n_{\mathrm{in}}+n} - (\operatorname{Id} + \eta \xi)\|_{L^2_{\mu^{\mathrm{sad}}}}\\
        &=\eta \|\xi\|_{L^2_{\mu^{\mathrm{sad}}}} + \|Y^{n_{\mathrm{in}}+n} - \operatorname{Id}\|_{L^2_{\mu^{n_{\mathrm{in}}}}}\\
        &\leq \eta \|\xi\|_{L^2_{\mu^{\mathrm{sad}}}} + \sum_{k=0}^{n-1} \|Y^{n_{\mathrm{in}}+k+1} - Y^{n_{\mathrm{in}}+k}\|_{L^2_{\mu^{n_{\mathrm{in}}}}}\\
        &= \eta \|\xi\|_{L^2_{\mu^{\mathrm{sad}}}} + \sum_{k=0}^{n-1} \|\Psi_\tau[\mu^{n_{\mathrm{in}}+k}]\circ Y^{n_{\mathrm{in}}+k}- Y^{n_{\mathrm{in}}+k}\|_{L^2_{\mu^{n_{\mathrm{in}}}}}\\
        &= \eta \|\xi\|_{L^2_{\mu^{\mathrm{sad}}}} + \sum_{k=0}^{n-1} \|\Psi_\tau[\mu^{n_{\mathrm{in}}+k}]- \operatorname{Id}\|_{L^2_{\mu^k}}\\
        &=\eta \|\xi\|_{L^2_{\mu^{\mathrm{sad}}}} + \tau\sum_{k=0}^{n-1} \|(\operatorname{H}_{\mu^{n_{\mathrm{in}}+k}}^2+\beta\operatorname{I_{d \times d}})^{-\frac{1}{2}} \nabla_\mu F(\mu^{n_{\mathrm{in}}+k})\|_{L^2_{\mu^k}}\\
        &\leq \eta \|\xi\|_{L^2_{\mu^{\mathrm{sad}}}} + \tau\sqrt{n}\left(\sum_{k=0}^{n-1} \|(\operatorname{H}_{\mu^{n_{\mathrm{in}}+k}}^2+\beta\operatorname{I_{d \times d}})^{-\frac{1}{2}} \nabla_\mu F(\mu^{n_{\mathrm{in}}+k})\|_{L^2_{\mu^k}}^2\right)^{1/2}\\
        &\leq \eta \|\xi\|_{L^2_{\mu^{\mathrm{sad}}}} + \tau\sqrt{n}\left(\frac{2}{\tau\sqrt{\beta}}(F(\mu^{n_{\mathrm{in}}})-F(\mu^{n_{\mathrm{in}}+n}))\right)^{1/2}\\
        &\leq \eta \|\xi\|_{L^2_{\mu^{\mathrm{sad}}}} + \sqrt{2}\sqrt{\frac{\tau n}{\sqrt{\beta}}}\left(2F_0\right)^{1/2}\\
        &\leq \eta \kappa + 2\sqrt{\tau n_{\mathrm{out}}}\beta^{-1/4} F_0^{1/2},
    \end{aligned}
    \end{equation}
    where the fourth inequality follows from the Cauchy-Schwarz inequality, the fifth inequality follows from Lemma \ref{lem:local-wsfn-descent} and the sixth inequality follows from the contradiction hypothesis. Similarly, 
    \begin{equation}
        \label{evolution_of_tilde_y_k}
        W_2(\tilde \mu^{n_{\mathrm{in}}+n}, \mu^{\mathrm{sad}})\leq\|\tilde{Z}^{n_{\mathrm{in}}+n} - \operatorname{Id}\|_{L^2_{\mu^{\mathrm{sad}}}} \leq \eta \kappa + \eta |c| + 2\sqrt{\tau n_{\mathrm{out}}}\beta^{-1/4}F_0^{1/2}.
    \end{equation}
    In the following we will show that $\mu^{n_{\mathrm{in}}+n}$ and $\tilde \mu^{n_{\mathrm{in}}+n}$ cannot be too far apart. To this end, since $\mu^{n_{\mathrm{in}}+n} = (Z^{n_{\mathrm{in}}+n})_{\#}\mu^{\mathrm{sad}}$ and $\tilde \mu^{n_{\mathrm{in}}+n} = (\tilde Z^{n_{\mathrm{in}}+n})_{\#}\mu^{\mathrm{sad}}$, we may define 
    \begin{align*}
        w^{n_{\mathrm{in}}+n} := \tilde{Z}^{n_{\mathrm{in}}+n} - Z^{n_{\mathrm{in}}+n} \in L_{\mu^{\mathrm{sad}}}^2,
    \end{align*}
    and prove that $\|w^{n_{\mathrm{in}}+n}\|_{L_{\mu^{\mathrm{sad}}}^2}$ cannot be too large. In particular, we have
    \begin{equation}
    \begin{aligned}
        &w^{n_{\mathrm{in}}+n+1} - w^{n_{\mathrm{in}}+n}\\ 
        &=  \tilde{Z}^{n_{\mathrm{in}}+n+1} - \tilde Z^{n_{\mathrm{in}}+n} - (Z^{n_{\mathrm{in}}+n+1 }- Z^{n_{\mathrm{in}}+n})\\
        &= - \tau (\operatorname{H}_{\tilde \mu^{n_{\mathrm{in}}+n}}^2 +\beta\operatorname{I_{d \times d}})^{-\frac{1}{2}}\nabla_\mu F(\tilde{\mu}^{n_{\mathrm{in}}+n}) \circ \tilde{Z}^{n_{\mathrm{in}}+n} + \tau (\operatorname{H}_{\mu^{n_{\mathrm{in}}+n}}^2 +\beta\operatorname{I_{d \times d}})^{-\frac{1}{2}}\nabla_\mu F(\mu^{n_{\mathrm{in}}+n}) \circ Z^{n_{\mathrm{in}}+n} \\
        &= -\tau(\operatorname{K}_{\mu^{\mathrm{sad}}}^2+\beta\operatorname{I_{d \times d}})^{-\frac{1}{2}}\operatorname{K}_{\mu^{\mathrm{sad}}}(\tilde Z^{n_{\mathrm{in}}+n}-Z^{n_{\mathrm{in}}+n}) -\tau R_\beta(Z^{n_{\mathrm{in}}+n},\tilde Z^{n_{\mathrm{in}}+n})\\
        &=-\tau (\operatorname{K}_{\mu^{\mathrm{sad}}}^2+\beta\operatorname{I_{d \times d}})^{-\frac{1}{2}}\operatorname{K}_{\mu^{\mathrm{sad}}}w^{n_{\mathrm{in}}+n} -\tau R_\beta(Z^{n_{\mathrm{in}}+n},\tilde Z^{n_{\mathrm{in}}+n}),
        \label{recurrence_formula} 
    \end{aligned}
     \end{equation}
     where the penultimate equality follows from Lemma \ref{lem:preconditioned-gradient-expansion-sharp} and $R_\beta(Z^{n_{\mathrm{in}}+n},\tilde Z^{n_{\mathrm{in}}+n})$ satisfies 
     \begin{align*}
    \|R_\beta(Z^{n_{\mathrm{in}}+n},\tilde Z^{n_{\mathrm{in}}+n}))\|_{L^2_{\mu^{\mathrm{sad}}}}
    &\leq
    \Bigg[\frac{2L_H}{\sqrt{\beta}}\left(\|Z^{n_{\mathrm{in}}+n}-\operatorname{Id}\|_{L^2_{\mu^{\mathrm{sad}}}}+\|\tilde Z^{n_{\mathrm{in}}+n}-\operatorname{Id}\|_{L^2_{\mu^{\mathrm{sad}}}}\right)\\
        &+ \frac{L_H^2}{2\pi\beta}\left(\|Z^{n_{\mathrm{in}}+n}-\operatorname{Id}\|_{L^2_{\mu^{\mathrm{sad}}}}^2+\|\tilde Z^{n_{\mathrm{in}}+n}-\operatorname{Id}\|_{L^2_{\mu^{\mathrm{sad}}}}^2
        \right)\\
        &+\left(\frac{R_F}{\sqrt{\beta}}+\frac{2L_H}{\pi\beta}
        \right)\|\nabla_\mu F(\mu^{\mathrm{sad}},\cdot)\|_{L^2_{\mu^{\mathrm{sad}}}}
    \Bigg]\|\tilde Z^{n_{\mathrm{in}}+n}-Z^{n_{\mathrm{in}}+n}\|_{L^2_{\mu^{\mathrm{sad}}}}.
\end{align*}
The assumed bound on $\operatorname{M}_{\mu^{\mathrm{sad}}}$ verifies condition~\eqref{eq:multiplication-part-controlled-at-reference} of Lemma~\ref{lem:preconditioned-gradient-expansion-sharp}. That lemma applies to arbitrary reference measures in $\mathcal P_2(\mathbb R^d)$.
    By induction on $n$, the recurrence formula \eqref{recurrence_formula} implies
    \begin{align*}
        w^{n_{\mathrm{in}}+n} &= \left(\operatorname{I_{d \times d}} -\tau (\operatorname{K}_{\mu^{\mathrm{sad}}}^2+\beta\operatorname{I_{d \times d}})^{-\frac{1}{2}}\operatorname{K}_{\mu^{\mathrm{sad}}}\right)^{n} w^{n_{\mathrm{in}}}\\ 
        &- \tau\sum_{k=0}^{n-1} \left(\operatorname{I_{d \times d}} -\tau (\operatorname{K}_{\mu^{\mathrm{sad}}}^2+\beta\operatorname{I_{d \times d}})^{-\frac{1}{2}}\operatorname{K}_{\mu^{\mathrm{sad}}}\right)^{n-k-1} R_\beta(Z^{n_{\mathrm{in}}+k},\tilde Z^{n_{\mathrm{in}}+k}).
    \end{align*}
    We have $\|\nabla_\mu^2 F(\mu^{\mathrm{sad}}, \cdot,\cdot)\|_{L_{\mu^{\mathrm{sad}} \otimes \mu^{\mathrm{sad}}}^2} < \infty$, so by \cite[Theorem VI.23]{ReedSimon1980}, the operator $\operatorname{K}_{\mu^{\mathrm{sad}}}$ is Hilbert-Schmidt, hence it is compact by \cite[Theorem VI.22]{ReedSimon1980}. Moreover, it is self-adjoint since it is bounded and symmetric by Assumption \ref{assumption:smooth-F}. By self-adjointness, $\underline{\lambda}(\operatorname{K}_{\mu^{\mathrm{sad}}}) \in \sigma(\operatorname{K}_{\mu^{\mathrm{sad}}}).$ Since $\underline{\lambda}(\operatorname{K}_{\mu^{\mathrm{sad}}}) < -\delta < 0,$ by compactness, $\lambda_0=\underline{\lambda}(\operatorname{K}_{\mu^{\mathrm{sad}}})$ is an eigenvalue of $\operatorname{K}_{\mu^{\mathrm{sad}}}$. Thus we may choose a corresponding unit eigenvector $q_0 \in L_{\mu^{\mathrm{sad}}}^2,$ i.e.,
\[
\operatorname{K}_{\mu^{\mathrm{sad}}}q_0=\lambda_0 q_0,\quad \lambda_0 < -\delta.
\]
Since $q_0$ is an eigenvector of $\operatorname{K}_{\mu^{\mathrm{sad}}}$ with corresponding eigenvalue $\lambda_0$, it follows that $q_0$ is also an eigenvector of $\operatorname{I_{d \times d}} -\tau (\operatorname{K}_{\mu^{\mathrm{sad}}}^2+\beta\operatorname{I_{d \times d}})^{-\frac{1}{2}}\operatorname{K}_{\mu^{\mathrm{sad}}}$ with corresponding eigenvalue $1-\tau \frac{\lambda_0}{\sqrt{\lambda_0^2 + \beta}}$ for $\tau > 0$. Thus, using the fact that $w^{n_{\mathrm{in}}} = \tilde{Z}^{n_{\mathrm{in}}} - Z^{n_{\mathrm{in}}} = \eta c q_0$, we get
\begin{equation*}
        \left(\operatorname{I_{d \times d}} -\tau (\operatorname{K}_{\mu^{\mathrm{sad}}}^2+\beta\operatorname{I_{d \times d}})^{-\frac{1}{2}}\operatorname{K}_{\mu^{\mathrm{sad}}}\right)^{n} w^{n_{\mathrm{in}}} = \left(1-\tau \frac{\lambda_0}{\sqrt{\lambda_0^2 + \beta}}\right)^{n} w^{n_{\mathrm{in}}}.
\end{equation*}
Therefore, using $\|w^{n_{\mathrm{in}}}\|_{L_{\mu^{\mathrm{sad}}}^2} = \eta|c|$, we have
\begin{equation*}
    \left\|\left(\operatorname{I_{d \times d}} -\tau (\operatorname{K}_{\mu^{\mathrm{sad}}}^2+\beta\operatorname{I_{d \times d}})^{-\frac{1}{2}}\operatorname{K}_{\mu^{\mathrm{sad}}}\right)^{n} w^{n_{\mathrm{in}}}\right\|_{L^2_{\mu^{\mathrm{sad}}}} = \eta|c|\left(1-\tau \frac{\lambda_0}{\sqrt{\lambda_0^2 + \beta}}\right)^{n}.
\end{equation*}
Now we upper bound $\|\left(\operatorname{I_{d \times d}} -\tau (\operatorname{K}_{\mu^{\mathrm{sad}}}^2+\beta\operatorname{I_{d \times d}})^{-\frac{1}{2}}\operatorname{K}_{\mu^{\mathrm{sad}}}\right)\|_{\mathrm{op}}$ using spectral functional calculus. The key observation is that the function $f(x) = \frac{x}{\sqrt{\beta+x^2}}$ is non-decreasing on $\mathbb R$.

Since $\operatorname{K}_{\mu^{\mathrm{sad}}}$ is self-adjoint, we have $\sigma(\operatorname{K}_{\mu^{\mathrm{sad}}}) \subset \mathbb R$. Let $\Omega_{\operatorname{K}}$ be the Borel $\sigma$-algebra of subsets of $\sigma(\operatorname{K}_{\mu^{\mathrm{sad}}})$. By \cite[1.1. Definition, Chapter IX]{conway1994course}, a spectral measure for $(\sigma(\operatorname{K}_{\mu^{\mathrm{sad}}}), \Omega_{\operatorname{K}}, L_{\mu^{\mathrm{sad}}}^2)$ is defined as a function $E:\Omega_{\operatorname{K}} \to \mathcal{B}(L_{\mu^{\mathrm{sad}}}^2)$ such that
\begin{itemize}
    \item for each $\Delta \in \Omega_{\operatorname{K}}$, $E(\Delta)$ is a projection,
    \item $E(\emptyset) = 0$ and $E(\sigma(\operatorname{K}_{\mu^{\mathrm{sad}}})) =1$,
    \item $E(\Delta_1 \cap \Delta_2) = E(\Delta_1)E(\Delta_2)$ for $\Delta_1, \Delta_2 \in \Omega_{\operatorname{K}}$,
    \item if $\{\Delta_n\}_{n \geq 1}$ are pairwise disjoint sets from $\Omega_{\operatorname{K}}$, then
    \begin{equation*}
        E\left(\cup_{n=1}^\infty \Delta_n \right) = \sum_{i=1}^\infty E(\Delta_n).
    \end{equation*}
\end{itemize}
By \cite[1.9. Lemma, Chapter IX]{conway1994course}, if $E$ is a spectral measure for $(\sigma(\operatorname{K}_{\mu^{\mathrm{sad}}}), \Omega_{\operatorname{K}}, L_{\mu^{\mathrm{sad}}}^2)$ and $\phi,\psi \in L_{\mu^{\mathrm{sad}}}^2$, then 
\begin{equation*}
    E_{\phi, \psi}(\Delta) := \langle E(\Delta)\phi, \psi \rangle_{L_{\mu^{\mathrm{sad}}}^2}
\end{equation*}
defines a countably additive measure on $\Omega_{\operatorname{K}}$ with $\|E\|_{\operatorname{TV}} = |E|(\sigma(\operatorname{K}_{\mu^{\mathrm{sad}}})) \leq \|\phi\|_{L_{\mu^{\mathrm{sad}}}^2}\|\psi\|_{L_{\mu^{\mathrm{sad}}}^2}$. 

Since $\operatorname{K}_{\mu^{\mathrm{sad}}}$ is self-adjoint on $L_{\mu^{\mathrm{sad}}}^2$, it is a normal operator. Hence, by \cite[2.2. The Spectral Theorem, Chapter IX]{conway1994course}, there is a unique spectral measure $E$ on the Borel subsets of $\sigma(\operatorname{K}_{\mu^{\mathrm{sad}}})$ such that 
\begin{equation*}
    \operatorname{K}_{\mu^{\mathrm{sad}}} = \int_{\sigma(\operatorname{K}_{\mu^{\mathrm{sad}}})} \lambda \mathrm{d}E(\lambda).
\end{equation*}
Let $f:\sigma(\operatorname{K}_{\mu^{\mathrm{sad}}}) \to \mathbb R$ be bounded $\Omega_{\operatorname{K}}$-measurable. Then, by \cite[1.10. Proposition, Chapter IX]{conway1994course}, there exists a unique operator $f(\sigma(\operatorname{K}_{\mu^{\mathrm{sad}}})) \in \mathcal{B}(L_{\mu^{\mathrm{sad}}}^2)$ such that for all $\phi,\psi \in L_{\mu^{\mathrm{sad}}}^2$,
\begin{equation*}
    \langle\phi, f(\operatorname{K}_{\mu^{\mathrm{sad}}})\psi\rangle_{L_{\mu^{\mathrm{sad}}}^2} = \int_{\sigma(\operatorname{K}_{\mu^{\mathrm{sad}}})} f(\lambda) \mathrm{d}\langle\phi,E(\lambda)\psi\rangle_{L_{\mu^{\mathrm{sad}}}^2}.
\end{equation*}
Choose $\phi=\psi \in L_{\mu^{\mathrm{sad}}}^2$ and $f_\tau(\lambda) = 1 - \tau (\lambda^2+\beta)^{-\frac{1}{2}}\lambda$, for $\tau > 0$, which is bounded on $\sigma(\operatorname{K}_{\mu^{\mathrm{sad}}})$ since $\operatorname{K}_{\mu^{\mathrm{sad}}}$ is bounded. Moreover, $f_\tau$ is non-increasing on $\mathbb R$. Hence, for any $\lambda \in \sigma(\operatorname{K}_{\mu^{\mathrm{sad}}})$, since $\lambda \geq \lambda_0$, we get
\begin{equation*}
    1 - \tau (\lambda^2+\beta)^{-\frac{1}{2}}\lambda= f_\tau(\lambda) \leq f_\tau(\lambda_0) = 1 - \tau (\lambda_0^2+\beta)^{-\frac{1}{2}}\lambda_0.
\end{equation*}
From $\tau \leq \frac{1}{2} < 1,$ it follows that $f_\tau(\lambda) > 0$, which implies
\begin{equation*}
    \left(1 - \tau (\lambda^2+\beta)^{-\frac{1}{2}}\lambda\right)^2 \leq \left(1 - \tau (\lambda_0^2+\beta)^{-\frac{1}{2}}\lambda_0\right)^2.
\end{equation*}
Hence, by \cite[4.7. Theorem, Chapter X]{conway1994course}, 
\begin{equation*}
\begin{aligned}
&\|(\operatorname{I_{d \times d}}-\tau(\operatorname{K}_{\mu^{\mathrm{sad}}}^2+\beta\operatorname{I_{d \times d}})^{-\frac{1}{2}}\operatorname{K}_{\mu^{\mathrm{sad}}})\psi\|_{L_{\mu^{\mathrm{sad}}}^2}^2 = \int_{\sigma(\operatorname{K}_{\mu^{\mathrm{sad}}})}\left(1-\tau\frac{\lambda}{\sqrt{\lambda^2+\beta}}\right)^{2}\mathrm{d}\langle \psi,E(\lambda)\psi\rangle_{L_{\mu^{\mathrm{sad}}}^2}\\
&\leq \int_{\sigma(\operatorname{K}_{\mu^{\mathrm{sad}}})}\left(1-\tau\frac{\lambda_0}{\sqrt{\lambda_0^2+\beta}}\right)^{2}\mathrm{d}\langle \psi,E(\lambda)\psi\rangle_{L_{\mu^{\mathrm{sad}}}^2}\\
&\leq \left(1-\tau\frac{\lambda_0}{\sqrt{\lambda_0^2+\beta}}\right)^{2}E(\sigma(\operatorname{K}_{\mu^{\mathrm{sad}}}))\|\psi\|_{L_{\mu^{\mathrm{sad}}}^2}^2\\
&=\left(1-\tau\frac{\lambda_0}{\sqrt{\lambda_0^2+\beta}}\right)^{2}\|\psi\|_{L_{\mu^{\mathrm{sad}}}^2}^2.
\end{aligned}
\end{equation*}
Therefore,
\begin{equation*}
    \|\operatorname{I_{d \times d}} -\tau (\operatorname{K}_{\mu^{\mathrm{sad}}}^2+\beta\operatorname{I_{d \times d}})^{-\frac{1}{2}}\operatorname{K}_{\mu^{\mathrm{sad}}}\|_{\mathrm{op}} \leq 1-\tau\frac{\lambda_0}{\sqrt{\lambda_0^2+\beta}}.
\end{equation*}
Putting everything together gives
    \begin{equation}
    \begin{aligned}
    \label{eq:difference wn}
    &\left|\|w^{n_{\mathrm{in}}+n}\|_{L^2_{\mu^{\mathrm{sad}}}} - \eta|c|\left(1-\tau \frac{\lambda_0}{\sqrt{\lambda_0^2 + \beta}}\right)^{n}\right|\\
    &\left\|w^{n_{\mathrm{in}}+n} - \left(\operatorname{I_{d \times d}} -\tau (\operatorname{K}_{\mu^{\mathrm{sad}}}^2+\beta\operatorname{I_{d \times d}})^{-\frac{1}{2}}\operatorname{K}_{\mu^{\mathrm{sad}}}\right)^{n} w^{n_{\mathrm{in}}}\right\|_{L^2_{\mu^{\mathrm{sad}}}}\\
    &\leq \tau \sum_{k=0}^{n-1} \|\operatorname{I_{d \times d}} -\tau (\operatorname{K}_{\mu^{\mathrm{sad}}}^2+\beta\operatorname{I_{d \times d}})^{-\frac{1}{2}}\operatorname{K}_{\mu^{\mathrm{sad}}}\|_{\mathrm{op}}^{n-k-1}\|R_\beta(Z^{n_{\mathrm{in}}+k},\tilde Z^{n_{\mathrm{in}}+k})\|_{L_{\mu^{\mathrm{sad}}}^2}\\
    &\leq \tau\left(1-\tau\frac{\lambda_0}{\sqrt{\lambda_0^2+\beta}}\right)^{n-k-1} \sum_{k=0}^{n-1} \|R_\beta(Z^{n_{\mathrm{in}}+k},\tilde Z^{n_{\mathrm{in}}+k})\|_{L_{\mu^{\mathrm{sad}}}^2}\\
    &\leq \frac{\tau \Delta}{1-\tau\frac{\lambda_0}{\sqrt{\lambda_0^2+\beta}}} \sum_{k=0}^{n-1} \left(1-\tau\frac{\lambda_0}{\sqrt{\lambda_0^2+\beta}}\right)^{n-k}\|w^{n_{\mathrm{in}}+k}\|_{L_{\mu^{\mathrm{sad}}}^2},
    \end{aligned}
    \end{equation}
    where the last inequality follows from the fact that
    \begin{align*}
        &\|R(Z^{n_{\mathrm{in}}+k},\tilde Z^{n_{\mathrm{in}}+k})\|_{L_{\mu^{\mathrm{sad}}}^2}\\ &\leq \underbrace{\Bigg(\frac{4L_{\operatorname{H}}}{\sqrt{\beta}}\left(2A+B\right)+\frac{L_{\operatorname{H}}^2}{\pi\beta}\left(5A^2+4AB+B^2\right)+\varepsilon\left(\frac{R_F}{\sqrt{\beta}}+\frac{2L_{\operatorname{H}}}{\pi\beta}\right)\Bigg)}_{:=\Delta > 0}\|w^{n_{\mathrm{in}}+k}\|_{L_{\mu^{\mathrm{sad}}}^2},
    \end{align*}
    due to \eqref{evolution_of_y_k}, \eqref{evolution_of_tilde_y_k}, the assumption $|c|\leq 2\kappa$, and $A:=\eta \kappa,$ $B:=2\sqrt{\tau n_{\mathrm{out}}}\beta^{-1/4}F_0^{1/2}.$
    
    For any $n \in \{0,..., n_{\mathrm{out}}\}$, set 
    \begin{equation*}
        a_{n_{\mathrm{in}}+n} := \left(1-\tau \frac{\lambda_0}{\sqrt{\lambda_0^2 + \beta}}\right)^{-n}\|w^{n_{\mathrm{in}}+n}\|_{L_{\mu^{\mathrm{sad}}}^2},\ b=\frac{\tau \Delta}{1-\tau \frac{\lambda_0}{\sqrt{\lambda_0^2 + \beta}}},\ d=\eta|c|.
    \end{equation*}
    Observe that $a_{n_{\mathrm{in}}} = \|w^{n_{\mathrm{in}}}\|_{L_{\mu^{\mathrm{sad}}}^2} = \eta|c| = d$. Thus, since $b>0$, it can be proved by induction on $n$ that the inequality
    \begin{equation*}
        a_{n_{\mathrm{in}}+n} \leq d + b\sum_{k=0}^{n-1} a_{n_{\mathrm{in}}+k},
    \end{equation*}
    obtained from \eqref{eq:difference wn}, implies
    \begin{equation*}
        a_{n_{\mathrm{in}}+n} \leq d(1+b)^n,
    \end{equation*}
    for all $n \in \{0,..., n_{\mathrm{out}}\}$. Therefore,
    \begin{equation}
    \label{upper bound wn}
        \left(1-\tau \frac{\lambda_0}{\sqrt{\lambda_0^2 + \beta}}\right)^{-n}\|w^{n_{\mathrm{in}}+n}\|_{L_{\mu^{\mathrm{sad}}}^2} \leq \eta|c|\left(1+\frac{\tau \Delta}{1-\tau \frac{\lambda_0}{\sqrt{\lambda_0^2 + \beta}}}\right)^n.
    \end{equation}
From this, using \eqref{eq:difference wn} and \eqref{upper bound wn} gives
    \begin{align*}
        \left(1-\tau \frac{\lambda_0}{\sqrt{\lambda_0^2 + \beta}}\right)^{-n}\|w^{n_{\mathrm{in}}+n}\|_{L^2_{\mu^{\mathrm{sad}}}}&\geq \eta|c|- \frac{\tau \Delta}{1-\tau \frac{\lambda_0}{\sqrt{\lambda_0^2 + \beta}}} \sum_{k=0}^{n-1} \left(1-\tau \frac{\lambda_0}{\sqrt{\lambda_0^2 + \beta}}\right)^{-k}\|w^{n_{\mathrm{in}}+k}\|_{L_{\mu^{\mathrm{sad}}}^2}\\
        &\geq \eta|c| - \frac{\tau \Delta}{1-\tau \frac{\lambda_0}{\sqrt{\lambda_0^2 + \beta}}} \eta|c|\sum_{k=0}^{n-1} \left(1+\frac{\tau \Delta}{1-\tau \frac{\lambda_0}{\sqrt{\lambda_0^2 + \beta}}}\right)^k\\
        &= \eta|c|- \frac{\tau \Delta}{1-\tau \frac{\lambda_0}{\sqrt{\lambda_0^2 + \beta}}} \eta|c|\frac{\left(1+\frac{\tau \Delta}{1-\tau \frac{\lambda_0}{\sqrt{\lambda_0^2 + \beta}}}\right)^{n} - 1}{\frac{\tau \Delta}{1-\tau \frac{\lambda_0}{\sqrt{\lambda_0^2 + \beta}}}}\\
        &=\eta|c|\left(2- \left(1+\frac{\tau \Delta}{1-\tau \frac{\lambda_0}{\sqrt{\lambda_0^2 + \beta}}}\right)^{n}\right)\\
        &\geq \eta|c|\left(2- \left(1+\tau \Delta\right)^{n}\right)\\
        &\geq \eta|c|\left(2-\exp{\left(n_{\mathrm{out}}\tau \Delta\right)}\right),
    \end{align*}
    where the penultimate inequality follows from the fact that $1-\tau \frac{\lambda_0}{\sqrt{\lambda_0^2 + \beta}} \geq 1$ since $\lambda_0 <-\delta < 0$, and the last inequality follows from the standard inequality $(1+x)^n \leq e^{nx}$.

    The aim is now to prove that $\tau_{\mathrm{sad}}n_{\mathrm{out}}\Delta \leq \log\frac{3}{2}.$ We use the coupled scaling $\beta_{\mathrm{sad}}=b\delta^2$ for $b > 0$ and keep all dependence on $\beta_{\mathrm{sad}}$ explicit. Under the contradiction hypothesis, for every $n\in\{0,\ldots,n_{\mathrm{out}}\}$,
\begin{equation}
\label{eq:new-radius-first}
    \left\|Z^{n_{\mathrm{in}}+n}-\operatorname{Id}\right\|_{L^2_{\mu^{\mathrm{sad}}}}\leq A+B.
\end{equation}
Similarly, since $|c|\leq2\kappa$,
\begin{align}
\label{eq:new-radius-second}
    \left\|\widetilde Z^{n_{\mathrm{in}}+n}-\operatorname{Id}\right\|_{L^2_{\mu^{\mathrm{sad}}}} \leq\eta\kappa+\eta|c|+B \leq 3A+B.
\end{align}
We first verify that the local bootstrap in Lemma \ref{lem:episode-radius-bootstrap} closes. From the definition of $\eta$,
\begin{align}
\label{eq:new-A-exact}
    A = \eta\kappa = \frac{2F_0}{\varepsilon+\sqrt{\varepsilon^2+2C_{\operatorname H}F_0}} =\frac{\sqrt{\varepsilon^2+2C_{\operatorname H}F_0} - \varepsilon}{C_{\operatorname H}} \leq \sqrt{\frac{2F_0}{C_{\operatorname H}}}.
\end{align}
Using the definition of $F_0$, we obtain
\begin{equation}
\label{eq:new-A-upper}
    A \leq \sqrt{\frac{2c_F}{C_{\operatorname H}}}\frac{\log(3/2)\beta_{\mathrm{sad}}^{3/4}}{L_{\operatorname H}(\tau_{\mathrm{sad}}n_{\mathrm{out}})^{3/2}}.
\end{equation}
Moreover,
\begin{equation}
\label{eq:new-B-exact}
    B = 2\sqrt{c_F}\frac{\log(3/2)\sqrt{\beta_\mathrm{sad}}}{L_{\operatorname H}\tau_\mathrm{sad}n_{\mathrm{out}}}.
\end{equation}
Hence,
\begin{align*}
    3A+B \leq \frac{\sqrt{\beta_\mathrm{sad}}}{L_{\operatorname H}}
    \sqrt{c_F}\log\frac{3}{2}\left(\frac{2}{\tau_\mathrm{sad}n_{\mathrm{out}}} +
        \frac{ 3\sqrt2\beta_\mathrm{sad}^{1/4}
        }{\sqrt{C_{\operatorname H}}(\tau_\mathrm{sad}n_{\mathrm{out}})^{3/2}
        }
    \right) \leq \frac{\sqrt{\beta_\mathrm{sad}}}{2L_{\operatorname H}},
\label{eq:new-bootstrap-closed}
\end{align*}
where the final inequality follows from
$\tau_\mathrm{sad}n_{\mathrm{out}}\geq\tau_\mathrm{sad}$, $\beta_\mathrm{sad}\leq1$, and \eqref{eq:cF-bootstrap-condition}. Therefore both trajectories remain inside the ball of radius $\sqrt{\beta_\mathrm{sad}}/L_{\operatorname H}$ throughout the episode. Lemma \ref{lem:local-wsfn-descent} hence applies up to time $n_{\mathrm{out}}$ by the first exit argument.

Applying Lemma \ref{lem:preconditioned-gradient-expansion-sharp} at the reference measure $\mu^{\mathrm{sad}}$, and using $\left\|\nabla_\mu F(\mu^{\mathrm{sad}},\cdot)\right\|_{L^2_{\mu^{\mathrm{sad}}}}\leq \varepsilon,$ \eqref{eq:new-radius-first} and \eqref{eq:new-radius-second} imply
\begin{equation*}
\left\|R_\beta\left(Z^{n_{\mathrm{in}}+n},\widetilde Z^{n_{\mathrm{in}}+n}\right)\right\|_{L^2_{\mu^{\mathrm{sad}}}} \leq
\Delta\left\|w^{n_{\mathrm{in}}+n}\right\|_{L^2_{\mu^{\mathrm{sad}}}},
\end{equation*}
where
\begin{align*}
    \Delta = \frac{8L_{\operatorname H}A}{\sqrt{\beta_\mathrm{sad}}}
    +\frac{4L_{\operatorname H}B}{\sqrt{\beta_\mathrm{sad}}} +
    \frac{L_{\operatorname H}^2}{2\pi\beta_\mathrm{sad}}\left(10A^2+8AB+2B^2\right)+ \varepsilon\left(\frac{R_F}{\sqrt{\beta_\mathrm{sad}}} +\frac{2L_{\operatorname H}}{\pi\beta_\mathrm{sad}}\right).
\end{align*}
We now estimate $\tau_\mathrm{sad}n_{\mathrm{out}}\Delta$. By \eqref{eq:new-A-upper} and \eqref{eq:new-B-exact},
\begin{align*}
    \tau_\mathrm{sad}n_{\mathrm{out}}\Delta &\leq
    8\sqrt{c_F}\log \frac32 + 8\sqrt{\frac{2c_F}{C_{\operatorname H}}}
    \frac{\log(3/2)\beta_\mathrm{sad}^{1/4}}{(\tau_\mathrm{sad}n_{\mathrm{out}})^{1/2}}+\frac{10c_F\log^2(3/2)}{\pi C_{\operatorname H}}\frac{\sqrt{\beta_\mathrm{sad}}}{(\tau_\mathrm{sad}n_{\mathrm{out}})^2}\\
    &+
    \frac{8\sqrt2c_F\log^2(3/2)}{\pi\sqrt{C_{\operatorname H}}}
    \frac{\beta_\mathrm{sad}^{1/4}}{(\tau_\mathrm{sad}n_{\mathrm{out}})^{3/2}}
    +\frac{4c_F\log^2(3/2)}{\pi \tau_\mathrm{sad}n_{\mathrm{out}}}+
    \tau_\mathrm{sad}n_{\mathrm{out}}\varepsilon
    \left(\frac{R_F}{\sqrt{\beta_{\mathrm{sad}}}}+\frac{2L_{\operatorname H}}{\pi\beta_\mathrm{sad}}\right).
\end{align*}
Since $\beta_\mathrm{sad}\leq1$ and $\tau_\mathrm{sad}n_{\mathrm{out}}\geq\tau_\mathrm{sad}$, the first five terms on the right-hand side are bounded above by
\begin{align*}
    \log\frac{3}{2}\Bigg[8\sqrt{c_F}+ 8\sqrt{\frac{2c_F}{C_{\operatorname H}\tau_\mathrm{sad}}}+ c_F\log\left(\frac{3}{2}\right)\left(\frac{10}{\pi C_{\operatorname H}\tau_\mathrm{sad}^2}+\frac{8\sqrt2}{\pi\sqrt{C_{\operatorname H}}\,\tau_\mathrm{sad}^{3/2}}+\frac{4}{\pi\tau_\mathrm{sad}}\right)\Bigg]\leq\frac{3}{4}\log \frac32,
\end{align*}
where we used \eqref{eq:cF-remainder-condition}. Moreover, \eqref{eq:new-epsilon-condition} gives
\begin{equation*}
    \tau_\mathrm{sad}n_{\mathrm{out}}\varepsilon
    \left(\frac{R_F}{\sqrt{\beta_{\mathrm{sad}}}}+\frac{2L_{\operatorname H}}{\pi\beta_\mathrm{sad}}
    \right)\leq\frac{1}{4}\log\frac 32.
\end{equation*}
Thus,
\begin{equation*}
    \tau_\mathrm{sad}n_{\mathrm{out}}\Delta \leq  \log\frac{3}{2}.
\end{equation*}

Let $\rho_0 := 1- \tau_\mathrm{sad}\frac{\lambda_0}{\sqrt{\lambda_0^2+\beta_\mathrm{sad}}}.$ Since $\lambda_0<0$, we have $\rho_0>1$. The discrete perturbation estimate derived from the recurrence for
$w^{n_{\mathrm{in}}+n}$ therefore gives, for every $n\leq n_{\mathrm{out}}$,
\begin{align*}
    \rho_0^{-n} \left\| w^{n_{\mathrm{in}}+n}\right\|_{L^2_{\mu^{\mathrm{sad}}}}
    &\geq \eta|c|\left(2-\exp(\tau_\mathrm{sad}n_{\mathrm{out}}\Delta)\right) \geq \eta|c|\left(2-\exp\left(\log\frac32\right)\right) = \frac{\eta|c|}{2}.
\end{align*}
Since $\lambda_0\leq-\delta$,
\begin{align*}
    \rho_0 =1+ \tau_\mathrm{sad}\frac{|\lambda_0|}{\sqrt{\lambda_0^2+\beta_\mathrm{sad}}}
    \geq 1+ \tau_\mathrm{sad} \frac{\delta}{\sqrt{\delta^2+\beta_\mathrm{sad}}} = \varrho.
\end{align*}
Therefore,
\begin{equation}
\label{eq:new-lower-w}
    \left\|w^{n_{\mathrm{in}}+n}\right\|_{L^2_{\mu^{\mathrm{sad}}}}\geq \frac{\eta|c|}{2} \varrho^n.
\end{equation}

On the other hand, the triangle inequality in $L^2_{\mu^{\mathrm{sad}}}$ and
\eqref{eq:new-radius-first}, \eqref{eq:new-radius-second} give
\begin{align*}
    \left\|w^{n_{\mathrm{in}}+n}\right\|_{L^2_{\mu^{\mathrm{sad}}}} \leq 4A+2B = 4\eta\kappa + 4\sqrt{\tau_\mathrm{sad}n_{\mathrm{out}}}\beta_\mathrm{sad}^{-1/4}F_0^{1/2}.
\end{align*}
Furthermore, by \eqref{eq:new-A-exact}, $A\leq\sqrt{\frac{2F_0}{C_{\operatorname H}}}.$ The smallness condition \eqref{eq:new-small-beta-condition}, together with $\tau_\mathrm{sad}n_{\mathrm{out}}\geq\tau_\mathrm{sad}$, implies
\begin{equation*}
    A \leq \sqrt{\tau_\mathrm{sad}n_{\mathrm{out}}}\beta_\mathrm{sad}^{-1/4}F_0^{1/2}.
\end{equation*}
Hence
\begin{equation}
\label{eq:new-upper-w}
    \left\|w^{n_{\mathrm{in}}+n}\right\|_{L^2_{\mu^{\mathrm{sad}}}}
    \leq 8 \sqrt{\tau_\mathrm{sad}n_{\mathrm{out}}}\beta_\mathrm{sad}^{-1/4}F_0^{1/2}.
\end{equation}

We now evaluate the lower and upper bounds at $n=n_{\mathrm{out}}$. From
\eqref{eq:new-lower-w},
\begin{equation*}
\left\|w^{n_{\mathrm{in}}+n_{\mathrm{out}}}\right\|_{L^2_{\mu^{\mathrm{sad}}}}
    \geq \frac{\eta|c|}{2}\varrho^{n_{\mathrm{out}}}.
\end{equation*}
For every $x>0$ and $m\geq1$,
\begin{equation*}
    (1+x)^m \geq \sqrt{em\log(1+x)}(1+x)^{m/2}.
\end{equation*}
Applying this inequality with $x=\tau_\mathrm{sad}\widetilde\delta$ and $m=n_{\mathrm{out}}$ gives
\begin{equation}
\label{eq:new-refined-lower}
\begin{aligned}
\left\|w^{n_{\mathrm{in}}+n_{\mathrm{out}}}\right\|_{L^2_{\mu^{\mathrm{sad}}}}
    \geq\frac{\eta|c|}{2}\sqrt{en_{\mathrm{out}}\log\varrho
    }\varrho^{n_{\mathrm{out}}/2}.
\end{aligned}
\end{equation}
Combining \eqref{eq:new-refined-lower} and \eqref{eq:new-upper-w} gives
\begin{equation}
\label{eq:new-isolated-exp}
    \varrho^{n_{\mathrm{out}}/2} \leq \frac{16\sqrt{\tau_\mathrm{sad}}F_0^{1/2}}{\eta|c|\sqrt e\beta_\mathrm{sad}^{1/4}\log^{1/2}\varrho}.
\end{equation}

It remains to eliminate $F_0^{1/2}/\eta$. By \eqref{eq:new-epsilon-condition}, $\varepsilon \leq \frac34 \sqrt{2C_{\operatorname H}F_0}.$ Therefore,
\begin{align*}
    \eta\kappa = \frac{2F_0}{\varepsilon+\sqrt{\varepsilon^2+2C_{\operatorname H}F_0}}
    \geq \frac12\sqrt{\frac{2F_0}{C_{\operatorname H}}}.
\end{align*}
Equivalently,
\begin{equation*}
\label{eq:new-F0-over-eta}
    \frac{F_0^{1/2}}{\eta}\leq \kappa \sqrt{2C_{\operatorname H}}.
\end{equation*}
Substituting this estimate into \eqref{eq:new-isolated-exp}, and using
$|c|\geq c_{\mathrm{ep}}$, gives
\begin{equation*}
    \varrho^{n_{\mathrm{out}}/2} \leq
    \frac{16\sqrt{2C_{\operatorname H}\tau_\mathrm{sad}}\kappa}{\sqrt e \beta_\mathrm{sad}^{1/4}c_{\mathrm{ep}}\log^{1/2}\varrho}.
\end{equation*}
Taking logarithms, we obtain
\begin{equation*}
    n_{\mathrm{out}} \leq \frac{2}{\log\varrho}\log\left(
        \frac{16\sqrt{2C_{\operatorname H}\tau_\mathrm{sad}}\kappa}{\sqrt e\beta_\mathrm{sad}^{1/4}c_{\mathrm{ep}}\log^{1/2}\varrho}\right).
\end{equation*}
This contradicts the definition \eqref{eq:new-def-nout} of $n_{\mathrm{out}}$. The contradiction proves that at least one of the two trajectories decreases the objective by at least $2F_0$ within
$n_{\mathrm{out}}$ iterations.   
\end{proof}
The previous proposition gives a deterministic mechanism for escaping a saddle once the perturbation has a sufficiently large component in the unstable eigendirection. We now formalize this escape phase as a saddle-point episode and quantify its success probability.

Informally, a saddle-point episode begins when the iterate is near an $(\varepsilon,\delta)$-saddle and a perturbation is injected. The episode is said to be successful if, after running the WSFN iteration for the prescribed number of steps, the objective has decreased by at least a fixed amount. 

\begin{definition}[Perturbation episode with early acceptance]
\label{def:saddle-point-episode}
Let $\bar\mu\in\mathcal P_2(\mathbb R^d)$ satisfy $\|\nabla_\mu F(\bar\mu,\cdot)\|_{L^2_{\bar\mu}}\leq\varepsilon$. Draw $\xi\sim\mathcal N_{L^2_{\bar\mu}}(0,\operatorname{K}_{\bar\mu}^2),$ $\nu^0:=(\operatorname{Id}+\eta\xi)_\#\bar\mu.$ Define the auxiliary trajectory
\begin{equation*}
     \nu^{\ell+1} := \left(\operatorname{Id}-\tau_{\mathrm{sad}}\left(\operatorname{H}_{\nu^\ell}^2+\beta_{\mathrm{sad}}\operatorname{I_{d\times d}}\right)^{-1/2}\nabla_\mu F(\nu^\ell,\cdot)\right)_\#\nu^\ell, \quad
    0\leq\ell<n_{\mathrm{out}}.
\end{equation*}
Define the episode stopping index by
\begin{equation*}
    \ell_{\mathrm{ep}} := \min\left(\left\{\ell\in\{0,\ldots,n_{\mathrm{out}}\}: F(\bar\mu)-F(\nu^\ell)\geq F_0\right\} \cup\{n_{\mathrm{out}}\}\right), \quad \mu_{\mathrm{out}}:=\nu^{\ell_{\mathrm{ep}}}.
\end{equation*}
The algorithm computes only the prefix $\nu^0,\ldots,\nu^{\ell_{\mathrm{ep}}}$.
The episode is successful if $F(\bar\mu)-F(\mu_{\mathrm{out}})\geq F_0.$ Upon success, the algorithm accepts $\mu_{\mathrm{out}}$ and returns to the gradient norm test. Otherwise, it terminates and returns the stored measure $\bar\mu$. When $\bar\mu$ is an $(\varepsilon,\delta)$-saddle point,
we call this a saddle point episode. In the notation of Proposition~\ref{discrete_lower_bound_metric}, $\bar\mu=\mu^{\mathrm{sad}}$ and $\nu^\ell=\mu^{n_{\mathrm{in}}+\ell}$.
\end{definition}
\begin{proposition}[Escape from saddle point episode]
    \label{discrete_decrease_around_saddle}
   Let Assumptions~\ref{assumption:smooth-F} and~\ref{assumption:Hessian-lipschitz} hold. Let $\mu^{\mathrm{sad}}\in\mathcal P_2(\mathbb R^d)$ satisfy $\|\nabla_\mu F(\mu^{\mathrm{sad}},\cdot)\|_{L^2_{\mu^{\mathrm{sad}}}}\le\varepsilon,$ $\lambda_0:=\underline{\lambda}(\operatorname{K}_{\mu^{\mathrm{sad}}})<-\delta.$ Assume that, for some $R_F>0$, $\|\operatorname{M}_{\mu^{\mathrm{sad}}}\|_{\mathrm{op}}\le R_F\|\nabla_\mu F(\mu^{\mathrm{sad}},\cdot)\|_{L^2_{\mu^{\mathrm{sad}}}}.$ Let $q_0\in L^2_{\mu^{\mathrm{sad}}}$ be a unit eigenvector of $\operatorname{K}_{\mu^{\mathrm{sad}}}$ corresponding to $\lambda_0$. Fix \(\zeta_{\mathrm{ep}}\in(0,1)\), and choose parameters $\beta$, $\tau$, $\kappa$, $n_\mathrm{out}$, $F_0$ and $\eta$ as in \eqref{def_beta_tau_sad}--\eqref{eq:new-def-eta}. Suppose that the small gradient condition \ref{eq:new-epsilon-condition} holds. Let $\xi\sim\mathcal N_{L^2_{\mu^{\mathrm{sad}}}}(0,\operatorname{K}_{\mu^{\mathrm{sad}}}^2)$, and define the WSFN trajectory $\mu^{n_{\mathrm{in}}+n}$, $0\leq n\leq n_{\mathrm{out}}$, initialized at $\mu^{n_{\mathrm{in}}}:=(\operatorname{Id}+\eta\xi)_\#\mu^{\mathrm{sad}}.$ Let $\ell_{\mathrm{ep}}$ be the stopping index from Definition~\ref{def:saddle-point-episode}, with $\bar\mu=\mu^{\mathrm{sad}}$ and $\nu^\ell=\mu^{n_{\mathrm{in}}+\ell}$, and set $\mu_{\mathrm{out}}:=\mu^{n_{\mathrm{in}}+\ell_{\mathrm{ep}}}.$ Then
   \begin{equation*}
       \mathbb P\left(F(\mu^{\mathrm{sad}})-F(\mu_{\mathrm{out}})\geq F_0\right) \geq 1-\frac{3\zeta_{\mathrm{ep}}}{4}.
   \end{equation*}
\end{proposition}
\begin{proof}
    By Lemma \ref{lem:Gaussian-L2}, for every $t > 0,$
\begin{equation*}
    \mathbb P\left(\|\xi\|_{L^2_\mu}^2 > \mathrm{Tr}(\operatorname{K}_\mu^2) +2 \sqrt{\mathrm{Tr}(\operatorname{K}_\mu^4)t} + 2\|\operatorname{K}_\mu\|_{\mathrm{op}}^2t\right) \leq e^{-t}.
\end{equation*}
Since $\mathrm{Tr}(\operatorname{K}_\mu^4) \leq \|\nabla_\mu^2F(\mu,\cdot,\cdot)\|_{L^2_{\mu \otimes \mu}}^4 \leq C_{\operatorname{K}}^4,$ $\|\operatorname{K}_\mu\|_{\mathrm{op}}^2 \leq \|\nabla_\mu^2F(\mu,\cdot,\cdot)\|_{L^2_{\mu \otimes \mu}}^2 \leq C_{\operatorname{K}}^2,$ and setting $\kappa$ as in \eqref{eq:def_kappa}, i.e., 
\begin{equation*}
    \kappa \geq C_{\operatorname{K}}\left(1+2\sqrt{\log \frac{4}{\zeta_{\mathrm{ep}}}}+2\log \frac{4}{\zeta_{\mathrm{ep}}}\right)
\end{equation*}
    it follows from the choice $t=\log(4/\zeta_{\mathrm{ep}})$ that
    \begin{equation*}
        \mathbb P \left(\|\xi\|_{L^2_{\mu^{\mathrm{sad}}}} > \kappa\right) \leq \frac{\zeta_{\mathrm{ep}}}{4}.
    \end{equation*}
    Thus $\|\xi\|_{L^2_{\mu^{\mathrm{sad}}}} \leq \kappa$ with probability at least $1 - \frac{\zeta_{\mathrm{ep}}}{4}$. 

    We apply Proposition~\ref{discrete_lower_bound_metric} to the trajectory continued up to $n_{\mathrm{out}}$ steps. Define
    \begin{equation*}
        \mathcal{E}:=\left\{\exists n\in\{1,\ldots,n_{\mathrm{out}}\}: F(\mu^{n_{\mathrm{in}}})-F(\mu^{n_{\mathrm{in}}+n})\geq 2F_0\right\}.
    \end{equation*}
    We first prove that $\mathbb P(\mathcal E)\geq 1-\frac{\zeta_{\mathrm{ep}}}{2}.$ Let $\mathcal{A} = \mathcal{E}^c$ be the event that the scheme does not escape the saddle point, meaning $F(\mu^{n_{\mathrm{in}}}) - F(\mu^{n_{\mathrm{in}}+n}) < 2 F_0$ for all $n \in \{1, ..., n_{\mathrm{out}}\}$. Define the bounded norm event by
\[
    \mathcal{B} := \left\{ \|\xi\|_{L_{\mu^{\mathrm{sad}}}^2} \leq \kappa \right\}.
\]
We have $\mathbb{P}(\mathcal{B}) \geq 1 - \frac{\zeta_{\mathrm{ep}}}{4}$. Using the union bound, we get
\[
    \mathbb{P}(\mathcal{E}) \geq \mathbb{P}(\mathcal{E} \cap \mathcal{B}) = \mathbb{P}(\mathcal{B}) - \mathbb{P}(\mathcal{A} \cap \mathcal{B}).
\]
To lower bound $\mathbb{P}(\mathcal{E})$, we must upper bound $\mathbb{P}(\mathcal{A} \cap \mathcal{B})$. We decompose the perturbation $\xi$ into a component along the unit eigenvector $q_0$ and an orthogonal component $\xi_\perp$ so that
\[
    \xi = \langle q_0, \xi \rangle_{L_{\mu^{\mathrm{sad}}}^2} q_0 + \xi_\perp,
\]
and $\langle q_0, \xi_\perp \rangle_{L_{\mu^{\mathrm{sad}}}^2} = 0$.

For each fixed orthogonal component $\xi_\perp$, we define the ``stuck region'' $I_{\xi_\perp}$ as the set scalars for which the perturbation $\xi$ remains bounded and the scheme fails to escape, that is
\[
    I_{\xi_\perp} := \left\{ x \in \mathbb R| x q_0 + \xi_\perp \in \mathcal{A} \cap \mathcal{B} \right\}.
\]
Take $x_1, x_2 \in I_{\xi_\perp}$. Let $\xi_1 = x_1 q_0 + \xi_\perp$ and $\xi_2 = x_2 q_0 + \xi_\perp$. Let $c = x_2 - x_1$, so $\xi_2 = \xi_1 + c q_0$. By definition, both $\xi_1$ and $\xi_2$ are in $\mathcal{B}$, meaning $\|\xi_1\|_{L_{\mu^{\mathrm{sad}}}^2} \leq \kappa$ and $\|\xi_2\|_{L_{\mu^{\mathrm{sad}}}^2} \leq \kappa$. The triangle inequality gives
\[
    |c| = |x_2 - x_1| \leq 2\kappa.
\]
According to Proposition \ref{discrete_lower_bound_metric}, if $\|\xi_1\|_{L_{\mu^{\mathrm{sad}}}^2} \leq \kappa$ and $c_{\mathrm{ep}} \leq |c| \leq 2\kappa$, then at least one of the initializations ($\mu^{n_{\mathrm{in}}}$ or $\tilde{\mu}^{n_{\mathrm{in}}}$) must escape. However, since we assumed both $x_1$ and $x_2$ are in the stuck region $\mathcal{A}$, this escape condition cannot hold. Because we already know $|c| \leq 2\kappa$, it must be the case that
\[
    |x_2 - x_1| < c_{\mathrm{ep}}.
\]
This implies that the maximum distance between any two points in $I_{\xi_\perp}$ is bounded. Thus, $I_{\xi_\perp}$ is contained within an interval of length $c_{\mathrm{ep}}$. By Lemma \ref{lem:independence}, $\langle q_0, \xi \rangle_{L_{\mu^{\mathrm{sad}}}^2} \sim \mathcal{N}(0, \lambda_0^2)$, with density
\[
    p(x) = \frac{1}{\sqrt{2\pi}|\lambda_0|} \exp\left(-\frac{x^2}{2\lambda_0^2}\right),
\]
and $\langle q_0, \xi \rangle_{L_{\mu^{\mathrm{sad}}}^2}$ is independent from $\xi_\perp$. The density is bounded by its value at $x = 0$, so $p(x) \leq \frac{1}{\sqrt{2\pi}|\lambda_0|}$. Due to independence, we can now bound the conditional probability of $\langle q_0, \xi \rangle_{L_{\mu^{\mathrm{sad}}}^2}$ falling into the stuck region by 
\[
     \mathbb{P}(\mathcal{A} \cap \mathcal{B}|\xi_\perp) = \mathbb{P}(\langle q_0, \xi \rangle_{L_{\mu^{\mathrm{sad}}}^2} \in I_{\xi_\perp}| \xi_\perp) = \int_{I_{\xi_\perp}} p(x) \mathrm dx \leq \frac{1}{\sqrt{2\pi}|\lambda_0|} c_{\mathrm{ep}} = \frac{\delta}{|\lambda_0|}\frac{\zeta_{\mathrm{ep}}}{4} <\frac{\zeta_{\mathrm{ep}}}{4},
\]
since $|\lambda_0| > \delta.$ Taking the expectation over the orthogonal component $\xi_\perp$, the unconditional probability of being in $\mathcal{A} \cap \mathcal{B}$ is
\[
    \mathbb{P}(\mathcal{A} \cap \mathcal{B}) = \mathbb{E}_{\xi_\perp} \left[ \mathbb{P}(\langle q_0, \xi \rangle_{L_{\mu^{\mathrm{sad}}}^2} \in I_{\xi_\perp}|\xi_\perp) \right] \leq \frac{\zeta_{\mathrm{ep}}}{4}.
\]
We substitute this back into our initial probability bound
\[
    \mathbb{P}(\mathcal{E}) \geq \mathbb{P}(\mathcal{B}) - \mathbb{P}(\mathcal{A} \cap \mathcal{B}) \geq \left(1 - \frac{\zeta_{\mathrm{ep}}}{4}\right) - \frac{\zeta_{\mathrm{ep}}}{4} = 1 - \frac{\zeta_{\mathrm{ep}}}{2}.
\]
Define the perturbation cost event $\mathcal D:=\left\{F(\mu^{n_{\mathrm{in}}})-F(\mu^{\mathrm{sad}})\leq F_0\right\}.$ By Lemma~\ref{not_very_increase}, $\mathbb P(\mathcal D)\geq 1-\frac{\zeta_{\mathrm{ep}}}{4}.$ On $\mathcal E\cap \mathcal D$, there exists $n\in\{1,\ldots,n_{\mathrm{out}}\}$ such that
\begin{equation*}
    F(\mu^{\mathrm{sad}})-F(\mu^{n_{\mathrm{in}}+n})= F(\mu^{\mathrm{sad}})-F(\mu^{n_{\mathrm{in}}}) + F(\mu^{n_{\mathrm{in}}})-F(\mu^{n_{\mathrm{in}}+n}) \geq -F_0+2F_0=F_0.
\end{equation*}
Hence the set of indices satisfying the acceptance criterion is nonempty. By the definition of $\ell_{\mathrm{ep}}$, the algorithm accepts an iterate no later than this index, and its output satisfies
\begin{equation*}
     F(\mu^{\mathrm{sad}})-F(\mu_{\mathrm{out}})\geq F_0.
\end{equation*}
Therefore,
\begin{equation*}
    \mathbb P\left(F(\mu^{\mathrm{sad}})-F(\mu_{\mathrm{out}})\geq F_0\right) \geq \mathbb P(\mathcal E\cap \mathcal D) \geq 1-\mathbb P(\mathcal E^c)-\mathbb P(\mathcal D^c) \geq 1-\frac{3\zeta_{\mathrm{ep}}}{4}.
\end{equation*}
\end{proof}

\subsection{Proofs of the convergence theorems}
\label{subsection:proofs_conv_theorems}

We are now ready to prove the main convergence results. The argument combines two facts: outside the saddle region, large gradient steps decrease the objective by a uniform amount, while inside the saddle region, successful saddle-point episodes also give a decrease. To turn this into a global theorem, the Algorithm \ref{alg:episode-wsfn} uses phase specific parameters $\beta_{\mathrm{sad}}= b\delta^2,$ $\tau=\tau_{\mathrm{sad}}\in (0,\frac{1}{2}]$ during saddle episodes, while the large gradient phase retains parameters satisfying its own global descent condition $\tau_{\mathrm{grad}} \leq C_{\operatorname H}^{-1}\sqrt{\beta_{\mathrm{grad}}}.$

\begin{theorem}[Polynomial time approximate second-order stationarity]
\label{thm:polynomial-time-sosp-full}
Let Assumptions~\ref{assumption:smooth-F} and~\ref{assumption:Hessian-lipschitz} hold. Fix $\mu^0\in\mathcal P_2(\mathbb R^d)$. Assume that Assumption~\ref{assumption_regularity_wg} holds on $S=\mathcal L_0$, with $\varepsilon_0, R_F>0$. Fix $0<\varepsilon\le\varepsilon_0$, $\delta>0$, $\beta_{\mathrm{grad}}>0$, and $\zeta\in(0,1)$, and set $F_{\min}:=F(\mu^0)-\inf_{\mu\in\mathcal P_2(\mathbb R^d)}F(\mu)>0.$ Run Algorithm~\ref{alg:episode-wsfn} with $0<\tau_{\mathrm{grad}}\le C_{\mathrm H}^{-1}\sqrt{\beta_{\mathrm{grad}}},$ and the episode parameters $\beta_{\mathrm{sad}}=\Theta(\delta^2),$ $\tau_{\mathrm{sad}}\in (0,\frac{1}{2}],$ $\kappa=\tilde O(1), n_\mathrm{out}=\tilde O(1), F_0=\tilde \Omega(\delta^3), \eta=\tilde \Theta(\frac{\delta^3}{\varepsilon+\delta^{3/2}}),$ prescribed in \eqref{def_beta_tau_sad}-\eqref{eq:new-def-eta}, subject to \eqref{eq:new-epsilon-condition}, using the same constant $R_F$ at every episode. Choose these parameters jointly with $M_{\max}\in\mathbb N$ and $\zeta_{\mathrm{ep}}\in(0,1)$
such that $M_{\max} \ge \left\lfloor\frac{F_{\min}}{F_0}\right\rfloor+1,$ and $\frac34M_{\max}\zeta_{\mathrm{ep}}\le\zeta.$ At every initiated episode, draw a fresh Gaussian perturbation conditionally on the preceding history, with covariance $\operatorname{K}_{\bar\mu}^2$ on $L^2_{\bar\mu}$, where $\bar\mu$ is the episode's starting measure. Then the algorithm terminates and, with probability at least $1-\zeta$, returns an $(\varepsilon,\delta)$-second-order critical point $\widehat\mu$. The number $N_{\mathrm{SOSP}}$ of WSFN transport steps satisfies
\begin{equation}
\label{eq:sosp-exact-complexity}
    N_{\mathrm{SOSP}} \leq \frac{2(C_{\operatorname H}^2+\beta_{\mathrm{grad}})}{\tau_{\mathrm{grad}}\sqrt{\beta_{\mathrm{grad}}}}\frac{F_{\min}}{\varepsilon^2} + n_{\mathrm{out}}\left(\frac{F_{\min}}{F_0}+1\right).
\end{equation}
Hence, with the large gradient parameters and other problem constants fixed,
\begin{equation}
\label{eq:sosp-soft-complexity}
    N_{\mathrm{SOSP}} = \tilde O\left(F_{\min}\left(\frac{1}{\varepsilon^2} +\frac{1}{\delta^3}\right)\right).
\end{equation}
\end{theorem}

\begin{proof}
We first verify that Proposition~\ref{discrete_decrease_around_saddle} applies uniformly whenever an episode is initiated at an $(\varepsilon,\delta)$-saddle. Every accepted outer update decreases the objective. This follows from Lemma~\ref{descent lemma} for a large gradient step, and from the acceptance criterion $ F(\bar\mu)-F(\mu^{\mathrm{out}})\ge F_0$ for a successful perturbation episode. Hence, every outer iterate, and every episode starting measure, belongs to $\mathcal L_0$.

Let $\bar\mu_j$ be the starting measure of an initiated $j$-th episode, and let $\mathcal F_j$ denote the history immediately before its perturbation is drawn. By the triggering condition, $\|\nabla_\mu F(\bar\mu_j,\cdot)\|_{L^2_{\bar\mu_j}}\le\varepsilon\le\varepsilon_0.$ Since $\bar\mu_j\in\mathcal L_0$, Assumption~\ref{assumption_regularity_wg} gives
\begin{equation*}
    \|\operatorname{M}_{\bar\mu_j}\|_{\mathrm{op}} \le R_F \|\nabla_\mu F(\bar\mu_j,\cdot)\|_{L^2_{\bar\mu_j}}.
\end{equation*}
The constant $R_F$ is independent of $j$ and of the realized history. Conditionally on $\mathcal F_j$, the reference measure $\bar\mu_j$ is fixed and the perturbation has the Gaussian law required by Proposition~\ref{discrete_decrease_around_saddle}. Therefore, on the event
$\underline{\lambda}(\operatorname{K}_{\bar\mu_j})<-\delta$, Proposition~\ref{discrete_decrease_around_saddle} implies
\begin{equation}
\label{eq:conditional-episode-success}
    \mathbb P\left(F(\bar\mu_j)-F(\mu_j^{\mathrm{out}})\ge F_0|\mathcal F_j\right) \ge 1-\frac{3\zeta_{\mathrm{ep}}}{4}.
\end{equation}
Further, we divide the proof into five steps. 

\emph{Step 1: uniform decrease at every large gradient step.} 
For \(\mu\in\mathcal P_2(\mathbb R^d)\), define 
\[P_\mu := \left(\operatorname{H}_\mu^2+\beta_\mathrm{grad}\operatorname{I_{d\times d}} \right)^{-1/2}.
\] 
By Lemma \ref{descent lemma}, every WSFN step satisfies 
\begin{equation} 
\label{eq:sosp-descent-lemma} 
F(\mu^+) \leq F(\mu) - \frac{\tau_\mathrm{grad}\sqrt{\beta_\mathrm{grad}}}{2} \left\| P_\mu \nabla_\mu F(\mu,\cdot) \right\|_{L^2_\mu}^2. 
\end{equation} 
Assumption \ref{assumption:smooth-F} and Lemma \ref{lem:uniform_bound_tildeH} imply $\|\operatorname{H}_\mu\|_{\mathrm{op}} \leq C_{\operatorname{H}}.$ Therefore, $\operatorname{H}_\mu^2 + \beta_\mathrm{grad}\operatorname{I_{d\times d}} \preceq (C_{\operatorname{H}}^2+\beta_\mathrm{grad})\operatorname{I_{d\times d}}.$ Since both sides are strictly positive, inversion reverses the operator order, so that $\left(\operatorname{H}_\mu^2 + \beta_\mathrm{grad}\operatorname{I_{d\times d}} \right)^{-1} \succeq \frac{1}{C_{\operatorname{H}}^2+\beta_\mathrm{grad}} \operatorname{I_{d\times d}}.$ Consequently, 
\[ 
\begin{aligned} 
\left\| P_\mu \nabla_\mu F(\mu,\cdot) \right\|_{L^2_\mu}^2 &= \left\langle \nabla_\mu F(\mu,\cdot), \left( \operatorname{H}_\mu^2 + \beta_\mathrm{grad}\operatorname{I_{d\times d}} \right)^{-1} \nabla_\mu F(\mu,\cdot) \right\rangle_{L^2_\mu} \\ &\geq \frac{1}{C_{\operatorname{H}}^2+\beta_\mathrm{grad}} \left\| \nabla_\mu F(\mu,\cdot) \right\|_{L^2_\mu}^2. \end{aligned} 
\] 
Substituting this estimate into \eqref{eq:sosp-descent-lemma} gives \begin{equation*} 
F(\mu)-F(\mu^+) \geq \frac{\tau_\mathrm{grad}\sqrt{\beta_\mathrm{grad}}} {2(C_{\operatorname{H}}^2+\beta_\mathrm{grad})} \left\| \nabla_\mu F(\mu,\cdot) \right\|_{L^2_\mu}^2. 
\end{equation*} 
In particular, every outer iterate satisfying $\left\| \nabla_\mu F(\mu,\cdot) \right\|_{L^2_\mu} > \varepsilon$ produces an objective decrease strictly larger than $G_0 := \frac{\tau_\mathrm{grad}\sqrt{\beta_\mathrm{grad}}} {2(C_{\operatorname{H}}^2+\beta_\mathrm{grad})} \varepsilon^2.$ Let \(N_{\mathrm{grad}}\) denote the total number of large gradient steps before termination. Since the objective is bounded below, $N_{\mathrm{grad}}G_0 \leq F(\mu^0) - \inf_{\mu}F(\mu) = F_{\min}.$ Hence 
\begin{equation} 
\label{eq:sosp-number-gradient} 
N_{\mathrm{grad}} \leq \frac{F_{\min}}{G_0} = \frac{2(C_{\operatorname{H}}^2+\beta_\mathrm{grad})} {\tau_\mathrm{grad}\sqrt{\beta_\mathrm{grad}}} \frac{F_{\min}}{\varepsilon^2}. 
\end{equation} 

\emph{Step 2: the number of successful episodes is finite.} Call an episode successful if its realized output satisfies $F(\bar\mu)-F(\mu^{\mathrm{out}}) \geq F_0.$ Whenever an episode is successful, the next outer iterate is \(\mu^{\mathrm{out}}\), and therefore the outer objective decreases by at least \(F_0\). Let \(N_{\mathrm{succ}}\) be the total number of successful episodes before termination. Telescoping the decreases over all successful episodes and all large gradient steps gives $N_{\mathrm{grad}}G_0 + N_{\mathrm{succ}}F_0 \leq F_{\min}.$ In particular, 
\begin{equation} 
\label{eq:sosp-number-successful-episodes} 
N_{\mathrm{succ}} \leq \frac{F_{\min}}{F_0}. 
\end{equation} 
Every outer stage that does not terminate is either a large gradient step or a successful episode. By \eqref{eq:sosp-number-gradient} and \eqref{eq:sosp-number-successful-episodes}, only finitely many such stages can occur. The algorithm must therefore eventually encounter an unsuccessful episode and terminate. There is exactly one unsuccessful episode, namely the final episode. Consequently, the total number \(N_{\mathrm{ep}}\) of initiated episodes satisfies 
\begin{equation} 
\label{eq:sosp-number-total-episodes} 
N_{\mathrm{ep}} = N_{\mathrm{succ}}+1 \leq \left\lfloor \frac{F_{\min}}{F_0} \right\rfloor+1 \leq M_{\max}. 
\end{equation} 
Let $\ell_{\mathrm{ep},j}\leq n_{\mathrm{out}}$ denote the number of WSFN transport steps performed in the $j$-th episode. Since each large gradient step uses one transport step,
\begin{equation*}
N_{\mathrm{SOSP}} = N_{\mathrm{grad}} + \sum_{j=1}^{N_{\mathrm{ep}}}\ell_{\mathrm{ep},j}\leq N_{\mathrm{grad}} + n_{\mathrm{out}}N_{\mathrm{ep}}.
\end{equation*}
Combining \eqref{eq:sosp-number-gradient}  and \eqref{eq:sosp-number-total-episodes} therefore gives
\begin{equation*}
    N_{\mathrm{SOSP}} \leq \frac{2(C_{\operatorname H}^2+\beta_{\mathrm{grad}})F_{\min}}{\tau_{\mathrm{grad}}\sqrt{\beta_{\mathrm{grad}}}\varepsilon^2}
    + n_{\mathrm{out}}\left(\frac{F_{\min}}{F_0}+1\right),
\end{equation*}
which is \eqref{eq:sosp-exact-complexity}. 

\emph{Step 3: simultaneous high probability control of all true saddle episodes.} 
If fewer than \(j\) episodes are initiated, define all events associated with the \(j\)-th episode to be empty. Let \(\mathcal B_j\) be the event that: 
\begin{enumerate} 
\item the \(j\)-th episode is initiated at an \((\varepsilon,\delta)\)-saddle \(\bar\mu_j\), i.e., $\left\| \nabla_\mu F(\bar\mu_j,\cdot) \right\|_{L^2_{\bar\mu_j}} \leq\varepsilon,$ $\underline{\lambda}(\operatorname{K}_{\bar\mu_j})<-\delta,$ 
\item the episode is unsuccessful, i.e., $F(\bar\mu_j)-F(\mu_j^{\mathrm{out}}) < F_0.$ 
\end{enumerate} 
Conditionally on \(\mathcal F_j\), the point \(\bar\mu_j\) is fixed and the perturbation used in that episode is fresh. Therefore, on the event that \(\bar\mu_j\) is an \((\varepsilon,\delta)\)-saddle, the uniform escape guarantee \eqref{eq:conditional-episode-success} gives $\mathbb P \left( \mathcal B_j|\mathcal F_j \right) \leq \frac{3\zeta_{\mathrm{ep}}}{4}.$ Taking expectations yields $\mathbb P(\mathcal B_j) \leq \frac{3\zeta_{\mathrm{ep}}}{4}.$ Define the good event $\mathcal G := \bigcap_{j=1}^{M_{\max}} \mathcal B_j^c.$ By the union bound and the choice of \(\zeta_{\mathrm{ep}}\), 
\[ 
\begin{aligned} 
\mathbb P(\mathcal G^c) \leq \sum_{j=1}^{M_{\max}} \mathbb P(\mathcal B_j)\leq M_{\max} \frac{3\zeta_{\mathrm{ep}}}{4}= M_{\max} \frac{3}{4} \frac{4\zeta}{3M_{\max}} = \zeta. 
\end{aligned} 
\] 
Thus, $\mathbb P(\mathcal G)\geq1-\zeta.$ On \(\mathcal G\), every episode initiated at an \((\varepsilon,\delta)\)-saddle is successful. 

\emph{Step 4: the returned iterate is approximately second-order stationary.} Let \(\widehat\mu\) denote the iterate returned by the algorithm. By the stopping rule, the final episode was initiated at \(\widehat\mu\), and hence \begin{equation} 
\label{eq:sosp-output-small-gradient} 
\left\| \nabla_\mu F(\widehat\mu,\cdot) \right\|_{L^2_{\widehat\mu}} \leq \varepsilon. 
\end{equation} 
Moreover, the final episode was unsuccessful: 
\begin{equation} 
\label{eq:sosp-output-failed-episode} 
F(\widehat\mu)-F(\mu^{\mathrm{out}}) < F_0. 
\end{equation} 
We now work on the good event \(\mathcal G\). Suppose, for contradiction, that $\underline{\lambda}(\operatorname{K}_{\widehat\mu})<-\delta.$ Together with \eqref{eq:sosp-output-small-gradient}, this would imply that \(\widehat\mu\) is an \((\varepsilon,\delta)\)-saddle. By the definition of \(\mathcal G\), every episode initiated at such a point is successful, so we would have $F(\widehat\mu)-F(\mu^{\mathrm{out}}) \geq F_0,$ contradicting \eqref{eq:sosp-output-failed-episode}. Therefore, $\underline{\lambda}(\operatorname{K}_{\widehat\mu})\geq-\delta.$ Combining this with \eqref{eq:sosp-output-small-gradient} proves that \(\widehat\mu\) is an \((\varepsilon,\delta)\)-second-order critical point on an event of probability at least \(1-\zeta\). 

\emph{Step 5: asymptotic form of the complexity bound.} By \eqref{eq:new-def-f0} and $\beta_\mathrm{sad}=b\delta^2$,
\begin{align*}
    n_{\mathrm{out}}\left(\frac{F_{\min}}{F_0}+1\right)=
    \frac{L_{\operatorname H}^2\tau_\mathrm{sad}^3}{c_F\log^2(3/2)b^{3/2}}
    \frac{n_{\mathrm{out}}^4}{\delta^3}F_{\min}+n_{\mathrm{out}}.
\end{align*}
Since $n_{\mathrm{out}} = \tilde O(1),$ from \eqref{eq:sosp-exact-complexity}, 
\[ 
N_{\mathrm{SOSP}} = \tilde O\left(\frac{C_{\operatorname{H}}^2+\beta_\mathrm{grad}} {\tau_\mathrm{grad}\sqrt{\beta_\mathrm{grad}}} \frac{F_{\min}}{\varepsilon^2} + \delta^{-3}F_{\min}+1\right). 
\] 
For the fixed initial measure $\mu^0$, we assume that $F_{\min}>0$ is independent of the tolerances. Since $F_0(\delta)\to0$ as $\delta\downarrow0$, there exists $\delta_0>0$, possibly depending on $F_{\min}$, such that $F_0(\delta)\le F_{\min}$ whenever $0<\delta\le\delta_0$. Hence,
\begin{equation*}
    n_{\mathrm{out}} \left(\frac{F_{\min}}{F_0}+1\right) \le 2n_{\mathrm{out}}\frac{F_{\min}}{F_0},
\end{equation*}
so the terminal episode is absorbed into the episode complexity term, giving \eqref{eq:sosp-soft-complexity}.
\end{proof}

\begin{lemma}[Sublevel control of the episode iterates]
\label{lem:outer-iterates-sublevel}
Let Assumptions~\ref{assumption:smooth-F} and~\ref{assumption:Hessian-lipschitz} hold. Choose $\beta_{\mathrm{grad}}>0$ and $0<\tau_{\mathrm{grad}}\leq C_{\operatorname H}^{-1}\sqrt{\beta_{\mathrm{grad}}},$ and choose the episode parameters as in \eqref{def_beta_tau_sad}-\eqref{eq:new-def-eta}, subject to \eqref{eq:new-epsilon-condition}. Consider Algorithm~\ref{alg:episode-wsfn} with early acceptance. Let $\mu^0 \in \mathcal{P}_2(\mathbb R^d)$ and let $\bar\mu^0:=\mu^0,\bar\mu^1,\ldots$ denote the outer iterates of the WSFN algorithm, i.e., the iterates at which the gradient $\nabla_\mu F$ is evaluated. If an episode is initiated at \(\bar\mu^r\), write $\nu_r^0:=\left(\operatorname{Id}+\eta\xi_r\right)_\#\bar\mu^r$ for its perturbed initialization and
\[
\nu_r^{\ell+1} :=\left(\operatorname{Id}-\tau\left(\operatorname{H}_{\nu_r^\ell}^2+\beta\operatorname{I_{d\times d}}
\right)^{-1/2}\nabla_\mu F(\nu_r^\ell,\cdot)\right)_\#\nu_r^\ell,
\quad 0\leq\ell<\ell_{\mathrm{ep},r},
\]
for the temporary iterates within that episode. Define $\mathcal L_0:=\left\{\mu\in\mathcal P_2(\mathbb R^d): F(\mu)\leq F(\mu^0)\right\},$ and let $\mathcal E_\kappa:= \bigcap_{r:\ \text{episode }r\text{ initialized}}\left\{\|\xi_r\|_{L^2_{\bar\mu^r}}
\leq\kappa\right\}.$ Then:
\begin{enumerate}
    \item Every outer iterate and the terminal returned iterate belong to
    \(\mathcal L_0\), i.e., $ F(\bar\mu^r)\leq F(\mu^0)$ for every outer iteration \(r\).

    \item On the event $\mathcal E_\kappa$, every computed temporary iterate in every perturbation episode satisfies $ F(\nu_r^\ell)\leq F(\mu^0)+F_0,$ $0\leq\ell\leq\ell_{\mathrm{ep},r}.$
\end{enumerate}
Hence, both conclusions hold on the intersection of \(\mathcal E_\kappa\) with the event that all episodes initiated at \((\varepsilon,\delta)\)-saddles are successful.
\end{lemma}

\begin{proof}
We first prove the bound on the temporary episode iterates and then prove invariance of the initial sublevel set by induction over the outer iterations.

\emph{Step 1: control of the perturbation increase.}
Consider an episode initiated at an outer iterate \(\bar\mu^r\). By the
algorithmic trigger, $\left\|\nabla_\mu F(\bar\mu^r,\cdot)\right\|_{L^2_{\bar\mu^r}} \leq \varepsilon.$ On the event \(\mathcal E_\kappa\), the corresponding perturbation
satisfies $\|\xi_r\|_{L^2_{\bar\mu^r}} \leq \kappa.$ Applying the second-order smoothness estimate of Lemma \ref{lemma:smoothness-of-F-2} to $\nu_r^0
= \left(\operatorname{Id}+\eta\xi_r\right)_\#\bar\mu^r$ gives
\[
\begin{aligned}
F(\nu_r^0) \leq F(\bar\mu^r) + \eta \left\langle \nabla_\mu F(\bar\mu^r,\cdot),\xi_r\right\rangle_{L^2_{\bar\mu^r}} + \frac{C_{\operatorname{H}}}{2}
\eta^2\|\xi_r\|_{L^2_{\bar\mu^r}}^2 \leq F(\bar\mu^r) + \eta\varepsilon\kappa + \frac{C_{\operatorname{H}}}{2}\eta^2\kappa^2.
\end{aligned}
\]
It remains to verify that the last two terms sum to \(F_0\) in \eqref{eq:new-def-f0}. Set $s := \sqrt{\varepsilon^2+2C_HF_0}.$ The choice of \(\eta\) gives $\eta\kappa = \frac{2F_0}{\varepsilon+s} = \frac{s-\varepsilon}{C_H},$ where the second equality follows from $(s-\varepsilon)(s+\varepsilon)
= 2C_HF_0.$ Therefore,
\[
\begin{aligned}
\eta\varepsilon\kappa + \frac{C_H}{2}\eta^2\kappa^2=
\frac{\varepsilon(s-\varepsilon)}{C_H} + \frac{(s-\varepsilon)^2}{2C_H} =
\frac{(s-\varepsilon)(s+\varepsilon)}{2C_H} =
F_0.
\end{aligned}
\]
Hence
\begin{equation}
\label{eq:episode-initialization-sublevel}
F(\nu_r^0) \leq F(\bar\mu^r)+F_0.
\end{equation}
Note that this argument uses only the small gradient condition and the bound on the perturbation norm. It does not require \(\bar\mu^r\) to induce negative curvature.

\emph{Step 2: control of the computed episode iterates.}
Fix an episode $r$ and work on $\mathcal E_\kappa$. Set $r_{\mathrm{loc}}:=\frac{\sqrt{\beta_{\mathrm{sad}}}}{L_{\operatorname H}}.$ The parameter choices \eqref{eq:new-small-beta-condition}, \eqref{eq:new-def-f0}, \eqref{eq:cF-bootstrap-condition}, and \eqref{eq:new-def-eta} give
\begin{equation*}
    \eta\kappa+2\sqrt{\tau_{\mathrm{sad}}n_{\mathrm{out}}}\beta_{\mathrm{sad}}^{-1/4}F_0^{1/2} \leq \frac{r_{\mathrm{loc}}}{2}<r_{\mathrm{loc}}.
\end{equation*}
Indeed, this follows from the same parameter estimate used to bound $3A+B$ in the proof of Proposition~\ref{discrete_lower_bound_metric}, with $A=\eta\kappa,$ $B=2\sqrt{\tau_{\mathrm{sad}}n_{\mathrm{out}}}\beta_{\mathrm{sad}}^{-1/4}F_0^{1/2}.$ Write $\nu_r^j=(Z_r^j)_\#\bar\mu_r$, where
\begin{equation*}
    Z_r^0=\operatorname{Id}+\eta\xi_r, \quad Z_r^{j+1}= \left(\operatorname{Id}-\tau_{\mathrm{sad}}d_r^j\right)\circ Z_r^j,
\end{equation*}
and $d_r^j := \left(\operatorname{H}_{\nu_r^j}^2 +\beta_{\mathrm{sad}}\operatorname{I_{d\times d}}\right)^{-1/2}\nabla_\mu F(\nu_r^j,\cdot).$ For every $k<\ell_{\mathrm{ep},r}$, the acceptance criterion has failed at every $j\leq k$. Thus 
\begin{equation*}
    F(\bar\mu_r)-F(\nu_r^j)<F_0, \quad 0\leq j\leq k.
\end{equation*}
Combining this with the perturbation cost bound \eqref{eq:episode-initialization-sublevel} gives
\begin{equation*}
    F(\nu_r^0)-F(\nu_r^j) = F(\nu_r^0)-F(\bar\mu_r)+F(\bar\mu_r)-F(\nu_r^j)<2F_0, \quad 0\leq j\leq k.
\end{equation*}
For $k\geq1$, apply Lemma~\ref{lem:episode-radius-bootstrap} to this prefix with $N=k$. For $k=0$, use $\|Z_r^0-\operatorname{Id}\|_{L^2_{\bar\mu_r}}
\leq\eta\kappa$ directly. In either case,
\begin{equation*}
    \|Z_r^k-\operatorname{Id}\|_{L^2_{\bar\mu_r}} \leq \eta\kappa+2\sqrt{\tau_{\mathrm{sad}}k}\beta_{\mathrm{sad}}^{-1/4}F_0^{1/2} <r_{\mathrm{loc}}.
\end{equation*}
Since \eqref{eq:new-epsilon-condition} implies $\varepsilon\leq \frac{\beta_{\mathrm{sad}}}{L_{\operatorname H}}$, Lemma~\ref{lem:local-wsfn-descent} applies at the starting point of every computed update, including the final update of a successful episode. Hence,
\begin{equation*}
    F(\nu_r^{k+1})\leq F(\nu_r^k), \quad 0\leq k<\ell_{\mathrm{ep},r}.
\end{equation*}
Together with \eqref{eq:episode-initialization-sublevel}, this gives
\begin{equation}
\label{eq:episode-trajectory-sublevel}
    F(\nu_r^\ell) \leq F(\nu_r^0) \leq F(\bar\mu_r)+F_0, \quad
    0\leq\ell\leq\ell_{\mathrm{ep},r}.
\end{equation}
Thus, whenever $F(\bar\mu_r)\leq F(\mu^0)$, all computed temporary iterates satisfy $F(\nu_r^\ell)\leq F(\mu^0)+F_0$.

\emph{Step 3: invariance of the initial sublevel set.} We now prove by induction over the outer iterations that $F(\bar\mu^r)\leq F(\mu^0).$ The assertion is immediate at \(r=0\), since $\bar\mu^0=\mu^0.$ Suppose that $F(\bar\mu^r)\leq F(\mu^0)$ for some outer iteration \(r\). There are two cases.

\emph{Case 1: large gradient step.} If $\left\|\nabla_\mu F(\bar\mu^r,\cdot)
\right\|_{L^2_{\bar\mu^r}} > \varepsilon,$ the algorithm performs one WSFN step and sets
\begin{equation*}
    \bar\mu^{r+1} = \left(\operatorname{Id}- \tau_\mathrm{grad}\left(
\operatorname{H}_{\bar\mu^r}^2 + \beta_\mathrm{grad}\operatorname{I_{d\times d}} \right)^{-1/2} \nabla_\mu F(\bar\mu^r,\cdot)\right)_\#\bar\mu^r.
\end{equation*}
By Lemma \ref{descent lemma}, $F(\bar\mu^{r+1}) \leq F(\bar\mu^r)
\leq F(\mu^0).$

\emph{Case 2: saddle point episode.} Suppose instead that $\left\|
\nabla_\mu F(\bar\mu^r,\cdot)\right\|_{L^2_{\bar\mu^r}} \leq \varepsilon,$ so that an episode is initiated. If the episode is successful, then $ F(\bar\mu_r)-F(\nu_r^{\ell_{\mathrm{ep},r}})\geq F_0.$ The algorithm accepts $\bar\mu_{r+1}=\nu_r^{\ell_{\mathrm{ep},r}}$, and hence $F(\bar\mu_{r+1}) \leq F(\bar\mu_r)-F_0 \leq F(\mu^0).$

If the episode is unsuccessful, the algorithm rejects the temporary trajectory and returns the stored pre-perturbation iterate $\widehat\mu = \bar\mu^r.$ By the induction hypothesis, $F(\widehat\mu) = F(\bar\mu^r) \leq F(\mu^0).$ The algorithm then terminates, so there are no subsequent outer iterates. Thus, in every case, each accepted outer iterate and the terminal returned iterate belong to \(\mathcal L_0\). Combining this fact with \eqref{eq:episode-trajectory-sublevel} proves both conclusions.
\end{proof}

\begin{proof}[Proof of Corollary \ref{cor:global-proximity-under-benignity}]
Let \(\mathcal G_{\mathrm{SOSP}}\) denote the event on which the conclusion
of Theorem \ref{thm:polynomial-time-sosp} holds. Thus, $\mathbb P(\mathcal G_{\mathrm{SOSP}}) \geq 1-\zeta,$ and, on \(\mathcal G_{\mathrm{SOSP}}\), the algorithm terminates and returns an iterate \(\widehat\mu\) satisfying
\begin{equation}
\label{eq:output-sosp-corollary}
\left\|\nabla_\mu F(\widehat\mu,\cdot)\right\|_{L^2_{\widehat\mu}} \leq \varepsilon, \quad \underline{\lambda}(\operatorname{K}_{\widehat\mu}) \geq -\delta.
\end{equation}
By Lemma \ref{lem:outer-iterates-sublevel}, every outer iterate of the algorithm belongs to \(\mathcal L_0\). In particular, since the algorithm returns the stored pre-perturbation outer iterate when the final episode is unsuccessful, we have
\begin{equation}
\label{eq:output-in-initial-sublevel}
\widehat\mu \in \mathcal L_0.
\end{equation}
This conclusion holds on every sample path on which the algorithm
terminates.

We now work on the event \(\mathcal G_{\mathrm{SOSP}}\). Suppose, for a
contradiction, that  $\widehat\mu \notin \mathcal N_{\alpha_{\mathrm{ben}}}(\mathcal M^*).$ Since \(\widehat\mu\in\mathcal L_0\) by
\eqref{eq:output-in-initial-sublevel}, the \((\varepsilon,\delta,\alpha_{\mathrm{ben}})\)-saddle property on \(\mathcal L_0\) implies that either $\left\|\nabla_\mu F(\widehat\mu,\cdot)\right\|_{L^2_{\widehat\mu}}>\varepsilon,$ or $\underline{\lambda}(\operatorname{K}_{\widehat\mu}) < -\delta.$ Both alternatives contradict \eqref{eq:output-sosp-corollary}. Therefore, $\widehat\mu \in \mathcal N_{\alpha_{\mathrm{ben}}}(\mathcal M^*)$ on \(\mathcal G_{\mathrm{SOSP}}\).

Hence, $\mathbb P\left(\widehat\mu \in \mathcal N_{\alpha_{\mathrm{ben}}}(\mathcal M^*)\right) \geq \mathbb P(\mathcal G_{\mathrm{SOSP}}) \geq 1-\zeta.$ Finally, by the definition of $\mathcal N_{\alpha_{\mathrm{ben}}}
(\mathcal M^*)$, this means that
\begin{equation*}
    \operatorname{dist}_{W_2}(\widehat\mu,\mathcal M^*) = \inf_{\mu^*\in\mathcal M^*}W_2(\widehat\mu,\mu^*)\le \alpha_{\mathrm{ben}}.
\end{equation*}
The complexity bound is inherited directly from Theorem~\ref{thm:polynomial-time-sosp}.
\end{proof}

\begin{proof}[Proof of Theorem \ref{thm:local-spectral-behavior-wsfn}]
Since \(\mu^*\) is a global minimizer, $\nabla_\mu F(\mu^*,\cdot)=0$ in $L^2_{\mu^*}.$ Moreover, \(\operatorname{M}_{\mu^*}=0\), \(\mu^*\)-a.e., and hence $\operatorname{H}_{\mu^*}=\operatorname{K}_{\mu^*}.$ The second-order necessary optimality condition in Proposition \ref{prop:necessary-second-order-optimality-minimizers} gives \(\operatorname{K}_{\mu^*}\succeq0\).

We first identify the derivative of the update map. Define
the pulled-back Wasserstein gradient by $G(v):=\nabla_\mu F(\mu_v,\cdot)\circ Z_v.$ By Lemma \ref{lemma:smoothness gradient of F}, which follows from the pulled-back gradient differentiation identity associated with Proposition \ref{proposition:Wasserstein-hessian-appendix} and the Hessian estimate of Lemma \ref{lem:lip-curves},
\begin{equation*}
   \|G(v)-G(0) - \operatorname{H}_{\mu^*}v\|_{L_{\mu^*}^2} \leq \frac{L_{\operatorname{H}}}{2}\|v\|_{L_{\mu^*}^2}^2.
\end{equation*}
Since \(\mu^* \in \mathcal{P}_2^{\mathrm{ac}}(\mathbb R^d)\) is a global minimizer, hence local, Theorem \ref{thm:first-order-optimality-minimizers} gives $G(0)=0.$ Moreover, because $\operatorname{H}_{\mu^*}=\operatorname{K}_{\mu^*},$ we get
\begin{equation}
\label{eq:pulled-back-gradient-expansion}
    G(v) = \operatorname{K}_{\mu^*}v+R_G(v),
\end{equation}
for some $R_G(v)$ such that $\|R_G(v)\|_{L^2_{\mu^*}}\leq \frac{L_{\operatorname{H}}}{2}\|v\|_{L^2_{\mu^*}}^2,$ for all \(v \in L^2_{\mu^*}\).

For $v \in L^2_{\mu^*},$ let $\widehat{H}(v):=\widetilde{\operatorname{H}}_{\mu^*,1}$ denote the pulled-backed Hessian at $\mu_v,$ and set $\widehat{P}(v):= \left(\widehat{H}(v)^2+\beta \operatorname{I_{d \times d}}\right)^{-\frac{1}{2}}.$ Lemma \ref{lem:intertwining_preconditioner} shows that $\widehat{P}(v)$ intertwines with the preconditioner $\operatorname{H}_{\mu_v},$ so that
\begin{equation*}
    \widehat{P}(v)J_v = J_v(\operatorname{H}_{\mu_v}^2+\beta \operatorname{I_{d \times d}})^{-\frac{1}{2}},
\end{equation*}
where $J_v:L_{\mu_v}^2 \to L_{\mu^*}^2$ is the composition operator given by $J_v w:=w \circ (\operatorname{Id}+v)$ for $w \in L_{\mu_v}^2.$ Moreover, Lemma \ref{lem:lip-curves} and the identity $\operatorname{H}_{\mu^*}=\operatorname{K}_{\mu^*}$ give
\begin{equation*}
    \|\widehat{H}(v)-\operatorname{K}_{\mu^*}\|_{\mathrm{op}} \leq L_{\operatorname{H}}\|v\|_{L^2_{\mu^*}}.
\end{equation*}
Applying Lemma \ref{lem:regularized-wsfn-spectral-maps} gives
\begin{equation}
\label{eq:pulled-back-preconditioner-expansion}
    \left\|\widehat P(v) - (\operatorname{K}_{\mu^*}^2+\beta \operatorname{I_{d \times d}})^{-1/2}
    \right\|_{\mathrm{op}} \leq \frac{2L_{\operatorname{H}}}{\pi\beta}\|v\|_{L^2_{\mu^*}}.
\end{equation}
Finally, since $(\lambda^2 +\beta)^{-1/2} \leq \beta^{-1/2}$ for every $\lambda \in \mathbb R$, continuous functional calculus gives $\|\widehat P(v)\|_{\mathrm{op}}\leq\beta^{-1/2}.$

Using \eqref{eq:pulled-back-gradient-expansion} and \eqref{eq:pulled-back-preconditioner-expansion}, we obtain
\begin{align*}
    \widehat P(v)G(v)&=
    (\operatorname{K}_{\mu^*}^2+\beta \operatorname{I_{d \times d}})^{-1/2}\operatorname{K}_{\mu^*}v\\
    &+ \left(\widehat P(v) - (\operatorname{K}_{\mu^*}^2+\beta \operatorname{I_{d \times d}})^{-1/2}\right)\operatorname{K}_{\mu^*}v + \widehat P(v)R_G(v),
\end{align*}
and therefore
\begin{align*}
    \left\|
        \widehat P(v)G(v)
        -
        (\operatorname{K}_{\mu^*}^2+\beta \operatorname{I_{d \times d}})^{-1/2}\operatorname{K}_{\mu^*}v
    \right\|_{L^2_{\mu^*}}
    &\leq \left(\frac{2L_{\operatorname{H}}}{\pi\beta}\|\operatorname{K}_{\mu^*}\|_{\mathrm{op}} + \|\widehat P(v)\|_{\mathrm{op}}\frac{L_{\operatorname{H}}}{2}\right)\|v\|_{L^2_{\mu^*}}^2\\
    &\leq L_{\operatorname{H}}\left(\frac{2C_{\operatorname{K}}}{\pi\beta} + \frac{1}{2\sqrt{\beta}}\right)\|v\|_{L^2_{\mu^*}}^2.
\end{align*}
Substitution into $\Psi_{\tau,\beta}$ proves that there exists $C=C(L_{\operatorname H},C_{\operatorname K},\beta)>0$ such that, for every $v\in L^2_{\mu^*}$,
\begin{equation*}
    \Psi_{\tau,\beta}(v) = \mathcal L_{\tau,\beta}v+\mathcal R_{\tau,\beta}(v), \quad \|\mathcal R_{\tau,\beta}(v)\|_{L^2_{\mu^*}} \leq C\|v\|_*^2.
\end{equation*}

We next analyze the operator \(\mathcal L_{\tau,\beta}\). Let \(E\) be the spectral measure of \(\operatorname{K}_{\mu^*}\). By the spectral functional calculus,
\[
    \mathcal L_{\tau,\beta} = \int_{\sigma(\operatorname{K}_{\mu^*})}\left(1-\tau\frac{\lambda}{\sqrt{\lambda^2+\beta}}\right) \mathrm dE(\lambda).
\]
Then, by \cite[1.10. Proposition, Chapter IX]{conway1994course}, for every \(v\in L^2_{\mu^*}\) and every \(n\geq1\),
\begin{equation}
\label{eq:iterated-linearized-wsfn}
    \mathcal L_{\tau,\beta}^n v = \int_{\sigma(\operatorname{K}_{\mu^*})}\left(1-\tau\frac{\lambda}{\sqrt{\lambda^2+\beta}}\right)^n\mathrm dE(\lambda)v.
\end{equation}
For \(\lambda>0\), $0 \leq 1-\tau\frac{\lambda}{\sqrt{\lambda^2+\beta}} <1,$ whereas the multiplier equals \(1\) at \(\lambda=0\). Dominated convergence in \eqref{eq:iterated-linearized-wsfn} gives
\[
    \lim_{n \to \infty}\mathcal L_{\tau,\beta}^n v =E(\{0\})v = P_{\ker \operatorname{K}_{\mu^*}}v,
\]
where $P_{\ker \operatorname{K}_{\mu^*}}v$ is the projection of $v$ onto the null space $\operatorname{K}_{\mu^*}.$ If $\operatorname{K}_{\mu^*}$ is injective, then $\ker \operatorname{K}_{\mu^*} =\{0\}$ and $\lim_{n \to \infty}\mathcal L_{\tau,\beta}^n v=0.$

Because \(\operatorname{K}_{\mu^*}\) is compact on the infinite dimensional space \(L^2_{\mu^*}\), one has \(0\in\sigma(\operatorname{K}_{\mu^*})\) by \cite[Theorem VI. 15]{ReedSimon1980}. Hence $1\in\sigma(\mathcal L_{\tau,\beta}),$ and therefore $ \|\mathcal L_{\tau,\beta}^n\|_{\mathrm{op}}=1.$

Now assume that $(\operatorname{Id}-P_{\ker \operatorname{K}_{\mu^*}})v=\operatorname{K}_{\mu^*}^s w$ for some $s>0$ and $w\in L^2_{\mu^*}.$ Recall that \(\operatorname{K}_{\mu^*}\) is bounded, self-adjoint, and positive semidefinite, with $\|\operatorname{K}_{\mu^*}\|_{\mathrm{op}}\leq C_{\operatorname{K}}.$ The linearized WSFN operator can be written through the functional calculus of \(\operatorname{K}_{\mu^*}\) as $\mathcal L_{\tau,\beta} = \ell_{\tau,\beta}(\operatorname{K}_{\mu^*}),$ where $\ell_{\tau,\beta}(\lambda):= 1-\tau\frac{\lambda}{\sqrt{\lambda^2+\beta}}$ for $\lambda \in \sigma(\operatorname{K}_{\mu^*})$.

We first show that $\mathcal L_{\tau,\beta}^n P_{\ker \operatorname{K}_{\mu^*}}=P_{\ker \operatorname{K}_{\mu^*}}.$ Indeed, for every \(u\in\ker \operatorname{K}_{\mu^*}\), one has \(\operatorname{K}_{\mu^*}u=0\), and therefore
\[
    (\operatorname{K}_{\mu^*}^2+\beta \operatorname{I_{d \times d}})^{-1/2}\operatorname{K}_{\mu^*}u=0.
\]
It follows that
\[
    \mathcal L_{\tau,\beta}u = u-\tau(\operatorname{K}_{\mu^*}^2+\beta \operatorname{I_{d \times d}})^{-1/2}\operatorname{K}_{\mu^*}u = u.
\]
Thus, \(\mathcal L_{\tau,\beta}\) acts as the identity on \(\ker \operatorname{K}_{\mu^*}\). Since \(P_{\ker \operatorname{K}_{\mu^*}}u\in\ker \operatorname{K}_{\mu^*}\) for every \(u\in L^2_{\mu^*}\), we obtain
\[
    \mathcal L_{\tau,\beta}^nP_{\ker \operatorname{K}_{\mu^*}}u=P_{\ker \operatorname{K}_{\mu^*}}u,
\]
Alternatively, using the spectral calculus, $P_{\ker \operatorname{K}_{\mu^*}}=\mathbbm{1}_{\{0\}}(\operatorname{K}_{\mu^*})$ and hence
\[
    \mathcal L_{\tau,\beta}^nP_{\ker \operatorname{K}_{\mu^*}}
    = \ell_{\tau,\beta}(\operatorname{K}_{\mu^*})^n\mathbbm{1}_{\{0\}}(\operatorname{K}_{\mu^*}) = \ell_{\tau,\beta}(0)^n\mathbbm{1}_{\{0\}}(\operatorname{K}_{\mu^*}) = P_{\ker \operatorname{K}_{\mu^*}},
\]
because \(\ell_{\tau,\beta}(0)=1\).

We now decompose $v=P_{\ker \operatorname{K}_{\mu^*}}v+(\operatorname{Id}-P_{\ker \operatorname{K}_{\mu^*}})v$ for every $v \in L_{\mu^*}^2.$ Using $\mathcal L_{\tau,\beta}^n P_{\ker \operatorname{K}_{\mu^*}}=P_{\ker \operatorname{K}_{\mu^*}}$ and the source condition $(\operatorname{Id}-P_{\ker \operatorname{K}_{\mu^*}})v=\operatorname{K}_{\mu^*}^sw,$ we obtain
\begin{align*}
    \mathcal L_{\tau,\beta}^nv-P_{\ker \operatorname{K}_{\mu^*}}v
    &=
    \mathcal L_{\tau,\beta}^nP_{\ker \operatorname{K}_{\mu^*}}v
    +
    \mathcal L_{\tau,\beta}^n(\operatorname{Id}-P_{\ker \operatorname{K}_{\mu^*}})v
    -
    P_{\ker \operatorname{K}_{\mu^*}}v\\
    &=
    \mathcal L_{\tau,\beta}^n(\operatorname{Id}-P_{\ker \operatorname{K}_{\mu^*}})v\\
    &=
    \mathcal L_{\tau,\beta}^n\operatorname{K}_{\mu^*}^sw.
\end{align*}
Because both \(\mathcal L_{\tau,\beta}\) and \(\operatorname{K}_{\mu^*}^s\) are functions of \(\operatorname{K}_{\mu^*}\), the spectral calculus gives
\begin{align}
    \left\|
        \mathcal L_{\tau,\beta}^n\operatorname{K}_{\mu^*}^s
    \right\|_{\mathrm{op}} &=
    \sup_{\lambda\in\sigma(\operatorname{K}_{\mu^*})}
    \lambda^s
    \left|
        1-\tau\frac{\lambda}{\sqrt{\lambda^2+\beta}}
    \right|^n \leq
    \sup_{0\leq\lambda\leq C_{\operatorname{K}}}
    \lambda^s
    \left(
        1-\tau\frac{\lambda}{\sqrt{\lambda^2+\beta}}
    \right)^n.
    \label{eq:source-spectral-multiplier}
\end{align}
Here we used \(\sigma(\operatorname{K}_{\mu^*})\subseteq[0,C_{\operatorname{K}}]\). Moreover, since \(\tau\in(0,1]\), \(\beta>0\), and \(\lambda\geq0\),
\[
    0 \leq \tau\frac{\lambda}{\sqrt{\lambda^2+\beta}} <1.
\]
The elementary inequality \(1-x\leq e^{-x}\) implies
\begin{align*}
    \left(
        1-\tau\frac{\lambda}{\sqrt{\lambda^2+\beta}}
    \right)^n
    \leq
    \exp\left(
        -\tau n\frac{\lambda}{\sqrt{\lambda^2+\beta}}
    \right).
\end{align*}
For \(0\leq\lambda\leq C_{\operatorname{K}}\), we also have
\[
    \sqrt{\lambda^2+\beta}
    \leq
    \sqrt{C_{\operatorname{K}}^2+\beta},
\]
and consequently
\[
    \frac{\lambda}{\sqrt{\lambda^2+\beta}}
    \geq
    \frac{\lambda}{\sqrt{C_{\operatorname{K}}^2+\beta}}.
\]
It follows that
\begin{equation}
\label{eq:exponential-spectral-bound}
    \left(
        1-\tau\frac{\lambda}{\sqrt{\lambda^2+\beta}}
    \right)^n
    \leq
    \exp\left(-\frac{\tau n}{\sqrt{C_{\operatorname{K}}^2+\beta}}\lambda\right).
\end{equation}
Combining \eqref{eq:source-spectral-multiplier} and
\eqref{eq:exponential-spectral-bound}, we obtain
\[
    \left\|
        \mathcal L_{\tau,\beta}^n\operatorname{K}_{\mu^*}^s
    \right\|_{\mathrm{op}}
    \leq
    \sup_{\lambda\geq0}
    \lambda^s
    \exp\left(-\frac{\tau n}{\sqrt{C_{\operatorname{K}}^2+\beta}}\lambda\right).
\]

For every \(a>0\) and $s>0$, the function $\lambda\mapsto\lambda^se^{-a\lambda}$ attains its maximum over \([0,\infty)\) at \(\lambda=\frac{s}{a}\), and
\[
    \sup_{\lambda\geq0}\lambda^se^{-a\lambda}
    =
    \left(\frac{s}{ea}\right)^s.
\]
Taking \(a=\frac{\tau n}{\sqrt{C_{\operatorname{K}}^2+\beta}}\) gives
\[
    \left\|
        \mathcal L_{\tau,\beta}^n\operatorname{K}_{\mu^*}^s
    \right\|_{\mathrm{op}}
    \leq
    \left(
        \frac{s\sqrt{C_{\operatorname{K}}^2+\beta}}{e\tau n}
    \right)^s.
\]
Therefore,
\begin{align*}
    \left\|
        \mathcal L_{\tau,\beta}^nv-P_{\ker \operatorname{K}_{\mu^*}}v
    \right\|_{L^2_{\mu^*}}
    &=
    \left\|
        \mathcal L_{\tau,\beta}^n\operatorname{K}_{\mu^*}^sw
    \right\|_{L^2_{\mu^*}} \leq
    \left(
        \frac{s\sqrt{C_K^2+\beta}}{e\tau n}
    \right)^s
    \|w\|_{L^2_{\mu^*}},
\end{align*}
as claimed.

Finally, on the range of \(P_m\), every spectral value of $\operatorname{K}_{\mu^*}$ satisfies \(\lambda\geq m>0\). Since $\lambda\mapsto 1-\tau\frac{\lambda}{\sqrt{\lambda^2+\beta}}$ is non-increasing on \([0,\infty)\),
\[
    \|\mathcal L_{\tau,\beta}^nP_m\|_{\mathrm{op}} \leq\left(1-\tau\frac{m}{\sqrt{m^2+\beta}}\right)^n,
\]
as required.
\end{proof}

\begin{remark}[Interpretation of Theorem \ref{thm:local-spectral-behavior-wsfn}]
\label{rem:interpretation-local-spectral-behavior-wsfn}
The multiplier $r_{\tau,\beta}(\lambda)=1-\tau\frac{\lambda}{\sqrt{\lambda^2+\beta}}$ distinguishes two curvature regimes. When $\lambda\gg\sqrt{\beta}$, it is approximately $1-\tau$. When $0<\lambda\ll\sqrt{\beta}$, it is approximately $1-\tau\frac{\lambda}{\sqrt{\beta}}$. Thus, regularized WSFN approximately normalizes sufficiently large positive curvature, while nearly flat directions remain weakly contracted.

The absence of a uniform contraction on the full space must be distinguished from the behavior on the positive spectral subspace. If $\operatorname K_{\mu^*}$ has infinitely many positive eigenvalues, they accumulate at zero, and no common geometric contraction factor applies to all positive curvature directions. If $\operatorname K_{\mu^*}$ has finite rank, its nonzero spectrum is instead separated from zero, so the positive spectral component contracts geometrically. In either case, the null-space component is unchanged by the linearized dynamics.

The source condition $(\operatorname{Id}-P_{\ker \operatorname K_{\mu^*}})v=\operatorname K_{\mu^*}^s w$ is a H\"older-type spectral regularity condition, as used in spectral regularization, inverse problems, and statistical learning; see \cite[Section~2.3]{BAUER200752}, \cite[Assumption~2]{LIN2020868}, and \cite[Exercise~2.12]{blanchard2017optimal}. On a positive eigenmode with eigenvalue $\lambda$, it introduces the factor $\lambda^s$ and thereby suppresses the contribution of nearly flat directions. Larger $s$ imposes stronger spectral regularity and gives the faster linearized rate $n^{-s}$.

Therefore, the spectral band and source condition estimates provide different quantitative conclusions. The former restricts attention to curvature bounded away from zero, whereas the latter controls the distribution of the initial perturbation across the positive spectrum.
\end{remark}

\subsection{Local spectral analysis and nonlinear limitations}
\label{app:local-spectral-analysis}
This appendix derives perturbative estimates for a deterministic continuation of WSFN, and explains the additional requirements for transferring the linearized rates to nonlinear iterates. Throughout, fix a global minimizer $\mu^* \in \mathcal P_2^{\mathrm{ac}}(\mathbb R^d)$, $\tau\in(0,1]$, and $\beta>0$. For fixed $\tau$ and $\beta$, write
\begin{equation*}
    \Phi_{\tau,\beta}(\mu) := \left(\operatorname{Id}-\tau(\operatorname H_\mu^2+\beta \operatorname{I_{d \times d}})^{-1/2}\nabla_\mu F(\mu,\cdot)\right)_\#\mu.
\end{equation*}
Algorithm~\ref{alg:episode-wsfn} terminates when it returns $\widehat\mu$. The iterates considered below constitute an additional deterministic continuation from that measure, with a fixed pair $(\tau,\beta)\in(0,1]\times(0,\infty)$, not a further phase already guaranteed by the algorithm. This pair need not coincide with either phase specific pair used before termination.

\begin{theorem}[Post entry perturbative estimates]
\label{thm:global-to-local-spectral-wsfn}
Let Assumptions \ref{assumption:smooth-F} and \ref{assumption:Hessian-lipschitz} hold. Let $\widehat\mu\in\mathcal P_2(\mathbb R^d)$ and
define $v^0:=\operatorname T_{\mu^*}^{\widehat\mu}-\operatorname{Id},$ $v^{n+1}:=\Psi_{\tau,\beta}(v^n),$ for $n \geq 0.$ Then $\rho^n:=(\operatorname{Id}+v^n)_\#\mu^*$ satisfies $\rho^0=\widehat\mu,$ $\rho^{n+1}=\Phi_{\tau,\beta}(\rho^n),$ and there exists \(C(L_{\operatorname{H}},C_{\operatorname{K}},\beta)>0\) such that
\begin{equation}
\label{eq:post-hitting-nonlinear-recursion}
v^{n+1} = \mathcal L_{\tau,\beta}v^n + \mathcal R_{\tau,\beta}(v^n),
\end{equation}
with $\left\|\mathcal R_{\tau,\beta}(v^n)\right\|_{L^2_{\mu^*}} \leq C\|v^n\|_{L^2_{\mu^*}}^2.$ Moreover:
\begin{itemize}
    \item If, in addition, $\left(\operatorname{Id}-P_{\ker \operatorname{K}_{\mu^*}}\right)v^0=\operatorname{K}_{\mu^*}^s w$ for some $w \in L_{\mu^*}^2$ and $s > 0,$ then
\begin{equation}
\label{eq:post-hitting-source-estimate}
\left\|v^n-P_{\ker \operatorname{K}_{\mu^*}}v^0\right\|_{L^2_{\mu^*}} \leq \left(\frac{s\sqrt{C_{\operatorname{K}}^2+\beta}}{e\tau n}\right)^s\|w\|_{L^2_{\mu^*}} + C\sum_{j=0}^{n-1}\|v^j\|_{L^2_{\mu^*}}^2.
\end{equation}
    \item For \(m>0\), let $q_m:= 1-\tau\frac{m}{\sqrt{m^2+\beta}}.$ Then
\begin{equation}
\label{eq:post-hitting-spectral-band-estimate}
\|P_mv^n\|_{L^2_{\mu^*}} \leq q_m^n \|P_mv^0\|_{L^2_{\mu^*}} + C\sum_{j=0}^{n-1} q_m^{n-1-j}\|v^j\|_{L^2_{\mu^*}}^2.
\end{equation}
\end{itemize}
\end{theorem}

\begin{proof}
Since \(\mu^*\) is absolutely continuous, the optimal transport map $\operatorname T_{\mu^*}^{\widehat\mu}$ exists. Its properties give $\widehat\mu = (\operatorname{T}_{\mu^*}^{\widehat \mu})_\#\mu^* = (\operatorname{Id}+v^0)_\#\mu^*.$

Starting from \(\bar\mu^0:=\widehat \mu\), consider the WSFN update
\begin{equation*}
\label{eq:post-hitting-deterministic-wsfn}
\bar\mu^{n+1} := \left(
\operatorname{Id} - \tau(\operatorname{H}_{\bar\mu^n}^2+\beta \operatorname{I_{d \times d}})^{-\frac{1}{2}}\nabla_\mu F(\bar\mu^n,\cdot)\right)_\#\bar\mu^n, \quad n\geq0.
\end{equation*}
Let \(\Psi_{\tau,\beta}\) denote the pulled-back WSFN map centered at \(\mu^*\), and define recursively
\begin{equation*}
v^{n+1} := \Psi_{\tau,\beta}(v^n), \quad n \geq 0.
\end{equation*}
Suppose inductively that \(\bar\mu^n=(\operatorname{Id}+v^n)_\#\mu^*\). The case $n=0$ holds by construction. By the definition of the WSFN update,
\[
\begin{aligned}
\bar\mu^{n+1} &=
\left(\operatorname{Id}-\tau(\operatorname{H}_{\bar\mu^n}^2+\beta \operatorname{I_{d \times d}})^{-\frac{1}{2}}
\nabla_\mu F(\bar\mu^n,\cdot)\right)_\#\bar\mu^n
\\
&=\left(\left[\operatorname{Id}-\tau(\operatorname{H}_{\bar\mu^n}^2+\beta \operatorname{I}_{d \times d})^{-\frac{1}{2}}\nabla_\mu F(\bar\mu^n,\cdot)\right]\circ(\operatorname{Id}+v^n)\right)_\#\mu^*
\\
&=\left(\operatorname{Id}+\Psi_{\tau,\beta}(v^n)\right)_\#\mu^*
=(\operatorname{Id}+v^{n+1})_\#\mu^*.
\end{aligned}
\]
This proves \(\bar\mu^n=(\operatorname{Id}+v^n)_\#\mu^*\) for all $n \geq 0$ by induction. The identity $\Phi_{\tau,\beta}(\mu_v) =(\operatorname{Id}+\Psi_{\tau,\beta}(v))_\#\mu^*$ shows inductively that $\rho^n$ follows the deterministic WSFN update.

Applying Theorem~\ref{thm:local-spectral-behavior-wsfn} at the selected
minimizer \(\mu^*\) gives \eqref{eq:post-hitting-nonlinear-recursion}. Iterating this identity yields
\begin{equation}
\label{eq:post-hitting-variation-of-constants}
v^n = \mathcal L_{\tau,\beta}^n v^0 + \sum_{j=0}^{n-1}\mathcal L_{\tau,\beta}^{n-1-j} \mathcal R_{\tau,\beta}(v^j),
\end{equation}
for every \(n\geq1\). Since $\left\|\mathcal L_{\tau,\beta}^\ell\right\|_{\mathrm{op}}
\leq1,$ $\ell\geq0,$ the remainder estimate in \eqref{eq:post-hitting-nonlinear-recursion} gives
\begin{equation*}
\left\|v^n - \mathcal L_{\tau,\beta}^n v^0 \right\|_{L^2_{\mu^*}}
\leq C\sum_{j=0}^{n-1}\|v^j\|_{L^2_{\mu^*}}^2.
\end{equation*}
Now, if \((\operatorname{Id}-P_{\ker \operatorname{K}_{\mu^*}})v^0=\operatorname{K}_{\mu^*}^s w\), for some $s > 0$ and $w \in L_{\mu^*}^2,$ the source condition estimate in Theorem \ref{thm:local-spectral-behavior-wsfn} gives
\[
\left\|\mathcal L_{\tau,\beta}^n v^0 - P_{\ker \operatorname{K}_{\mu^*}}v^0\right\|_{L^2_{\mu^*}} \leq\left(\frac{s\sqrt{C_{\operatorname{K}}^2+\beta}}{e\tau n}\right)^s \|w\|_{L^2_{\mu^*}}.
\]
Combining this with \eqref{eq:post-hitting-variation-of-constants} gives
\begin{align*}
    &\left|\left\|v^n - P_{\ker \operatorname{K}_{\mu^*}}v^0\right\|_{L^2_{\mu^*}} - \left\|\sum_{j=0}^{n-1}\mathcal L_{\tau,\beta}^{n-1-j} \mathcal R_{\tau,\beta}(v^j)\right\|_{L^2_{\mu^*}}\right|\\ &\leq \left\|v^n - P_{\ker \operatorname{K}_{\mu^*}}v^0-\sum_{j=0}^{n-1}\mathcal L_{\tau,\beta}^{n-1-j} \mathcal R_{\tau,\beta}(v^j)\right\|_{L^2_{\mu^*}}\\ 
    &\leq \left(\frac{s\sqrt{C_{\operatorname{K}}^2+\beta}}{e\tau n}\right)^s \|w\|_{L^2_{\mu^*}}.
\end{align*}
Hence
\begin{align*}
    \|v^n - P_{\ker \operatorname{K}_{\mu^*}}v^0\|_{L^2_{\mu^*}} \leq \left(\frac{s\sqrt{C_{\operatorname{K}}^2+\beta}}{e\tau n}\right)^s \|w\|_{L^2_{\mu^*}} + \sum_{j=0}^{n-1}\left\|\mathcal L_{\tau,\beta}^{n-1-j} \mathcal R_{\tau,\beta}(v^j)\right\|_{L^2_{\mu^*}}
\end{align*}
Since \(\|\mathcal L_{\tau,\beta}^\ell\|_{\mathrm{op}}\leq1\), and the quadratic remainder estimate on $\mathcal R_{\tau,\beta}(v^j)$ proves \eqref{eq:post-hitting-source-estimate}.

Finally, since $P_m=\mathbbm{1}_{[m,\infty)}(\operatorname{K}_{\mu^*})$ and $\mathcal L_{\tau,\beta} = \ell_{\tau,\beta}(\operatorname{K}_{\mu^*}),$ $ \ell_{\tau,\beta}(\lambda) := 1-\tau\frac{\lambda}{\sqrt{\lambda^2+\beta}},$ $\lambda \in \sigma(\operatorname{K}_{\mu^*}),$ both operators are Borel functions of the same self-adjoint bounded operator \(\operatorname{K}_{\mu^*}\). \cite[Theorem VII.2 (a)]{ReedSimon1980} implies $P_m\mathcal L_{\tau,\beta} = \mathcal L_{\tau,\beta}P_m.$ Theorem \ref{thm:local-spectral-behavior-wsfn} hence gives
\[
\left\|\mathcal L_{\tau,\beta}^\ell P_m\right\|_{\mathrm{op}} \leq q_m^\ell.
\]
Applying \(P_m\) to \eqref{eq:post-hitting-variation-of-constants} therefore proves \eqref{eq:post-hitting-spectral-band-estimate}.
\end{proof}

\paragraph{Connection with the global guarantee.}
On the event in Corollary~\ref{cor:global-proximity-under-benignity}, the returned measure $\widehat\mu$ belongs to $\mathcal N_{\alpha_{\mathrm{ben}}}(\mathcal M^*)$. To obtain a small initial displacement relative to a fixed minimizer, one must additionally select $\mu^*\in\mathcal M^*\cap\mathcal P_2^{\mathrm{ac}}(\mathbb R^d)$ that is close to $\widehat\mu$. For example, if $W_2(\widehat\mu,\mu^*)\leq\alpha_{\mathrm{ben}}$, then $\|v^0\|_{L_{\mu^*}^2}= W_2(\widehat\mu,\mu^*)\leq\alpha_{\mathrm{ben}}.$ The fact that $\widehat\mu \in \mathcal N_{\alpha_{\mathrm{ben}}}(\mathcal M^*)$ alone does not ensure that a nearby minimizer is absolutely continuous or that the distance to the minimizer set is attained. The estimates above are pathwise and do not provide a further nonlinear convergence rate. They also do not establish that the continuation remains close to the selected minimizer. Although $\operatorname{Id}+v^0$ is an optimal transport map, the recursively constructed maps $\operatorname{Id}+v^n$ need not be optimal for $n\geq1$.

\begin{remark}[Additional requirements for nonlinear rate transfer]
\label{rem:nonlinear-spectral-rates-appendix}
\emph{Accumulation and spectral transfer of nonlinear errors.}
From the proof of Theorem \ref{thm:global-to-local-spectral-wsfn},
\[
\left\|v^n - \mathcal L_{\tau,\beta}^n v^0 \right\|_{L^2_{\mu^*}}
\leq C\sum_{j=0}^{n-1}\|v^j\|_{L^2_{\mu^*}}^2.   
\]
There is no uniform geometric damping in this bound. In particular, small nonlinear errors can generate components in nearly flat directions. The null-space component can also change. Since $P_{\ker \operatorname K_{\mu^*}}\mathcal L_{\tau,\beta}=P_{\ker \operatorname K_{\mu^*}}$,
\[
    P_{\ker \operatorname K_{\mu^*}}v^n = P_{\ker \operatorname K_{\mu^*}}v^0
    + \sum_{j=0}^{n-1}P_{\ker \operatorname K_{\mu^*}}\mathcal R_{\tau,\beta}(v^j).
\]
Thus, $P_{\ker \operatorname K_{\mu^*}}v^0$ is the limit of the linearized
iteration, not an automatically preserved component or prescribed limit of the nonlinear iteration.

\emph{Spectral band invariance.} One sufficient structural route to a nonlinear spectral band rate would be local invariance of $V_m:=\operatorname{Ran}P_m$, that is, $\Psi_{\tau,\beta}\left(V_m\cap B_{L_{\mu^*}^2}(0,r)\right) \subseteq V_m,$ where $B_{L_{\mu^*}^2}(0,r) := \{v\in L_{\mu^*}^2:\|v\|_{L_{\mu^*}^2}<r\}$ for the open ball of radius $r>0$ centered at the zero vector field. Because $\mathcal L_{\tau, \beta}$ already preserves $V_m$, this is equivalent to $(\operatorname{Id}-P_m)\mathcal R_{\tau, \beta}(v)=0,$ for $v\in V_m\cap B_{L_{\mu^*}^2}(0,r).$ Under this additional condition,
\[
    \|\Psi_{\tau,\beta}(v)\|_{L_{\mu^*}^2} \leq \left(q_m+C\|v\|_{L_{\mu^*}^2}\right)\|v\|_{L_{\mu^*}^2}.
\]
Choosing $r$ sufficiently small that $q_m+Cr<1$ would therefore give geometric decay for sufficiently small initial data in $V_m$. This invariance is not a consequence of the present assumptions. It rules out nonlinear transfer from the band $[m,\infty)$ into the complementary spectral subspace $[0,m)$ and requires compatibility between the nonlinear terms and the spectral decomposition of the Hessian frozen at $\mu^*$.

\emph{Source space control.} A possible route to nonlinear source condition
rates would instead require quantitative control in the source space, for example $(\operatorname{Id}-P_{\ker \operatorname K_{\mu^*}})v^j=\operatorname K_{\mu^*}^s w_j,$ with $\sup_{j\geq0}\|w_j\|_{L_{\mu^*}^2}<\infty,$ or an estimate on the positive spectral part of the nonlinear remainder of the form
\[
    (\operatorname{Id}-P_{\ker \operatorname K_{\mu^*}})\mathcal R_{\tau, \beta}(v)\in\operatorname{Ran}(\operatorname K_{\mu^*}^s),
    \quad
    \left\|
        \operatorname K_{\mu^*}^{-s}(\operatorname{Id}-P_{\ker \operatorname K_{\mu^*}})\mathcal R_{\tau, \beta}(v)
    \right\|_{L_{\mu^*}^2}
    \leq C_s\|v\|_{L_{\mu^*}^2}^2.
\]
Here $\operatorname K_{\mu^*}^{-s}$ is the spectral inverse with domain $\operatorname{Ran}(\operatorname K_{\mu^*}^s)$, a subspace of $(\ker \operatorname K_{\mu^*})^\perp$. It is not assumed to be bounded or defined on the whole orthogonal complement. Such a weighted estimate is substantially stronger than the available ambient norm bound. Indeed, suppose that $\operatorname K_{\mu^*}$ has positive eigenpairs $\operatorname K_{\mu^*}e_j=\lambda_j e_j,$ $\|e_j\|_{L_{\mu^*}^2}=1,$ $\lambda_j\downarrow0.$ For $h_j:=\lambda_j^{s/2}e_j$,
\[
    \|h_j\|_{L_{\mu^*}^2}=\lambda_j^{s/2}\to0,
    \quad
    \|\operatorname K_{\mu^*}^{-s}h_j\|_{L_{\mu^*}^2}=
    \lambda_j^{-s/2}\to\infty.
\]
Thus, an error can be arbitrarily small in $L_{\mu^*}^2$ while arbitrarily large in the source norm. This obstruction concerns the case of positive eigenvalues accumulating at zero, which is true for $\operatorname K_{\mu^*}$ since it is Hilbert-Schmidt. Source space preservation or a weighted remainder estimate would still need to be combined with control of the accumulated errors in \eqref{eq:post-hitting-variation-of-constants}, including the null-space component. Neither displayed condition alone establishes the nonlinear rate $n^{-s}$.

These invariance and source space conditions are strong structural requirements and can be difficult to verify. They illustrate possible routes to nonlinear rates, but cannot be considered as assumptions established by the present analysis. Accordingly, the unconditional rates in Theorem~\ref{thm:local-spectral-behavior-wsfn} are stated for the linearized dynamics.
\end{remark}

\section{Scope of the improved saddle dependence}
\label{app:saddle-parameter-dependence}
This section explains the improvement of the saddle contribution from $\tilde O(\delta^{-4})$ to $\tilde O(\delta^{-3})$. The distinction is between the time required to escape an individual saddle neighborhood and the objective decrease used to count such episodes. Throughout, the regularity constants are fixed independently of $\delta$. 

We consider the saddle episode analysis with $\beta_\mathrm{sad}=b\delta^2,$ $b>0,$ $\tau_{\mathrm{sad}}\in(0,\frac{1}{2}],$ where $b$ and $\tau_{\mathrm{sad}}$ are fixed. The small gradient conditions and the local descent and first-exit estimates of that analysis are understood to hold. An episode uses at most $n_{\mathrm{out}}$ transport steps and accepts an output only after certifying a net decrease of at least $F_0$ relative to its pre-perturbation iterate. 

\subsection{Episode duration and objective decrease counting}
\label{app:saddle-duration-counting}
We briefly recall the counting argument that was used in the proof of Theorem \ref{thm:polynomial-time-sosp}.Let $F_{\min} := F(\mu^0)-\inf_{\mu\in\mathcal P_2(\mathbb R^d)}F(\mu).$ Suppose all accepted non-episode steps are descending. If $N_{\mathrm{succ}}$ episodes are successful, telescoping the accepted objective decreases gives
\[
    N_{\mathrm{succ}}F_0\leq F_{\min}.
\]
For an algorithm that terminates at its first unsuccessful episode, there is at most one additional episode. Hence the number $N_{\mathrm{sad}}$ of transport steps spent in episodes satisfies
\begin{equation}
\label{eq:app-saddle-counting}
    N_{\mathrm{sad}} \leq n_{\mathrm{out}}\left(\frac{F_{\min}}{F_0}+1\right).
\end{equation}
This counting inequality is deterministic. The high-probability escape guarantee is used separately to justify that termination does not occur at an $(\varepsilon,\delta)$-saddle.

For comparison, the discrete-time PWGF analysis of \cite[Theorem~5.5 and Appendix~E.2, equation~(38)]{yamamoto2025hessianguided} uses an episode budget $n_{\mathrm{out}}$ of order $\tilde O(\delta^{-1})$ and a certified decrease $F_0$ at the cubic scale. Their resulting saddle term is $\tilde O(F_{\min}\delta^{-4})$, with regularity constants suppressed.

For WSFN, the revised parameter choices in \eqref{eq:new-def-nout} and \eqref{eq:new-def-f0} have the form
\begin{equation*}
    n_{\mathrm{out}}=\tilde O(1), \quad F_0 = c_F\frac{\beta_\mathrm{sad}^{3/2}}{L_{\operatorname H}^2(\tau_\mathrm{sad}n_{\mathrm{out}})^3},
\end{equation*}
where $c_F>0$ is independent of $\delta$ and may absorb fixed numerical factors such as $\log^2(3/2)$. Substituting $\beta_\mathrm{sad}=b\delta^2$ into \eqref{eq:app-saddle-counting} gives
\begin{align*}
    N_{\mathrm{sad}} \leq \frac{L_{\operatorname H}^2\tau_\mathrm{sad}^3}
    {c_Fb^{3/2}}\frac{n_{\mathrm{out}}^4}{\delta^3}F_{\min} +
    n_{\mathrm{out}} = \tilde O\left(\frac{L_{\operatorname H}^2F_{\min}}{\delta^3} +1\right).
\end{align*}
Note that the one-power improvement comes from episode duration, not from increasing the polynomial scale of the certified decrease.

\subsection{Curvature normalization removes slow amplification}
\label{app:saddle-normalized-amplification}
Consider an exact critical point $\mu^* \in \mathcal{P}_2^{\mathrm{ac}}(\mathbb R^d).$ Then $\operatorname{H}_{\mu^*}=\operatorname{K}_{\mu^*}$, and consider a unit eigenvector $q \in L_{\mu^*}^2$ with $\operatorname{K}_{\mu^*}q=-\gamma q$, where $\gamma>\delta>0$. By Table \ref{tab:critical-point-multipliers}, the frozen linearized Wasserstein gradient descent iteration is
\begin{equation*}
    v^{k+1} = \left(\operatorname{I_{d \times d}}-\tau_{\mathrm{GD}}\operatorname{H}_{\mu^*}\right)v^k, \quad v^0\in L^2_{\mu^*},
\end{equation*}
where $\tau_{\mathrm{GD}}>0$. Here ``frozen'' means that the Hessian is kept fixed at $\mu^*$, not evaluated along the nonlinear trajectory.

Define the signed amplitude of the displacement in the direction $q$ by $a_k:=\langle v^k,q\rangle_{L^2_{\mu^*}}.$ Thus, $a_kq$ is the orthogonal projection of $v^k$ onto $\operatorname{span}\{q\}$, and its norm is $|a_k|$. By self-adjointness of $\operatorname{H}_{\mu^*}$ and the eigenvector identity, we obtain
\begin{align*}
    a_{k+1} = \left\langle \left(\operatorname{I_{d \times d}}-\tau_{\mathrm{GD}}\operatorname{H}_{\mu^*}\right)v^k,
        q \right\rangle_{L^2_{\mu^*}} = \left\langle v^k, \left(\operatorname{I_{d \times d}}-\tau_{\mathrm{GD}}\operatorname{H}_{\mu^*}\right)q\right\rangle_{L^2_{\mu^*}} = (1+\tau_{\mathrm{GD}}\gamma)a_k.
\end{align*}
Iterating this scalar recurrence gives
\begin{equation}
\label{eq:unstable-mode-geometric-growth}
    a_k = (1+\tau_{\mathrm{GD}}\gamma)^k a_0, \quad |a_k|=(1+\tau_{\mathrm{GD}}\gamma)^k|a_0|.
\end{equation}
In particular, if $v^0=a_0q$, then $v^k=a_kq$ for every $k$, so $\|v^k\|_{L^2_{\mu^*}}=|a_k|$.

Fix a target amplitude $r>0$ and suppose $0<|a_0|<r$. Define the amplification time by $ N_{\mathrm{GD}}(r,a_0) := \min\left\{k\in\mathbb N_0: |a_k|\geq r\right\}.$ The threshold $r$ is measured in the reference $L^2_{\mu^*}$ norm. When $v^0=a_0q$, it is the radius of the displacement ball that the linearized trajectory must leave. For a general initial displacement, it is a threshold for the component along $q$.

Using \eqref{eq:unstable-mode-geometric-growth}, $|a_k|\geq r$ if and only if $|a_0|(1+\tau_{\mathrm{GD}}\gamma)^k\geq r.$ Using $\log(1+\tau_{\mathrm{GD}}\gamma)>0$, solving for $k$ gives
\begin{equation*}
    N_{\mathrm{GD}}(r,a_0) = \left\lceil \frac{\log(|a_0|^{-1}r)}{\log(1+\tau_{\mathrm{GD}}\gamma)}
    \right\rceil.
\end{equation*}
The numerator is the logarithm of the required amplification factor, whereas the denominator is the logarithmic growth per iteration. If $a_0=0$, this component remains zero under the frozen dynamics and never reaches a positive threshold. For $\gamma=\Theta(\delta)$ and a fixed admissible $\tau_{\mathrm{GD}}$, we have $\log(1+\tau_{\mathrm{GD}}\gamma) =\Theta(\tau_{\mathrm{GD}}\delta),$ for sufficiently small $\delta,$ and hence
\begin{equation*}
    N_{\mathrm{GD}}(r,a_0) =O\left(\frac{1}{\tau_{\mathrm{GD}}\delta}\log\left(\frac{r}{|a_0|}\right)\right).
\end{equation*}
This isolates the inverse curvature $\delta^{-1}$ slowdown. The formula is exact for the frozen linearized component around $\mu^*$, not for the nonlinear Wasserstein gradient descent trajectory. Transferring this amplification mechanism to the latter requires the tools from \cite{yamamoto2025hessianguided}.

WSFN instead has the frozen multiplier
\begin{equation*}
    a_{k+1} = \left(1+\tau_\mathrm{sad}\frac{\gamma}{\sqrt{\gamma^2+\beta_\mathrm{sad}}}\right)a_k.
\end{equation*}
With $\beta_\mathrm{sad}=b\delta^2$,
\begin{equation}
\label{eq:app-uniform-unstable-multiplier}
    1+\tau_\mathrm{sad} \frac{\gamma}{\sqrt{\gamma^2+\beta_\mathrm{sad}}}
    \geq 1+\frac{\tau_\mathrm{sad}}{\sqrt{1+b}} > 1.
\end{equation}
Thus the amplification takes $O(\log(r/|a_0|))$ steps, without a polynomial factor in $\delta^{-1}$. This rescaling that preserves curvature sign is the mechanism behind saddle-free Newton preconditioning. The nonlinear escape proof is still needed to control Hessian variation and the nonzero gradient at an approximate saddle, as the frozen multiplier alone is not an escape theorem. In particular, an order-one episode stepsize must be justified by local descent, instead of being inferred from
\eqref{eq:app-uniform-unstable-multiplier}; see Lemma \ref{lem:local-wsfn-descent}

\subsection{Why the objective decrease scale is cubic}
\label{app:saddle-cubic-depth}
Preconditioning changes the update direction but does not change the objective. Suppose first that $\mu^*\in \mathcal{P}_2^{\mathrm{ac}}(\mathbb R^d).$ is an exact critical point and $q\in L^2_{\mu^*}$ is a unit vector satisfying $ \langle \operatorname{H}_{\mu^*}q,q\rangle_{L^2_{\mu^*}} = -\gamma,$ where $\gamma > \delta > 0.$ Lemma~\ref{lemma:smoothness-of-F-2} gives
\begin{equation*}
    F((\operatorname{Id}+rq)_\#\mu^*)-F(\mu^*) \leq -\frac{\gamma}{2}r^2 + \frac{L_{\operatorname H}}{6}r^3.
\end{equation*}
The quadratic term and the cubic error balance when $r=\Theta(\frac{\gamma}{L_{\operatorname H}})$. For example, taking $r=\frac{\gamma}{L_{\operatorname H}}$ gives
\begin{equation*}
    F(\mu^*) - F((\operatorname{Id}+rq)_\#\mu^*) \geq \frac{\gamma^3}{3L_{\operatorname H}^2}.
\end{equation*}
Thus, at the weakest relevant curvature scale $\gamma=\Theta(\delta)$, the objective depth scale is asymptotically equivalent to $ \frac{\delta^3}{L_{\operatorname H}^2}.$ The same scale survives at an approximate saddle. If $\operatorname{K}_\mu q=\lambda q$ with $\lambda<-\delta$, $\|q\|_{L^2_\mu}=1$, and $ \|\operatorname{M}_\mu\|_{\mathrm{op}} \leq R_F\varepsilon \leq \frac{\delta}{2},$ then $\langle \operatorname{H}_\mu q,q\rangle_{L^2_\mu}\leq-\frac{\delta}{2}$. Choose the sign of $q$ so that
$\langle\nabla_\mu F(\mu),q\rangle\leq0$. The same cubic Taylor estimate, with
$r=\frac{\delta}{2L_{\operatorname H}}$, gives
\begin{equation*}
    F(\mu) - F((\operatorname{Id}+rq)_\#\mu) > \frac{\delta^3}{24L_{\operatorname H}^2}.
\end{equation*}
The additional small gradient restrictions in the WSFN proof also control the sensitivity of the varying preconditioner.

\subsection{Sharpness of the cubic decrease scale}
\label{app:saddle-cubic-example}
The cubic scale is not only a feature of the Taylor upper bound. The following example shows that the entire available objective decrease can have precisely this order.

\begin{proposition}[A rank-one example with cubic objective depth]
\label{prop:cubic-depth-example}
Fix $L>0$ and $\bar\gamma>0$. For $\gamma\in(0,\bar\gamma]$, define $ f_\gamma(x) := \frac{\gamma^3}{L^2}[\cos(\frac{Lx}{\gamma})-1],$ and $F_\gamma(\mu) := f_\gamma(m(\mu)),$ with $m(\mu):= \int_{\mathbb R}x\mu(\mathrm dx).$ Then $F_\gamma$ is bounded below and Wasserstein regular. Its multiplication Hessian part vanishes, its integral Hessian part is rank at most one, Assumption \ref{assumption:smooth-F} holds uniformly in $\gamma$, and Assumption \ref{assumption:Hessian-lipschitz} is satisfied with constant $L$. At every centered $\mu^*\in\mathcal P_2^{\mathrm{ac}}(\mathbb R)$, $\nabla_\mu F_\gamma(\mu^*,\cdot)=0,$ $\operatorname{K}_{\mu^*}q=-\gamma q,$ for $q(x)\equiv1,$ and
\begin{equation}
\label{eq:app-exact-cubic-depth}
    F_\gamma(\mu^*) - \inf_{\nu\in\mathcal P_2(\mathbb R)}F_\gamma(\nu) = \frac{2\gamma^3}{L^2}.
\end{equation}
\end{proposition}

\begin{proof}
Direct differentiation gives $f_\gamma'(x) = -\frac{\gamma^2}{L}\sin\left(\frac{Lx}{\gamma}\right),$ $f_\gamma''(x)=-\gamma\cos\left(\frac{Lx}{\gamma}\right),$ $f_\gamma'''(x)= L\sin\left(\frac{Lx}{\gamma}\right).$ Thus $\|f_\gamma''\|_\infty\leq\bar\gamma$ and $f_\gamma''$ is $L$-Lipschitz, uniformly in $\gamma$. The Wasserstein derivatives are $\nabla_\mu F_\gamma(\mu,x) = f_\gamma'(m(\mu)),$ $\nabla\nabla_\mu F_\gamma(\mu,x)=0,$ and $\nabla_\mu^2F_\gamma(\mu,x,y) = f_\gamma''(m(\mu)).$ These identities follow by differentiating the mean along transport curves. They give
\begin{equation*}
    \operatorname{M}_\mu=0, \quad \operatorname{K}_\mu[v](x) = f_\gamma''(m(\mu)) \int_{\mathbb R}v(y)\mu(\mathrm dy).
\end{equation*}
In particular, 
\begin{equation*}
    \|\operatorname{K}_\mu\|_{\mathrm{op}} = |f_\gamma''(m(\mu))| \leq \bar\gamma.
\end{equation*}
For any coupling $\pi\in\Pi(\mu,\nu)$,
\begin{equation*}
    |m(\nu)-m(\mu)| = \left|\int(y-x)\pi(\mathrm dx,\mathrm dy)\right| \leq
    \left(\int|y-x|^2\pi(\mathrm dx,\mathrm dy)\right)^{1/2}.
\end{equation*}
Since the kernel is constant in its spatial variables, its $L^2_{\pi\otimes\pi}$ difference is bounded by $L$ times the right-hand side. This proves the claimed coupling Lipschitz estimate. The multiplication part
contributes zero.

At a centered measure, $f_\gamma'(0)=0$ and $f_\gamma''(0)=-\gamma$. Since $\|q\|_{L^2_{\mu^*}}=1$ for $q\equiv1$, the asserted eigenvalue identity follows. Finally, $f_\gamma(0)=0$ and $\inf_x f_\gamma(x)=-\frac{2\gamma^3}{L^2}.$ This infimum is attained by translating $\mu^*$ to have mean $\frac{\pi\gamma}{L}$, which proves \eqref{eq:app-exact-cubic-depth}.
\end{proof}
Taking $\gamma=2\delta$ gives a point with curvature strictly below $-\delta$ whose total available decrease is only $\frac{16\delta^3}{L^2}$. Hence, uniformly over this regularity class, one cannot guarantee a decrease of order $\omega(\delta^3)$ from every such point.

\subsection{Why smaller regularization does not improve this balance}
\label{app:saddle-beta-balance}

The normalized curvature is bounded above by one. Thus, once $\beta$ is comparable to $\delta^2$, decreasing $\beta$ further cannot produce another polynomial improvement in the unstable multiplier. At the same time, the estimates in Lemma \ref{lem:regularized-wsfn-spectral-maps} deteriorate as $\beta$ decreases. This trade-off can be seen by formally optimizing the leading expression supplied by the escape proof. Fix $\tau_{\mathrm{sad}}\in(0,\frac{1}{2}]$ independently of $\beta$ and $\delta$. Suppressing logarithms, the balances are $\tau_\mathrm{sad}n_{\mathrm{out}} = \tilde O(\tilde\delta^{-1}),$ and $F_0=\tilde \Omega(\frac{\beta^{3/2}}{L_{\operatorname H}^2(\tau_\mathrm{sad}n_{\mathrm{out}})^3}),$ with $\tilde\delta=\frac{\delta}{\sqrt{\delta^2+\beta}}.$ The episode counting bound from the proof of Theorem \ref{thm:polynomial-time-sosp} is
\begin{equation*}
    N_{\mathrm{sad}} \leq F_{\min}\frac{n_{\mathrm{out}}}{F_0} + n_{\mathrm{out}}.
\end{equation*}
Since $F_0>0$, we have $\frac{1}{F_0} = \tilde O\left(\frac{L_{\operatorname H}^2(\tau_\mathrm{sad}n_{\mathrm{out}})^3}{\beta^{3/2}}\right).$ Using $n_{\mathrm{out}}$, we therefore obtain
\begin{align*}
    \frac{n_{\mathrm{out}}}{F_0} = \tilde O\left(\frac{\tau_\mathrm{sad}L_{\operatorname H}^2}{\beta^{3/2}}n_{\mathrm{out}}^4\right) = \tilde O\left(\frac{L_{\operatorname H}^2}{\beta^{3/2}}\tilde\delta^{-4}\right).
\end{align*}
By the definition of $\tilde\delta$,
\begin{equation*}
    \tilde\delta^{-4} = \left(\frac{\sqrt{\delta^2+\beta}}{\delta}\right)^4
    = \frac{(\delta^2+\beta)^2}{\delta^4}.
\end{equation*}
Thus,
\begin{equation*}
    \frac{n_{\mathrm{out}}}{F_0} = \tilde O\left(L_{\operatorname H}^2\frac{(\delta^2+\beta)^2}{\beta^{3/2}\delta^4}\right).
\end{equation*}
Suppressing logarithmic factors and fixed multiplicative constants, this motivates the cost proxy
\begin{equation}
\label{eq:app-beta-cost-proxy}
    \mathcal C(\beta,\delta) := L_{\operatorname H}^2\frac{(\delta^2+\beta)^2}
    {\beta^{3/2}\delta^4}.
\end{equation}
This proxy describes the leading coefficient of $F_{\min}$ in the upper bound. The additive cost of the terminal episode remains bounded by
$n_{\mathrm{out}} =\tilde O(\tilde\delta^{-1})$. To optimize the proxy over $\beta$ for fixed $\delta>0$, introduce the dimensionless ratio $ z:=\frac{\beta}{\delta^2}>0,$ so that $\beta=z\delta^2.$ Then $(\delta^2+\beta)^2 = \delta^4(1+z)^2,$ $\beta^{3/2}=z^{3/2}\delta^3.$ Substituting these identities into \eqref{eq:app-beta-cost-proxy} gives
\begin{align*}
    \mathcal C(\beta,\delta) = L_{\operatorname H}^2\frac{\delta^4(1+z)^2}{z^{3/2}\delta^3\delta^4} = \frac{L_{\operatorname H}^2}{\delta^3}h(z),
    \quad h(z):=\frac{(1+z)^2}{z^{3/2}}.
\end{align*}
Since $h(z)>0$, logarithmic differentiation gives $\log h(z)=2\log(1+z)-\frac32\log z,$ and hence
\begin{align*}
    \frac{h'(z)}{h(z)} = \frac{2}{1+z}-\frac{3}{2z} = \frac{4z-3(1+z)}{2z(1+z)} = \frac{z-3}{2z(1+z)}.
\end{align*}
The denominator is positive for $z>0$. Therefore, $h'(z)<0$ on $(0,3)$ and $h'(z)>0$ on $(3,\infty)$, so $z=3$ is the unique global minimizer. Its value is $h(3)=\frac{16}{3\sqrt3}.$ Equivalently, $\operatorname{argmin}_{\beta>0} \mathcal C(\beta,\delta)= 3\delta^2,$ and $\min_{\beta>0}\mathcal C(\beta,\delta) =\frac{16}{3\sqrt3}\frac{L_{\operatorname H}^2}{\delta^3}.$ Thus the leading order proxy selects $\beta=\Theta(\delta^2)$ and retains the exponent $\delta^{-3}$. Taking $\beta$ much smaller or much larger than $\delta^2$ worsens this proxy, since $ h(z)\sim z^{-3/2}$ as $z\downarrow0,$ and $h(z)\sim z^{1/2},$ as $z\to\infty.$

The value $\beta=3\delta^2$ minimizes only the displayed proxy with $\tau_\mathrm{sad}$ held fixed. Logarithmic factors and the remaining admissibility conditions can change the preferred constant. Moreover, minimizing an upper bound proxy does not establish a lower bound on the actual iteration complexity.

\subsection{What is sharp and what might improve further}
\label{app:saddle-scope}

Proposition~\ref{prop:cubic-depth-example} establishes sharpness of the uniform objective decrease scale, not an $\Omega(\delta^{-3})$ iteration lower bound for WSFN or for arbitrary second-order methods. In that example, $F_{\min}$ is of order $\delta^3$, so it does not force $\delta^{-3}$ successive episodes. A complexity lower bound with a fixed initial objective gap $F_\mathrm{min}$ would need a separate construction and an argument excluding faster ways of reaching approximate second-order stationarity.

Stronger global information could improve the episode count without changing the local escape rate. Such information is additional to qualitative benignity since the absence of nonglobal local minima does not itself bound the number of saddle neighborhoods encountered by a trajectory.

Higher-order regularity and higher-order models provide another route. For Euclidean unconstrained optimization, \cite{cartis2020concise} show that adaptive regularization with Taylor models of order $p\geq2$, for a lower bounded objective with Lipschitz $p$-th derivative, attains first and second-order tolerances $(\varepsilon,\delta)$ with evaluation complexity $O\left(\max\left\{\varepsilon^{-(p+1)/p},\delta^{-(p+1)/(p-1)}\right\}\right),$ with the other problem and algorithm parameters fixed.
The second-order exponent is $3$ for $p=2$ and $2$ for $p=3$. These are Euclidean evaluation complexity results under their own model solution requirements, and do not imply consequences of the present Wasserstein analysis. They nevertheless show why the cubic exponent should not be presented as a universal barrier independent of smoothness, oracle information, or algorithm design. Indeed, the example above has
$\|f_\gamma^{(4)}\|_\infty=\frac{L^2}{\gamma}$, so its third derivative is not Lipschitz with a constant uniform in $\gamma$. It therefore does not satisfy the stronger uniform regularity needed for the $p=3$ comparison.

Finally, the improved saddle contribution must be distinguished from the leading total complexity. Theorem~\ref{thm:polynomial-time-sosp} contains a large gradient contribution proportional to $F_{\min}\varepsilon^{-2}$, while
condition~\ref{eq:new-epsilon-condition} requires $\varepsilon\le \frac{b\delta^2}{L_{\mathrm H}}$. Hence, the total upper bound is $\tilde O(F_{\min}\varepsilon^{-2}+1)$ throughout the admissible regime. At the largest admissible polynomial scale $\varepsilon=\tilde\Theta(\delta^2)$, this becomes $\tilde O(F_{\min}\delta^{-4})$ for fixed $F_{\min}>0$.

Therefore, WSFN improves the certified cost of saddle episodes without improving the leading polynomial order of the present total upper bound. Improving that leading order requires a sharper analysis of the large gradient phase or a weaker restriction linking $\varepsilon$ and $\delta$. A further reduction of the episode contribution alone would not change the leading order of the current guarantee.

\begin{proposition}[Limitation of the large gradient phase]
\label{prop:large-gradient-obstruction}
Under the restriction $\tau_{\mathrm{grad}}\leq C_{\operatorname H}^{-1}\sqrt{\beta_{\mathrm{grad}}}$, the large gradient phase admits $\Omega(\varepsilon^{-2})$ worst case examples. These examples satisfy Assumptions~\ref{assumption_regularity_wg}, \ref{assumption:smooth-F}, and~\ref{assumption:Hessian-lipschitz} with constants independent of $\varepsilon$, and have initial objective gap $F_{\min}$ bounded above and below by positive constants.
\end{proposition}
\begin{proof}
Choose a nonnegative smooth function $q\in C_c^\infty((-1,0))$ with $\|q\|_\infty=1$, and set $a:=\int_{\mathbb R}q(t)\mathrm dt>0$. For $0<\varepsilon<\frac{a}{4}$, define $f_\varepsilon$ by
\begin{equation*}
    f_\varepsilon(0)=0, \quad f_\varepsilon'(x) = 2\varepsilon-\int_x^0 q(t)\mathrm dt.
\end{equation*}
Then $f_\varepsilon''=q$, so the Hessian bound and its Lipschitz constant are independent of $\varepsilon$. Moreover, $f_\varepsilon$ is convex and bounded below, with a minimizer in $(-1,0)$. For $x\geq0$, we have $f_\varepsilon(x)=2\varepsilon x$.

Define the functional on $\mathcal P_2(\mathbb R)$ by $F_\varepsilon(\mu):=f_\varepsilon(m_\mu),$ with $m_\mu:=\int_{\mathbb R}x\mathrm d\mu(x).$ Its Wasserstein derivatives satisfy $ \nabla_\mu F_\varepsilon(\mu,x) =f_\varepsilon'(m_\mu),$ $\operatorname M_\mu=0,$ $(\operatorname K_\mu v)(x) =q(m_\mu)\int_{\mathbb R}v\mathrm d\mu.$ Hence $\|\operatorname K_\mu\|_{\mathrm{HS}}\leq1$, and Assumptions \ref{assumption_regularity_wg}, \ref{assumption:smooth-F} are automatic. Assumption \ref{assumption:Hessian-lipschitz} follows from
\begin{equation*}
    |q(m_\nu)-q(m_\mu)| \leq \|q'\|_\infty |m_\nu-m_\mu| \leq \|q'\|_\infty d_\gamma, 
\end{equation*}
where $\gamma\in\Pi(\mu,\nu).$ Thus the required regularity constants are uniform in $\varepsilon$.

Initialize with an absolutely continuous measure of mean $m_0=(2\varepsilon)^{-1}$, for example $\mu^0=\mathcal N(m_0,1)$. Since $\min f_\varepsilon\in[-2\varepsilon,0]$,
\begin{equation*}
     1 \leq F_\varepsilon(\mu^0)-\inf_\mu F_\varepsilon(\mu) \leq 1+2\varepsilon.
\end{equation*}
While $m_k\geq0$, the gradient norm equals $2\varepsilon>\varepsilon$ and $\operatorname H_{\mu^k}=0$. Consequently, only large gradient steps are performed, and
\begin{equation*}
     m_{k+1} = m_k-\frac{2\tau_{\mathrm{grad}}\varepsilon}{\sqrt{\beta_{\mathrm{grad}}}},
\end{equation*}
so that
\begin{equation*}
    m_k-m_{k+1} \leq \frac{2\varepsilon}{C_{\operatorname H}}.
\end{equation*}
Reaching a negative mean therefore requires at least $\frac{C_{\operatorname H}}{4\varepsilon^2}$ steps. The iterates are translations of the initial measure, so absolute continuity is preserved.
\end{proof}
\end{document}